\pdfoutput=1
\RequirePackage{ifpdf}
\ifpdf 
\documentclass[pdftex]{sigma}
\else
\documentclass{sigma}
\fi

\usepackage{enumitem}
\usepackage{amsxtra}
\usepackage{mathrsfs}
\usepackage[all]{xy}
\input{glyphtounicode.tex}
\numberwithin{equation}{section}

\newtheorem{Theorem}{Theorem}[section]
\newtheorem{Corollary}[Theorem]{Corollary}
\newtheorem{Lemma}[Theorem]{Lemma}
\newtheorem{Proposition}[Theorem]{Proposition}
\newtheorem{Assumption}[Theorem]{Assumption}
\newtheorem{Caution}[Theorem]{Caution}
\newtheorem{Conjecture}[Theorem]{Conjecture}

{\theoremstyle{definition}
\newtheorem{Definition}[Theorem]{Definition}

\newtheorem{Notation}[Theorem]{Notation}
\newtheorem{Example}[Theorem]{Example}
\newtheorem{Remark}[Theorem]{Remark} }

\makeatletter
\newcommand{\proofstep}[1]{%
 \par 
 \addvspace{\medskipamount}
 \textit{#1\@addpunct{.}}\enspace\ignorespaces
}
\makeatother

\makeatletter
\def\namedlabel#1#2{\begingroup
 #2%
 \def\@currentlabel{#2}%
 \phantomsection\label{#1}\endgroup
}
\makeatother

\newcommand{\C}{\mathbb{C}}
\newcommand{\Z}{\mathbb{Z}}
\newcommand{\N}{\mathbb{N}}
\newcommand{\Q}{\mathbb{Q}}
\newcommand{\R}{\mathbb{R}}
\newcommand{\T}{\mathbb{T}}
\newcommand{\OO}{\mathbb{O}}
\newcommand{\bS}{\mathbb{S}}
\newcommand{\PP}{\mathbb{P}}
\newcommand{\LL}{\mathbb{L}}
\newcommand{\D}{\mathbb{D}}
\newcommand{\tD}{\widetilde{\D}}
\newcommand{\Dsm}{\D_{\rm sm}}
\newcommand{\tDsm}{\tD_{\rm sm}}
\newcommand{\intDsmtimes}{\D_{\rm sm}^{\times \circ}}

\newcommand{\cA}{\mathcal{A}}

\newcommand{\cB}{\mathcal{B}}
\newcommand{\cO}{\mathcal{O}}
\newcommand{\hcO}{\widehat{\cO}}
\newcommand{\cOan}{\cO^{\rm an}}
\newcommand{\cU}{\mathcal{U}}
\newcommand{\cUss}{\cU^{\rm ss}}
\newcommand{\cV}{\mathcal{V}}
\newcommand{\cVsm}{\cV^{\rm sm}}
\newcommand{\cVsmss}{\cV^{\rm sm,ss}}
\newcommand{\cVss}{\cV^{\rm ss}}
\newcommand{\Vss}{V^{\rm ss}}
\newcommand{\cC}{\mathcal{C}}
\newcommand{\cS}{\mathcal{S}}
\newcommand{\cE}{\mathcal{E}}
\newcommand{\cF}{\mathcal{F}}

\newcommand{\cD}{\mathcal{D}}
\newcommand{\cG}{\mathcal{G}}
\newcommand{\cM}{\mathcal{M}}

\newcommand{\cMnd}{\cM^{\rm nd}}
\newcommand{\cMsm}{\cM^{\rm sm}}
\newcommand{\cMtimes}{\cM^{\times}}
\newcommand{\cMsmlf}{\cM^{\rm sm,lf}}
\newcommand{\cMsmss}{\cM^{\rm sm,ss}}
\newcommand{\cMsmtimes}{\cM^{\rm sm,\times}}
\newcommand{\cMsmnd}{\cM^{\rm sm,nd}}
\newcommand{\hcM}{\widehat{\cM}}
\newcommand{\cR}{\mathcal{R}}
\newcommand{\cY}{\mathcal{Y}}
\newcommand{\cYsm}{\cY^{\rm sm}}
\newcommand{\cYtimes}{\cY^{\times}}

\newcommand{\cYsmtimes}{\cY^{\rm sm,\times}}

\newcommand{\sfD}{\mathsf{D}}
\newcommand{\sfG}{\mathsf{G}}

\newcommand{\scrS}{\mathscr{S}}
\newcommand{\scrD}{\mathscr{D}}
\newcommand{\scrG}{\mathscr{G}}
\newcommand{\scrA}{\mathscr{A}}
\newcommand{\scrR}{\mathscr{R}}
\newcommand{\scrRan}{\scrR^{\rm an}}
\newcommand{\scrT}{\mathscr{T}}
\newcommand{\scrC}{\mathscr{C}}
\newcommand{\scrQ}{\mathscr{Q}}
\newcommand{\scrQan}{\scrQ^{\rm an}}

\newcommand{\hS}{{\widehat{S}}}
\newcommand{\hxi}{{\hat{\xi}}}
\newcommand{\hb}{{\hat{b}}}
\newcommand{\chb}{\check{b}}
\newcommand{\hR}{{\widehat{R}}}
\newcommand{\hSigma}{\widehat{\Sigma}}
\newcommand{\hPhi}{\widehat{\Phi}}
\newcommand{\hU}{{\widehat{U}}}
\newcommand{\hGamma}{\widehat{\Gamma}}
\newcommand{\hB}{\widehat{B}}
\newcommand{\hQ}{\widehat{Q}}

\newcommand{\tC}{\widetilde{C}}
\newcommand{\tCC}{{\widetilde{\C}}}
\newcommand{\tdaleth}{\widetilde{\daleth}}
\newcommand{\tXi}{\widetilde{\Xi}}
\newcommand{\tcM}{\widetilde{\cM}}
\newcommand{\tcMsm}{{\tcM}^{\rm sm}}
\newcommand{\tzero}{\tilde{0}}
\newcommand{\tpr}{\widetilde{\pr}}
\newcommand{\tx}{\tilde{x}}
\newcommand{\ttau}{\tilde{\tau}}
\newcommand{\tnabla}{\widetilde{\nabla}}
\newcommand{\tPhi}{\widetilde{\Phi}}
\newcommand{\tU}{\widetilde{U}}
\newcommand{\tV}{\widetilde{V}}
\newcommand{\tR}{\widetilde{R}}
\newcommand{\ovvarsigma}{\overline{\varsigma}}

\newcommand{\bx}{\mathbf{x}}
\newcommand{\bs}{\mathbf{s}}
\newcommand{\bN}{\mathbf{N}}
\newcommand{\bM}{\mathbf{M}}
\newcommand{\bSigma}{\mathbf{\Sigma}}
\newcommand{\hbSigma}{{\widehat{\bSigma}}}
\newcommand{\by}{\mathbf{y}}
\newcommand{\bD}{\mathbf{D}}
\newcommand{\bw}{\mathbf{w}}
\newcommand{\bU}{\mathbf{U}}
\newcommand{\bv}{\mathbf{v}}
\newcommand{\bu}{\mathbf{u}}
\newcommand{\bH}{\mathbf{H}}

\newcommand{\frs}{\mathfrak{s}}
\newcommand{\frX}{\mathfrak{X}}
\newcommand{\hfrX}{\widehat{\frX}}
\newcommand{\frD}{\mathfrak{D}}
\newcommand{\frakm}{\mathfrak{m}}
\newcommand{\frakman}{\frakm^{\rm an}}

\newcommand{\frB}{\mathfrak{B}}

\newcommand{\frS}{\mathfrak{S}}

\newcommand{\ttq}{\mathtt{q}}
\newcommand{\ttp}{\mathtt{p}}
\newcommand{\tti}{\mathtt{i}}
\newcommand{\ttj}{\mathtt{j}}

\newcommand{\Cr}{\operatorname{Cr}}
\newcommand{\Crsm}{\Cr^{\rm sm}}
\newcommand{\Crss}{\Cr^{\rm ss}}
\newcommand{\vol}{\operatorname{vol}}
\newcommand{\crit}{\operatorname{cr}}
\newcommand{\Lef}{\operatorname{\mathsf{Lef}}}
\newcommand{\Lefsm}{\Lef^{\rm sm}}
\newcommand{\mir}{\operatorname{mir}}
\newcommand{\Mir}{\operatorname{Mir}}
\newcommand{\Miran}{\Mir^{\rm an}}
\newcommand{\hMir}{\widehat{\Mir}}
\newcommand{\slide}{\operatorname{\mathtt{s}}}
\newcommand{\Spec}{\operatorname{Spec}}
\newcommand{\Spf}{\operatorname{Spf}}
\newcommand{\id}{\operatorname{id}}
\newcommand{\ev}{\operatorname{ev}}
\newcommand{\pt}{{\operatorname{pt}}}
\newcommand{\Frac}{\operatorname{Frac}}
\newcommand{\Lie}{\operatorname{Lie}}
\newcommand{\Cone}{\operatorname{Cone}}
\newcommand{\PD}{\operatorname{PD}}
\newcommand{\Hom}{\operatorname{Hom}}
\newcommand{\Ext}{\operatorname{Ext}}
\newcommand{\End}{\operatorname{End}}
\newcommand{\rank}{\operatorname{rank}}
\newcommand{\Sym}{\operatorname{Sym}}
\newcommand{\ch}{\operatorname{ch}}
\newcommand{\tch}{\widetilde{\ch}}
\newcommand{\Ch}{\operatorname{Ch}}
\newcommand{\Ker}{\operatorname{Ker}}
\newcommand{\Cok}{\operatorname{Cok}}
\newcommand{\Image}{\operatorname{Im}}
\newcommand{\Pic}{\operatorname{Pic}}
\newcommand{\Aut}{\operatorname{Aut}}

\newcommand{\age}{\operatorname{age}}
\newcommand{\Bx}{\operatorname{Box}}
\newcommand{\CR}{{\rm CR}}
\newcommand{\inv}{\operatorname{inv}}
\newcommand{\Bri}{\operatorname{Bri}}
\newcommand{\Briequiv}{\operatorname{Bri_{\T}}}
\newcommand{\Briequivan}{\overline{\operatorname{Bri_{\T}^{an}}}}
\newcommand{\Brian}{\overline{\operatorname{{Bri}^{an}}}}
\newcommand{\Brism}{\Bri^{\rm sm}}
\newcommand{\Loc}{\operatorname{Loc}}
\newcommand{\pr}{\operatorname{pr}}

\newcommand{\Fan}{\operatorname{\mathfrak{Fan}}}

\newcommand{\CPL}{\operatorname{CPL}}

\newcommand{\cpl}{\operatorname{cpl}}
\newcommand{\PLZ}{\operatorname{PL}_\Z}
\newcommand{\plZ}{\operatorname{pl}_\Z}
\newcommand{\Bl}{\operatorname{Bl}}
\newcommand{\QDM}{\operatorname{QDM}}
\newcommand{\QDMan}{\operatorname{QDM}^{\rm an}}
\newcommand{\Asymp}{\operatorname{Asym}}

\newcommand{\std}{{\rm std}}
\newcommand{\diag}{\operatorname{diag}}
\newcommand{\ori}{\operatorname{\mathsf{ori}}}
\newcommand{\Fockan}{\operatorname{\mathfrak{AFock}}}
\newcommand{\NE}{\operatorname{NE}}
\newcommand{\hNE}{\operatorname{\widehat{NE}}}

\newcommand{\OEf}{\operatorname{OE}}
\newcommand{\hOEf}{\operatorname{\widehat{OE}}}
\newcommand{\Gr}{\operatorname{Gr}}
\newcommand{\gr}{\operatorname{gr}}
\newcommand{\grad}{\operatorname{grad}}

\newcommand{\unit}{\boldsymbol{1}}
\newcommand{\Laa}{\boldsymbol{\Lambda}}

\newcommand{\iu}{\boldsymbol{\mathrm{i}}}
\newcommand{\st}{{\rm st}}

\def\corr#1{\left\langle#1 \right\rangle}
\def\parfrac#1#2{\frac{\partial #1}{\partial #2}}
\def\floor#1{\lfloor #1 \rfloor}

\begin{document}
\allowdisplaybreaks

\newcommand{\arXivNumber}{1906.00801}

\renewcommand{\thefootnote}{}

\renewcommand{\PaperNumber}{032}

\FirstPageHeading

\ShortArticleName{Global Mirrors and Discrepant Transformations for Toric Deligne--Mumford Stacks}

\ArticleName{Global Mirrors and Discrepant Transformations\\ for Toric Deligne--Mumford Stacks\footnote{This paper is a~contribution to the Special Issue on Integrability, Geometry, Moduli in honor of Motohico Mulase for his 65th birthday. The full collection is available at \href{https://www.emis.de/journals/SIGMA/Mulase.html}{https://www.emis.de/journals/SIGMA/Mulase.html}}}

\Author{Hiroshi IRITANI}
\AuthorNameForHeading{H.~Iritani}

\Address{Department of Mathematics, Graduate School of Science, Kyoto University,\\ Kitashirakawa-Oiwake-cho, Sakyo-ku, Kyoto, 606-8502, Japan}

\Email{\href{mailto:iritani@math.kyoto-u.ac.jp}{iritani@math.kyoto-u.ac.jp}}
\URLaddress{\url{https://www.math.kyoto-u.ac.jp/~iritani/}}

\ArticleDates{Received June 13, 2019, in final form March 29, 2020; Published online April 22, 2020}

\Abstract{We introduce a global Landau--Ginzburg model which is mirror to several toric Deligne--Mumford stacks and describe the change of the Gromov--Witten theories under discrepant transformations. We prove a formal decomposition of the quantum coho\-mo\-logy D-modules (and of the all-genus Gromov--Witten potentials) under a discrepant toric wall-crossing. In the case of weighted blowups of weak-Fano compact toric stacks along toric centres, we show that an analytic lift of the formal decomposition corresponds, via the $\hGamma$-integral structure, to an Orlov-type semiorthogonal decomposition of topological $K$-groups. We state a conjectural functoriality of Gromov--Witten theories under discrepant transformations in terms of a Riemann--Hilbert problem.}


\Classification{14N35; 14J33; 53D45}

\vspace{-2mm}

{\small \renewcommand{\baselinestretch}{0.7} \tableofcontents

}

\renewcommand{\thefootnote}{\arabic{footnote}}
\setcounter{footnote}{0}

\section{Introduction}
It is a very interesting problem to study how
Gromov--Witten invariants (or quantum cohomo\-lo\-gy) change
under birational transformations.
When the birational transformation is \emph{crepant}
(or a~\emph{$K$-equivalence}), a conjecture of
Yongbin Ruan \cite{Ruan:crepant} says
that the quantum cohomology of $K$-equivalent spaces
should be related to each other
by analytic continuation in quantum parameters
(see, e.g., \cite{Bryan-Graber,Coates-Iritani:Fock,
CIT:wallcrossing,Lee-Lin-Wang:flops, AMLi-Ruan}).
In this paper, we are concerned with \emph{discrepant}
transformations, or more precisely, birational maps
$\varphi \colon \frX_+ \dasharrow \frX_-$
between smooth Deligne--Mumford stacks
such that there exist
projective birational morphisms
$f_\pm \colon \hfrX \to \frX_\pm$ satisfying
$f_- = \varphi \circ f_+$ and
that $f_+^*K_{\frX_+} - f_-^*K_{\frX_-}$ is a non-zero
effective divisor.
In this case, we write ``$K_{\frX_+} > K_{\frX_-}$'' by a slight
abuse of notation
\[
\xymatrix{
& \hfrX \ar[ld]_{f_+}\ar[rd]^{f_-} & \\
\frX_+ \ar@{-->}[rr]^{\varphi}& & \frX_-.
}
\]
This includes the case where $\varphi$ is a
blowup along a smooth subvariey.
In this case, we do not expect that
the quantum cohomology of $\frX_+$ and $\frX_-$
are related by analytic continuation
because their ranks are different.
Instead, we expect that the quantum cohomology of
$\frX_+$ would contain the quantum cohomology of $\frX_-$
as a \emph{direct summand} after analytic continuation.
This is analogous to the conjecture
that $D^b(\frX_+)$ contains $D^b(\frX_-)$ as a
\emph{semiorthogonal summand}
\cite{BFK:VGIT, Bonda-Orlov:SOD,Kawamata:log_crepant,Kawamata:Dcat_birat}, where $D^b(\frX_\pm)$ denotes the derived category
of coherent sheaves on $\frX_\pm$.
In this paper, we describe a decomposition of quantum
cohomology D-modules (and of all genus Gromov--Witten potentials)
for toric Deligne--Mumford stacks under
discrepant transformations.
Moreover, we show in special cases that
the decomposition of quantum cohomology D-modules
is induced by
a semiorthogonal decomposition of derived categories
(or more precisely of topological $K$-groups)
via the \emph{$\hGamma$-integral structure} \cite{Iritani:Integral,
KKP:Hodge}.
We also formulate a general conjecture in view of
our results in the toric case.

\subsection{Quantum cohomology D-modules}
Our central objects of study are
\emph{quantum $($cohomology$)$ D-modules}.
Let $\frX$ be a smooth Deligne--Mumford stack.
The genus-zero Gromov--Witten invariants define a family
of (super)commu\-ta\-ti\-ve product structures $\star_\tau$
on the orbifold cohomology group $H^*_\CR(\frX)$
parametrized by $\tau \in H^*_\CR(\frX)$; this is called
quantum cohomology.
The product $\star_\tau$ then defines a
meromorphic flat connection~$\nabla$ called the \emph{quantum connection}
(or \emph{Dubrovin connection})
on the trivial $H^*_\CR(\frX)$-bundle over~$H^*_\CR(\frX)
\times \C$. It is given by the formulae
\begin{gather*}
\nabla_{\parfrac{}{\tau^i}} =\parfrac{}{\tau^i} + \frac{1}{z} \phi_i\star_\tau, \\
\nabla_{z\parfrac{}{z}} = z \parfrac{}{z} - \frac{1}{z} E\star_\tau+ \mu,
\end{gather*}
where $(\tau,z)$ represents a point on the base $H^*_\CR(\frX) \times \C$,
$\{\tau^i\}$ are linear co-ordinates on $H^*_\CR(\frX)$
dual to a basis $\{\phi_i\}$, $E$ is the so-called Euler vector
field, and $\mu$ is the (constant) grading operator.
This connection is self-dual with respect to the pairing $P$
between the fibres at $(\tau,-z)$ and $(\tau,z)$
induced by the orbifold Poincar\'e pairing.
The \emph{quantum D-module}\footnote
{Here we assume the convergence and the
analyticity of quantum cohomology (which are true
for toric DM stacks); the superscript ``an'' means
analytic.} $\QDMan(\frX)$ of $\frX$
is, roughly speaking, the module of sections of this
vector bundle equipped with the meromorphic flat
connection $\nabla$ and the pairing $P$
(see Sections~\ref{subsec:QDM} and~\ref{subsec:Hukuhara-Turrittin} for the details).
We also obtain the \emph{formal} quantum D-module
(or more precisely, the quantum D-module \emph{completed in} $z$)
by restricting $\QDMan(\frX)$ to the formal neighbourhood of $z=0$:
\[
\overline{\QDMan}(\frX)
= \QDMan(\frX) \otimes_{\cOan[z]}\cOan[\![z]\!].
\]
When a torus $\T$ acts on $\frX$, we can also define
the \emph{$\T$-equivariant} quantum D-module
$\QDMan_\T(\frX)$ and its formal version
$\overline{\QDMan_\T}(\frX)$; they are
deformation of the non-equivariant quantum D-modules.

\subsection{Global Landau--Ginzburg models and toric wall-crossings}
A smooth semiprojective toric Deligne--Mumford (DM) stack $\frX_\bSigma$
(in the sense of Borisov--Chen--Smith~\cite{BCS})
can be
defined as the geometric invariant theory (GIT)
quotient of a vector space~$\C^S$
by a torus $\LL_{\C^\times} = \LL\otimes \C^\times$,
where $\LL$ is a free $\Z$-module of finite rank.
The torus $\LL_{\C^\times}$ acts on $\C^S$ via
a group homomorphism $\LL_{\C^\times}
\to (\C^\times)^S$, and by dualizing it,
we get the family $\pr\colon (\C^\times)^S
\to \LL^\star \otimes \C^\times$ of tori
equipped with the function
$F = \sum\limits_{b\in S} u_b$ on $(\C^\times)^S$,
where $u_b$ is the $b$th co-ordinate on $(\C^\times)^S$
\[
\xymatrix{
(\C^\times)^S \ar[rr]^{F} \ar[d]^{\pr} &
& \C.
\\
\LL^\star \otimes \C^\times & &
}
\]
This is the Landau--Ginzburg (LG) model mirror to $\frX_\bSigma$ introduced
by Givental \cite{Givental:ICM}.
Using the secondary fan of Gelfand--Kapranov--Zelevinsky
\cite{GKZ:discriminants},
we can partially compactify this LG model to an LG model of the form
(see Section~\ref{subsec:LG})
\[
\xymatrix{
\cY \ar[rr]^{F} \ar[d]^{\pr} & & \C,
\\
\cM & &
}
\]
where $\cY$, $\cM$ are possibly singular toric DM stacks
(in the sense of Tyomkin~\cite{Tyomkin:tropical}).
The base $\cM$ of this LG model
corresponds to the (extended) K\"ahler moduli space of $\frX_\bSigma$
(i.e., the base of the quantum D-module)
and contains a distinguished point~$0_\bSigma$ called the
\emph{large radius limit point} of~$\frX_\bSigma$.
It also contains,\footnote{In fact, by choosing a suitable presentation of $\frX_\bSigma$
as the GIT quotient $[\C^S/\!/\LL_{\C^\times}]$,
we can arrange
that the large radius limit point of any given smooth toric DM stack
$\frX_{\bSigma'}$ having the same affinization and
the same generic
stabilizer as $\frX_{\bSigma}$ appears in the base space $\cM$.}
as torus-fixed points, the large radius limit points
$0_{\bSigma'}$ of several other toric DM stacks $\frX_{\bSigma'}$
which can be obtained from $\frX_{\bSigma}$
by varying the stability condition for the GIT quotient
$[\C^S/\!/\LL_{\C^\times}]$.
Hodge-theoretic mirror symmetry for toric DM stacks
established by
Coates--Corti-Iritani--Tseng \cite{CCIT:mirrorthm, CCIT:MS}
implies that, for each smooth toric DM stack $\frX_{\bSigma'}$
whose large radius limit point appears in~$\cM$,
we have a mirror map defined
on the formal neighbourhood of~$0_{\bSigma'}$
\[
\mir \colon \ (\cM,0_{\bSigma'})\sphat \longrightarrow
\text{a partial compactification of
$H^*_\CR(\frX_{\bSigma'})/2\pi\iu H^2(\frX_{\bSigma'},\Z)$}
\]
and a mirror isomorphism
\[
\Bri(F)\sphat_{\bSigma'} \cong \mir^*\QDMan(\frX_{\bSigma'}).
\]
Here $\Bri(F)$ denotes
the \emph{Brieskorn module} associated with the LG potential $F$
and the sub/su\-per\-scripts ${}_{\bSigma'}\sphat$
means the completion at the large-radius limit point $0_{\bSigma'}$
(see Sections~\ref{subsec:Bri}--\ref{subsec:completion}).
Using the convergence result from~\cite{Iritani:coLef,CCIT:MS},
we show that this mirror isomorphism can be extended to a
small analytic neighbourhood of~$0_{\bSigma'}$ as an
isomorphism\footnote{The analytified Brieskorn module $\Brian(F)_{\bSigma'}$
is analytic in the $\cM$-direction
but is still formal in $z$; similarly the formal quantum D-module
$\overline\QDMan(\frX_{\bSigma'})$ is formal in $z$ but
is still analytic in $\tau$, so they can be compared. }
between a certain \emph{analy\-ti\-fied} Brieskorn module
$\Brian(F)_{\bSigma'}$
and the formal quantum $D$-module
$\overline{\QDMan}(\frX_{\bSigma'})$
(see Theo\-rem~\ref{thm:an_mirror_isom}).
This enables us to compare the quantum D-modules of various
birational models of~$\frX_\bSigma$ over the mirror moduli
space~$\cM$.
Then we arrive at the following result
(in this theorem, we do not assume compactness of~$\frX_\pm$ or (semi-)positivity of~$c_1(\frX_\pm)$).

\begin{Theorem}[{Theorem \ref{thm:decomp_QDM}}]
\label{thm:decomp_QDM_introd}
Let $\varphi\colon \frX_+\dasharrow \frX_-$ be a discrepant
transformation between semiprojective toric DM stacks
induced by a single wall-crossing in the space of GIT stability
conditions. Suppose that $K_{\frX_+}>K_{\frX_-}$.
Then we have a formal decomposition of the
$\T$-equivariant quantum D-modules
\begin{equation}\label{eq:decomp_QDM_introd}
\mir_+^*\overline{\QDMan_\T}(\frX_+) \cong
\mir_-^* \overline{\QDMan_\T}(\frX_-) \oplus \scrR
\end{equation}
over a non-empty open subset $\cU_0'$ of $\cM\times \Lie \T$,
where $\mir_\pm$ denotes the mirror map for $\frX_\pm$
and~$\scrR$ is a locally free $\cOan_{\cU_0'}[\![z]\!]$-module equipped
with a meromorphic flat connection and a pairing.
\end{Theorem}

This theorem is a generalization of the
result of Gonz\'alez--Woodward \cite{Gonzalez-Woodward:tmmp}
who showed a~decomposition of the
quantum cohomology \emph{algebras}
under a running of the toric minimal model programme.
In this theorem, we consider analytic continuation
over a neighbourhood of a toric curve $\cC
\subset \cM$
connecting the large radius limit points $0_\pm$ for $\frX_\pm$.
When $\varphi$ is an isomorphism in codimension one (``flip''),
the curve $\cC$ is asymptotically close, near the large radius limit
points~$0_\pm$,
to the curve in the boundary of the K\"ahler moduli space
given by the extremal curve class.
When $\varphi$ (or $\varphi^{-1}$)
contracts a divisor, the curve $\cC$ is asymptotically close to
the curve corresponding to the extremal class near $0_+$
(resp.~$0_-$)
and to the line spanned by a cohomology class of degree
greater than 2 near $0_-$
(resp.~of degree less than 2 near $0_+$),
see Remark~\ref{rem:curve_asymptotics}.
In either case, at least one of the mirror maps $\mir_\pm$ involves
negative degree variables with respect to the Euler vector field
(an instance of the \emph{generalized} mirror transformation
\cite{Coates-Givental,Iritani:genmir, Jinzenji:genmir})
and the formal decomposition occurs over the base
of the \emph{big} quantum cohomology in general.
We also note that the decomposition~\eqref{eq:decomp_QDM_introd}
is defined only over the formal power series ring $\C[\![z]\!]$
and the completion in $z$ is unavoidable.
In fact, as Theorem \ref{thm:analytic_lift_introd} below shows,
the Stokes structure does not admit an orthogonal
decomposition.

Using the Givental--Teleman formula
\cite{BCR:crepant_open, Givental:quadratic, Teleman:semisimple,
Zong:Givental_GKM},
we obtain a decomposition for the (all-genus)
\emph{ancestor Gromov--Witten potentials}.
The result is stated in terms of Givental's quantization formalism;
we refer to Section~\ref{subsec:comp_all_genera} for the notation.
\begin{Theorem}[Theorem \ref{thm:ancestor}]
Let $\frX_\pm$ be toric DM stacks
as in Theorem $\ref{thm:decomp_QDM_introd}$.
Let $\scrA_{\pm,\tau}$ denote the ancestor potentials
of $\frX_\pm$ at $\tau\in H^*_{\CR,\T}(\frX_\pm)$.
For $(q,\chi) \in \cM \times \Lie \T$ in a non-empty open set,
we have
\[
T_\bs \hU_{q,\chi} \scrA_{+,\mir_+(q,\chi)}
= \scrA_{-,\mir_-(q,\chi)}\otimes \scrT^{\otimes \rank \scrR},
\]
where $\chi$ is the $\T$-equivariant parameter,
$\hU_{q,\chi}$ is the quantization of a
symplectic transformation associated with
the decomposition \eqref{eq:decomp_QDM_introd},
$\scrT$ is the Witten--Kontsevich tau-function
$($the ancestor potential of a~point$)$ and
$T_\bs$ is a certain shift operator.
\end{Theorem}

\subsection{Analytic lift and Orlov's decomposition}
The non-equivariant quantum D-module has (in general)
irregular singularities at $z=0$ and
the formal quantum D-module misses analytic information
such as the Stokes structure at $z=0$.
For a compact toric DM stack, at least in the non-equivariant limit
and over the semisimple locus,
the formal structure of the quantum D-module is very poor\footnote
{This does not mean that Theorem \ref{thm:decomp_QDM_introd} is trivial.
It compares the \emph{equivariant} quantum D-modules
over the open set $\cU_0'$ that contains the \emph{non-semisimple} loci.
It is also important that the two quantum D-modules
are connected through the explicit mirror LG model.},
since it is determined only by eigenvalues of the
Euler multiplication.
We will restore the missing information by describing
the \emph{analytic lift}
of the formal decomposition \eqref{eq:decomp_QDM_introd}.
By the Hukuhara--Turrittin theorem (see Section~\ref{subsec:Hukuhara-Turrittin}),
the decomposition~\eqref{eq:decomp_QDM_introd} in the non-equivariant limit
can be locally lifted to an analytic isomorphism:\footnote{More precisely, the analytic lift is defined over
functions in $z$ which admit asymptotic expansions along
an angular sector (with vertex at $z=0$);
such functions form a sheaf $\cA$ over the real oriented
blowup of $\C$ at the origin \cite{Sabbah:isomonodromic}.}
\begin{equation}
\label{eq:analytic_lift_introd}
\mir_+^*\QDMan(\frX_+) \big|_{B\times I}
\cong\mir_-^*\QDMan(\frX_-)\big|_{B\times I} \oplus \scrRan,
\end{equation}
where $B$ is a small open subset of $\cM$ and $I$ is an angular
sector $\{z\colon |\arg(z) -\phi|<\frac{\pi}{2}+\epsilon\}$
with $\epsilon>0$.
We call it the \emph{analytic lift} or a
\emph{sectorial decomposition};
its uniqueness is ensured by the fact that
the angle of the sector is bigger than $\pi$.
The analytic lift induces a decomposition
(depending on $B$ and $I$)
of the local system
underlying $\QDMan(\frX_+)$.
On the other hand, the \emph{$\hGamma$-integral structure}
\cite{Iritani:Integral, KKP:Hodge} identifies
the complexified topological $K$-group $K(\frX)\otimes \C$
with the space of multi-valued
flat sections of the quantum D-module; for toric stacks,
it corresponds to the integral structure on the GKZ system
identified by Borisov--Horja \cite{Borisov-Horja:FM}.
We show in some special cases
that the decomposition of the local system given by the analytic lift
corresponds to a semiorthogonal decomposition
of the topological $K$-group
$K(\frX_+)$ via the $\hGamma$-integral structure.
An important ingredient here is the fact \cite{Iritani:Integral}
that the $\hGamma$-integral structure coincides
with the natural integral structure of the
Brieskorn module under mirror symmetry.
By describing the analytic lift in terms of mirror oscillatory integrals
and studying the relationship between the local systems of Lefschetz
thimbles (see Theorems \ref{thm:inclusion} and \ref{thm:decomp_Lef_K}),
we obtain the following result.

\begin{Theorem}[Theorems \ref{thm:functoriality},
\ref{thm:decomp_Lef_K}, and \ref{thm:sectorial_Orlov}]
\label{thm:analytic_lift_introd}
Let $\frX_-$ be a weak-Fano compact toric stack
satisfying a~mild technical assumption as described in
Section~$\ref{subsec:functoriality_nota}$ and
let $\varphi \colon \frX_+ \to \frX_-$ be a~weighted
blowup along a toric substack $Z \subset \frX_-$.
We assume that $\frX_+$ is also weak-Fano.
Then there exist a submersion $f$ from an open set $W$ of
$H^*_\CR(\frX_+)$ to $H^*_\CR(\frX_-)$ and
an angular sector~$I$ $($of angle greater than $\pi)$
such that we have an analytic decomposition over the sector
\[
\QDMan(\frX_+)\big|_{W\times I}
\cong \scrRan_{-J} \oplus \cdots \oplus
\scrRan_{-1} \oplus f^*\QDMan(\frX_-) \big|_{W\times I},
\]
which induces, via the $\hGamma$-integral structure,
a semiorthogonal decomposition
of the $K$-group
\begin{equation}
\label{eq:decomp_K_introd}
K(\frX_+) = K(Z)_{-J} \oplus \cdots \oplus K(Z)_{-1} \oplus
\varphi^* K(\frX_-),
\end{equation}
where $K(Z)_k= \cO(-k E) \otimes i_{E_*}\varphi_E^* K(Z)
\subset K(\frX_+)$. Here $E$ is the exceptional divisor
of $\varphi$, $\varphi_E = \varphi|_E\colon E \to Z$,
$i_E \colon E \to \frX_+$ is the inclusion and
$J = k_0+\cdots+k_c -1$
when a fibre of $\varphi_E \colon E \to Z$ is given
by the weighted projective space $\PP(k_0,\dots,k_c)$.
\end{Theorem}
\begin{Remark}\quad
\begin{enumerate}\itemsep=0pt
\item[(1)] In this theorem, we allow $Z$ to be of codimension one
(i.e., a toric divisor);
in this case $\frX_+$ is obtained from $\frX_-$ by
a root construction (see \cite{Cadman:tangency}) along $Z$.

\item[(2)] When $\frX_-$ is a smooth projective variety and
$\varphi \colon \frX_+ \to \frX_-$ is the blowup along
a smooth subvariety $Z$, the decomposition~\eqref{eq:decomp_K_introd} is induced by
Orlov's semiorthogonal decomposition~\cite{Orlov:proj}
for $D^b(\frX_+)$.
In general, the Orlov-type decomposition \eqref{eq:decomp_K_introd}
arises from a sectorial decomposition of the quantum D-module
at a point which is far from the large radius limit point.
On the other hand,
we can idenfity explicitly the locus in the mirror modu\-li
space~$\cM$ where the analytic lift~\eqref{eq:analytic_lift_introd}
induces the pull-back $\varphi^* \colon K(\frX_-) \to K(\frX_+)$
in $K$-theory (Theorem~\ref{thm:functoriality}).

\item[(3)] The result suggests that each residual piece $\scrRan_i$
of the sectorial decomposition should be related to
the quantum D-module of the blowup centre $Z$;
they certainly have the same formal structure, but we do not
know if the Stokes structures are related
(although we expect from homological mirror symmetry
that they should be related).

\item[(4)] We need the weak-Fano assumption when we apply
results from~\cite{Iritani:Integral}.
We hope that the same result holds without such assumptions,
but it may require some technical advances.

\item[(5)] In Section~\ref{subsec:conj_discussion}, we formulate
a general conjecture relating the decomposition of
topological $K$-groups and that of quantum D-modules
under discrepant transformations.
Under the conjecture, the decomposition of
$K$-groups in principle determines the relationship between
the quantum D-modules of~$\frX_+$ and~$\frX_-$
(including the map~$f$). This involves solving a~Riemann--Hilbert
problem; see Proposition~\ref{prop:RH}.
\end{enumerate}
\end{Remark}

\subsection{Related works}
We mention some of the earlier works that are closely related to
the present paper.

The relationship between quantum cohomology and
derived category has been suggested by Dubrovin \cite{Dubrovin:ICM}.
Our Theorem \ref{thm:analytic_lift_introd} can be viewed as a
variation on this theme
(see also Gamma conjecture \cite{GGI:gammagrass} or
Dubrovin-type conjecture \cite{Sanda-Shamoto}).
Bayer \cite{Bayer:semisimple} showed that the semisimplicity
of quantum cohomology is preserved under blowup at a point
in connection with Dubrovin's conjecture \cite{Dubrovin:ICM}.
His computation \cite[Lemma 3.4.2]{Bayer:semisimple}
for the spectral cover is compatible
with the picture in this paper.

As mentioned earlier, Gonz\'alez--Woodward \cite{Gonzalez-Woodward:tmmp}
showed a decomposition of toric quantum cohomology
under flips; they used the same LG model mirrors
(combined with the quantum Kirwan map) to analyze
the change of quantum cohomology.
Charest--Woodward \cite{Charest-Woodward:Floer_flip}
and Sanda \cite{Sanda:thesis, Sanda:computation_QC}
discussed (orthogonal) decomposition of Fukaya category
and of quantum cohomology under flips/blowups.
From a categorical viewpoint,
the quantum cohomology (or the formal quantum D-module) should arise
as the Hochschild (resp.~negative cyclic) homology of the Fukaya category,
and a decomposition of the Fukaya category should induce
a (formal) decomposition of the quantum cohomology/D-module
via the open-closed map.

Acosta--Shoemaker \cite{Acosta-Shoemaker:blowup_LG,
Acosta-Shoemaker:toric_birational}
(see also an earlier work of Acosta \cite{Acosta:asymp})
studied discrepant
wall-crossings for toric Gromov--Witten theory in the same
setting as ours.
They compared the Givental $I$-functions of~$\frX_\pm$
over the mirror moduli space $\cM$.
Writing~$I_\pm$ for the $I$-functions of~$\frX_\pm$,
they showed that the asymptotic expansion of~$I_+$
near the large-radius limit point~$0_-$ is related to~$I_-$
by a linear transformation
$L\colon H^*_\CR(\frX_+) \to H^*_\CR(\frX_-)$,
i.e., $L I_+(q)\sim I_-(q)$
as $q\to 0_-$ along some (one-dimensional) angular direction.
We expect that their linear transformation $L$ should correspond to
the projection between the quantum cohomology
local systems for $\frX_\pm$ associated with an
analytic lift~\eqref{eq:analytic_lift_introd}.
They also dealt with the case of complete intersections in toric stacks
(which we do not cover in this paper).
More recently,
Lee--Lin--Wang \cite[Section~6]{Lee-Lin-Wang:StringMath15},
\cite{Lee-Lin-Wang:flips_I}
have announced a decomposition of quantum D-modules
under flips and showed it for local toric flips.
In these approaches \cite{Acosta:asymp,Acosta-Shoemaker:blowup_LG,
Acosta-Shoemaker:toric_birational, Lee-Lin-Wang:StringMath15,
Lee-Lin-Wang:flips_I},
they studied irregular singularities on the base~$\cM$ directly
whereas we studied those on the $z$-plane;
the singularities at $z=0$ and $q=0_-$ are closely related,
but those at $z=0$ have a simpler structure.

Recently, Clingempeel--Le~Floch--Romo \cite{Clingempeel-LeFloch-Romo}
compared the hemisphere partition functions
(which in our language correspond to certain solutions
of the quantum D-modules)
of the cyclic quotient singularities
$\C^2/\mu_n$ and their Hirzebruch--Jung resolutions.
They discussed the relation to a semiorthogonal
decomposition of the derived categories, extending the work of
Herbst--Hori--Page~\cite{HHP} to the anomalous
(discrepant) case.
Their examples are complementary to ours:
the Hirzebruch--Jung resolutions are type (II-ii)
discrepant transformations whereas
transformations in Theorem~\ref{thm:analytic_lift_introd}
are of type (II-i) or (III) (see Remark~\ref{rem:3_types} for these types).

Under homological mirror symmetry, the derived categories
of coherent sheaves for toric stacks correspond to
the Fukaya--Seidel categories of the mirror LG models.
Kerr \cite{Kerr:weighted},
Diemer--Katzarkov--Kerr \cite{DKK:symplectomorphism,
DKK:compactification_LG}
and Ballard--Diemer--Favero--Katzarkov--Kerr
\cite{BDFKK:Mori_nonFano}
(see also \cite{BFK:VGIT})
studied semiorthogonal decompositions
of the Fukaya--Seidel categories of the LG mirrors
under toric wall-crossings; they conjectured
\cite[Conjecture~1]{BDFKK:Mori_nonFano}
that semiorthogonal decompositions for the Fukaya--Seidel categories
and for the derived categories
should match up under mirror symmetry.
Theorem~\ref{thm:analytic_lift_introd} can be viewed as an
evidence for their conjecture on the level of enumerative mirror symmetry.
We note that they introduced a similar toric compactification
of the moduli space of LG models and
studied how the critical values assemble
under deformation in a more general setting than ours
(a relevant discussion appears in Section~\ref{subsec:inclusion} in this paper).

After the author finished a draft of this paper, he heard a talk
of Kontsevich~\cite{Kontsevich:HSE2019}
who studied (in joint project with Katzarkov and Pantev)
the change of quantum cohomology under blowups from a
similar perspective and
gave an application to birational geometry.

The author apologizes for the long delay in preparing the paper
since the original announcement \cite{Iritani:KIAS}
(see also~\cite{Iritani:MIT})
in June 2008. There were many technical issues in proving
our results in this generality.
Advances made in joint work \cite{CCIT:mirrorthm,CCIT:MS}
with Coates, Corti and Tseng
and in joint work \cite{GGI:gammagrass}
with Galkin and Golyshev are essential
in this paper.

\section{Preliminaries}
In this section, in order to fix notation,
we review quantum cohomology, quantum D-modules
and the Gamma-integral structure.
Our main interest in this paper lies in the case
where $\frX$ is a~toric DM stack, but
all the materials in this section make sense for a general smooth
DM stack~$\frX$ satisfying mild assumptions.

\subsection{Quantum cohomology}\label{subsec:qcoh}
Gromov--Witten theory for smooth Deligne--Mumford stacks
(or symplectic orbifolds) has been developed by Chen--Ruan~\cite{Chen-Ruan:orbGW}
and Abramovich--Graber--Vistoli \cite{AGV1,AGV:GW}.
We use the algebro-geometric approach in~\cite{AGV1,AGV:GW}.

Let $\frX$ be a smooth DM (Deligne--Mumford) stack over $\C$.
We write $X$ for the coarse moduli space of $\frX$.
Recall that the \emph{inertia stack} $I \frX$ is the fibre product
$\frX\times_{\frX\times \frX} \frX$ of the two diagonal
morphisms $\frX \to \frX \times \frX$.
A point on $I\frX$ is given by a pair $(x,g)$ of a point $x\in \frX$
and a stabilizer $g\in \Aut(x)$.
We write $I\frX = \bigsqcup\limits_{v\in \Bx} \frX_v$ for the decomposition
of $I\frX$ into connected components, where $\Bx$ is the index set.
We have a distinguished element $0\in \Bx$ that corresponds to
the \emph{untwisted sector} $\frX_0 \cong \frX$ consisting of
points $(x,g=1)$ with the trivial stabilizer.
The \emph{orbifold cohomology}
$H_\CR^*(\frX)$ of Chen and Ruan \cite{Chen-Ruan:new_coh}
is defined to be
\begin{equation}
\label{eq:orb_coh}
H_\CR^*(\frX) := H^{*- 2\age}(I\frX,\C) =
\bigoplus_{v\in \Bx} H^{*-2 \age(v)}(\frX_v,\C),
\end{equation}
where $\age \colon I\frX \to \Q_{\ge 0}$ is a locally constant
function giving a shift of degrees and we write $\age(v) = \age|_{\frX_v}$
(see Section~\ref{subsubsec:orb_coh} for the age in the case of toric DM stacks).
The right-hand side means the cohomology group of the underlying
complex analytic space of $I\frX$ and we use complex coefficients
unless otherwise specified.
We also restrict ourselves to cohomology classes of ``even parity'', i.e., we
only consider cohomology classes of even degrees\footnote{An element
$\alpha\in H^*_\CR(\frX)$ having ``even parity'' does
not imply that the degree of $\alpha$ as an orbifold cohomology class is even;
it means that the degree of $\alpha$ is even \emph{as an element of}
$H^*(I\frX)$.} on~$I\frX$.
For toric DM stacks, every orbifold cohomology class has even parity.
When $\frX$ is proper, the orbifold Poincar\'e pairing
on $H_\CR^*(\frX)$ is defined to be
\begin{equation}\label{eq:orb_Poincare_pairing}
(\alpha,\beta) := \int_{I\frX} \alpha \cup \inv^*\beta,
\end{equation}
where $\inv \colon I\frX \to I\frX$ is the involution sending
$(x,g)$ to $\big(x,g^{-1}\big)$.
For $d \in H_2(\frX,\Q)$ and $l \in \Z_{\ge 0}$,
let $\frX_{g,l,d}$ denote the moduli stack of genus-$g$ twisted
stable maps to $\frX$ of degree $d$ and with~$l$ marked points
(this was denoted by $\mathcal{K}_{g,l}(\frX,d)$ in~\cite{AGV:GW}).
It carries a virtual fundamental class $[\frX_{g,l,d}]_{\rm vir}
\in A_*(\frX_{g,l,d},\Q)$ and the evaluation maps
$\ev_i \colon \frX_{g,l,d} \to \overline{I}\frX$, $i=1,\dots,l$
to the rigidified cyclotomic inertia stack $\overline{I}\frX$
(see \cite[Section~3.4]{AGV:GW}).
When the moduli stack $\frX_{g,l,d}$ is proper (this happens
when $\frX$ has a projective coarse moduli space), we define
\emph{genus-zero descendant Gromov--Witten invariants} by
\[
\big\langle \alpha_1\psi^{k_1},\dots,\alpha_l\psi^{k_l}\big\rangle_{g,l,d}
:= \int_{[\frX_{g,l,d}]_{\rm vir}}
\prod_{i=1}^l \ev_i^*(\alpha_i) \psi_i^{k_i},
\]
where $\alpha_1,\dots,\alpha_l \in H^*_\CR(\frX)$, $d\in H_2(\frX,\Q)$
and $\psi_i$ denotes the $\psi$-class (see \cite[Section~8.3]{AGV:GW}) at the $i$th marking.
Here note that the rigidified cyclotomic inertia stack $\overline{I}\frX$ has
the same coarse moduli space as $I\frX$, and thus $\alpha_i$ can be
regarded as a cohomology class of~$\overline{I}\frX$.

We assume that there exists a
finitely generated monoid $\Laa_+ \subset H_2(\frX,\Q)$
such that $\R_{\ge 0}\Laa_+$ is a strictly convex full-dimensional cone
in $H_2(\frX,\R)$
and that $\Laa_+$ contains classes of any orbifold stable curves.
We also assume that $\Laa_+$ is saturated,
i.e., $\Laa_+ = \Laa \cap \R_{\ge 0} \Laa_+$
for $\Laa := \Z \Laa_+$.
The monoid $\Laa_+$ for toric DM stacks will be described
in Section~\ref{subsec:ext_refined_fanseq}
and will be denoted by $\Laa^\bSigma_+$ there.
For a ring $K$, let $K[\![\Laa_+]\!]$ denote the completion of
$K[\Laa_+]$ consisting of all formal sums
$\sum\limits_{d \in \Laa_+} a_d Q^d$ with $a_d \in K$.
The variable $Q$ here is called the \emph{Novikov variable}.
We also choose a homogeneous basis $\{\phi_i\}_{i=0}^s$ of $H^*_\CR(\frX)$
and introduce linear co-ordinates $\{\tau^i\}_{i=0}^s$
on $H^*_\CR(\frX)$ as $\big(\tau^0,\dots,\tau^s\big)
\mapsto \tau = \sum\limits_{i=0}^s \tau^i \phi_i$.
We write $K[\![\tau]\!] = K[\![\tau^0,\dots,\tau^s]\!]$ for any ring $K$.
The \emph{quantum product} $\alpha\star \beta$
of classes $\alpha,\beta\in H^*_\CR(\frX)$
is defined so that for every $\gamma \in H^*_\CR(\frX)$, we have
\[
(\alpha\star\beta, \gamma) :=
\sum_{d\in \Laa_+}
\sum_{l \ge 0}
\corr{\alpha,\beta,\gamma,\tau,\dots,\tau}_{0,l+3,d} \frac{Q^d}{l!},
\]
where $\tau = \sum\limits_{i=0}^s \tau^i\phi_i \in H^*_\CR(\frX)$ is a parameter.
This definition makes sense when $X$ is projective
and $\alpha \star \beta$ lies in
$H^*_\CR(\frX) \otimes \C[\![\Laa_+]\!][\![\tau]\!]$.
The quantum product $\star$ is known to be commutative and associative
and defines a commutative ring structure on
$H^*_\CR(\frX) \otimes \C[\![\Laa_+]\!][\![\tau]\!]$.
When we want to emphasize the dependence of $\star$ on the parameter $\tau$,
we shall write $\alpha \star_\tau \beta$ in place of $\alpha \star \beta$.

The Gromov--Witten invariants and the
quantum product can be generalized to the equivariant setting
or to a non-projective space $\frX$.
Suppose that an algebraic torus $\T$ acts on $\frX$.
The equivariant orbifold cohomology $H_{\CR,\T}^*(\frX)$ is
defined to be the $\T$-equivariant cohomology of $I\frX$ with
the same degree shift as before:
\[
H_{\CR,\T}^*(\frX) := H^{*-2 \age}_\T(I\frX,\C)
= \bigoplus_{v\in \Bx} H^{*-2 \age(v)}_\T(\frX_v,\C).
\]
Let $E\T \to B\T$ denote a universal $\T$-bundle.
We assume that
\begin{itemize}\itemsep=0pt
\item[(a)] $I\frX$ is equivariantly formal, i.e., the Serre spectral sequence
for $I\frX \times_{\T} E\T \to B\T$ collapses at the $E^2$-term (over $\Q$);
this implies that
$H_{\CR,\T}^*(\frX)$ is a free
$R_\T := H^*_\T(\pt,\C)$-module of rank $\dim H_\CR^*(\frX)$;
\item[(b)] the $\T$-fixed set of $\frX$ is projective;
\item[(c)] the evaluation maps $\ev_i \colon \frX_{0,l,d} \to \overline{I}\frX$ are proper.
\end{itemize}
By the assumption (a), the $\C$-basis $\{\phi_i\}_{i=0}^s\subset H^*_\CR(\frX)$
can be lifted to an $R_\T$-basis of $H^*_{\CR,\T}(\frX)$,
which we denote by the same symbol.
This induces $R_\T$-linear co-ordinates $(\tau^0,\dots,\tau^s)
\mapsto \tau = \sum\limits_{i=0}^s \tau^i\phi_i$
on $H_{\CR,\T}^*(\frX)$ as before.
Under the assumption (b), we can define the equivariant orbi\-fold
Poincar\'e pairing and
the equivariant Gromov--Witten invariants
via the virtual localization formula \cite{Graber-Pandharipande}.
They take values in
the fraction field $S_\T :=\Frac(R_\T)$ of $R_\T$.
Under the assumption~(c), we can define the quantum product
using the push-forward by the proper map $\ev_3$
as follows:
\[
\alpha \star \beta = \sum_{d\in \Laa_+}
\sum_{l \ge 0} \inv^* \PD {\ev_3}_*\left( \ev_1^*(\alpha) \ev_2^*(\beta)
\prod_{j=1}^l \ev_{3+j}^*(\tau) \cap [\frX_{0,l+3,d}]_{\rm vir}
\right) \frac{Q^d}{l!},
\]
where $\PD$ stands for the Poincar\'e duality on $I\frX$.
Therefore we get the quantum product $\star$ defined on
the space
$H_{\CR,\T}^*(\frX)\otimes_{R_\T} R_\T[\![\Laa_+]\!][\![\tau]\!]$
without inverting equivariant parameters, whereas the Gromov--Witten
invariants themselves lie in $S_\T$ in general.
\begin{Remark}
As remarked in~\cite{Bryan-Graber}, the assumption~(c) is satisfied
if the coarse moduli space~$X$ is semi-projective, i.e., projective over an affine variety.
We will impose the semi-projectivity assumption on toric DM stacks.
\end{Remark}

\subsection{K\"ahler moduli space}\label{subsec:Kaehler_moduli}
In this section, we specialize the Novikov variable $Q$ to one
with the aid of the divisor equation, and introduce the K\"ahler moduli
space $\cM_{\rm A}(\frX)$ parameterizing the quantum product.

As before, let $\phi_0,\phi_1,\dots,\phi_s$ be a homogeneous
basis of $H^*_\CR(\frX)$.
We assume that $\phi_0=\unit$ is the identity class
and that $\{\phi_1,\dots,\phi_r\}$, $r\le s$ is a basis of the degree-two
untwisted sector $H^2(\frX) \subset H^2_\CR(\frX)$.
We write $\tau = \sigma + \tau'$ with
$\sigma = \sum\limits_{i=1}^r \tau^i \phi_i\in H^2(\frX)$ and
$\tau' = \tau^0 \phi_0 + \sum\limits_{i=r+1}^s \tau^i \phi_i$.
The divisor equation \cite[Theorem~8.3.1]{AGV:GW} implies
\[
(\alpha\star_\tau\beta, \gamma) = \sum_{d\in \Laa_+}
\sum_{l=0}^\infty
\corr{\alpha,\beta,\gamma,\tau',\dots,\tau'}_{0,l+3,d}
\frac{Q^d e^{\sigma \cdot d}}{l!}.
\]
This shows that the quantum product $\star_\tau$ depends only
on the equivalence class
\[
[\tau] \in H^*_\CR(\frX)/2\pi\iu \Laa^\star,
\]
where $\Laa^\star \subset H^2(\frX,\Q)$ denotes the dual lattice of $\Laa
=\Z \Laa_+$. Note that $d\in \Laa$ defines a~function
$q^d \colon [\tau] \mapsto e^{\sigma\cdot d}=
\exp\big(\sum\limits_{i=1}^r \tau^i (\phi_i\cdot d)\big)$
on $H^*_\CR(\frX)/2\pi\iu \Laa^\star$,
and that each co-ordinate $\tau^i$ with $i\notin \{1,\dots, r\}$
also defines a function on $H^*_\CR(\frX)/2\pi\iu \Laa^\star$.
These functions define an open embedding of
$H^*_\CR(\frX)/2\pi\iu \Laa^\star$ into the
following space:
\[
\cM_{\rm A}(\frX) :=
\Spec \C[\Laa_+][\tau'] = \Spec \C[\Laa_+]\big[\tau^0,\tau^{r+1},\dots,\tau^{s}\big],
\]
where the element of $\C[\Laa_+]$ corresponding to $d\in \Laa_+$
represents the function $q^d = e^{\sigma \cdot d}$.
The space $\cM_{\rm A}(\frX)$ gives a partial compactification
of $H_\CR^*(\frX)/2\pi\iu \Laa^\star$ depending on the choice
of the monoid $\Laa_+$.
By setting $Q=1$, we may view $\star_\tau$ as a family of
products parameterized by the formal neighbourhood
of the ``origin'' (that is, $q^d = 0$ for all non-zero $d\in \Laa_+$
and $\tau^0=\tau^{r+1} = \cdots = \tau^s =0$)
in $\cM_{\rm A}(\frX)$. This origin is called the \emph{large
radius limit point}.

Quantum cohomology for smooth DM stacks has
the additional symmetry called the \emph{Galois symmetry}
\cite[Section~2.2]{Iritani:Integral}.
Let $H^2(\frX,\Z)$ denote the sheaf cohomology of the constant sheaf~$\Z$ on the topological stack (orbifold) underlying~$\frX$;
an element $\xi \in H^2(\frX,\Z)$
corresponds to a~topological orbi-line bundle $L_\xi$ over $\frX$.
For a connected component $\frX_v$ of $I\frX$,
the stabilizer along~$\frX_v$ acts on fibres of $L_\xi$ by
a constant scalar $\exp(2\pi\iu f_v(\xi))$ for some
$f_v(\xi) \in [0,1) \cap \Q$. This number~$f_v(\xi)$ is called
the \emph{age} of $L_\xi$ along $\frX_v$.
We define a map $g(\xi) \colon H_\CR^*(\frX) \to H_\CR^*(\frX)$
by
\[
g(\xi)(\tau) = \left( \tau_0 - 2\pi \iu \xi \right)
\oplus \bigoplus_{v\neq 0} e^{2\pi\iu f_v(\xi)} \tau_v,
\]
where $\tau_v$ denotes the $H^*(\frX_v)$-component
of $\tau\in H^*_{\CR}(\frX)$ in the decomposition~\eqref{eq:orb_coh}.
Let $d g(\xi) \in \End(H_\CR^*(\frX))$ denote
the derivative of the map $g(\xi)$; it is given by
$d g(\xi) \tau = \bigoplus\limits_{v\in \Bx} e^{2\pi\iu f_v(\xi)} \tau_v$.
Gromov--Witten invariants satisfy
(see \cite[Proposition 2.3]{Iritani:Integral})
\begin{equation}
\label{eq:Galois}
\big\langle dg(\xi)(\alpha_1) \psi^{k_1},\dots, dg(\xi)(\alpha_l) \psi^{k_l}\big\rangle _{g,l,d}
= e^{2\pi\iu \xi \cdot d}
\big\langle \alpha_1 \psi^{k_1},\dots,\alpha_l \psi^{k_l}\big\rangle_{g,l,d}
\end{equation}
and thus the quantum product satisfies
\[
dg(\xi)( \alpha \star_\tau \beta ) = (dg(\xi) \alpha) \star_{g(\xi)(\tau)}
(dg(\xi) \beta).
\]
The map $g(\xi)$ induces the action of $H^2(\frX,\Z)/\Laa^\star$
on $H_\CR^*(\frX)/2\pi\iu \Laa^\star$; this action naturally
extends to the partial compactification $\cM_{\rm A}(\frX)$.
In view of this symmetry, we can regard
the quantum products $\star_\tau$ as a family of products
parametrized by the formal neighbourhood of the origin
(large radius limit) of the stack:
\[
\big[\cM_{\rm A}(\frX)\big/
\big(H^2(\frX,\Z)/\Laa^\star\big)\big].
\]
We refer to $\cM_{\rm A}(\frX)$ or to the quotient stack above
as the \emph{K\"ahler moduli space}, where the subscript A stands for
the A-model.

The above construction can be adapted to the
equivariant quantum cohomology.
We choose a homogeneous $R_\T$-basis
$\{\phi_i\}_{i=0}^s$ of $H_{\CR,\T}^*(\frX)$
such that
\begin{equation}
\label{eq:basis_equiv_lift}
\phi_0=\unit \qquad \text{and} \qquad
\{\phi_1,\dots,\phi_r\} \subset H^2_\T(\frX)
\qquad \text{with $r = \dim H^2(\frX)\le s$}.
\end{equation}
Then the non-equivariant limits of $\phi_1,\dots,\phi_r$
form a basis of $H^2(\frX)$.
The basis $\{\phi_i\}_{i=0}^s$
defines $R_\T$-linear
co-ordinates $\{\tau^i\}_{i=0}^s$ on $H_{\CR,\T}^*(\frX)$ as before.
The equivariant K\"ahler moduli spaces are given
by replacing the ground ring $\C$ with $R_\T$.
\begin{equation}
\label{eq:equiv_Kaehler_moduli}
\cM_{\rm A,\T}(\frX) := \Spec R_\T[\Laa_+]\big[\tau^0,\tau^{r+1},\dots,t^s\big].
\end{equation}
It is fibred over $\Spec R_\T \cong \Lie \T$.
The group $H^2(\frX,\Z)/\Laa^\star$ acts on $\cM_{\rm A,\T}(\frX)$
through the isomorphism $\cM_{\rm A,\T}(\frX) \cong
\cM_{\rm A}(\frX) \times \Spec R_\T$.
The $\T$-equivariant quantum product $\star_\tau$ can be
viewed as a family of product structures parameterized
by the formal neighbourhood of the origin
in $\big[\cM_{\rm A,\T}(\frX)/\big(H^2(\frX,\Z)/\Laa^\star\big)\big]$.

\begin{Remark}
\label{rem:equiv_coh}
Unlike the non-equivariant case, the \emph{equivariant} K\"ahler
moduli space $\cM_{\rm A,\T}(\frX)$ is \emph{not} a partial
compactification of $H^*_{\CR,\T}(\frX)/2\pi\iu \Laa^\star$;
we do not even have a natural map $H^*_{\CR,\T}(\frX) \to
\cM_{\rm A,\T}(\frX)$. If we regard $H^*_{\CR,\T}(\frX)$ as a
locally free coherent sheaf over $\Spec R_\T$ and write
$\bH_{\CR,\T}(\frX)=
\Spec\big(\Sym^\bullet_{R_\T}\big(H^*_{\CR,\T}(\frX)^\vee\big)\big)$
for the total space
of the associated vector bundle (where $(-)^\vee$ stands for the
dual as an $R_\T$-module), then we have an open dense embedding
$\bH_{\CR,\T}(\frX)/2\pi\iu\Laa^\star \to \cM_{\rm A,\T}(\frX)$.
In this paper,
we take the view that the equivariant quantum product
$\star_\tau$ is parametrized by
points in $\bH_{\CR,\T}(\frX)$ rather than by
equivariant cohomology classes.
Note that equivariant cohomology classes correspond to
sections of $\bH_{\CR,\T}(\frX) \to \Spec R_\T$.
\end{Remark}

\begin{Remark}
The \emph{equivariant} K\"ahler moduli space $\cM_{\rm A,\T}(\frX)$
depends on the choice of a basis $\phi_1,\dots,\phi_r$.
In fact, in the equivariant case, we can replace the basis
$\{\phi_i\}_{i=0}^s$ with a new basis
$\{\tilde{\phi}_i\}_{i=0}^s$ of the form
\[
\tilde{\phi}_i =
\begin{cases}
\phi_i + c_i \phi_0, & i \in \{1,\dots,r\}, \\
\phi_i, & i \notin \{1,\dots,r\},
\end{cases}
\]
for some $c_i \in H^2_\T(\pt)$ without violating the homogeneity.
Then the corresponding $0$-th co-ordinate $\ttau^0
= \tau^0 - \sum\limits_{i=1}^r c_i \tau^i$ does not lie in the
ring $R_\T[\Laa_+]\big[\tau^0,\tau^{r+1},\dots,\tau^s\big]$.
In other words, the construction of the equivariant K\"ahler moduli
space requires the choice of a splitting of the sequence
$0 \to H^2_\T(\pt) \to H^2_\T(\frX) \to H^2(\frX) \to 0$.
\end{Remark}

\begin{Remark}
In the above discussion, we considered the specialization $Q=1$
and the equivalence class
of $\tau$ in $H^*_{\CR}(\frX)/2\pi\iu \Laa^\star$.
This is equivalent to considering the restriction to
$\tau^1 = \cdots =\tau^r =0$
and the substitution of $q^d$ for $Q^d$.
\end{Remark}

\begin{Remark}
Henceforth we specialize the Novikov variable $Q$ to one in
the quantum pro\-duct~$\star_\tau$, unless otherwise stated.
\end{Remark}

\subsection{Quantum D-module}\label{subsec:QDM}
The quantum product defines a meromorphic flat connection
on a vector bundle (with fibre orbifold cohomology) over the
K\"ahler moduli space, called the \emph{quantum connection}.
The quantum connection, the grading and the
orbifold Poincar\'e pairing constitute
the \emph{quantum D-module} of~$\frX$.

We start by explaining the equivariant version; we get the non-equivariant version
by taking non-equivariant limit.
As before, we fix a homogeneous $R_\T$-free basis $\{\phi_0,\dots,\phi_s\}$ of
$H^*_{\CR,\T}(\frX)$ that satisfies \eqref{eq:basis_equiv_lift}.
We use this basis to construct the equivariant K\"ahler moduli
spaces $\cM_{\rm A,\T}(\frX)$.
Consider the vector bundle
\begin{equation}
\label{eq:quantum_D-mod_bundle}
\bH_{\CR,\T}(\frX)
\times_{\Spec R_\T} (\cM_{\rm A,\T}(\frX) \times \C_z) \to
\cM_{\rm A,\T}(\frX) \times \C_z
\end{equation}
of rank equal to $\dim H_\CR^*(\frX)$, where $\C_z := \Spec \C[z]$
is the complex plane with co-ordinate $z$ and
$\bH_{\CR,\T}(\frX)$ denotes a vector bundle over
$\Spec R_\T$ corresponding to $H^*_{\CR,\T}(\frX)$
(see Remark~\ref{rem:equiv_coh}).
The Galois symmetry in the previous section induces the
$(H^2(\frX,\Z)/\Laa^\star)$-action on this vector bundle
defined by the map $dg(\xi)\times g(\xi) \times \id_{\C_z}$.
By the Galois symmetry, this vector bundle
descends to a vector bundle on the quotient stack
$\big[\cM_{\rm A,\T}(\frX)/\big(H^2(\frX,\Z)/\Laa^\star\big)\big] \times \C_z$.
The \emph{quantum connection} is a meromorphic flat partial
connection on this vector bundle over the formal neighbourhood
of the origin in $\cM_{\rm A,\T}(\frX)$ (times~$\C_z$); it is given by
\[
\nabla = d + z^{-1} \sum_{i=0}^s (\phi_i\star_\tau) d\tau^i.
\]
This connection is partial in the sense that it is defined only
in the $\tau$-direction and not in the $z$-direction or in the
direction of equivariant parameters
(the first term~$d$ means the relative differential
over $\Spec R_\T[z]$).
The connection has simple poles along $z=0$ and has logarithmic
singularities along the toric boundary of $\Spec \C[\Laa_+]$.
Note that $d\tau^i$ with $1\le i\le r$ defines a~logarithmic
$1$-form on $\Spec \C[\Laa_+]$.
In a more formal language, the module
\[
H^*_{\CR,\T}(\frX) [z][\![\Laa_+]\!][\![\tau']\!]
:=
H^*_{\CR,\T}(\frX) \otimes_{R_\T} R_\T[z][\![\Laa_+]\!]
\big[\!\big[\tau^0,\tau^{r+1},\dots,\tau^s\big]\!\big]
\]
(which we regard as the module of sections of the bundle
\eqref{eq:quantum_D-mod_bundle} over the formal neighbourhood
of the origin in $\cM_{\rm A,\T}(\frX)$)
is equipped with $s+1$ operators
$\nabla_{\parfrac{}{\tau^i}}
= \parfrac{}{\tau^i} + z^{-1} \phi_i \star_\tau$,
$i=0,1,\dots,s$:
\[
\nabla_{\parfrac{}{\tau^i}}\colon
H^*_{\CR,\T}(\frX)[z][\![\Laa_+]\!][\![\tau']\!]
\to z^{-1} H^*_{\CR,\T}(\frX)[z][\![\Laa_+]\!][\![\tau']\!]
\]
that commute each other:
$\big[\nabla_{\parfrac{}{\tau^i}},\nabla_{\parfrac{}{\tau^j}}\big]= 0$.

Let $\{\chi_1,\dots,\chi_k\}$ denote a basis
of $H^2_\T(\pt,\C)$ so that $R_\T = \C[\chi_1,\dots,\chi_k]$.
For $\xi \in H^2(\frX,\C)$, we write $\xi q \parfrac{}{q}$ for the
derivation of $\C[\![\Laa_+]\!]$ given by
$\big(\xi q\parfrac{}{q}\big) q^d = (\xi \cdot d) q^d$.
The \emph{Euler vector field}~$\cE$ is the following
derivation of $R_\T[\![\Laa_+]\!][\![\tau']\!]$
\begin{equation}\label{eq:Euler_A}
\cE = c_1(\frX) q\parfrac{}{q} + \sum_{i\in \{0,r+1,\dots,s\}}
\left(1- \frac{\deg \phi_i}{2}\right) \tau^i \parfrac{}{\tau^i}
+ \sum_{i=1}^k \chi_i \parfrac{}{\chi_i}.
\end{equation}
The \emph{grading operator} $\Gr\in
\End_\C(H_{\CR,\T}(\frX)[z][\![\Laa_+]\!][\![\tau']\!])$
is defined to be
\begin{equation}
\label{eq:grading_A}
\Gr( f(q,\tau',\chi,z) \phi_i )
= \left(\left(\cE + z\parfrac{}{z}\right) f(q,\tau',\chi,z)\right) \phi_i +
\frac{\deg \phi_i}{2} f(q,\tau',\chi,z) \phi_i,
\end{equation}
where $f(q,\tau',\chi,z) \in R_\T[z][\![\Laa_+]\!][\![\tau']\!]$.
This is compatible with the quantum connection in the sense that
\[
\left[ \nabla_{\parfrac{}{\tau^i}}, \Gr \right] =
\left(1 - \frac{\deg\phi_i}{2}\right) \nabla_{\parfrac{}{\tau^i}}
\left( = \nabla_{\big[\parfrac{}{\tau^i},\cE\big]} \right).
\]
Let $P$ denote the pairing between the fibres of the bundle
\eqref{eq:quantum_D-mod_bundle} at $(q,\tau',\chi,-z)$
and at $(q,\tau',\chi,z)$
induced by the orbifold Poincar\'e pairing.
This gives the following $R_\T[\![\Laa_+]\!][\![\tau']\!]$-bilinear
pairing on the module
$H^*_{\CR,\T}(\frX)[z][\![\Laa_+]\!][\![\tau']\!]$:
\begin{equation}
\label{eq:pairing_A}
P(\alpha, \beta) := (\alpha(-z), \beta(z)) \in S_\T[z][\![\Laa_+]\!][\![\tau']\!],
\end{equation}
where $(\cdot,\cdot)$ denotes the orbifold Poincar\'e pairing~\eqref{eq:orb_Poincare_pairing}.
The pairing $P$ satisfies
\begin{gather*}
P(f(q,\tau',\chi,-z) \alpha, \beta) = P(\alpha,f(q,\tau',\chi,z)\beta)
= f(q,\tau',\chi,z) P(\alpha,\beta), \\
P(\beta,\alpha) = P(\alpha,\beta) |_{z\to -z}
\end{gather*}
for $f(q,\tau',\chi,z) \in R_\T[z][\![\Laa_+]\!][\![\tau']\!]$
and the compatibility equations with $\nabla$ and $\cE$:
\begin{gather*}
\parfrac{}{\tau^i} P(\alpha,\beta ) =
P\big(\tnabla_{\parfrac{}{\tau^i}} \alpha, \beta\big)
+ P\big(\alpha, \nabla_{\parfrac{}{\tau^i}} \beta \big), \\
\left( z\parfrac{}{z} + \cE\right) P(\alpha,\beta)
 = P(\Gr \alpha,\beta) + P(\alpha,\Gr \beta) - (\dim \frX) P(\alpha,\beta),
\end{gather*}
where $\tnabla= \nabla|_{z\to -z}$.
The structures $\nabla$, $\Gr$, $P$ are equivariant
with respect to the Galois symmetry of $H^2(\frX,\Z)/\Laa^\star$.
We call the $\big(H^2(\frX,\Z)/\Laa^\star\big)$-equivariant quadruple
\begin{equation}\label{eq:QDM}
\QDM_\T(\frX) :=
\left(H^*_{\CR,\T}(\frX)[z][\![\Laa_+]\!][\![\tau']\!], \nabla, \Gr, P \right)
\end{equation}
the \emph{$\T$-equivariant quantum D-module} of $\frX$.

\begin{Remark}
The flat connection $\nabla$ defines on the space
$H^*_{\CR,\T}(\frX)[z][\![\Laa_+]\!][\![\tau']\!]$
the structure of a module over the ring of differential operators
\[
R_\T[z][\![\Laa_+]\!][\![\tau']\!]\left\langle z\parfrac{}{\tau^0},
z\parfrac{}{\tau^1}, \dots,
z\parfrac{}{\tau^s} \right\rangle,
\]
where $z\parfrac{}{\tau^i}$ acts by
the connection $z \nabla_{\parfrac{}{\tau^i}}$.
This is often called a \emph{$z$-connection}
\cite{Sabbah:twistor_D_modules}.
\end{Remark}

The non-equivariant quantum D-module is the restriction of
the equivariant one to the origin $0\in \Spec R_\T = \Lie \T$.
It is a quadruple
\[
\QDM(\frX) := \big(
H^*_\CR(\frX) [z][\![\Laa_+]\!][\![\tau']\!],
\nabla, \Gr, P \big),
\]
where the pairing $P$ is defined only when $\frX$ is proper.
In the non-equivariant case, we can define the connection in the
$z$-direction by
\begin{equation}\label{eq:conn_z}
\nabla_{z\parfrac{}{z}} := \Gr - \nabla_{\cE} - \frac{\dim \frX}{2}
= z \parfrac{}{z} - \frac{1}{z} (E\star_\tau) + \mu
\end{equation}
and this preserves the pairing $P$.
Here $\cE$ denotes the non-equivariant limit of~\eqref{eq:Euler_A} (and thus does not contain the term
$\sum\limits_{i=1}^n \chi_i\parfrac{}{\chi_i}$)
and
\[
E=c_1(\frX) + \sum_{i\in \{0,r+1,\dots,s\}}
\left(1-\frac{1}{2} \deg \phi_i\right) \tau^i \phi_i
\]
is the section of $\QDM(\frX)$ corresponding to $\cE$ and
$\mu \in \End(H^*_\CR(\frX))$ is the endomorphism
given by $\mu(\phi_i) =
\big(\frac{1}{2} \deg\phi_i - \frac{1}{2} \dim \frX\big) \phi_i$.

\subsection[$\hGamma$-integral structure]{$\boldsymbol{\hGamma}$-integral structure}\label{subsec:integral}

The $\hGamma$-integral structure \cite[Section~2.4]{Iritani:Integral},
\cite[Proposition 3.1]{KKP:Hodge}
is an integral lattice in the space of flat sections of the quantum D-module.
The integral lattice is identified with the topological $K$-group of $\frX$.
We review its definition only in the non-equivariant case (see \cite[Section~3]{CIJ}
for the equivariant $\hGamma$-integral structure).

For simplicity, we assume that the quantum product $\star_\tau$
is convergent in a neighbourhood of~the large radius limit point.
Then the quantum connection is also analytic in the same neighbourhood;
it has no singularities on the intersection of the open subset
$\big[H^*_\CR(\frX)/H^2(\frX,\Z)\big]$
\linebreak $\subset \big[\cM_{\rm A}(\frX)/\big(H^2(\frX,\Z)/\Laa^\star\big)\big]$
and the convergence domain of $\star_\tau$.
Introduce the following \linebreak
$\End(H^*_\CR(\frX))$-valued function:
\[
L(\tau,z) \phi_i = e^{-\sigma/z}
\phi_i +
\sum_{k=0}^s
\sum_{\substack{n\ge 0, \, d\in \Laa_+ \\ (n,d)\neq (0,0)}}
\corr{\phi_k,\tau',\dots,\tau',\frac{e^{-\sigma/z}\phi_i}
{-z-\psi}}_{0,l+2,d} \frac{e^{\sigma \cdot d}}{l!} \phi^k,
\]
where we write $\tau = \sigma+ \tau' \in H^*_\CR(\frX)$,
$\sigma = \sum\limits_{i=1}^r \tau^i\phi_i \in H^2(\frX)$,
$\tau' = \tau^0 \phi_0 + \sum\limits_{i=r+1}^s \tau^i \phi_i$ as before,
$\{\phi^k\}$ is the dual basis of $\{\phi_k\}$ with respect to
the orbifold Poincar\'e pairing, and $1/(-z-\psi)$ should be expanded
in the series $\sum\limits_{k=0}^\infty (-z)^{-k-1} \psi^k$.
We also set
\[
z^{-\mu} z^{c_1(\frX)} := e^{-\mu \log z} e^{c_1(\frX) \log z}.
\]
Then by \cite[Proposition 2.4]{Iritani:Integral},
$L(\tau,z) z^{-\mu} z^{c_1(\frX)}
\phi_i$, $i=0,\dots,s$ form a basis of (multi-valued) $\nabla$-flat sections
(which are flat also in the $z$-direction with respect to~\eqref{eq:conn_z}),
i.e., $L(\tau,z) z^{-\mu} z^{c_1(\frX)}$ is a fundamental
solution of the quantum connection.

We introduce the Chern character and the $\hGamma$-class
for a smooth DM stack $\frX$. Let $\pi \colon I\frX \to \frX$ denote the
natural projection. Recall the decomposition $I\frX = \bigsqcup\limits_{v\in \Bx}
\frX_v$ into twisted sectors.
For an orbi-vector bundle $V$ on $\frX$, the stabilizer along
$\frX_v$ acts on $(\pi^*V)|_{\frX_v}$ and
decomposes it into the sum of eigenbundles
\[
\pi^*V|_{\frX_v} = \bigoplus_{0\le f <1} V_{v,f},
\]
where the stabilizer acts on $V_{v,f}$ by $\exp(2\pi\iu f)$.
The Chern character of $V$ is defined to be
\[
\tch(V) = \bigoplus_{v\in \Bx} \sum_{0\le f<1} e^{2\pi\iu f} \ch(V_{v,f})
\in H^*(I\frX).
\]
Consider now the tangent bundle $V=T\frX$ and let
$\delta_{v,f,j}$, $j=1,\dots,\rank(V_{v,f})$ denote the
Chern roots of $T\frX_{v,f}$.
The $\hGamma$-class of $\frX$ is defined to be
\[
\hGamma_\frX = \bigoplus_{v\in \Bx}
\prod_{0\le f<1} \prod_{j=1}^{\rank(T \frX_{v,f})}
\Gamma(1-f+ \delta_{v,f,j})
\in H^*(I\frX),
\]
where in the right-hand side we expand the Euler
$\Gamma$-function $\Gamma(z)$
in series at $z=1-f$. This is an algebraic cohomology class defined over
transcendental numbers.
\begin{Definition}[{\cite[Definition 2.9]{Iritani:Integral}}]
\label{def:s}
Let $K(\frX)$ denote the Grothendieck group of topological
orbi-vector bundles on $\frX$.
For $V \in K(\frX)$, we define a (multi-valued)
flat section $\frs_V(\tau,z)$ of the quantum D-module by
\[
\frs_V(\tau,z) = \frac{1}{(2\pi)^{\dim \frX/2}}
L(\tau,z) z^{-\mu} z^{c_1(\frX)} \big(
\hGamma_\frX \cup (2\pi\iu)^{\deg_0/2} \inv^*\tch(V) \big),
\]
where $\deg_0$ denotes the grading operator on $H^*(I\frX)$
\emph{without age shift}, i.e., $\deg_0(\alpha) = p \alpha$ for
$\alpha \in H^{p}(I\frX)$. The map $V \mapsto \frs_V$ defines
an integral lattice in the space of flat sections, which we call
the \emph{$\hGamma$-integral structure}.
\end{Definition}

Important properties of the $\hGamma$-integral structure are
as follows
\cite[Proposition 2.10]{Iritani:Integral}:
\begin{itemize}\itemsep=0pt
\item it is monodromy-invariant around the large radius limit point: we have
\begin{equation}\label{eq:Galois_line_tensor}
dg(\xi)^{-1} \frs_V(g(\xi)(\tau),z) =\frs_{V\otimes L_\xi}(\tau,z),
\end{equation}
where $L_\xi$ is the line bundle corresponding to $\xi\in H^2(\frX,\Z)$;
therefore it defines a $\Z$-local system underlying the quantum D-module;

\item it intertwines the Euler pairing with the orbifold Poincar\'e
pairing:
if $V_1$, $V_2$ are holomorphic orbi-vector bundles on $\frX$, we have
\begin{equation}\label{eq:Euler_Poincare}
\chi(V_1,V_2) = \big(\frs_{V_1} \big(\tau,e^{-\pi\iu} z\big), \frs_{V_2} (\tau,z) \big),
\end{equation}
where $\chi(V_1,V_2) =
\sum\limits_{i=0}^{\dim \frX}(-1)^i \dim \Ext^i(V_1,V_2)$.
When $V_1$, $V_2$ are $C^\infty$ orbi-vector bundles,
we can define $\chi(V_1,V_2)$ as
the index of a Dirac operator on $V_1^\vee \otimes V_2 \otimes
\Omega^{0,\bullet}_\frX$
and this formula holds under mild assumptions,
see \cite[Remark~2.8]{Iritani:Integral}.
Note that the right-hand side does not depend on $(\tau,z)$ since
$\frs_{V_1}$, $\frs_{V_2}$ are flat sections.
\end{itemize}
\begin{Remark}\label{rem:branch_L}
The fundamental solution $L(\tau,z) z^{-\mu} z^{c_1(\frX)}$ is
multi-valued as a function of
$(\tau,z)\in \big[H^*_\CR(\frX)/H^2(\frX;\Z)\big]\times \C^\times$,
but it has a standard determination when
$\tau$ is the image of a \emph{real} class in
$H^*_\CR(\frX;\R) = H^*(I\frX;\R)$ sufficiently close
to the large radius limit point and $z$ is \emph{positive real}.
We will sometimes use such a point as a base point.
\end{Remark}

\section{Global Landau--Ginzburg mirrors of toric DM stacks}
In this section, we construct a global Landau--Ginzburg model
(LG model) which is simultaneously mirror to several smooth toric DM stacks.
For background materials on toric (Deligne--Mumford) stacks,
we refer the reader to \cite{BCS, CLS,FMN, Iwanari1, Iwanari2, Tyomkin:tropical}.

\subsection{Toric data}\label{subsec:data}
Throughout the paper, we fix the data $(\bN,\Pi)$, where
\begin{itemize}\itemsep=0pt
\item $\bN$ is a finitely generated abelian group of rank $n$
(possibly having torsion) and
\item $\Pi$ is a full-dimensional, convex,
rational polyhedral cone in $\bN_\R := \bN\otimes \R$.
\end{itemize}
We do not require that $\Pi$ is strictly convex;
$\Pi$ will be the support of the fan of a smooth toric DM stack.
To construct a mirror family of LG models,
we choose a finite subset $S \subset \bN$ such that
\begin{itemize}\itemsep=0pt
\item $S$ generates the cone
$\Pi$ over $\R_{\ge 0}$, i.e., $\Pi = \sum\limits_{b\in S} \R_{\ge 0} b$.
\end{itemize}
The set $S$ specifies the set of monomials appearing
in the LG model.
We make the following technical assumption to ensure that
the base of the mirror family
has no generic stabilizers, that is, generic LG models have
no automorphisms of diagonal symmetry.
\begin{Assumption}
\label{assump:S_generates_N}
$S$ generates $\bN$ as an abelian group, i.e., $\bN = \Z S$.
\end{Assumption}
This is not an essential restriction.
In fact, if $S$ does not satisfy this assumption, we can first construct
the mirror family by taking a bigger set $S' \supset S$
satisfying the assumption,
and then restrict the family to the subspace of the base
corresponding to $S$ (then, the subspace has generic stabilizers).

\begin{Notation}
We write $\bN_{\rm tor}$ for the torsion part of $\bN$ and
write $\overline{\bN} = \bN/\bN_{\rm tor}$ for the torsion-free quotient.
For $v\in \bN$, we write $\overline{v}$ for the image of $v$
in $\overline{\bN}$.
The subscripts $\Q$, $\R$, etc.\ means the tensor product with
$\Q$, $\R$ over $\Z$, e.g., $\bN_\R = \bN \otimes_\Z \R$.
For subsets $A \subset \bN$ and $B \subset \bN_\R$, we write
$A \cap B := \{a\in A\colon \overline{a} \in B\}$.
\end{Notation}

\begin{Definition}\label{def:stacky_fan_adapted_to_S}
A \emph{stacky fan adapted to} $S$ is
a triple $\bSigma = (\bN,\Sigma, R)$ such that
\begin{itemize}\itemsep=0pt
\item[(i)] $\Sigma$ is a rational simplicial fan defined on the
vector space $\bN_\R$;
\item[(ii)] the support $|\Sigma| = \bigcup\limits_{\sigma\in \Sigma} \sigma$
of the fan $\Sigma$ equals $\Pi = \sum\limits_{b\in S} \R_{\ge 0} \overline{b}$;
\item[(iii)] there exists a strictly convex piecewise linear function
$\eta \colon \Pi\to \R$ which is linear on each cone of $\Sigma$;
\item[(iv)] $R \subset S$ is a subset such that
the map $R \ni b \mapsto \R_{\ge 0} \overline{b}$
gives a bijection between $R$ and the set
$\Sigma(1)$ of one-dimensional cones of $\Sigma$.
\end{itemize}
The data $(\bN,\Sigma,R)$ gives a \emph{stacky fan} in the sense of
Borisov, Chen and Smith \cite{BCS}.
We write~$\frX_\bSigma$ for the smooth toric DM stack
(toric stack for short) defined by $\bSigma$.
We also set $R(\bSigma) := R$ (``rays'') and
$G(\bSigma) := S \setminus R$ (``ghost rays'').
We denote by $\Fan(S)$ the set of stacky fans adapted to~$S$.
\end{Definition}
\begin{Remark}The above conditions (ii), (iii) imply that the corresponding toric DM
stack $\frX_\bSigma$ is \emph{semiprojective} \cite[Section~7.2]{CLS},
that is, the coarse moduli space
$X_\Sigma$ of $\frX_\bSigma$ is projective over the affine variety
$\Spec H^0(X_\Sigma,\cO)$
and has a torus fixed point.
Conversely, any semiprojective toric DM stack arises from
some toric data $(\bN, \Pi, S, \Sigma, R)$ in this section.
\end{Remark}

\subsubsection{Toric DM stacks}
\label{subsubsec:toric_stacks}
Let $\bSigma = (\bN,\Sigma,R)\in \Fan(S)$ be a stacky
fan adapted to $S$.
When we start from the stacky fan~$\bSigma$,
the set $G(\bSigma) = S \setminus R$ can be viewed as the
data of an \emph{extended stacky fan}
in the sense of Jiang~\cite{Jiang}. The extended stacky fan
is given by the pair
$(\bSigma, G(\bSigma))$ of the stacky fan $\bSigma$ and
the finite subset $G(\bSigma)\subset \bN \cap |\Sigma|$.
We recall a definition of the smooth toric DM stack $\frX_\bSigma$
\cite{BCS, Jiang} in terms of the
extended stacky fan.
The stacky fan $\bSigma$ defines the \emph{fan sequence}
\begin{equation}
\label{eq:fanseq}
\begin{CD}
0 @>>> \LL^\bSigma @>>> \Z^{R(\bSigma)} @>{\beta(\bSigma)}>> \bN,
\end{CD}
\end{equation}
where $\beta(\bSigma)\colon \Z^{R(\bSigma)} \to \bN$
sends the basis $e_b \in \Z^{R(\bSigma)}$
corresponding to $b\in R(\bSigma)$ to $b\in \bN$
and $\LL^\bSigma := \Ker(\beta(\bSigma))$.
The dual sequence
\begin{equation}\label{eq:divseq}
\begin{CD}
0 @>>> \bM @>>> \big(\Z^{R(\bSigma)}\big)^\star @>>> \big(\LL^\bSigma\big)^\vee
\end{CD}
\end{equation}
is called the \emph{divisor sequence}, where
$\big(\LL^\bSigma\big)^\vee := H^1(\Cone(\beta(\bSigma))^\star)$
is the Gale dual of $\beta(\bSigma)$ (see~\cite{BCS}).
The \emph{extended fan sequence} is the sequence
\begin{equation}\label{eq:ext_fanseq}
\begin{CD}
0 @>>> \LL @>>> \Z^S @>{\beta}>> \bN @>>>0,
\end{CD}
\end{equation}
where the map $\beta \colon \Z^S \to \bN$ sends
the basis $e_b \in \Z^S$
corresponding to $b\in S$ to $b\in \bN$ and
$\LL:=\Ker(\beta)$.
Note that $\beta$ is surjective by Assumption~\ref{assump:S_generates_N}.
The dual sequence
\begin{equation}
\label{eq:ext_divseq}
\begin{CD}
0 @>>> \bM @>>> \big(\Z^S\big)^\star @>{D}>> \LL^\star
\end{CD}
\end{equation}
is called the \emph{extended divisor sequence}, where
$\bM := \Hom(\bN,\Z)$ and $D \colon \big(\Z^S\big)^\star \to
\LL^\star$ is dual to $\LL \to \Z^S$
and $\Cok(D) \cong \Ext^1(\bN,\Z)$.
The torus $\LL_{\C^\times} := \LL \otimes \C^\times$ acts on $\C^S$
via the natural map $\LL_{\C^\times}
\to (\C^\times)^S$ induced by $\LL \to \Z^S$.
The toric stack $\frX_\bSigma$ is defined as a GIT quotient
of $\C^S$ by the $\LL_{\C^\times}$-action.
Set
\[
\scrS_\bSigma := \{I \subset R(\bSigma)\colon \text{the cone spanned by $I$
belongs to $\Sigma$}\}
\]
and define the open subset $U_\bSigma \subset \C^S$ as
\[
U_\bSigma = \C^S \setminus \bigcup_{I\subset S, I \notin \scrS_\bSigma}
\C^{S\setminus I},
\]
where we regard $\C^I$ with $I\subset S$ as a co-ordinate subspace
of $\C^S$ (we set $\C^\varnothing = \{0\}$).
Since every element of $\scrS_\bSigma$ is contained in~$R(\bSigma)$,
we may also write
\[
U_\bSigma = \left (
\C^{R(\bSigma)} \setminus \bigcup_{I \subset R(\bSigma),
I \notin \scrS_\bSigma} \C^{R(\bSigma) \setminus I} \right) \times
(\C^\times)^{G(\bSigma)}.
\]
We define
\[
\frX_\bSigma := [U_\bSigma/\LL_{\C^\times}].
\]
Since $\LL_{\C^\times}$ acts on $U_\bSigma$ with at most finite stabilizers,
$\frX_\bSigma$ is a smooth DM stack.
The toric stack $\frX_\bSigma$ depends only on $\bSigma =(\bN,\Sigma,R)$
and does not depend on the choice of the extension
$G(\bSigma) = S\setminus R$ (see \cite{Jiang}).
The coarse moduli space $X_\Sigma$ of $\frX_\bSigma$ is the toric
variety associated with the fan $\Sigma$.
The $(\C^\times)^S$-action on $U_\bSigma$ induces
the $\T$-action on $\frX_\bSigma$, where $\T$ is the torus
\[
\T := (\C^\times)^S/\Image\big(\LL_{\C^\times} \to (\C^\times)^S\big)
\cong \bN \otimes \C^\times.
\]

\subsubsection{Picard group and second (co)homology}
\label{subsubsec:Pic_2nd}
A character $\xi \in \Hom(\LL_{\C^\times},\C^\times)
=\LL^\star$ of $\LL_{\C^\times}$ defines a line bundle
$L_\xi=[(U_\bSigma\times \C)/\LL_{\C^\times}]$ over~$\frX_\bSigma$,
where $\LL_{\C^\times}$ acts on the second factor $\C$ by the character $\xi$.
We denote by $D_b := D(e_b^\star)$
the image of the standard basis $e_b^\star \in \big(\Z^S\big)^\star$ under~$D$.
We can see that $D_b$ with $b\in G(\bSigma)$ yields
the trivial line bundle on $\frX_\bSigma$, and the correspondence
$\xi \mapsto L_\xi$ gives the identification
\begin{equation}\label{eq:Pic}
\LL^\star\Big/\sum_{b\in G(\bSigma)} \Z D_b \cong \Pic(\frX_\bSigma).
\end{equation}
This group is isomorphic to the Gale dual $\big(\LL^\bSigma\big)^\vee$
of $\beta(\bSigma)$ appearing in~\eqref{eq:divseq}.
The torsion free quotient of $\Pic(\frX_\bSigma)$ can be identified
with the ordinary dual $\big(\LL^\bSigma\big)^\star$.
Moreover we have the identifications
\begin{gather}
H^2(\frX_\bSigma,\Q) \cong \big(\LL^\bSigma_\Q\big)^\star, \nonumber\\
H^2_\T(\frX_\bSigma,\Q) \cong \big(\Q^{R(\bSigma)}\big)^\star, \nonumber\\
H^2_\T(\pt,\Q) \cong \bM_\Q,\label{eq:second_coh}
\end{gather}
so that the divisor sequence \eqref{eq:divseq} over $\Q$ is identified with
\[
\begin{CD}
0 @>>> H^2_\T(\pt,\Q) @>>> H^2_\T(\frX_\bSigma,\Q)
@>>> H^2(\frX_\bSigma,\Q) @>>> 0.
\end{CD}
\]
We write $\overline{D}_b$ for the image of $D_b\in \LL^\star$
in $\big(\LL^\bSigma_\Q\big)^\star \cong H^2(\frX_\bSigma,\Q)$.
This is the class of a toric divisor.
We also write
\[
R_\T := H^*_\T(\pt,\C) \cong \Sym(\bM_\C)
\]
for the $\T$-equivariant cohomology of a point.

\subsubsection{Orbifold cohomology}\label{subsubsec:orb_coh}
For a cone $\sigma$ of $\Sigma$, we introduce
$\Bx(\sigma) \subset \bN$ as
\[
\Bx(\sigma) = \bigg\{v \in \bN \colon
\text{$\overline{v}$ is of the form
$\sum\limits_{b \in \sigma\cap R(\bSigma)} c_b \overline{b}$
for some $c_b\in [0,1)$} \bigg\}
\]
and set $\Bx(\bSigma) = \bigcup\limits_{\sigma \in \Sigma} \Bx(\sigma)$.
The set $\Bx(\bSigma)$ parametrizes connected components of the
inertia stack $I\frX_\bSigma$~\cite{BCS}.
We write $\frX_{\bSigma,v}$ for the component of $I\frX_\bSigma$
corresponding to $v\in \Bx(\bSigma)$.
The \emph{age} of a box element $v\in \Bx(\sigma)$ is given by
$\age(v) := \sum\limits_{b\in R(\bSigma) \cap \sigma} c_b$
when we write $\overline{v} = \sum\limits_{b\in R(\bSigma) \cap \sigma}
c_b \overline{b}$ with $c_b\ge 0$.
The \emph{orbifold cohomology}
$H_{\CR}^*(\frX_\bSigma)$ of Chen and Ruan~\cite{Chen-Ruan:new_coh}
is given by (as a graded vector space):
\begin{equation}\label{eq:orb_coh_toric}
H_{\CR}^*(\frX_\bSigma) =
\bigoplus_{v \in \Bx(\bSigma)} H^{*-2 \age (v)} (\frX_{\bSigma,v},\C).
\end{equation}
As a ring, it is generated by the fundamental classes $\unit_v$
on $I \frX_{\bSigma,v}$
with $v\in \Bx(\bSigma)$ and the toric divisor classes
$\overline{D}_b \in H^2(\frX_\bSigma,\Q)$
with $b\in R(\bSigma)$. Here we regard $\overline{D}_b$ as a class
supported on the untwisted sector $\frX_{\bSigma,0} = \frX_\bSigma$.
The $\T$-equivariant orbifold cohomology $H_{\CR,\T}^*(\frX_\bSigma)$
is defined by replacing each factor in the right-hand side of
\eqref{eq:orb_coh_toric}
with the $\T$-equivariant cohomology $H_\T^{*-2\age(v)} (\frX_{\bSigma,v})$.

\subsection{Landau--Ginzburg model}
\label{subsec:LG}
We construct a global family of mirror LG models
which are simultaneously
mirror to all the toric stacks $\frX_\bSigma$ with
$\bSigma \in \Fan(S)$.

The \emph{uncompactified LG model}
\cite{Givental:ICM, Hori-Vafa, Iritani:Integral} is the family of tori
\[
\begin{CD}
\Hom(\bN,\C^\times) @>>> (\C^\times)^S \\
@. @VV{\pr}V \\
@. \LL^\star \otimes \C^\times
\end{CD}
\]
obtained from the extended fan
sequence \eqref{eq:ext_fanseq} by applying
$\Hom(-,\C^\times)$,
together with the function $F \colon (\C^\times)^S \to \C$
\[
F = \sum_{b\in S} u_b,
\]
where $u_b$ denotes the $\C^\times$-valued co-ordinate on $(\C^\times)^S$
given by the projection to the $b$th factor.

We shall partially compactify this family to include all
the large radius limit points of $\frX_\bSigma$ with $\bSigma \in \Fan(S)$.
We construct partial compactifications of $(\C^\times)^S$ and
$\LL^\star \otimes \C^\times$
as possibly \emph{singular} toric DM stacks
in the sense of Tyomkin \cite{Tyomkin:tropical}.
According to Tyomkin \cite[Section~4.1]{Tyomkin:tropical},
a singular toric DM stack can be described by
\emph{toric stacky data} $(L,\Xi,\daleth)$ such that:
\begin{itemize}\itemsep=0pt
\item $L$ is a finitely generated free abelian group;
\item $\Xi$ is a (not necessarily simplicial) rational fan on $L_\R = L\otimes \R$;
\item $\daleth \subset |\Xi|$ is a subset in the support $|\Xi|$ of $\Xi$
such that for each cone $\sigma \in \Xi$, there exists a finite
index sublattice $L(\sigma) \subset L$ such that $\daleth \cap \sigma =
L(\sigma) \cap \sigma$. We call $\daleth$ the \emph{integral structure}
of the toric stacky data.
\end{itemize}
Note that this is a generalization of a stacky fan $(\bN,\Sigma,R)$
of Borisov, Chen and Smith \cite{BCS} when the group $\bN$ has no torsion.
For a given stacky fan $(\bN,\Sigma,R)$ with free $\bN$,
we can assign a toric stacky data $(\bN,\Sigma,\daleth)$ by
taking $\daleth$ to be the union of the monoids $\sigma_\Z
= \Z_{\ge 0} (R \cap \sigma)$ for all $\sigma \in \Sigma$.
Tyomkin constructed a singular toric DM stack
from $(L,\Xi,\daleth)$ by gluing affine charts; its coarse moduli
space is the toric variety $X_\Xi$ associated with the fan $\Xi$.
We refer the reader to \cite[Section~4.1]{Tyomkin:tropical} for the details
(the construction of the affine charts in our case will be reviewed in
Section~\ref{subsec:localchart_LG}).

\begin{Notation}
For $c =(c_b)_{b\in S} \in \big(\R^S\big)^\star$, we define a convex
piecewise linear function $\eta_c \colon \Pi \to \R$ by
\[
\eta_c(v) := \max\left\{
\varphi(v) \colon \varphi \in \bM_\R=\Hom(\bN_\R,\R),
\, \varphi(\overline{b}) \le c_b \,
(\forall\, b\in S) \right\}.
\]
This is well-defined when there exists $\varphi\in \bM_\R$
such that $\varphi(\overline{b}) \le c_b$ for all $b\in S$;
in particular if $c_b \ge 0$ for all $b\in S$.
The graph of $\eta_c$ is the union of ``lower faces'' of the
convex cone in
$\bN_\R \oplus \R$ generated by $(\overline{b},c_b)$, $b\in S$
and $(0,1)$.
\end{Notation}
For a stacky fan $\bSigma=(\bN,\Sigma,R)$ adapted to $S$,
we define a full-dimensional strictly-convex cone
$\CPL_+(\bSigma) \subset \big(\R^S\big)^\star$ by
\begin{equation}
\label{eq:CPL}
\CPL_+(\bSigma) := \left\{c \in
\big(\R^S\big)^\star \colon
\begin{matrix}
\text{$c_b \ge 0$ $(\forall\, b\in S)$,
$\eta_c$ is linear on each cone of $\Sigma$}, \\
\text{$\eta_c(\overline{b}) = c_b$ for $b\in R(\bSigma)$}
\end{matrix}
\right\},
\end{equation}
where $\CPL$ stands for ``convex piecewise linear'' (notation borrowed
from~\cite{Oda-Park}) and the sub\-script~$+$ means non-negative.
Note that $\eta_c$ in the definition is determined only by
$\{c_b\colon b\in R(\bSigma)\}$. Note also
that $c_b \ge \eta_c(\overline{b})$ for $b\in G(\bSigma)$.
We define the integral structure $\tdaleth\subset \big(\R^S\big)^\star$ as
\[
\tdaleth:= \left\{c \in \big(\Z^S\big)^\star\colon
\text{$c_b \ge 0$ $(\forall\, b\in S)$, and $\eta_c(\overline{b})\in \Z$
$(\forall\, b \in \bN \cap \Pi)$}\right\}.
\]
Note that $\tdaleth \cap \CPL_+(\bSigma)$ equals the intersection
of the following sublattice of $\big(\Z^S\big)^\star$
\begin{equation}
\label{eq:PLZ}
\PLZ(\bSigma) :=
\left\{ c\in \big(\Z^S\big)^\star\colon
\text{$\forall\, \sigma \in \Sigma$, $\exists\, m_\sigma \in \bM$
s.t.~$m_\sigma(b) = c_b$ ($\forall\, b\in R(\bSigma) \cap \sigma$)}
\right\}
\end{equation}
with the cone $\CPL_+(\bSigma)$. We also define
\begin{equation}\label{eq:cpl}
\cpl(\bSigma) := D(\CPL_+(\bSigma)), \qquad
\daleth := D\big(\tdaleth\big),
\end{equation}
where $D \colon \big(\R^S\big)^\star \to \LL^\star_\R$ is the
map appearing in the extended divisor sequence \eqref{eq:ext_divseq}.
The cone $\cpl(\bSigma)$ consists of convex piecewise linear functions
(with respect to $\bSigma$) modulo linear functions.
It is easy to check that
$\CPL_+(\bSigma) = D^{-1}(\cpl(\bSigma)) \cap (\R_{\ge 0})^S$
and that $\daleth \cap \cpl(\bSigma)$ is the intersection of
the finite index sublattice $\plZ(\bSigma) := \PLZ(\bSigma)/\bM$ of $\LL^\star$
with the cone $\cpl(\bSigma)$.

\begin{Definition}[partially compactified LG model]\label{def:LG}
Let $\tXi$ be the fan in $\big(\R^S\big)^\star$ consisting of the maximal cones
$\CPL_+(\bSigma)$ with $\bSigma \in \Fan(S)$ and
their faces. Let $\Xi$ be the fan in $\LL^\star_\R$ consisting of the maximal cones
$\cpl(\bSigma)$ with $\bSigma \in \Fan(S)$ and their faces.
\begin{enumerate}\itemsep=0pt
\item Define $\cY$ to be the singular toric DM stack
corresponding to the toric stacky data \linebreak $\big(\big(\Z^S\big)^\star, \tXi, \tdaleth\big)$.
We call $\cY$ the \emph{total space} of the LG model.

\item Define $\cM$ to be the singular toric DM stack
corresponding to the toric stacky data $(\LL^\star, \Xi,\daleth)$.
We call $\cM$ the \emph{secondary toric stack}.

\item The map $D\colon \big(\Z^S\big)^\star \to \LL^\star$ defines
a map between these toric stacky data, and thus induces
a toric morphism $\pr\colon \cY \to \cM$.
Let $u_b \colon \cY \to \C$ with $b\in S$ denote the regular function defined
by $e_b\in \Z^S$, and define the LG potential to be
$F = \sum\limits_{b\in S} u_b$.
We call the pair $(\pr \colon \cY \to \cM, F)$
the (partially compactified) \emph{LG model} associated to~$S$.
\end{enumerate}
\end{Definition}

\begin{Definition}\label{def:large_radius_limit}
Each stacky fan $\bSigma \in \Fan(S)$ gives rise to a~torus-fixed point $0_\bSigma$ of the secondary toric stack $\cM$. We call $0_\bSigma$ the \emph{large radius limit point} of $\frX_\bSigma$.
\end{Definition}

\begin{Remark}\quad
\begin{enumerate}\itemsep=0pt
\item[(1)] The fan $\Xi$ on $\LL^\star_\R$ defined by the cones $\cpl(\bSigma)$
is called the \emph{secondary fan} or the \emph{GKZ fan} after the
work of Gelfand, Kapranov and Zelevinsky~\cite{GKZ:discriminants}
(see also Oda--Park~\cite{Oda-Park}).
The fan $\tXi$ gives a lift of the secondary fan $\Xi$ to $\big(\R^S\big)^\star$.
The support of $\tXi$ is the positive orthant $(\R_{\ge 0})^S$,
and therefore $\cY$ can be viewed as an iterated weighted blowup of $\C^S$.
The fibre of the map $D\colon |\tXi| = (\R_{\ge 0})^S \to \LL_\R^\star$ at
an interior point $\omega$ of $\cpl(\bSigma)\subset \LL^\star_\R$
can be identified with the image of the moment map
$\mu \colon X_\Sigma \to \bM_\R$
of the $\T$-action on
$X_\Sigma$ with respect to the reduced symplectic form associated with $\omega$.
Therefore the fan $\tXi$ can be viewed as the total space of
the moment polytope fibration over the secondary fan $\Xi$.

\item[(2)] Diemer, Katzarkov and Kerr \cite{DKK:symplectomorphism} introduced
a closely related (but slightly different) compactification of the
LG model in the case where~$S$ lies in a hyperplane of integral
distance one from the origin.
In this case (i.e., when~$S$ lies in a hyperplane of height one),
our space $\cY$ can be obtained from their \emph{total Lafforgue stack}~\cite{DKK:symplectomorphism} by contracting a divisor
(the zero-section), at least on the level of coarse moduli spaces.
Their \emph{secondary stack} \cite{DKK:symplectomorphism}
and our secondary toric stack are the same on the level of coarse
moduli spaces (the coarse moduli spaces are the toric variety
defined by the fan~$\Xi$), however it is not clear to the author if the stack structures are the same.
\end{enumerate}
\end{Remark}

\begin{Remark}
Cones of $\tXi$, $\Xi$ can be described more explicitly as follows.
A \emph{possibly degenerate fan}~\cite{Oda-Park} on $\bN_\R$
is a finite collection $\Sigma$ of convex (but not necessarily strictly
convex) rational polyhedral cones in $\bN_\R$ such that (1) if $\sigma\in \Sigma$
and $\tau$ is a face of $\sigma$, then $\tau \in \Sigma$; and
(2) if $\sigma, \tau \in \Sigma$ then the intersection
$\sigma \cap \tau$ is a common face of $\sigma$ and $\tau$.
Let $\Sigma$ be a possibly degenerate fan on $\bN_\R$.
Each cone $\sigma\in \Sigma$ contains the linear subspace
$V=\sigma \cap (-\sigma)$ as a face, and the linear subspace $V$
does not depend on $\sigma$. When $V=0$, $\Sigma$ is a fan in the usual sense.
A \emph{spanning set}~\cite{Oda-Park} of $\Sigma$
is a finite subset $R' \subset \bN$ such that each cone $\sigma \in \Sigma$
is generated by a subset of~$R'$ over $\R_{\ge 0}$.
Let $(\Sigma,R',\sigma)$ be a triple such that $\Sigma$ is a possibly
degenerate fan on $\bN_\R$ with support $|\Sigma| = \Pi$
which admits a strictly convex piecewise linear function $\eta \colon \Pi \to \R$
linear on each cone of $\Sigma$,
$R' \subset S$ is a spanning set of $\Sigma$, and $\sigma \in \Sigma$
is a cone. For such a triple, we define the cone $\CPL_+(\Sigma,R',\sigma)
\subset \big(\R^S\big)^\star$ as
\[
\CPL_+(\Sigma,R',\sigma) = \left\{
c\in \big(\R^S\big)^\star\colon
\begin{matrix}
c_b\ge 0 \, (\forall\, b\in S),
\ \text{$\eta_c$ is linear on each cone of $\Sigma$,} \\
\eta_c(\overline{b}) = c_b \, (\forall\, b\in R'), \
\eta_c|_\sigma = 0
\end{matrix}
\right\}.
\]
When $\bSigma = (\bN,\Sigma,R)$ is a stacky fan adapted to $S$, $R$ is a spanning
set for $\Sigma$ and
we have $\CPL_+(\bSigma) = \CPL_+(\Sigma,R,\{0\})$.
Then the fan $\tXi$ consists of the cones $\CPL_+(\Sigma,R',\sigma)$
and the fan~$\Xi$ consists of the cones $\cpl(\Sigma,R') =
D(\CPL_+(\Sigma,R',\sigma))$ (which are independent of $\sigma$).
\end{Remark}

\subsection{Extended refined fan sequence and extended Mori cone}\label{subsec:ext_refined_fanseq}
The \emph{refined fan sequence}~\cite{CCIT:MS} is an extension of the fan sequence~\eqref{eq:fanseq} by a finite group. In this section we describe an extended version of
the refined fan sequence for $\bSigma \in \Fan(S)$
(extended by ghost vectors $G(\bSigma) = S\setminus R(\bSigma)$).
This will be used to describe a local chart of the
global LG model $(\pr\colon \cY \to \cM,F)$.
\begin{Notation}\label{nota:Psi}
Define a function $\Psi^\bSigma \colon \Pi \to (\R_{\ge 0})^{S}$ as follows.
For $v\in \Pi$, we take a cone $\sigma \in \Sigma$ containing $v$,
and write $v = \sum\limits_{b\in R(\bSigma) \cap \sigma} c_b \overline{b}$.
Then $\Psi^\bSigma(v) = \big(\Psi^\bSigma_b(v)\big)_{b\in S}$ is given by
\[
\Psi^\bSigma_b(v)
= \begin{cases}
c_b, &\text{if $b\in R(\bSigma) \cap \sigma$}, \\
0, & \text{otherwise}.
\end{cases}
\]
The map $\Psi^\bSigma$ gives a section of the map $\beta \colon (\R_{\ge 0})^S \to \Pi$,
i.e., $\beta \circ \Psi^\bSigma = \id_\Pi$.
\end{Notation}
We define $\OO(\bSigma)\subset \Q^S \oplus \bN$ to be the
subgroup:
\[
\OO(\bSigma) :=
\sum_{v\in \bN \cap \Pi} \Z \big(\Psi^\bSigma(v), v\big)
+ \sum_{b\in G(\bSigma)} \Z (e_b, b)
\]
and define $\Laa(\bSigma)\subset \LL_\Q$ to be
\[
\Laa(\bSigma) := \big\{\lambda \in \Q^S\colon (\lambda,0) \in \OO(\bSigma)\big\}.
\]
Note that $\OO(\bSigma)$ is contained in
$\big\{(\lambda,v) \in \Q^S\oplus \bN \colon \beta(\lambda) = \overline{v}\big\}$.
These groups define the \emph{extended refined fan sequence}:
\begin{equation}\label{eq:ext_refined_fanseq}
\begin{CD}
0 @>>> \Laa(\bSigma) @>>> \OO(\bSigma) @>>> \bN @>>> 0,
\end{CD}
\end{equation}
where the map $\OO(\bSigma) \to \bN$ is given by the second projection.
This is compatible with the extended fan sequence
\eqref{eq:ext_fanseq} under the inclusions
$\LL \subset \Laa(\bSigma)$, $\Z^S \subset \OO(\bSigma)$,
where the second inclusion sends $e_b$ to $(e_b,b)$.
This sequence splits because the torsion part
of $\OO(\bSigma)$ is isomorphic to the torsion part $\bN_{\rm tor}$
of $\bN$.
The original \emph{refined fan sequence} in \cite{CCIT:MS} corresponds to the case
where $G(\bSigma) =\varnothing$; it is
the exact sequence
\begin{equation}
\label{eq:refined_fanseq}
\begin{CD}
0 @>>> \Laa^\bSigma @>>> \OO^{\bSigma} @>>> \bN @>>> 0,
\end{CD}
\end{equation}
where
\begin{gather*}
\OO^\bSigma := \OO(\bSigma)\cap \big( \Q^{R(\bSigma)} \oplus \bN\big)
= \sum_{v\in \Pi\cap \bN} \Z\big(\Psi^\bSigma(v),v\big), \\
\Laa^\bSigma := \Laa(\bSigma) \cap \Q^{R(\bSigma)}
= \big\{\lambda \in \Q^{R(\bSigma)}\colon (\lambda,0) \in \OO^\bSigma\big\}.
\end{gather*}
Note that $\Laa^\bSigma$ is a lattice in $\LL^\bSigma_\Q
= \LL^\bSigma \otimes \Q$, where $\LL^\bSigma$ is the lattice
appearing in the fan sequence~\eqref{eq:fanseq}.
For $b\in G(\bSigma)$, we define
$\delta^\bSigma_b : = e_b - \Psi^\bSigma(b) \in \Q^S$.
Then $\delta^\bSigma_b$ lies in $\Laa(\bSigma)$ and
we have the following decompositions:
\begin{gather}\label{eq:ext_decomp}
\OO(\bSigma) = \OO^{\bSigma} \oplus \bigoplus_{b\in G(\bSigma)}
\Z \delta^\bSigma_b, \qquad
\Laa(\bSigma) = \Laa^{\bSigma} \oplus \bigoplus_{b\in G(\bSigma)}
\Z \delta^\bSigma_b.
\end{gather}
\begin{Remark}[{\cite[Section~2.4]{CCIT:MS}}]
Under the identification $\LL^\bSigma_\Q \cong H_2(\frX_\bSigma,\Q)$
in \eqref{eq:second_coh}, the lattice
$\Laa^\bSigma\subset \LL^\bSigma_\Q$ contains classes of
all orbifold stable maps to $\frX_\bSigma$.
We also have an isomorphism $\big(\Laa^\bSigma\big)^\star \cong \Pic(X_\Sigma)$
\cite[Lemma~4.8]{CCIT:MS}, where $\Pic(X_\Sigma)$ is the Picard group
of the coarse moduli space~$X_\Sigma$.
On the other hand, we expect that $\OO^{\bSigma}$ is the set of classes
of orbi-discs in $H_2(\frX_\bSigma,L;\Q) \cong \Q^{R(\bSigma)}$
with boundaries
in a Lagrangian torus fibre $L \subset \frX_\bSigma$.
The notation $\OO$ indicates `open'.
\end{Remark}

We introduce the (extended) Mori cones and their open analogues.
Let $\Sigma(n)$ denote the set of $n$-dimensional (i.e., maximal) cones
of~$\Sigma$.
For $\sigma \in \Sigma(n)$, we define
\begin{gather}
\tC_{\bSigma,\sigma} := \big\{\lambda
\in \R^{S} \colon
\beta(\lambda) \in \sigma, \, \lambda_b \ge 0 \text{ for } b\notin
R(\bSigma) \cap \sigma \big\}, \nonumber\\
C_{\bSigma,\sigma} := \LL_\R \cap \tC_{\bSigma,\sigma}
= \{\lambda \in \LL_\R \colon \lambda_b \ge 0 \text{ for }
b\notin R(\bSigma) \cap \sigma \},\label{eq:tC_bSigma_sigma}
\end{gather}
where $\lambda_b$ with $b\in S$
denotes the $b$th component of $\lambda \in \R^S$.
We define the \emph{extended Mori cone}
$\hNE(\frX_\bSigma)$ and its open analogue (cone of ``open'' curves)
$\hOEf(\frX_\bSigma)$ by
\begin{gather*}
\hOEf(\frX_\bSigma) := \sum_{\sigma \in \Sigma(n)} \tC_{\bSigma,\sigma}, \qquad
\hNE(\frX_\bSigma) := \sum_{\sigma \in \Sigma(n)} C_{\bSigma,\sigma}
= \hOEf(\frX_\bSigma) \cap \LL_\R.
\end{gather*}
The unextended versions are given by
\begin{gather*}
\OEf(\frX_\bSigma) := \hOEf(\frX_\bSigma) \cap \R^{R(\bSigma)}, \qquad
\NE(\frX_\bSigma) := \hNE(\frX_\bSigma) \cap \LL^\bSigma_\R.
\end{gather*}
The cone $\NE(\frX_\bSigma)$ in $\LL^\bSigma_\R \cong H_2(\frX_\bSigma,\R)$
is the usual Mori cone, that is, the cone spanned by
effective curves in $\frX_\bSigma$.
The corresponding monoids are given as follows
\begin{gather}
\OO(\bSigma)_+ :=\big\{(\lambda,v) \in \OO(\bSigma) \colon
\lambda \in \hOEf(\frX_\bSigma) \big\},\nonumber\\
\Laa(\bSigma)_+ := \Laa(\bSigma) \cap \hNE(\frX_\bSigma).\label{eq:O+Laa+}
\end{gather}
Similarly, the unextended versions are given by
\begin{gather*}
\OO^\bSigma_+ := \OO(\bSigma)_+ \cap \big(\Q^{R(\bSigma)}\oplus \bN\big)
= \big\{(\lambda,v) \in \OO^\bSigma\colon \lambda \in \OEf(\frX_\bSigma)\big\},
\\
\Laa_+^\bSigma := \Laa(\bSigma)_+ \cap \Q^{R(\bSigma)} =
\Laa^\bSigma \cap \NE(\frX_\bSigma).
\end{gather*}
These cones and monoids are compatible with the decompositions in
\eqref{eq:ext_decomp} (cf.~\cite[Lemma~3.2]{Iritani:Integral}).
It is easy to check that
\begin{alignat}{3}
&\hOEf(\frX_\bSigma) = \OEf(\frX_\bSigma) +
\sum_{b\in G(\bSigma)} \R_{\ge 0} \delta_b^\bSigma,
\qquad && \OO(\bSigma)_+= \OO^\bSigma_+ +
\sum_{b\in G(\bSigma)} \Z_{\ge 0} \delta_b^\bSigma,&\nonumber \\
&\hNE(\frX_\bSigma) = \NE(\frX_\bSigma) +
\sum_{b\in G(\bSigma)} \R_{\ge 0} \delta_b^\bSigma,
\qquad&& \Laa(\bSigma)_+= \Laa^\bSigma_+ +
\sum_{b\in G(\bSigma)} \Z_{\ge 0} \delta_b^\bSigma.&\label{eq:cone_monoid_decomp}
\end{alignat}
The following lemma follows immediately from {\cite[Lemma~2.7]{CCIT:MS}}.
\begin{Lemma}\label{lem:OO_Laa_+}
Consider a natural map $\OO(\bSigma)_+ \to \bN \cap \Pi$
given by the second projection. The fibre of this map at $v \in \bN \cap \Pi$
equals $\big(\Psi^\bSigma(\overline{v}),v\big) + \Laa(\bSigma)_+$.
\end{Lemma}

We introduce a pairing between $\Pic^\st(\frX_\bSigma) :=
\Pic(\frX_\bSigma)/\Pic(X_\Sigma)$ and
the lattices $\OO(\bSigma)$, $\Laa(\bSigma)$, where $\Pic(X_\Sigma)$ denotes
the Picard group of the coarse moduli space $X_\Sigma$
of $\frX_\bSigma$.
This pairing corresponds to the Galois symmetry for quantum cohomology
in Section~\ref{subsec:Kaehler_moduli}.
Note that we have
\begin{gather}
\label{eq:dual_of_Picst}
\Laa(\bSigma)/\LL \cong \OO(\bSigma)/\Z^S
\cong \OO^\bSigma/\Z^{R(\bSigma)},
\end{gather}
where the first isomorphism follows from the comparison
of the extended fan sequence~\eqref{eq:ext_fanseq} and
the extended refined fan sequence~\eqref{eq:ext_refined_fanseq};
the second isomorphism follows from the fact that
$\OO(\bSigma) = \OO^\bSigma \oplus \bigoplus\limits_{b\in G(\bSigma)}
\Z (e_b,b)$.
Recall from~\eqref{eq:Pic} that $\Pic(\frX) \cong \LL^\star/\sum\limits_{b\in G(\bSigma)}
\Z D_b$.
Since we have $\lambda_b = D_b \cdot \lambda \in \Z$ for $b\in G(\bSigma)$
and $\lambda \in \Laa(\bSigma)$,
the natural pairing $\LL^\star \times (\LL_\Q/\LL) \to \Q/\Z$
induces the pairing
\[
\age \colon \
\Pic(\frX_\bSigma) \times (\Laa(\bSigma)/\LL) \to \Q/\Z.
\]
Through the identification
$\Laa(\bSigma)/\LL \cong \OO^\bSigma/\Z^{R(\bSigma)}$ above,
this pairing coincides with the age pairing
$\Pic(\frX_\bSigma) \times \big(\OO^\bSigma/\Z^{R(\bSigma)}\big) \to \Q/\Z$
introduced in \cite[Section~4.3]{CCIT:MS}.
It follows from \cite[Lemma~4.7]{CCIT:MS} that
the age pairing descends to a perfect pairing
\[
\age \colon \ \Pic^\st(\frX_\bSigma) \times (\Laa(\bSigma)/\LL) \to \Q/\Z.
\]
Via \eqref{eq:dual_of_Picst}, we obtain the following pairings:
\begin{alignat}{3}
& \Pic^\st(\frX_\bSigma) \times \OO(\bSigma) \to \C^\times, \qquad&&
(\xi, x) \mapsto e^{2\pi\iu \age(\xi,x)} , & \nonumber\\
& \Pic^\st(\frX_\bSigma) \times \Laa(\bSigma) \to \C^\times, \qquad&&
(\xi,\lambda) \mapsto e^{2\pi\iu \age(\xi,\lambda)}.& \label{eq:Picst_action}
\end{alignat}

\begin{Remark}[{\cite[Lemma~4.7]{CCIT:MS}}]
For a box element $b\in \Bx(\bSigma)$,
$\age\big(\xi,\big(\Psi^\bSigma(b),b\big)\big)$ is the \emph{age} of the line bundle
$L_\xi$ corresponding to $\xi\in \Pic(\frX_\bSigma)$
along the sector corresponding to the box~$b$,
where $\big(\Psi^\bSigma(b),b\big) \in \OO^\bSigma$.
\end{Remark}

\subsection{Local charts of the LG model}\label{subsec:localchart_LG}
We describe the local charts of $\cY$ and $\cM$ corresponding to
$\bSigma \in \Fan(S)$.
By definition,
the local chart of $\cY$ corresponding to $\bSigma$ is given by
(see \cite[Section~4.1]{Tyomkin:tropical}):
\[
\cY_\bSigma = \big[
\Spec\left( \C[\CPL_+(\bSigma)^\vee \cap \PLZ(\bSigma)^\star]\right)
\big/ \cG_\bSigma \big],
\]
where
\begin{itemize}\itemsep=0pt
\item $\CPL_+(\bSigma)^\vee \subset \R^S$ is the dual cone
of $\CPL_+(\bSigma)\subset \big(\R^S\big)^\star$ (see~\eqref{eq:CPL}),
\item
$\PLZ(\bSigma)^\star = \Hom(\PLZ(\bSigma),\Z) \subset \Q^S$
is the dual lattice of $\PLZ(\bSigma)$ (see \eqref{eq:PLZ});
\item $\cG_\bSigma := \big(\Z^S\big)^\star/\PLZ(\bSigma)$ is a finite group;
it acts on $\C[\CPL_+(\bSigma)^\vee
\cap \PLZ(\bSigma)^\star]$ via the natural pairing
$\big(\Z^S\big)^\star \times \PLZ(\bSigma)^\star \to \Q \to \Q/\Z \subset \C^\times$.
\end{itemize}
The coarse moduli space of $\cY_\bSigma$ is $
\Spec\big(\C\big[\CPL_+(\bSigma)^\vee\cap \Z^S\big]\big)$.
Similarly, the local chart of $\cM$ corresponding to $\bSigma$ is given by:
\[
\cM_\bSigma = \left[
\Spec\left( \C[\cpl(\bSigma)^\vee \cap \plZ(\bSigma)^\star]\right)
\big/ \cG_\bSigma'\right ],
\]
where
\begin{itemize}\itemsep=0pt
\item $\cpl(\bSigma)^\vee\subset \LL_\R$ is the dual cone of $\cpl(\bSigma)$
(see \eqref{eq:cpl}),
\item $\plZ(\bSigma)^\star \subset \LL_\Q$ is the dual lattice of
$\plZ(\bSigma) = \PLZ(\bSigma)/\bM$ and
\item $\cG'_\bSigma :=\LL^\star/\plZ(\bSigma)$
acts on $\C[\cpl(\bSigma)^\vee \cap \plZ(\bSigma)^\star]$
via the natural pairing $\LL^\star \times \plZ(\bSigma)^\star
\to \Q \to \Q/\Z \subset \C^\times$.
\end{itemize}
The coarse moduli space of $\cM_\bSigma$
is $\Spec(\C[\cpl(\bSigma)^\vee \cap \LL])$.
Comparing the extended divisor sequence \eqref{eq:ext_divseq}
with the sequence $0 \to \bM \to \PLZ(\bSigma) \to \plZ(\bSigma)
\to 0$, we obtain the exact sequence
\begin{equation}\label{eq:G_Gprime}
\begin{CD}
0 @>>> \cG_\bSigma @>>> \cG_\bSigma' @>>>
\Ext^1(\bN,\Z) = \bN_{\rm tor}\sphat@>>> 0,
\end{CD}
\end{equation}
where $\bN_{\rm tor}\sphat = \Hom(\bN_{\rm tor},\C^\times)$
denotes the Pontrjagin dual of $\bN_{\rm tor}$.
We have $\Pic(\frX_\bSigma)\cong
\LL^\star/\sum\limits_{b\in G(\bSigma)} \Z D_b$ by~\eqref{eq:Pic},
and that the Picard group $\Pic(X_\Sigma)$ of the coarse moduli
space is given by
$\plZ(\bSigma)/\sum\limits_{b\in G(\bSigma)} \Z D_b$
(see \cite[Section~4.2]{CLS}).
Therefore
\[
\cG_\bSigma' \cong \Pic^\st(\frX_\bSigma) (= \Pic(\frX_\bSigma)/\Pic(X_\Sigma))
\]
and the above sequence~\eqref{eq:G_Gprime}
identifies $\cG_\bSigma$ with the subgroup
of $\Pic^\st(\frX_\bSigma)$ on which the generic stabilizers
$\bN_{\rm tor}$ of $\frX_\bSigma$ acts trivially.

We give another description of the local chart
$\cY_\bSigma \to \cM_\bSigma$, which shows that
the LG model on this chart is the same as the LG model
considered in \cite[Section~4]{CCIT:MS}.

\begin{Lemma}[duality of cones and lattices]\label{lem:dual_cone_lattice}\quad
\begin{enumerate}\itemsep=0pt
\item[$(1)$] The extended Mori cone $\hNE(\bSigma)$ is the dual cone
of $\cpl_+(\bSigma)$; similarly $\hOEf(\bSigma)$ is the dual cone of
$\CPL_+(\bSigma)$.
\item[$(2)$] The lattice $\overline{\OO(\bSigma)} :=
\OO(\bSigma)/\bN_{\rm tor}$ equals the dual lattice
$\PLZ(\bSigma)^\star$;
similarly $\Laa(\bSigma) \!=\! \plZ(\bSigma)^\star$.
Here we identify $\overline{\OO(\bSigma)}$ with the image of
the first projection $\OO(\bSigma) \to \Q^S$
$($recall that $\bN_{\rm tor}$ is the torsion part of
$\OO(\bSigma))$.
\end{enumerate}
\end{Lemma}
\begin{proof}First we prove the statement on the dual cones.
Observe that the cone $\CPL_+(\bSigma)$ can be written as the intersection
of the following simplicial cones $K_\sigma$ for all
maximal cones $\sigma \in \Sigma(n)$:
\[
K_\sigma=
\left \{c\in \big(\R^S\big)^\star\colon
\begin{matrix}
\text{$c_b \ge 0$ $(\forall\, b\in R(\bSigma) \cap \sigma)$, and
the linear function} \\
\text{$\varphi\colon \bN_\R \to \R$
defined by $\varphi(\overline{b}) = c_b$, $b\in R(\bSigma) \cap \sigma$} \\
\text{satisfies $\varphi(\overline{b}) \le c_b$ for all $b\in S$.}
\end{matrix}
\right \}.
\]
Recall that $\hOEf(\frX_\bSigma)$ is the sum of the
cones $\tC_{\bSigma,\sigma}$ defined in~\eqref{eq:tC_bSigma_sigma}.
Therefore,
in order to prove $\CPL_+(\bSigma)^\vee = \hOEf(\frX_\bSigma)$,
it suffices to show that $K_\sigma^\vee = \tC_{\bSigma,\sigma}$.
For $b\in S$, we write $\overline{b}
= \sum\limits_{b' \in R(\bSigma) \cap \sigma} A_{bb'} \overline{b'}$
with $A_{bb'} \in \Q$.
Then the cone $K_\sigma$ is defined by the linear inequalities:
\[
c_b
\ge
\begin{cases}
0, & \text{if $b \in R(\bSigma) \cap \sigma$}, \\
\displaystyle \sum_{b' \in R(\bSigma) \cap \sigma} A_{bb'} c_{b'}, &
\text{if $b\notin R(\bSigma)\cap \sigma$}.
\end{cases}
\]
On the other hand, for $\lambda \in \R^S$ and $c\in \big(\R^S\big)^\star$, we have
\begin{gather}\label{eq:c_lambda}
c\cdot \lambda = \sum_{b\in R(\bSigma)\cap \sigma} c_b
\bigg(\lambda_b + \sum_{b' \notin R(\bSigma) \cap \sigma}
A_{b' b} \lambda_{b'} \bigg)
+ \sum_{b'\notin R(\bSigma) \cap \sigma}
\bigg(c_{b'} -
\sum_{b \in R(\bSigma) \cap \sigma}
A_{b'b} c_{b}\bigg)\lambda_{b'}.\!\!
\end{gather}
Hence the dual cone $K_\sigma^\vee$ is defined by the inequalities
$\lambda_b \ge 0$ for $b\notin R(\bSigma) \cap \sigma$ and
$\lambda_b +\sum\limits_{b' \notin R(\bSigma) \cap \sigma}
A_{b'b} \lambda_{b'}\ge 0$ for $b \in R(\bSigma) \cap \sigma$.
The latter inequality is equivalent to $\beta(\lambda) \in \sigma$, and
thus $K_\sigma^\vee = \tC_{\bSigma,\sigma}$.
Hence $\CPL_+(\bSigma)^\vee = \hOEf(\frX_\bSigma)$.
The equality $\cpl(\bSigma)^\vee = \hNE(\frX_\bSigma)$ follows from this
and $D(\CPL_+(\bSigma)) = \cpl(\bSigma)$, $\hNE(\frX_\bSigma)
= \hOEf(\frX_\bSigma) \cap \LL_\R$.

Next we study the dual lattices of $\PLZ(\bSigma)$, $\plZ(\bSigma)$.
For $c\in \big(\R^S\big)^\star$ and a maximal cone $\sigma \in \Sigma(n)$,
let $m_{\sigma}(c) \in \bM_\R$ denote the unique element
satisfying $m_\sigma(c)\cdot b = c_b$ for all $b\in R(\bSigma) \cap \sigma$.
The lattice $\PLZ(\bSigma)$ is the intersection
of the following lattices $L_\sigma$ for all $\sigma \in \Sigma(n)$:
\begin{gather*}
L_\sigma = \big\{ c \in \big(\R^S\big)^\star\colon m_{\sigma}(c) \in \bM, \,
\text{$c_b \in \Z$ for all $b\notin R(\bSigma) \cap \sigma$}\big\}.
\end{gather*}
Therefore $\PLZ(\bSigma)^\star = \sum\limits_{\sigma \in \Sigma(n)} L_\sigma^\star$.
For $\lambda \in \R^S$ and $c\in \big(\R^S\big)^\star$, equation~\eqref{eq:c_lambda} can be rewritten as:
\[
c\cdot \lambda = m_{\sigma}(c) \cdot \beta(\lambda) +
\sum_{b'\notin R(\bSigma) \cap \sigma}
(c_{b'} - m_\sigma(c)\cdot b')\lambda_{b'}.
\]
Therefore we have
\[
L_\sigma^\star = \big\{ \lambda \in \R^S\colon \beta(\lambda) \in \overline{\bN}
= \bN/\bN_{\rm tor},
\text{$\lambda_b \in \Z$ for all $b\notin R(\bSigma) \cap \sigma$}\big\}.
\]
On the other hand,
$\overline{\OO(\bSigma)}=\sum\limits_{\sigma \in \Sigma(n)} L_\sigma^\star$
follows easily from the definition.
Thus $\PLZ(\bSigma)^\star = \overline{\OO(\bSigma)}$.
The equality $\plZ(\bSigma)^\star = \Laa(\bSigma)$ follows from this,
the extended refined fan sequence \eqref{eq:ext_refined_fanseq}
and $\plZ(\bSigma) = \PLZ(\bSigma)/\bM$.
\end{proof}

\begin{Proposition}
\label{prop:chart_LG}
The local chart $\cY_\bSigma \to \cM_\bSigma$ has the following
presentation:
\begin{gather*}
\cY_\bSigma \cong \big[ \Spec (\C[\OO(\bSigma)_+])
/ \Pic^\st(\frX_\bSigma)\big],\\
\cM_\bSigma \cong \big[ \Spec (\C[\Laa(\bSigma)_+])
/\Pic^\st(\frX_\bSigma)\big],
\end{gather*}
where $\Pic^\st(\frX_\bSigma) = \Pic(\frX_\bSigma)/\Pic(X_\Sigma)$
acts on $\C[\OO(\bSigma)_+]$ and $\C[\Laa(\bSigma)_+]$
via~\eqref{eq:Picst_action}.
\end{Proposition}
\begin{proof}
By the previous Lemma~\ref{lem:dual_cone_lattice}, we have
$\OO(\bSigma)_+/\bN_{\rm tor} = \CPL_+(\bSigma)^\vee
\cap \PLZ(\bSigma)^\star$ and
$\Laa(\bSigma)_+ = \cpl(\bSigma)^\vee \cap \plZ(\bSigma)^\star$.
Therefore we have the natural maps
\begin{gather*}
\Spec\C\big[\CPL_+(\bSigma)^\vee \cap \PLZ(\bSigma)^\star\big]
 \hookrightarrow \Spec \C[\OO(\bSigma)_+], \\
\Spec \C\big[\cpl(\bSigma)^\vee \cap \plZ(\bSigma)^\star\big]
 \cong \Spec \C[\Laa(\bSigma)_+].
\end{gather*}
The first map is the inclusion of a connected component.
We can see that the $\cG_\bSigma' =
\LL^\star/\plZ(\bSigma) \cong \Pic^\st(\frX_\bSigma)$
action on $\C[\cpl(\bSigma)^\vee \cap \plZ(\bSigma)^\star]$
and the $\Pic^\st(\frX_\bSigma)$ action on
$\C[\Laa(\bSigma)_+]$ are the same (both are induced by
the pairing $\LL^\star \times \LL_\Q \to \Q$).
Also the $\Pic^\st(\frX_\bSigma)$ action on $\Spec \C[\OO(\bSigma)_+]$
induces a permutation of its connected components, and
the subgroup preserving the component
$\Spec \C[\CPL_+(\bSigma)^\vee \cap \PLZ(\bSigma)^\star]$
is identified with $\cG_\bSigma$ (via the sequence \eqref{eq:G_Gprime}).
The proposition follows.
\end{proof}

It follows from the above proposition that the local chart
$\cY_\bSigma \to \cM_\bSigma$ is a quotient (by $\Pic^\st(\frX)$) of
the LG model considered in \cite[Section~4]{CCIT:MS}.
Indeed, the decompositions \eqref{eq:cone_monoid_decomp}
induce the isomorphisms
\begin{gather}
\Spec \C[\OO(\bSigma)_+] \cong \Spec \C\big[\OO^\bSigma_+\big]
\times \C^{G(\bSigma)},\nonumber \\
\Spec \C[\Laa(\bSigma)_+] \cong \Spec \C\big[\Laa^\bSigma_+\big]
\times \C^{G(\bSigma)}\label{eq:LG_local_chart}
\end{gather}
and the structure map $\cY_\bSigma \to \cM_\bSigma$ is induced by
the natural map $\Spec \C\big[\OO^\bSigma_+\big]
\to \Spec \C\big[\Laa^\bSigma_+\big]$.
In~\cite{CCIT:MS}, the LG model was first introduced on
$\Spec \C\big[\OO^\bSigma_+\big] \to \Spec \C\big[\Laa^\bSigma_+\big]$
and then deformed over the parameter space $\C^{G(\bSigma)}$.
The total deformation family there is identified with the uniformizing
chart of
the LG model $(\pr\colon \cY_\bSigma \to \cM_\bSigma,F)$
in the present paper.
See also the expression~\eqref{eq:LG_pot_co-ordinates} in explicit co-ordinates and
Remark~\ref{rem:notation_difference} below.

\subsection{Co-ordinate system on the local chart}
\label{subsec:coord_localchart_LG}
Using the presentation in Proposition \ref{prop:chart_LG},
we introduce a convenient co-ordinate system on the local chart
$\cY_\bSigma \to \cM_\bSigma$.
For $(\lambda,v) \in \OO(\bSigma)$,
we write $u^{(\lambda,v)}$
for the corresponding element in $\C[\OO(\bSigma)]$.
We set
\[
u_b := u^{(e_b,b)}\in \C[\OO(\bSigma)_+], \qquad
q^\lambda := u^{(\lambda,0)}\in \C[\Laa(\bSigma)]
\]
for $b\in S$ and $\lambda \in \Laa(\bSigma)$.
When $\lambda$ lies in $\Laa^\bSigma\subset \Laa(\bSigma)$,
we also write $\ttq^\lambda$
for $q^\lambda$.
We choose a splitting $\varsigma \colon \bN \to \OO^\bSigma$
of the refined fan sequence~\eqref{eq:refined_fanseq}
of the form $\varsigma(v) = (\ovvarsigma(v),v)$,
where $\ovvarsigma \colon \bN \to \Q^{R(\bSigma)}$ defines
a splitting of the fan sequence \eqref{eq:fanseq} over $\Q$.
For $v\in \bN$ and $b\in S$, we define
\begin{gather*}
x^v := u^{\varsigma(v)} \in \C[\OO(\bSigma)], \\
t_b := q^{\delta^\bSigma_b} =
q^{e_b- \Psi^\bSigma(b)} =
u^{(e_b-\Psi^\bSigma(b),0)}
\in \C[\Laa(\bSigma)_+].
\end{gather*}
Note that $x^v$ does not necessarily belong to $\C[\OO(\bSigma)_+]$
and that $t_b = 1$ for $b\in R(\bSigma)$.
Then we have for $b\in S$,
\[
u_b = t_b \ttq^{\lambda(b)} x^b,
\]
where $\lambda(b) := \Psi^\bSigma(b) - \ovvarsigma(b)
\in \Laa^\bSigma$.
Note that $\ttq^{\lambda(b)}x^b = u^{(\Psi^\bSigma(b),b)}
\in \C\big[\OO^\bSigma_+\big]$.
We can regard $q=(\ttq,t)$ as co-ordinates on the base $\cM_\bSigma$
and $x$ as co-ordinates on fibres of $\cY_\bSigma \to \cM_\bSigma$.
The LG potential
\begin{equation}\label{eq:LG_pot_co-ordinates}
F = \sum_{b\in S} u_b = \sum_{b\in R(\bSigma)} \ttq^{\lambda(b)} x^b
+ \sum_{b\in G(\bSigma)} t_b \ttq^{\lambda(b)} x^b
\end{equation}
can then be viewed as a family of Laurent polynomials in $x$ with
the fixed set $S$ of exponents.
Furthermore, we use the following co-ordinate expressions when necessary.
\begin{itemize}\itemsep=0pt
\item Choosing an isomorphism $\bN \cong \Z^n \times \bN_{\rm tor}$,
we write $x^b = x_1^{b_1} \cdots x_n^{b_n}x^\zeta$
when $b \in \bN$ corresponds to $(b_1,\dots,b_n,\zeta)
\in \Z^n \times \bN_{\rm tor}$; $x_1,\dots,x_n$ can be viewed
as co-ordinates along fibres of $\cY \to \cM$ and
$x^\zeta$ are roots of unity labelling connected components
of the fibre.
\item Choosing a $\Z$-basis $\lambda_1,\dots,\lambda_r$ of
$\Laa^\bSigma$, we write $\ttq_i = \ttq^{\lambda_i}
\in \C\big[\Laa^\bSigma\big]$;
then $\ttq_1,\dots,\ttq_r$ and $t_b$, $b\in G(\bSigma)$
together
form a co-ordinate system $q=(\ttq,t)$ on the base.
\end{itemize}

\begin{Remark}\label{rem:notation_difference}
We compare the notation of \cite{CCIT:MS, Iritani:shift_mirror} with the
present one. In these papers, the LG potential was given in the form
\[
F(x;y) = \sum_{b\in S} y_b w_b
= \sum_{b\in S} y_b Q^{\lambda(b)} x^b,
\]
where $\{y_b\}_{b\in S}$
are deformation parameters.
In the present paper, we set\footnote{In \cite{CCIT:MS, Iritani:shift_mirror}, we also wrote
$\{y_1,\dots,y_m\}$ for $\{y_b\}_{b\in R(\bSigma)}$
with $m=|R(\bSigma)|$.}
$y_b=1$ for all $b\in R(\bSigma)$.
The variables $Q$ and the other variables $y_b$, $b\in G(\bSigma)$
correspond to our $\ttq_1,\dots,\ttq_r$ and $t_b$ with
$b\in G(\bSigma)$, and $w_b$ corresponds to our
$u^{(\Psi^\bSigma(b),b)} = \ttq^{\lambda(b)} x^b$.
(Note that $w_b$ \emph{does not} correspond to $u_b$ in the present paper.)
\end{Remark}

\begin{Remark}[{\cite[Section~4.1]{CCIT:MS}}]
\label{rem:flatness_pr}
By Lemma \ref{lem:OO_Laa_+}, we have
\[
\C[\OO(\bSigma)_+]
= \bigoplus_{v\in \bN\cap\Pi} \C[\Laa(\bSigma)_+]
u^{(\Psi^\bSigma(\overline{v}),v)}.
\]
In particular the family $\pr \colon \cY \to \cM$ is flat.
The fact that $\C[\OO(\bSigma)_+]$ is a free $\C[\Laa(\bSigma)_+]$-module
played an important role
in establishing a mirror isomorphism in \cite{CCIT:MS}.
\end{Remark}

\subsection{Examples}
We give examples of partially compactified LG models
for surface singularities.

\subsubsection[$A_1$-singularity resolution]{$\boldsymbol{A_1}$-singularity resolution}
\label{subsubsec:A1}
We take $\bN = \Z^2$ and $S = \{(-1,1), (1,1) , (0,1)\}$.
There are two stacky fans $\bSigma_1$, $\bSigma_2$ adapted to~$S$ as shown
in Fig.~\ref{fig:A1}.
The stacky fan $\bSigma_1$ consists of one maximal cone
and gives rise to the toric stack $\frX_1 = \big[\C^2/\mu_2\big]$
(the $A_1$ singularity)
and the fan $\bSigma_2$ consists of two maximal cones
and gives rise to the crepant resolution $\frX_2 = \cO_{\PP^1}(-2)$
of~$\C^2/\mu_2$.
\begin{figure}[htbp]\centering
\includegraphics{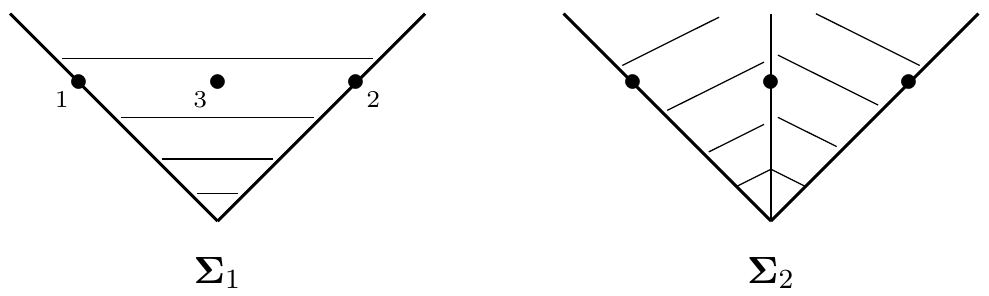}
\caption{The surface $A_1$-singularity (left) and its crepant resolution (right).}\label{fig:A1}
\end{figure}

The extended fan sequence is
\begin{equation}\label{eq:ext_fan_seq_A1}
\begin{CD}
0 @>>> \LL = \Z
@>{\left(\begin{smallmatrix} -1 \\ -1 \\ 2 \end{smallmatrix}\right)}>> \Z^3
@>{\left(\begin{smallmatrix} -1 & 1 & 0 \\
1 & 1 & 1
\end{smallmatrix}\right)}>> \bN = \Z^2 @>>>0
\end{CD}
\end{equation}
and the extended divisor sequence is its dual
\[
\begin{CD}
0 @>>> \bM @>>> \big(\Z^3\big)^\star
@>{D=(-1,-1,2)}>> \LL^\star = \Z
@>>> 0.
\end{CD}
\]
The fan $\tXi$ on the vector space $\big(\Z^3\big)^\star \otimes \R$
consists of the two
maximal cones:
\begin{gather*}
\CPL_+(\bSigma_1) = \big\{(c_1,c_2,c_3)\in (\R_{\ge 0})^3 \colon 2 c_3 \ge c_1 + c_2 \big\}, \\
\CPL_+(\bSigma_2) = \big\{(c_1,c_2,c_3) \in (\R_{\ge 0})^3 \colon 2 c_3 \le c_1 + c_2\big\}
\end{gather*}
and the integral structure $\tdaleth$ is given by
\[
\tdaleth = \big\{ (c_1,c_2,c_3) \in (\Z_{\ge 0})^3\colon
\min\big(c_3, \textstyle\frac{c_1+c_2}{2}\big)\in \Z \big\}.
\]
The map $D\colon \big(\Z^3\big)^\star \to \LL^\star$ induces the
secondary fan $\Xi$ on $\LL^\star_\R\cong \R$.
The fan $\Xi$ consists of the two maximal cones
$\cpl(\bSigma_1) = \R_{\ge 0}$, $\cpl(\bSigma_2) = \R_{\le 0}$
and the integral structure is given by
$\daleth = 2 \Z_{\ge 0} \cup \Z_{\le 0}$.
See Fig.~\ref{fig:A1_fan_LG}.
The extended refined fan sequence (see~\eqref{eq:ext_refined_fanseq})
associated to $\bSigma_1$ is given by
\[
\begin{CD}
0 @>>> \Laa(\bSigma_1) = \frac{1}{2} \Z @>{\left(
\begin{smallmatrix} -1 \\ -1 \\ 2 \end{smallmatrix}\right)}>>
\OO(\bSigma_1) = \Z^3 + \Z
\begin{pmatrix} \frac{1}{2} \\ \overset{}{\frac{1}{2}}
\\ \overset{}{0} \end{pmatrix}
@>{\left(\begin{smallmatrix} -1 & 1 & 0 \\
1 & 1 & 1 \end{smallmatrix}\right)}>> \bN = \Z^2@>>> 0,
\end{CD}
\]
which is a refinement of~\eqref{eq:ext_fan_seq_A1}.
The monoids\footnote{Note that $\OO(\bSigma)_+ = \OO(\bSigma) \cap
\CPL_+(\bSigma)^\vee$ and $\Laa(\bSigma)_+ = \Laa(\bSigma)
\cap \cpl(\bSigma)^\vee$. \label{foot:monoid_cone}}
corresponding to the (open) Mori cone
are given by $\Laa(\bSigma_1)_+ = \frac{1}{2}\Z_{\ge 0}$
and $\OO(\bSigma_1)_+ = \Z_{\ge 0} \big\langle e_1, e_2, \frac{1}{2}(e_1+e_2),
\delta_3 \big\rangle$
with $\delta_3 = \big(-\frac{1}{2},-\frac{1}{2},1\big)$.
By Proposition~\ref{prop:chart_LG},
the chart $\cY_{\bSigma_1} \to \cM_{\bSigma_1}$
is given by
\[
\xymatrix{
\cY_{\bSigma_1} = \big[\big\{(u_1,u_2,v,t)\colon
u_1u_2 = v^2\big\}\big/\mu_2
\big] \ar[d]^{t} \ar[rrr]^>(0.8){F = u_1 + u_2 + t v \quad} & & & \C, \\
\cM_{\bSigma_1} = [\C /\mu_2 ] & & & }
\]
where $u_1$, $u_2$, $v$, $t$ correspond to $e_1,e_2,\frac{1}{2}(e_1+e_2),\delta_3
\in \OO_+(\bSigma_1)$
and $\Pic^{\st}(\frX_{\bSigma_1}) =
\OO(\bSigma_1)/\Z^3 \cong \mu_2$ acts on these variables by
$(u_1,u_2,v,t) \mapsto (u_1,u_2,-v,-t)$.
On the other hand, the extended refined fan sequence for $\bSigma_2$
is the same as the extended fan sequence~\eqref{eq:ext_fan_seq_A1}
and one has $\Laa(\bSigma_2)_+ = \Z_{\le 0}$ and
$\OO(\bSigma_2)_+ = \Z_{\ge 0} \langle e_1,e_2,e_3, (1,1,-2)\rangle $.
Therefore, the chart $\cY_{\bSigma_2} \to \cM_{\bSigma_2}$ is given by
\[
\xymatrix{
\cY_{\bSigma_2} = \big\{ (u_1,u_2, u_3, q)\colon u_1u_2 = q u_3^2 \big\}
\ar[d]^{q} \ar[rrr]^>(0.8){F=u_1+u_2+u_3\quad} &&& \C, \\
\cM_{\bSigma_2} = \C
}
\]
where $u_1$, $u_2$, $u_3$, $q$ correspond respectively to $e_1$, $e_2$, $e_3$, $(1,1,-2)$.
The global LG model $(\cY \to \cM,F)$ is given by gluing
these charts by $u_3 = v t$, $q=t^{-2}$.
The base space is given by $\cM=\PP(1,2)$.

\begin{figure}\centering
\includegraphics{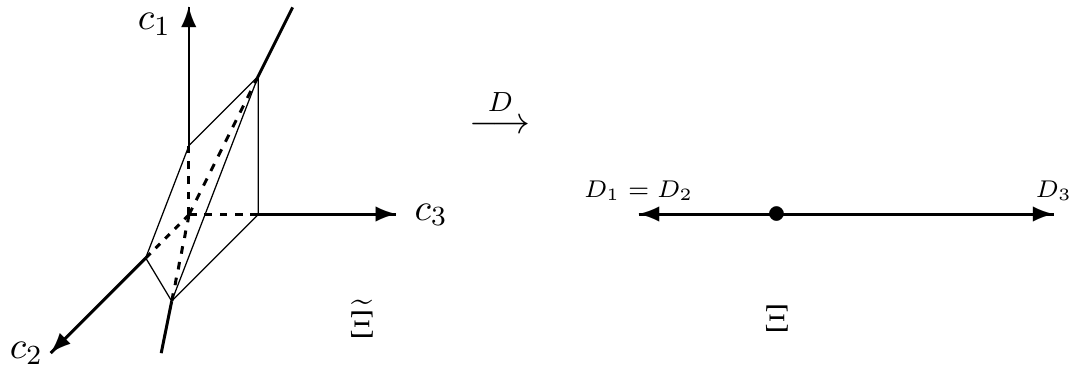}
\caption{The fans $\tXi$, $\Xi$ and the map $D$ between them.}\label{fig:A1_fan_LG}
\end{figure}

\subsubsection[Blowup of $\C^2$]{Blowup of $\boldsymbol{\C^2}$}\label{subsubsec:blowup}
We take $\bN = \Z^2$ and $S = \{(1,0),(0,1),(1,1)\}$.
The possible fan structures $\bSigma_1$, $\bSigma_2$
are shown in Fig.~\ref{fig:fan_blowup}.
The fan $\bSigma_1$ corresponds to $\C^2$ and
$\bSigma_2$ corresponds to the blowup $\Bl_0\big(\C^2\big)$ of $\C^2$
at the origin.
\begin{figure}[htbp]\centering
\includegraphics{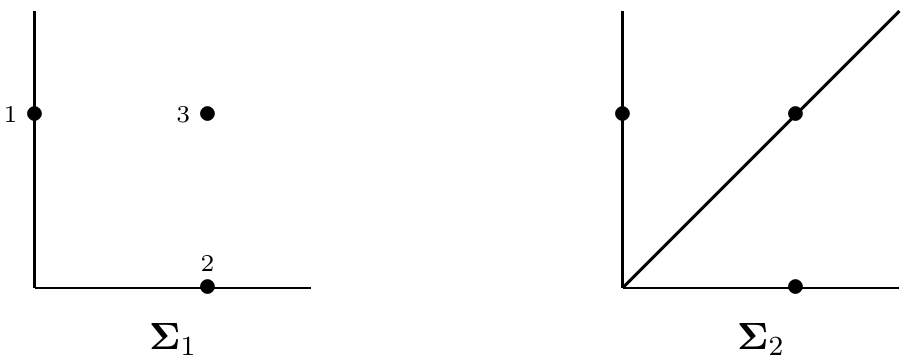}
\caption{$\C^2$ and its blowup at the origin.}\label{fig:fan_blowup}
\end{figure}

In this case, the LG model has no orbifold singularities.
The chart $\cY_{\bSigma_1} \to \cM_{\bSigma_1}$ is given by
\[
\xymatrix{
\cY_{\bSigma_1} = \{(u_1,u_2,u_3,t)\colon t u_1u_2 = u_3 \}
\ar[d]^{t} \ar[rrr]^>(0.8){F= u_1+ u_2+ u_3 \quad}
& & & \C. \\
\cM_{\bSigma_1} = \C & & &
}
\]
The chart $\cY_{\bSigma_2} \to \cM_{\bSigma_2}$ is given by
\[
\xymatrix{
\cY_{\bSigma_2} = \{(u_1,u_2,u_3,q) \colon u_1 u_2 = q u_3 \}
\ar[d]^{q} \ar[rrr]^>(0.8){F = u_1+u_2+u_3 \quad}
& & & \C. \\
\cM_{\bSigma_2} = \C & & &
}
\]
The two charts are glued by $t= q^{-1}$.

\subsubsection{Cyclic quotient singularity}\label{subsubsec:cyclic}
We take $\bN = \Z^2$ and $S = \{(0,1),(d,-1),(1,0)\}$ with $d\ge 3$. The case $d=1$ was considered in Section~\ref{subsubsec:blowup} (blowup of $\C^2$) and the case $d=2$ was considered in
Section~\ref{subsubsec:A1} ($A_1$-singularity).
As before, there are two fan structures
$\bSigma_1$, $\bSigma_2$ (see Fig.~\ref{fig:fan_cyclic_quot}).
The fan $\bSigma_1$ corresponds to $\big[\C^2/\mu_d\big]$
(of type $\frac{1}{d}(1,1)$) and $\bSigma_2$ corresponds to
its minimal resolution (the total space of~$\cO(-d)$ over~$\PP^1$).
\begin{figure}[htbp]\centering
\includegraphics{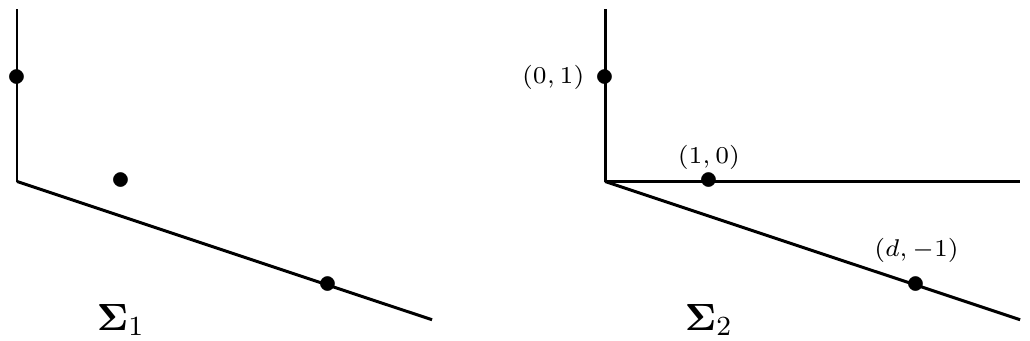}
\caption{Cyclic quotient singularity and its resolution.}\label{fig:fan_cyclic_quot}
\end{figure}

The chart $\cY_{\bSigma_1} \to \cM_{\bSigma_1}$ is given by
\[
\xymatrix
{
\cY_{\bSigma_1} = \big[
\big\{(u_1,u_2,v,t) \colon u_1u_2=v^d\big\} \big/\mu_d
\big] \ar[d]^{t} \ar[rrr]^>(0.8){F = u_1 + u_2 + t v \quad} & & & \C, \\
\cM_{\bSigma_1} = [\C/\mu_d] & & & }
\]
where $\Pic^\st(\frX_{\bSigma_1}) \cong
\mu_d$ acts by $(u_1,u_2,v,t) \mapsto (u_1,u_2,\zeta v, \zeta^{-1}t)$.
The chart $\cY_{\bSigma_2} \to \cM_{\bSigma_2}$ is given by
\[
\xymatrix
{
\cY_{\bSigma_2} =
\big\{(u_1,u_2,u_3,q) \colon u_1 u_2=u_3^d q \big\}
\ar[d]^{q} \ar[rrr]^>(0.8){F = u_1 + u_2 + u_3 \quad} & & & \C. \\
\cM_{\bSigma_2} = \C & & & }
\]
The two charts are glued by $u_3 = t v$ and $q =t^{-d}$.
After gluing, we get $\cM_{\bSigma_1} \cup \cM_{\bSigma_2}
= \PP(1,d)$. The pictures of the fans $\Xi$, $\tXi$ are similar to
Fig.~\ref{fig:A1_fan_LG}.

\section{Mirror symmetry}

In this section, we review mirror symmetry for smooth toric DM stacks
proved by Coates--Corti--Iritani--Tseng \cite{CCIT:mirrorthm, CCIT:MS}
and discuss its analytification.
We construct various versions (algebraic, completed, analytified)
of Brieskorn modules associated
with the Landau--Ginzburg model around the limit point $0_\bSigma$
and compare them with the quantum cohomology D-module
of the toric stack $\frX_\bSigma$.
We fix the data $(\bN,\Pi,S)$ from Section~\ref{subsec:data}.

\subsection{Brieskorn module}\label{subsec:Bri}
Let $(\pr\colon \cY \to \cM, F)$ be the partially compactified
LG model associated with the data $(\bN,\Pi,S)$
(see Definition~\ref{def:LG}).
We adapt the construction of the equivariant Brieskorn module
in \cite{CCIT:MS, Iritani:shift_mirror} to our context.

Note that the toric stacks $\cY$, $\cM$ have natural log structures defined by
their toric boundaries (see, e.g., \cite[Chapter~3]{Gross:tropical_book}).
With respect to these log structures,
the family $\cY \to \cM$ is log-smooth.
The sheaves of logarithmic one-forms and logarithmic vector fields on $\cY$
are globally free and given respectively by
\[
\Omega^1_{\cY} = \bigoplus_{b\in S} \cO_\cY
\frac{du_b}{u_b},
\qquad
\Theta_\cY = \bigoplus_{b\in S} \cO_\cY u_b \parfrac{}{u_b}.
\]
Let $x_1,\dots,x_n$ be the co-ordinates on fibres of $\cY \to \cM$
given by the choice of an isomorphism $\bN\cong \Z^n\times \bN_{\rm tor}$
(see Section~\ref{subsec:coord_localchart_LG}).
Then the sheaf of relative logarithmic $k$-forms\footnote
{The relative $k$-forms $\frac{dx_{i_1}}{x_{i_1}}\wedge \cdots \wedge
\frac{dx_{i_k}}{x_{i_k}}$ are independent of the choice of
a splitting $\varsigma$ in Section~\ref{subsec:coord_localchart_LG},
although the co-ordinates $x_i$ themselves depend on $\varsigma$.}
are
\[
\Omega^k_{\cY/\cM} = \bigoplus_{i_1<\cdots<i_k}
\cO_\cY \frac{dx_{i_1}}{x_{i_1}}\wedge \cdots \wedge
\frac{dx_{i_k}}{x_{i_k}}.
\]
The sheaf $\Theta_{\cY/\cM}$ of relative logarithmic vector fields
is generated by $x_1 \parfrac{}{x_1}, \dots, x_n\parfrac{}{x_n}$.
Let $\{\chi_1,\dots,\chi_n\}$ denote the basis of $\bM$ dual
to the chosen isomorphism $\overline{\bN} \cong \Z^n$; then
the relative vector field $x_i\parfrac{}{x_i}$ acts on functions
(on the chart $\cY_\bSigma$) as
\[
x_i \parfrac{}{x_i} \cdot u^{(\lambda,b)}
= (\chi_i\cdot b) u^{(\lambda,b)}
\qquad \text{for $(\lambda,b) \in \OO(\bSigma)_+$}.
\]
For $\xi \in \LL_\C^\star$, $\xi q \parfrac{}{q}$ denotes a vector field
on $\cM$ such that
\[
\xi q\parfrac{}{q} \cdot q^\lambda =
(\xi \cdot \lambda) q^\lambda
\qquad \text{for $\lambda \in \Laa(\bSigma)_+$}.
\]
We define a generator $\omega$ of $\Omega^n_{\cY/\cM}$ by
\[
\omega := \frac{1}{|\bN_{\rm tor}|} \frac{dx_1}{x_1}\wedge \cdots
\wedge \frac{dx_n}{x_n}.
\]
This is normalized so that
the integral over the maximal compact subgroup $\Hom\big(\bN,S^1\big)$
of $\Hom(\bN,\C^\times)$
equals~$(2\pi\iu)^n$.
Informally speaking, the non-equivariant Brieskorn module below
is a D-module on $\cM\times \C_z$ consisting of
certain cohomology classes of relative differential
forms $f \omega \in \pr_*\Omega^n_{\cY/\cM}$
such that oscillatory integrals
\[
[f \omega ] \longmapsto
\int_{\Gamma \subset \cY_q} e^{F/z} f(x,q,z) \omega
\]
are solutions to the D-module, where $\cY_q := \pr^{-1}(q)$.
In the equivariant case, the phase func\-tion~$F$ should be
replaced with $F - \sum\limits_{i=1}^n
\chi_i \log x_i$, see Remark~\ref{rem:twisted_de_Rham} below.

\begin{Definition}\label{def:Brieskorn}\quad
\begin{enumerate}\itemsep=0pt
\item[(1)]The \emph{equivariant Brieskorn module} $\Briequiv(F)$
is defined to be the $\cO_\cY[z]$-module $\Omega^n_{\cY/\cM}[z]
= \cO_\cY[z]\cdot \omega$ equipped
with the flat connection $
\nabla \colon \Briequiv(F) \to z^{-1}
\Briequiv(F) \otimes_{\cO_\cY} \Omega^1_\cY$
given by
\[
\nabla_V (f \omega) = \big(V(f) + z^{-1} V(F) f \big) \omega
\]
for $f\in \cO_\cY[z]$ and $V \in \Theta_\cY$.
We call $\nabla$ the \emph{Gauss--Manin connection}.
We let the equivariant parameter
$\chi_i \in \bM_\C \cong H^2_\T(\pt,\C)$
act on $\Briequiv(F)$ by $z \nabla_{x_i\parfrac{}{x_i}}$, i.e.,
\[
\chi_i \cdot f\omega := z \nabla_{x_i\parfrac{}{x_i}} (f\omega)
= \left( z x_i \parfrac{f}{x_i} +
\sum_{b\in S} (\chi_i \cdot b) u_b f \right) \omega.
\]
We shall often regard $\Briequiv(F)$ as a sheaf of modules over $\cM$ by
pushing it forward by $\pr\colon \cY\to \cM$; by abuse of notation we use
the same symbol to denote the pushed-forward sheaf.
The action of $\chi_i$ commutes with the $\cO_\cM[z]$-module structure
and thus $\Briequiv(F)$ has the structure of an $\cO_\cM\otimes R_\T[z]$-module
with $R_\T = H^*_\T(\pt,\C) \cong \Sym(\bM_\C)$.
The \emph{grading operator} $\Gr \in \End_\C(\Briequiv(F))$ is defined by
\begin{equation}\label{eq:grading_B}
\Gr (f\omega) = \left(
z\parfrac{f}{z} + \sum_{b\in S}u_b\parfrac{f}{u_b} \right) \omega
\end{equation}
for $f\in \cO_\cY[z]$.

\item[(2)] The (non-equivariant) \emph{Brieskorn module} $\Bri(F)$ is defined
to be the non-equivariant limit $\Briequiv(F)/\bM_\C \cdot \Briequiv(F)$
of the equivariant Brieskorn module.
This has the structure of an $\cO_\cM[z]$-module.
The flat connection and the grading operator
on $\Briequiv(F)$ descend to a flat connection
$\nabla \colon \Bri(F) \to z^{-1} \Bri(F) \otimes \Omega_\cM^1$,
also called the \emph{Gauss--Manin connection},
and an operator $\Gr \in \End_\C(\Bri(F))$.
\end{enumerate}
\end{Definition}

By definition, $\Briequiv(F)$ is isomorphic to the rank-one free module
$\cO_\cY[z]$
equipped with the flat connection $\nabla = d + d(F/z)\wedge$.
By pushing it forward by $\pr \colon \cY \to \cM$ and
forgetting the action of the fibre co-ordinates $x_1,\dots,x_n$,
we shall regard $\Briequiv(F)$ as an $\cO_\cM\otimes R_\T[z]$-module;
then $\Briequiv(F)$ is not of rank one as such.\footnote{This is not coherent as an $\cO_\cM\otimes R_\T[z]$-module in general;
we shall take its completion in the next section so that it has a finite
expected rank.}
We shall also regard $\Briequiv(F)$ as a flat connection (i.e., D-module)
over $\cM$.
First we regard it as a module over the ring of differential operators
\[
\scrD= \cO_\cM[z]\left\langle z u_b \parfrac{}{u_b}
\colon b\in S \right\rangle,
\]
where $z u_b \parfrac{}{u_b}$ act by
$z\nabla_{u_b\parfrac{}{u_b}}$.
Note that $\scrD$ contains $\cO_\cM \otimes R_\T[z]$ as its centre
via the map $\bM_\C \ni \chi_i \mapsto z x_i\parfrac{}{x_i}$.
By choosing a splitting $\big(\C^S\big)^\star \cong \bM_\C \oplus \LL_\C^\star$
of the extended divisor sequence~\eqref{eq:ext_divseq} tensored with $\C$, we can
lift a vector field $\xi q \parfrac{}{q}$ on $\cM$ given by $\xi \in \LL_\C^\star$
to a vector field $\hxi u\parfrac{}{u} = \sum\limits_{b\in S} \hxi_b
u_b\parfrac{}{u_b}$ on $\cY$, where
$\hxi = \big(\hxi_b\big)_{b\in S} \in \big(\C^S\big)^\star$ denotes the lift of
$\xi$ under the splitting.
By this splitting, we can regard $\Briequiv(F)$ as a module over
\[
\scrD \cong \cO_\cM \otimes R_\T[z]\left\langle z \xi q\parfrac{}{q}
\colon \xi \in \LL^\star_\C \right\rangle.
\]
A different choice of splittings shifts the action of
$z \xi q \parfrac{}{q} \in \Theta_\cM$ by an element of $\bM_\C$.
When the choice of a splitting is understood, we write
$z \nabla_{\xi q\parfrac{}{q}}$ for $z \nabla_{\hxi u\parfrac{}{u}}$.
The grading operator $\Gr$ satisfies
$[\Gr, z \nabla_{\xi q\parfrac{}{q}}] = z \nabla_{\xi q\parfrac{}{q}}$ and
\begin{align*}
\Gr( f(q,\chi,z) \Omega) & =
\left( \left(\cE + z\parfrac{}{z}\right) f(q,\chi,z)\right)
\Omega + f(q,\chi,z) \Gr(\Omega)
\end{align*}
for any $\Omega \in \Briequiv(F)$ and $f(q,z,\chi)
\in \cO_\cM\otimes R_\T[z]$, where $\cE$ is the Euler vector field
defined by
\begin{equation}
\label{eq:Euler_B}
\cE = \left(\sum_{b\in S} D_b\right) q\parfrac{}{q} +
\sum_{i=1}^n \chi_i \parfrac{}{\chi_i}.
\end{equation}
Recall here that $D_b\in \LL^\star$ is the image of
$e_i^\star \in \big(\Z^S\big)^\star$ under the map $D$ in \eqref{eq:ext_divseq}
and $\{\chi_1,\dots,\chi_n\}$
denotes a basis of $H^2_\T(\pt)$ so that $R_\T = \C[\chi_1,\dots,\chi_n]$.

\begin{Remark}The Brieskorn modules here were called \emph{Gauss--Manin systems}
in our previous papers \cite{CCIT:MS, Iritani:shift_mirror}.
We changed the name because the Gauss--Manin system usually
refers to the localization of the Brieskorn module by~$z$;
the Brieskorn modules were also called \emph{Brieskorn lattices} in \cite{Douai-Sabbah:I,Douai-Sabbah:II, Sabbah:tame, SaitoM:Brieskorn}.
\end{Remark}

\begin{Remark}\label{rem:Seidel}
The action of $x_1,\dots,x_n$ forgotten in the above process corresponds to
the Seidel representation (or shift operator) on quantum cohomology.
This defines the structure of a~\emph{difference module} with respect to
the equivariant parameters~$\chi_i$, i.e., the action of~$x_i$ shifts
$\chi_j$ as $\chi_j \mapsto \chi_j -\delta_{i,j} z$
(note that we have the commutation
relation $[\chi_i, x_j] = z \delta_{i,j} x_j$ as operators acting on
$\Briequiv(F)$).
See \cite{Iritani:shift_mirror}.
\end{Remark}

\begin{Remark}\label{rem:twisted_de_Rham}As done in \cite{CCIT:MS, Iritani:shift_mirror}, we can define
the equivariant and non-equivariant Brieskorn modules
in terms of the twisted (logarithmic) de Rham complex. We have
\begin{gather*}
\Briequiv(F) \cong \pr_*H^n(\Omega^\bullet_{\cY/\cM}[z][\chi_1,\dots,\chi_n],
z d + dF_\T \wedge), \\
\Bri(F) \cong \pr_*H^n(\Omega^\bullet_{\cY/\cM}[z], zd + dF\wedge)
\end{gather*}
with $F_\T= F -\sum\limits_{i=1}^n \chi_i \log x_i$, where
$H^n(-)$ means the cohomology sheaf of a complex of sheaves
(not the hypercohomology $R^n\pr_*$).
This definition involves the choice of co-ordinates
$x_1,\dots,x_n$, which corresponds to the choice of a splitting
as above.
To see that the first isomorphism holds, note that the $n$th cohomology is
the cokernel of $zd + dF_\T \wedge \colon \Omega^{n-1}_{\cY/\cM}[z][\chi]
\to \Omega^n_{\cY/\cM}[z][\chi]$ and the relations given
by $\Image(zd+dF_\T\wedge)$
define the action of $\chi_i$ on $\Omega^n_{\cY/\cM}[z]$.
The second isomorphism follows from the first.
\end{Remark}

\begin{Remark}
\label{rem:GMconn_coordinates}
Using the local co-ordinates $\ttq_1,\dots,\ttq_r$, $t_b$
with $b\in G(\bSigma)$
on $\cM$ from Section~\ref{subsec:coord_localchart_LG} and
the expression \eqref{eq:LG_pot_co-ordinates} for $F$,
the flat connection $\nabla$ of $\Bri(F)$ and $\Briequiv(F)$
can be written as
\begin{alignat*}{3}
&\nabla_{\ttq_i \parfrac{}{\ttq_i}} (f \omega)
=\left( \ttq_i\parfrac{f}{\ttq_i} +\frac{f}{z}
\sum_{b\in S} t_b \lambda(b)_i
\ttq^{\lambda(b)} x^b\right) \omega
\qquad && (1\le i\le r), & \\
& \nabla_{\parfrac{}{t_b}} (f \omega)
=\left( \parfrac{f}{t_b} + \frac{f}{z} \ttq^{\lambda(b)} x^b \right)
\omega \qquad && (b\in G(\bSigma)),
\end{alignat*}
where $f \in \cO_\cY[z]$ and $\lambda(b)_i$ denotes
the $i$th component of $\lambda(b) \in \Laa^\bSigma$
with respect to the chosen basis of $\Laa^\bSigma$.
Since $\ttq^{\lambda(b)} x^b = u^{(\Psi^\bSigma(b),b)}
\in \C[\OO(\bSigma)_+]$, it follows that $\nabla$
has no singularities along the divisor
$t_b = 0$ with $b\in G(\bSigma)$.
Therefore, a ``smaller'' log structure
(than the one given by toric boundaries) suffices to
describe the logarithmic singularities of $\nabla$.
Note that in the equivariant case, the choice of co-ordinates
$\ttq_i$, $t_b$, $x_i$ determines the splitting of the extended divisor
sequence.
\end{Remark}

\begin{Remark}\label{rem:conn_z}
In the non-equivariant Brieskorn module, the connection
$\nabla$ and the grading operator $\Gr$ together define
the connection $\nabla_{z\parfrac{}{z}} =
\Gr- \nabla_{\cE} - \dim \frX/2$ in the $z$-direction
as in the case of non-equivariant
quantum D-modules, where
$\cE = \big(\sum\limits_{b\in S} D_b\big) q\parfrac{}{q}$
denotes the non-equivariant Euler vector field.
Cf.~\eqref{eq:conn_z}.
\end{Remark}

\subsection{Completion and mirror isomorphism}
\label{subsec:completion}
We introduce a completion of the Brieskorn module
at the large radius limit point $0_\bSigma\in \cM_\bSigma$
of~$\frX_\bSigma$ (see Definition~\ref{def:large_radius_limit})
and recall a statement on mirror symmetry from~\cite{CCIT:MS}.

Let $\bSigma\in \Fan(S)$ be a stacky fan adapted to~$S$.
Let $\Briequiv(F)_\bSigma$, $\Bri(F)_\bSigma$ denote the
$\Pic^\st(\frX_\bSigma)$-equivariant
modules corresponding to $\Briequiv(F)$, $\Bri(F)$
on the affine chart $\cY_\bSigma \to \cM_\bSigma$
(see Proposition~\ref{prop:chart_LG}):
\begin{gather*}
\Briequiv(F)_\bSigma := \C[z][\OO(\bSigma)_+] \cdot \omega, \\
\Bri(F)_\bSigma := \Briequiv(F)_\bSigma\big/\bM_\C \cdot\Briequiv(F)_\bSigma.
\end{gather*}
Let $\frakm_\bSigma\subset \C[\Laa(\bSigma)_+]$
denote the maximal ideal corresponding to $0_\bSigma$,
i.e., the ideal generated by~$q^\lambda$ with $\lambda
\in \Laa(\bSigma)_+\setminus \{0\}$.
\begin{Definition}
The \emph{completed $($equivariant and non-equivariant$)$
Brieskorn modules at $0_\bSigma$} are defined to be
\begin{gather*}
\Briequiv(F)\sphat_\bSigma := \varprojlim_{k} \Briequiv(F)_\bSigma
/\frakm_\bSigma^k \Briequiv(F)_\bSigma, \\
\Bri(F)\sphat_\bSigma := \varprojlim_k
\Bri(F)_\bSigma /\frakm_\bSigma^k \Bri(F)_\bSigma.
\end{gather*}
The completed equivariant Brieskorn module $\Briequiv(F)\sphat$
is a free $R_\T[z][\![\Laa(\bSigma)_+]\!]$-module of rank
$\dim H^*_\CR(\frX)$ \cite[Theorem~4.26]{CCIT:MS}.
\end{Definition}

We write $\cM_\T = \cM \times \Spec R_\T=\cM\times \Lie \T$
and let
\[
\cM_{\T,\bSigma}\sphat := \Spf R_\T[\![\Laa(\bSigma)_+]\!]
\]
denote the formal neighbourhood of $0_\bSigma\times \Spec R_\T$
in $\cM_\T$.
We regard $\Briequiv(F)\sphat_\bSigma$
as a $\Pic^\st(\frX)$-equivariant module over $\cM_{\T,\bSigma}\sphat$.
We again choose a splitting $\varsigma \colon \bN \to \OO^\bSigma$ of
the refined fan sequence \eqref{eq:refined_fanseq}; via
the decomposition \eqref{eq:ext_decomp}, $\varsigma$ induces
a splitting of the extended refined fan sequence \eqref{eq:ext_refined_fanseq}
and that of the extended divisor sequence \eqref{eq:ext_divseq} over $\C$.
As explained in the previous section (Section~\ref{subsec:Bri}),
this splitting enables us to regard $\Briequiv(F)\sphat_\bSigma$
as a partial connection over $\cM_{\T,\bSigma}\sphat$.
(This partial connection was explicitly described in
Remark~\ref{rem:GMconn_coordinates} by choosing
co-ordinates $\ttq_i$, $t_b$, $x_i$ given by $\varsigma$.)

The splitting $\varsigma$ also defines a splitting
$H^2_{\T}(\frX_\bSigma) \cong H^2(\frX_\bSigma) \oplus H^2_\T(\pt)$
via the isomorphisms~\eqref{eq:second_coh}.
We choose an $R_\T$-basis $\{\phi_i\}_{i=0}^s$ of
$H_{\CR,\T}^*(\frX_\bSigma)$ so that the condition~\eqref{eq:basis_equiv_lift} is satisfied \emph{and} that
the splitting $H^2_\T(\frX_\bSigma) \cong \bigoplus\limits_{i=1}^r \C \phi_i
\oplus H^2_\T(\pt)$ given by this basis is the same as the splitting
induced by $\varsigma$.
This basis defines the equivariant K\"ahler moduli space
$\cM_{\rm A,\T}(\frX_\bSigma)$ and the equivariant
quantum D-module $\QDM_\T(\frX)$ \eqref{eq:QDM}, see
Sections~\ref{subsec:Kaehler_moduli}--\ref{subsec:QDM}.
Recall that $\QDM_\T(\frX)$
is an $H^2(\frX_\bSigma,\Z)/\big(\Laa^\bSigma\big)^\star$-equivariant module
over the formal neighbourhood of the origin in
$\cM_{\rm A,\T}(\frX_\bSigma)$; we denote this formal neighbourhood
by
\[
\cM_{\rm A,\T}(\frX_\bSigma)\sphat :=
\Spf R_\T\big[\!\big[\Laa^\bSigma_+\big]\!\big]
\big[\!\big[\tau^0,\tau^{r+1},\dots,\tau^s\big]\!\big].
\]
We have $H^2(\frX_\bSigma,\Z)/\big(\Laa^\bSigma\big)^\star \cong
\Pic^\st(\frX_\bSigma)$ because
$H^2(\frX_\bSigma,\Z) \cong \Pic(\frX_\bSigma)$
and
$\big(\Laa^\bSigma\big)^\star \cong \Pic(X_\Sigma)$
by \cite[Lemma~4.8]{CCIT:MS}.

\begin{Theorem}[{\cite[Theorems 4.28 and~6.11]{CCIT:MS}}]
\label{thm:mirror_isom}
Let $\frX_\bSigma$ be a smooth toric DM stack from
Defini\-tion~$\ref{def:stacky_fan_adapted_to_S}$.
There exists a $\Pic^\st(\frX_\bSigma)$-equivariant map
$($mirror map$)$
$\mir \colon \cM_{\T,\bSigma}\sphat \to
\cM_{\rm A,\T}(\frX_\bSigma)\sphat$
over $R_\T$ and a $\Pic^\st(\frX_\bSigma)$-equivariant
isomorphism $($mirror isomorphism$)$
\[
\Mir \colon \ \Briequiv(F)\sphat_\bSigma
\cong \mir^* \QDM_\T(\frX_\bSigma),
\]
such that
\begin{itemize}\itemsep=0pt
\item[{\rm (1)}] $\Mir$ identifies the Gauss--Manin connection on
$\Briequiv(F)\sphat_\bSigma$
with the quantum connection on $\mir^*\QDM_\T(\frX_\bSigma)$;
\item[{\rm (2)}] $\mir$
preserves the Euler vector fields $\mir_*(\cE) = \cE$
$($see \eqref{eq:Euler_A}, \eqref{eq:Euler_B}$)$
and $\Mir$ intertwines the grading operators
$\Gr \circ \Mir = \Mir \circ \Gr$
$($see \eqref{eq:grading_A}, \eqref{eq:grading_B}$)$;
\item[{\rm (3)}] $\Mir$ identifies the higher residue pairing
$($see {\rm \cite[Section~6]{CCIT:MS}} and Section~$\ref{subsubsec:higher_residue})$
with the pairing~$P$ \eqref{eq:pairing_A}
induced by the orbifold Poincar\'e pairing.
\end{itemize}
\end{Theorem}

\begin{Remark}The mirror map $\mir$ and the mirror isomorphism $\Mir$ are
obtained from those in~\cite{CCIT:MS}
by the restriction to $y_1 = \cdots =y_m =1$.
See also Remark~\ref{rem:notation_difference}.
\end{Remark}
\begin{Remark}[{\cite[(60)]{Iritani:Integral}}]
\label{rem:mirror_map_asymptotic}
We have
$\cM_{\T,\bSigma}\sphat \cong \Spf\big(R_\T\big[\!\big[\Laa^\bSigma_+\big]\!\big]\big)
\times \big(\C^{G(\bSigma)},0\big)\sphat$ \
by~\eqref{eq:LG_local_chart}.
Writing $(\ttq, t=\{t_v\}_{v\in G(\bSigma)})$
for the co-ordinates on $\cM_{\T,\bSigma}\sphat$
as in Section~\ref{subsec:coord_localchart_LG},
we have that the mirror map
$(\ttq, t) \mapsto \big(\hat\ttq, \tau' =
\tau^0\phi_0+\sum\limits_{i=r+1}^s \tau^i \phi_i\big)$
has the following asymptotic form
\begin{gather*}
\sum_{i=1}^r \phi_i \log \hat\ttq_i +\tau'
= \sum_{i=1}^r \phi_i \log \ttq_i +
\sum_{v\in G(\bSigma)} t_v \frD_v + O\big(\ttq,t^2\big),
\end{gather*}
 where
 \begin{gather*}
 \frD_v =
\prod_{b\in R(\bSigma)} \overline{D}_{b}^{\floor{\Psi_{b}(v)}}
\unit_{\langle v\rangle} \qquad \text{with}\quad
\langle v\rangle = v-\sum_{b\in R(\bSigma)} \floor{\Psi_b(v)} b
\in \Bx(\bSigma).
\end{gather*}
Here we assumed that
$\{\phi_i\}_{i=1}^r\subset \big(\LL^\bSigma_\C\big)^\star$
is chosen to be a $\Z$-basis of $\big(\Laa^\bSigma\big)^\star$
and that its dual basis defines co-ordinates $\ttq_i$, $i=1,\dots,r$ as in
Section~\ref{subsec:coord_localchart_LG};
$O\big(\ttq,t^2\big)$ denotes an element of the ideal
generated by $\ttq^\lambda$ $\big(\lambda \in \Laa^\bSigma_+\setminus\{0\}\big)$
and $t_{b_1} t_{b_2}$ ($b_1,b_2 \in G(\bSigma)$).
\end{Remark}

The above mirror isomorphism shows that
$\Briequiv(F)\sphat_\bSigma$ is a free
$R_\T[z][\![\Laa(\bSigma)_+]\!]$-module of rank
$\dim H_{\CR}^*(\frX_\bSigma)$.
This fact shows the following two propositions. The first one
proves a non-equivariant version of Theorem \ref{thm:mirror_isom}.
\begin{Proposition}
\label{prop:nonequiv_completion}
We have an isomorphism
\begin{equation*}
\Bri(F)\sphat_\bSigma
\cong
\Briequiv(F)_\bSigma\sphat/\bM_\C \Briequiv(F)_\bSigma\sphat.
\end{equation*}
In particular, the completed Brieskorn module
$\Bri(F)\sphat_\bSigma$ is a free module
of rank $\dim H^*_{\CR}(\frX)$ over
$\C[z][\![\Laa(\bSigma)_+]\!]$ and we have a
non-equivariant mirror isomorphism
\[
\Mir|_{\chi=0} \colon \
\Bri(F)\sphat_\bSigma \cong \mir^*\QDM(\frX_\bSigma),
\]
where $\mir$ denotes the non-equivariant limit
of the mirror map in Theorem~$\ref{thm:mirror_isom}$.
\end{Proposition}
\begin{proof}We apply Lemma \ref{lem:completion} below
to $K= \C[\Laa(\bSigma)_+]$, $I= \frakm_\bSigma$,
$N = \Briequiv(F)_\bSigma^{\oplus n}$,
$M= \Briequiv(F)_\bSigma$ and the map
$f = \bigoplus\limits_{i=1}^n \chi_i$. Then we find that
$\Cok(f)\sphat = \Bri(F)\sphat_\bSigma$
is the $\frakm_\bSigma$-adic completion of $\Cok(\hat{f}) =
\Briequiv(F)_\bSigma\sphat/
\bM_\C \Briequiv(F)_\bSigma\sphat$.
On the other hand, $\Cok(\hat{f})$ is a free
$\C[z][\![\Laa(\bSigma)_+]\!]$-module of finite rank,
and in particular
$\frakm_\bSigma$-adically complete.
This proves the proposition.
\end{proof}

\begin{Proposition}
\label{prop:Briequiv_mod_z}
We have
\[
\Briequiv(F)\sphat_\bSigma/z \Briequiv(F)\sphat_\bSigma \cong
\varprojlim_k \C[\OO(\bSigma)_+]/\frakm_\bSigma^k.
\]
\end{Proposition}
\begin{proof}
We apply Lemma \ref{lem:completion} to $K=\C[\Laa(\bSigma)_+]$,
$I=\frakm_\bSigma$, $N=M=\Briequiv(F)_\bSigma$ and the map
$f = z$. Then we find that the $\frakm_\bSigma$-adic completion of
$\Cok(f)=\C[\OO(\bSigma)_+]$ is isomorphic to the $\frakm_\bSigma$-adic
completion of $\Cok(\hat{f}) =
\Briequiv(F)\sphat_\bSigma/z \Briequiv(F)\sphat_\bSigma$.
Here $\Cok(\hat{f})$ is a free $R_\T[\![\Laa(\bSigma)_+]\!]$-module of finite rank,
and in particular $\frakm_\bSigma$-adically complete. The conclusion follows.
\end{proof}

\begin{Lemma}\label{lem:completion}
Let $K$ be a ring and $I\subset K$ be a finitely generated ideal.
Let $N$, $M$ be $K$-modules and $f\colon N \to M$ be a
homomorphism of $K$-modules.
Let $\widehat{N}$, $\widehat{M}$ denote the
$I$-adic completions of~$N$ and~$M$ respectively
and let $\hat{f} \colon \widehat{N} \to \widehat{M}$ denote the
induced homomorphism. Then the $I$-adic completion of~$\Cok(f)$ is
isomorphic to the $I$-adic completion of~$\Cok(\hat{f})$.
\end{Lemma}
\begin{proof}
Since $I$ is finitely generated, we have $\widehat{N}/I^k \widehat{N}
\cong N/I^k N$ and $\widehat{M}/I^k \widehat{M} \cong M/I^k M$
(see \cite[\href{http://stacks.math.columbia.edu/tag/05GG}{Lemma 05GG}]{stacks-project}).
Therefore the exact sequences $N \to M \to \Cok(f)\to 0$ and $\widehat{N} \to
\widehat{M} \to \Cok(\hat{f}) \to 0$ imply
(by the right-exactness of $\otimes_K(K/I^k)$) that
$\Cok(f)/I^k \Cok(f) \cong \Cok(\hat{f})/I^k \Cok(\hat{f})$.
The conclusion follows.
\end{proof}

\begin{Remark}[cf.~Remark \ref{rem:twisted_de_Rham}]\label{rem:twisted_de_Rham_formal}
In \cite{CCIT:MS, Iritani:shift_mirror},
the completed Brieskorn modules are described
as twisted de Rham cohomology. We have
\begin{gather*}
\Briequiv(F)\sphat_\bSigma
 \cong \pr_*H^n(\Omega^\bullet_{\cY/\cM}[z]\sphat \ [\chi_1,\dots,\chi_n],
z d + dF_\T \wedge), \\
\Bri(F)\sphat_\bSigma
 \cong \pr_*H^n(\Omega^\bullet_{\cY/\cM}[z]\sphat\ , zd + dF\wedge),
\end{gather*}
where $\sphat$ \ means the $\frakm_\bSigma$-adic completion.
These isomorphisms follow from the same argument as
in Remark~\ref{rem:twisted_de_Rham}.
Note that $(\Omega^\bullet_{\cY/\cM}[z]\sphat\ , zd+dF\wedge)$
is the Koszul complex associated with the action of $\chi_1,\dots,\chi_n$
on $\Briequiv(F)\sphat_\bSigma = \cO_\cY[z]\sphat\cdot \omega$.
Since we know that $\Briequiv(F)\sphat_\bSigma$
is a free $R_\T[z][\![\Laa(\bSigma)_+]\!]$-module,
$\chi_1,\dots,\chi_n$ form a regular sequence for the
module $\Briequiv(F)\sphat_\bSigma$. Thus we have
\[
H^i(\Omega^\bullet_{\cY/\cM}[z]\sphat\ , zd + dF\wedge) =0 \qquad
\text{for $i\neq n$.}
\]
Therefore, we have $\Bri(F)\sphat_\bSigma \cong
R^n\pr_*(\Omega^\bullet_{\cY/\cM}[z]\sphat\ , zd +d F\wedge)$.
\end{Remark}

\subsection{Analytification of the completed Brieskorn module}
\label{subsec:analytified_Bri}
In this section, we construct an analytification
of the completed Brieskorn module.
There is a trade-off\footnote{By \looseness=-1 the convergence result from \cite{CCIT:mirrorthm} reviewed in
Section~\ref{subsec:an_mirror_isom}, we can make the analytified
Brieskorn module fully analytic
both in the $\cM_\T$-direction and in the $z$-direction;
the full analytification is given by the analytic quantum D-module.
This analytic structure is in general different from that
of the original (algebraic) Brieskorn module;
the gauge transformation $(M_i^j(q,\chi,z))$
(see Part~\ref{item:Mir_matrix} in Section~\ref{subsec:an_mirror_isom})
connecting the Brieskorn module and the quantum D-module
is only a formal power series in~$z$. This stems from the
fact that the (extended) $I$-function~\cite{CCIT:mirrorthm}
is not necessarily convergent: see
\cite[Proposition~5.13]{Iritani:coLef} for a relevant
statement in the non-weak-Fano case.}
between the analyticity along $\cM_\T=\cM\times \Lie \T$
and that along the $z$-plane: the analytified Brieskorn module
is analytic in the $\cM_\T$-direction but formal in the variable $z$.

\subsubsection{Analytification of algebras}
\label{subsubsec:analytification_algebras}
First we study the restriction of $\Briequiv(F)$ to $z=0$.
By definition, $\Briequiv(F)/z \Briequiv(F)$ is isomorphic to
$\cO_\cY$ and the $R_\T \otimes \cO_\cM$-module structure on it
is induced by the map
\begin{equation}\label{eq:tpr}
\tpr := \left(\pr, x_1 \parfrac{F}{x_1},\dots, x_n \parfrac{F}{x_n}\right)
\colon \ \cY \to \cM_\T = \cM \times \Lie \T,
\end{equation}
where $\Lie \T = \Spec R_\T$ is identified with $\C^n$ via
the basis $\chi_1,\dots,\chi_n$ chosen in Section~\ref{subsec:Bri}.
Let $\tzero = \tzero_\bSigma \in \cY_\bSigma \subset \cY$ denote the
(unique) torus-fixed point on the chart $\cY_\bSigma$
such that $\pr(\tzero_\bSigma) = 0_\bSigma$:
it is defined by $u^{(\lambda,v)} =0$ for all non-torsion elements
$(\lambda,v) \in \OO(\bSigma)_+$.
On the uniformizing chart $\Spec \C[\OO(\bSigma)_+]$ of $\cY_\bSigma$,
$\tzero_\bSigma$ corresponds to
$|\bN_{\rm tor}|$ many points $\bN_{\rm tor}\sphat \cong
\Spec \C[\bN_{\rm tor}]$ which form a single $\Pic^\st(\frX_\bSigma)$-orbit
(recall the sequence \eqref{eq:G_Gprime}).
We mean by $\tzero_\bSigma$ either a single point in $\cY_\bSigma$
or the corresponding finite subset in $\Spec \C[\OO(\bSigma)_+]$
depending on the context.
We have $\tpr(\tzero_\bSigma) = (0_\bSigma, 0)$.
\emph{In order to avoid the heavy notation,
in this Section~$\ref{subsec:analytified_Bri}$, we will sometimes
omit the subscript $\bSigma$ for $0_\bSigma$ and $\tzero_\bSigma$,
writing $0 \in \cM_\bSigma$ for $0_\bSigma$
and $\tzero\in \cY_\bSigma$ for $\tzero_\bSigma$}.

\begin{Lemma}
\label{lem:fibre_at_0}
The set-theoretic fibre of $\tpr$ at $(0,0)$ is $\tzero$,
i.e., $\tzero = \tpr^{-1}(0,0)$.
\end{Lemma}
\begin{proof}
Note that
\[
x_i \parfrac{F}{x_i} = \sum_{b\in S} (\chi_i \cdot b) u_b
\in \sum_{b\in R(\bSigma)} (\chi_i\cdot b) u_b +
\frakm_\bSigma \C[\OO(\bSigma)_+],
\]
since $t_b \in \frakm_\bSigma$ for $b\in G(\bSigma)$ (see~\eqref{eq:LG_pot_co-ordinates}).
Therefore the scheme theoretic fibre at $(0,0)$ is the spectrum of
\begin{equation}\label{eq:scheme-theoretic_fibre}
\C[\OO(\bSigma)_+]\Big/
\bigg(\frakm_\bSigma \C[\OO(\bSigma)_+] + \bigg\langle
 \sum\limits_{b\in R(\bSigma)} (\chi_i \cdot b) u_b \colon 1\le i\le n
\bigg\rangle\bigg).
\end{equation}
By Remark \ref{rem:flatness_pr}, we have
\[
\C[\OO(\bSigma)_+]\Big/\frakm_\bSigma \C[\OO(\bSigma)_+]
\cong \bigoplus_{v\in \bN\cap \Pi} \C w_v \qquad
\text{with $w_v := \big[u^{(\Psi^\bSigma(\overline{v}),v)}\big]$},
\]
where $\Psi^\bSigma$ is given in Notation~\ref{nota:Psi} and
the product on the right-hand side is given by
\[
w_{v_1} \cdot w_{v_2} = \begin{cases}
w_{v_1+v_2}, & \text{if $v_1$, $v_2$ lie in the same cone of $\Sigma$}, \\
0, & \text{otherwise}.
\end{cases}
\]
It follows that the ring~\eqref{eq:scheme-theoretic_fibre} is precisely
the presentation of the orbifold cohomology ring
$H_{\CR}^*(\frX_\bSigma)$ due to Borisov--Chen--Smith~\cite{BCS}.
Since elements in $H_{\CR}^{>0}(\frX_\bSigma)$ are nilpotent,
the set-theoretical fibre at $(0,0)$ equals
$\Spec(H_\CR^0(\frX_\bSigma))
\cong \Spec(\C[\bN_{\rm tor}]) = \tzero$.
\end{proof}

We use the following elementary lemma on general topology.

\begin{Lemma}
\label{lem:topology}
Let $X$ and $Y$ be locally compact Hausdorff spaces and let
$f\colon X\to Y$ be a~continuous map.
Let $y_0\in Y$ be such that $f^{-1}(y_0)$ is compact.
Then
\begin{itemize}\itemsep=0pt
\item[{\rm (1)}] there exist an open neighbourhood $B$ of $f^{-1}(y_0)$
and an open neighbourhood $U$ of $y_0$
such that
$f(B)\subset U$ and $f|_B\colon B \to U$ is proper;
\item[{\rm (2)}] the family of open sets $\{f^{-1}(V)\cap B\colon \text{$V$
is an open neighbourhood of $y_0$}\}$
is a fundamental neighbourhood system of~$f^{-1}(y_0)$.
\end{itemize}
\end{Lemma}
\begin{proof}Since $X$ is locally compact, we can find a relatively compact open
neighbourhood~$C$ of~$f^{-1}(y_0)$.
Then $\overline{C}\setminus C$ is compact and hence
$f(\overline{C} \setminus C)$ is a compact set not containing $y_0$.
Since~$Y$ is Hausdorff, $f(\overline{C} \setminus C)$ is closed;
thus there exists an open neighbourhood $U$ of
$y_0$ such that $U \cap f(\overline{C} \setminus C) = \varnothing$.
Then one sees easily that $f^{-1}(U) \cap C = f^{-1}(U)
\cap \overline{C}$.
For every compact subset $K$ in $U$, we have
that $f^{-1}(K) \cap C = f^{-1}(K) \cap \overline{C}$ is compact.
This shows that $f|_{C \cap f^{-1}(U)}$ is proper.
It now suffices to set $B = C \cap f^{-1}(U)$ to conclude Part~(1).

To see Part (2), take an open neighbourhood $W\subset B$ of $f^{-1}(y_0)$.
Then $B\setminus W$ is closed in~$B$.
Since proper maps between locally compact Hausdorff spaces are closed,
$f(B\setminus W)$ is closed in~$U$ and does not contain~$y_0$.
Thus there exists an
open neighbourhood $V\subset U$ of $y_0$ such that
$V \cap f(B\setminus W)=\varnothing$.
Then $f^{-1}(V) \cap B \subset W$ and Part (2) follows.
\end{proof}

\begin{Lemma}\label{lem:proper}
There exist an analytic open neighbourhood $\cB$ of $\tzero$
in $\cY_\bSigma$ and an analytic open neighbourhood $\cU$
of $(0,0)$ in $\cM_\bSigma \times \Lie\T\subset \cM_\T$ such that
$\tpr(\cB) \subset \cU$ and that
the map \mbox{$\tpr|_{\cB} \colon \cB \to \cU$} is proper.
Moreover, for any sheaf $\cF$ on $\cB$ $($with respect to
the complex-analytic topology$)$, we have
$(\tpr_*\cF)_{(0,0)} = \cF_{\tzero}$,
where $(\tpr_*\cF)_{(0,0)}$ and $\cF_{\tzero}$
denote the stalks respectively at~$(0,0)$ and~$\tzero$.
\end{Lemma}
\begin{proof}We apply Lemma~\ref{lem:topology} for $f = \tpr$, $X= \cY$,
$Y= \cM_\T$ and $y_0= (0,0)$.
Note that Lemma~\ref{lem:fibre_at_0} gives $f^{-1}(y_0) = \tzero$.
The former statement follows from Part~(1) of Lemma~\ref{lem:topology} and
the latter statement on stalks follows from Part~(2).
\end{proof}

We will use the notation $\cOan$ to denote
the complex-analytic structure sheaf.
The following lemma shows that
$\tpr_*\cOan_{\cB}$ (which is coherent by Grauert's direct image theorem)
gives an analytification of
$\Briequiv(F)\sphat_\bSigma/z \Briequiv(F)\sphat_\bSigma$.
\begin{Lemma}\label{lem:analytification_algebras}
Let $\cB \subset \cY$ and $\cU\subset \cM_\T=\cM\times \Lie \T$
be as in Lemma~$\ref{lem:proper}$.
$($They are analytic open neighbourhoods of $\tzero$ and $(0,0)$
respectively.$)$
Let $(\tpr_*\cOan_{\cB})_{(0,0)}\sphat$
denote the completion of $\tpr_*\cOan_{\cB}$
at $(0,0)\in \cU$.
We have
\begin{equation}
\label{eq:tpr_B_completion}
(\tpr_*\cOan_{\cB})_{(0,0)}\sphat
\cong
\hcO_{\cY,\tzero}
\cong
\left( \Briequiv(F)\sphat_\bSigma/z \Briequiv(F)\sphat_\bSigma\right)
\otimes_{R_\T[\![\Laa(\bSigma)_+]\!]} \hcO_{\cU,
(0,0)}.
\end{equation}
Here
$\hcO_{\cU,(0,0)}$ is the
completion of $\cOan_{\cU}$ at $(0,0)$
and $\hcO_{\cY,\tzero}$ is
the completion of $\cO_{\cY}$ at $\tzero$;
more explicitly, they are
\[
\hcO_{\cU,(0,0)} =
\widehat{R_\T}[\![\Laa(\bSigma)_+]\!],
\qquad
\hcO_{\cY,\tzero} = \C[\![\OO(\bSigma)_+]\!],
\]
where $\widehat{R_\T}=\C[\![\chi_1,\dots,\chi_n]\!]$
is the completion of $R_\T = \C[\chi_1,\dots,\chi_n]$, and for a ring $K$,
$K[\![\OO(\bSigma)_+]\!]$ denotes the completion of
$K[\OO(\bSigma)_+]$ with respect to the ideal generated by
$u^{(\lambda,v)}$ with non-torsion $(\lambda,v)\in \OO(\bSigma)_+$.
\end{Lemma}
\begin{proof}We show that the natural maps{\samepage
\begin{gather}\label{eq:pr_B_Y}
(\tpr_*\cOan_{\cB})_{(0,0)}\sphat \to
\hcO_{\cY,\tzero}, \\
\label{eq:Briequiv_Y}
\left( \Briequiv(F)\sphat_\bSigma/z \Briequiv(F)\sphat_\bSigma\right)
\otimes_{R_\T[\![\Laa(\bSigma)_+]\!]} \hcO_{\cU,
(0,0)}
 \to \hcO_{\cY,\tzero}
\end{gather}
are} isomorphisms. First we prove that the map~\eqref{eq:pr_B_Y} is an isomorphism.
The left-hand side of~\eqref{eq:pr_B_Y} is the $\frakman_{(0,0)}$-adic
completion of $(\tpr_*\cOan_{\cB})_{(0,0)}$, where $\frakman _{(0,0)}\subset \cOan_{\cU,
(0,0)}$ is the ideal of $(0,0)$.
Recall from Lemma~\ref{lem:fibre_at_0} that the set-theoretic fibre of $\tpr$ at
$(0,0)$ is $\tzero$; also from Lemma~\ref{lem:proper} that
$(\tpr_*\cOan_{\cB})_{(0,0)}
=\cOan_{\cB,\tzero}$.
Therefore the ideal $\frakman_{\tzero}\subset
\cOan_{\cB,\tzero}$
of $\tzero$ is the radical of the ideal generated by
$\frakman_{(0,0)}$,
and the $\frakman_{(0,0)}$-adic topology on
$(\tpr_*\cOan_{\cB})_{(0,0)}
=\cOan_{\cB,\tzero}$ is equivalent to
the $\frakman_{\tzero}$-adic topology on the same space.
The completion with respect to the latter topology equals
$\hcO_{\cB,\tzero} = \hcO_{\cY,\tzero}$.

We 
prove similarly that the map \eqref{eq:Briequiv_Y} is an isomorphism.
Recall that $\Briequiv(F)\sphat_\bSigma/z\Briequiv(F)\sphat_\bSigma$ is
a finite free $R_\T[\![\Laa(\bSigma)_+]\!]$-module by
Theorem~\ref{thm:mirror_isom}.
Also $\hcO_{\cU, (0,0)}$ is the completion
of $R_\T[\![\Laa(\bSigma)_+]\!]$ with respect to
the ideal $\widehat{\frakm}_{(0,0)}$
generated by $\chi_1,\dots,\chi_n$
and~$\frakm_\bSigma$.
Therefore the left-hand side of~\eqref{eq:Briequiv_Y} is
the $\widehat{\frakm}_{(0,0)}$-adic completion of
$\Briequiv(F)\sphat_\bSigma/z\Briequiv(F)\sphat_\bSigma$
(see, e.g., \cite[\href{http://stacks.math.columbia.edu/tag/00MA}{Lemma~00MA}]{stacks-project}).
Moreover $\Briequiv(F)\sphat_\bSigma/z\Briequiv(F)\sphat_\bSigma$
is the $\frakm_\bSigma$-adic completion of $\C[\OO(\bSigma)_+]$
by Proposition \ref{prop:Briequiv_mod_z}.
Thus\footnote{Here we used the following fact. Let $I\subset J$ be
ideals of a ring $R$, let $M$ be an $R$-module
and assume that $I$ is finitely generated.
Let $M\sphat_I$ denote the $I$-adic completion of $M$
and let $(M\sphat_I)\sphat_{J}$ denote the $J$-adic completion
of $M\sphat_I$. Then $(M\sphat_I)\sphat_J \cong M\sphat_J$.
This follows from $M\sphat_I/J^k M\sphat_I \cong
(M\sphat_I/I^k M\sphat_I)/J^k (M\sphat_I/I^k M\sphat_I)
\cong (M/I^k)/J^k(M/I^k) \cong M/J^k M$
(where we used \cite[\href{http://stacks.math.columbia.edu/tag/05GG}
{Lemma 05GG}]{stacks-project} in the middle step).}
the left-hand side of \eqref{eq:Briequiv_Y} is the
$\frakm_{(0,0)}$-adic completion
of $\C[\OO(\bSigma)_+]$, where $\frakm_{(0,0)}
:= \frakm_\bSigma + (\chi_1,\dots,\chi_n)$ is an ideal of
$R_\T[\Laa(\bSigma)_+]$.
Again Lemma \ref{lem:fibre_at_0} implies that
the $\frakm_{(0,0)}$-adic completion
of $\C[\OO(\bSigma)_+]$ equals the
$\frakm_{\tzero}$-adic completion of it,
where $\frakm_{\tzero} \subset \C[\OO(\bSigma)_+]$
is the ideal of $\tzero$. This shows that
the map \eqref{eq:Briequiv_Y} is an isomorphism.
The lemma is proved.
\end{proof}

Because $\Briequiv(F)\sphat_\bSigma/z \Briequiv(F)\sphat_\bSigma$
is a free $R_\T[\![\Laa(\bSigma)_+]\!]$-module of rank
$\dim H_{\CR}^*(\frX_\bSigma)$ (by Theo\-rem~\ref{thm:mirror_isom}),
we have the following corollary.
\begin{Corollary}\label{cor:finite_free}The modules \eqref{eq:tpr_B_completion} are free
$\hcO_{\cU,(0,0)}$-modules of rank $\dim H_\CR^*(\frX_\bSigma)$.
\end{Corollary}

\begin{Corollary}\label{cor:tpr_locally_free}
Let $\cB \subset \cY$ and $\cU \subset \cM_\T$
be as in Lemma~$\ref{lem:proper}$. By shrinking~$\cB$,
$\cU$ if necessary, we have that
$\tpr_*\cOan_{\cB}$ is a locally free
$\cOan_{\cU}$-module
of rank $\dim H_{\CR}^*(\frX_\bSigma)$.
In particular, $\tpr \colon \cB \to \cU$ is a finite flat morphism.
\end{Corollary}
\begin{proof}Note that $\tpr_*\cOan_{\cB}$ is a coherent
$\cOan_{\cU}$-module by Grauert's Direct Image
Theorem and Lemma \ref{lem:proper}.
Thus it suffices to show that the stalk
$(\tpr_*\cOan_{\cB})_{(0,0)}$
is a free $\cOan_{\cU, (0,0)}$-module
of rank $\dim H_{\CR}^*(\frX_\bSigma)$.
By Corollary \ref{cor:finite_free}, we know that
the completion
$(\tpr_*\cOan_{\cB})\sphat_{(0,0)}$
is a free $\hcO_{\cU,(0,0)}$-module
of rank $\dim H_\CR^*(\frX_\bSigma)$.
The conclusion follows from the following (probably) well-known fact:
for a Noetherian local ring $(A,\frakm)$ and a finite $A$-module $M$,
$M$ is a free $A$-module of rank $r$ if and only if
its $\frakm$-adic completion
$\widehat{M} = M\otimes_A \widehat{A}$ is a free $\widehat{A}$-module of rank~$r$.
This follows, for instance, by combining the fact that a finite flat module over a local ring
is free, \cite[Theo\-rem~22.4(1)]{Matsumura:comm_ring_theory}
and $M/\frakm M \cong \widehat{M}/\frakm \widehat{M}$.
\end{proof}

\subsubsection{Analytification of D-modules}\label{subsubsec:analytification_D-module}

Next we construct an extension of the completed Brieskorn module $\Briequiv(F)_\bSigma\sphat$
to an analytic neighbourhood of $(0_\bSigma,0) \in \cM_\T = \cM\times \Lie \T$.
Let $\cU \subset \cM_\T$ and $\cB \subset \cY$
denote analytic open neighbourhoods of $(0,0)=(0_\bSigma,0)$ and
$\tzero=\tzero_\bSigma$
respectively as in Corollary \ref{cor:tpr_locally_free}.
Recall that $\tpr_*\cOan_{\cB}$ is a locally free
$\cOan_{\cU}$-module by Corollary \ref{cor:tpr_locally_free}.
\begin{Definition}[cf.~Definition \ref{def:Brieskorn}]
\label{def:analytic_Bri}
Let $\bSigma\in\Fan(S)$ be a stacky fan adapted to $S$.
\begin{enumerate}\itemsep=0pt
\item[(1)] The \emph{analytified equivariant Brieskorn module} around
the limit point $0=0_\bSigma$
is the sheaf over $\cU$
\[
\Briequivan(F)_\bSigma :=
\tpr_*\left(\cOan_{\cB}[\![z]\!]\right)\cdot \omega.
\]
We equip $\Briequivan(F)_\bSigma$ with an
$\cOan_{\cU}[\![z]\!]$-module
structure by the formula
\begin{equation}
\label{eq:chi_action_on_Bri}
\chi_i \cdot (f \omega) = \left(
x_i \parfrac{F}{x_i} f+ zx_i\parfrac{f}{x_i} \right) \omega
\end{equation}
for $f\in \tpr_*(\cOan_{\cB}[\![z]\!])$,
together with the usual multiplication of functions in $q$ and $z$
(where~$\chi_i$ is a co-ordinate on $\Lie \T$ and~$q$ is a co-ordinate on $\cM$,
see Sections~\ref{subsec:coord_localchart_LG} and~\ref{subsec:Bri}).
A flat connection $\nabla = d + d(F/z)\wedge$ on
$\cOan_{\cB}[\![z]\!]$
induces, via the choice of a splitting
$\big(\C^S\big)^\star \cong \bM_\C \oplus \LL_\C^\star$ of
the extended divisor sequence~\eqref{eq:ext_divseq} over $\C$,
a partial (logarithmic) flat connection on $\Briequivan(F)_\bSigma$
\[
\nabla \colon \ \Briequivan(F)_\bSigma \to z^{-1} \Briequivan(F)
\otimes_{\cOan_{\cU}} \Omega^{1,\rm an}_{\cU/\Lie \T}
\]
in the direction of $\cM$, see the discussion after Definition~\ref{def:Brieskorn}.
A \emph{grading operator}
$\Gr\in \End_\C(\Briequivan(F)_\bSigma)$ is defined
as before
\[
\Gr(f \omega) =\left(z\parfrac{f}{z} + \sum_{b\in S} u_b \parfrac{f}{u_b}\right) \omega.
\]
\item[(2)] The (non-equivariant) \emph{analytified Brieskorn module}
$\Brian(F)_\bSigma$ around the limit point~$0$
is defined to be the restriction of $\Briequivan(F)_\bSigma$ to
$\cV = (\cM \times \{0\}) \cap \cU$.
It is equipped with the flat connection $\nabla\colon
\Brian(F)_\bSigma \to z^{-1} \Brian(F)_\bSigma \otimes
\Omega^{1, \rm an}_{\cV}$ (independent of the
choice of a splitting) and the grading operator
$\Gr \in \End_\C\big(\Brian(F)_\bSigma\big)$. The connection $\nabla$
and $\Gr$ together give a flat connection in the $z$-direction
as in Remark~\ref{rem:conn_z}.
\end{enumerate}
\end{Definition}

\begin{Remark}
The overline for $\Briequivan(F)_\bSigma$,
$\Brian(F)_\bSigma$ indicates that they are completed in the
$z$-adic topology. Note also that these analytified Brieskorn modules depend
on the choice of $\bSigma \in \Fan(S)$.
\end{Remark}

\begin{Remark}\label{rem:module_structure}
The $\cOan_{\cU}\![\![z]\!]$-module structure on
$\Briequivan(F)_\bSigma$ is not a standard one on
$\tpr_*\!(\cOan_{\cB}\![\![z]\!])\omega.\!$
A~more precise definition of the module structure is described as follows:
we let a function $h=h(q,\chi_1,\dots,\chi_n,z)
\in \cOan_{\cU}[\![z]\!]$
act on a section $f(q,x,z) \omega$ of $\Briequivan(F)_\bSigma$ as
\[
h\cdot (f\omega) =
\left[ h\left(q, x_1\parfrac{F}{x_1} + z x_1\parfrac{}{x_1},\dots,
x_n\parfrac{F}{x_n} + z x_n \parfrac{}{x_n}, z\right) f(q,x,z) \right] \omega,
\]
where in the right-hand side we expand $h\big(
q, \textstyle x_1\parfrac{F}{x_1} + z x_1\parfrac{}{x_1},\dots,
x_n\parfrac{F}{x_n} + z x_n \parfrac{}{x_n}, z\big)$
in power series of $zx_i\parfrac{}{x_i}$ and apply it to $f$.
Note that $h\big(q,x_1\parfrac{F}{x_1},\dots,x_n\parfrac{F}{x_n},z\big)$
is the pull-back of $h(q,\chi_1,\dots,\chi_n,z)$
under $\tpr$ since $x_i \parfrac{F}{x_i} = \tpr^*(\chi_i)$.
The action is well-defined, since we allow
any formal power series in $z$.
\end{Remark}

\begin{Proposition}
\label{prop:Briequivan_finite_free}
$\Briequivan(F)_\bSigma$ is a locally free
$\cOan_{\cU}[\![z]\!]$-module of rank
$\dim H^*_{\CR}(\frX_\bSigma)$.
In particular $\Brian(F)_\bSigma$ is a locally free
$\cOan_{\cV}[\![z]\!]$-module of the same rank.
\end{Proposition}
\begin{proof}
This follows from Corollary \ref{cor:tpr_locally_free}.
Recall that the $\cOan_{\cU}[\![z]\!]$-module structure
on $\Briequivan(F)_\bSigma$ equals the standard one
on $\pr_*(\cO_{\cB}[\![z]\!])$ modulo $z$.
We show that any local basis
$s_1,\dots,s_N$ of $\tpr_*\cOan_{\cB}$ over
$\cOan_{\cU}$ gives rise to a local basis of
$\Briequivan(F)_\bSigma$ over $\cOan_{\cU}[\![z]\!]$.
That $s_1,\dots,s_N$ generate $\Briequiv(F)_\bSigma$
over $\cOan_{\cU}[\![z]\!]$
follows from the following fact.
Let $K$ be a ring with an ideal $\frakm$ and let~$M$ be a $K$-module.
Suppose that $K$ is $\frakm$-adically complete
and $M$ is Hausdorff with respect to the $\frakm$-adic topology,
that is, $\bigcap\limits_{i\ge 0} \frakm^i M = \{0\}$.
If $s_1,\dots,s_N\in K$ generate $M/\frakm M$ over $K/\frakm$,
then $s_1,\dots,s_N$ generate $M$ over $K$.
See for instance \cite[Corollary 2, Section~3, Chapter~VIII]{Zariski-Samuel}.
We apply this fact for $K = \cOan_{\cU}[\![z]\!]$,
$\frakm=z K$,
$M = \Briequivan(F)_\bSigma$.
On the other hand, suppose we have a relation $\sum\limits_{i=1}^N c_i s_i =0$
with $c_i \in \cOan_{\cU}[\![z]\!]$ and
$(c_1,\dots,c_N) \neq 0$.
Setting $c_i = c_{i,m}z^m + O\big(z^{m+1}\big)$
for some $m\ge 0$ and $c_{i,m}\in \cOan_{\cU}$ with
$(c_{1,m},\dots,c_{N,m}) \neq 0$,
we obtain a non-trivial relation $\sum\limits_{i=1}^N c_{i,m} s_i =0$ in
$\tpr_* \cOan_{\cB}$ over $\cOan_{\cU}$.
This is a contradiction. Thus $s_1,\dots,s_N$ are linearly independent over $\cOan_{\cU}[\![z]\!]$.
\end{proof}

The next proposition shows that $\Briequivan(F)_\bSigma$
is an analytification of $\Briequiv(F)\sphat_\bSigma$
(and thus justifies the name).

\begin{Proposition}[cf.~Lemma \ref{lem:analytification_algebras}]
\label{prop:analytification_D_modules}
Let $\big(\Briequivan(F)_\bSigma\big)_{(0,0)}\sphat$ denote
the $\frakman_{(0,0)}$-adic completion of
$\big(\Briequivan(F)_\bSigma\big)_{(0,0)}$,
where $\frakman_{(0,0)} \subset \cOan_{\cU,(0,0)}$
is the ideal of $(0,0)$.
Then
\begin{itemize}\itemsep=0pt
\item[{\rm (1)}]
$\big(\Briequivan(F)_\bSigma\big)_{(0,0)}\sphat$ has the
structure of an
$\hcO_{\cU,(0,0)}[\![z]\!]$-module;
\item[{\rm (2)}]
we have an isomorphism of (finite, free)
$\hcO_{\cU,(0,0)}[\![z]\!]$-modules
\begin{equation*}
\big(\Briequivan(F)_\bSigma\big)_{(0,0)}\sphat
\cong
\hcO_{\cY,\tzero}[\![z]\!]
\cong
\Briequiv(F)\sphat_\bSigma\otimes_{R_\T[z][\![\Laa(\bSigma)_+]\!]}
\hcO_{\cU,(0,0)}[\![z]\!],
\end{equation*}
\end{itemize}
where $\hcO_{\cU,(0,0)}$ and
$\hcO_{\cY,\tzero}$ are as in
Lemma~$\ref{lem:analytification_algebras}$.
In Part~$(2)$, the $\hcO_{\cU,(0,0)}[\![z]\!]$-module structure on
$\hcO_{\cY,\tzero}[\![z]\!]$ is defined similarly
to Definition~$\ref{def:analytic_Bri}$:
$\chi_i$ acts on it by the formula~\eqref{eq:chi_action_on_Bri}
and functions in $q$ and $z$ act in the usual way
$($see also Remark $\ref{rem:module_structure})$.
\end{Proposition}
\begin{proof}
$\big(\Briequivan(F)_\bSigma\big)\sphat_{(0,0)}$ is naturally
a module over the completion
$(\cOan_{\cU}[\![z]\!])\sphat_{(0,0)}$
of $(\cOan_{\cU}[\![z]\!])_{(0,0)}$
with respect to $\frakman_{(0,0)}$.
Hence Part~(1) follows from the fact\footnote{Note however that $(\cOan_{\cU}[\![z]\!])_{(0,0)}
\neq \cOan_{\cU,(0,0)}[\![z]\!]$.} that
$(\cOan_{\cU}[\![z]\!])_{(0,0)}\sphat =
\hcO_{\cU,(0,0)}[\![z]\!]$.

We show that the natural maps
\begin{equation}\label{eq:Briequivan_formal}
\big(\Briequivan(F)_\bSigma\big)\sphat_{(0,0)}
\to
\hcO_{\cY,\tzero}[\![z]\!]
\leftarrow
\Briequiv(F)\sphat_\bSigma\otimes_{R_\T[z][\![\Laa(\bSigma)_+]\!]}
\hcO_{\cU,(0,0)}[\![z]\!]
\end{equation}
are isomorphisms, where the first map is induced by the map
\[
\big(\Briequivan(F)_\bSigma\big)_{(0,0)} \cong
(\tpr_*\cOan_\cB[\![z]\!])_{(0,0)}
\overset{\text{Lemma \ref{lem:proper}}}{\cong}
(\cOan_{\cB}[\![z]\!])_{\tzero}
\to \hcO_{\cY,\tzero}[\![z]\!].
\]
That the maps in~\eqref{eq:Briequivan_formal} are isomorphisms
follows from the following
two facts: (a) all three mo\-dules in~\eqref{eq:Briequivan_formal}
are finite and free as $\hcO_{\cU,(0,0)}[\![z]\!]$-modules,
and~(b) the maps in~\eqref{eq:Briequivan_formal} are
isomorphisms modulo~$z$.
We already know that (b) holds by Lemma~\ref{lem:analytification_algebras}.
In fact, the maps in \eqref{eq:Briequivan_formal}
reduces to the isomorphisms in \eqref{eq:tpr_B_completion} modulo $z$.
Thus we only need to show (a).
Proposition \ref{prop:Briequivan_finite_free} implies
that $\big(\Briequivan(F)_\bSigma\big)_{(0,0)}$ is a
finite free $(\cO_{\cU}[\![z]\!])_{(0,0)}$-module,
and thus $\big(\Briequivan(F)_\bSigma\big)\sphat_{(0,0)}$
is a~finite free $(\cO_{\cU}[\![z]\!])\sphat_{(0,0)} =
\hcO_{\cU,(0,0)}[\![z]\!]$-module.
Also, since $\Briequiv(F)\sphat_\bSigma$ is a finite free
$R_\T[z][\![\Laa(\bSigma)_+]\!]$-module, the third term in
\eqref{eq:Briequivan_formal} is a finite free
$\hcO_{\cU,(0,0)}$-module.
The finite-freeness of $\hcO_{\cY,\tzero}[\![z]\!]$
follows from a discussion parallel to the proof of Proposition
\ref{prop:Briequivan_finite_free}. Indeed, we know from
Corollary \ref{cor:finite_free}
that $\hcO_{\cY,\tzero}$ is a finite free
$\hcO_{\cU,(0,0)}$-module, and
any basis of $\hcO_{\cY,\tzero}$
over $\hcO_{\cU,(0,0)}$
gives rise to a basis of $\hcO_{\cY,\tzero}[\![z]\!]$
over $\hcO_{\cU,(0,0)}[\![z]\!]$.
Part (2) is proved.
\end{proof}

By restricting the above isomorphism
to $\cV = \cU \cap (\cM\times \{0\})$
and using Proposition \ref{prop:nonequiv_completion}, we have
\begin{Corollary}
$\big(\Brian(F)_\bSigma\big)\sphat_{0}\cong\Bri(F)\sphat_\bSigma\otimes_{\C[z][\![\Laa(\bSigma)_+]\!]}
\hcO_{\cV,0}[\![z]\!]$.
\end{Corollary}

\begin{Remark}[cf.~Remarks \ref{rem:twisted_de_Rham} and \ref{rem:twisted_de_Rham_formal}]\label{rem:twisted_de_Rham_analytified}
As before, we can describe the analytified Brieskorn module
as a twisted de Rham cohomology.
The complex $\big(\Omega^{\bullet,\rm an}_{\cB/\cM}[\![z]\!],zd + dF\wedge\big)$
can be identified with the Koszul complex associated with
the action of $\chi_1,\dots,\chi_n$ on $\cOan_{\cB}[\![z]\!]$
given by~\eqref{eq:chi_action_on_Bri}.
Since $\chi_1,\dots,\chi_n$ form a regular sequence for
$\cOan_{\cB}[\![z]\!]$, we have
\[
H^i\big(\Omega^{\bullet,\rm an}_{\cB/\cM}[\![z]\!],zd + dF\wedge\big)
= \begin{cases}
0, & i \neq n, \\
\cOan_{\cB}[\![z]\!]/(\chi_1,\dots,\chi_n) \cOan_{\cB}[\![z]\!],
& i =n,
\end{cases}
\]
where the $n$th cohomology sheaf is supported on $\tpr^{-1}(\cM\times \{0\})
\cap \cB = \big\{u\in \cB \colon x_1\parfrac{F}{x_1}(u) =
\cdots = x_n \parfrac{F}{x_n}(u) = 0\big\}$.
Therefore we have
\[
\Brian(F)_\bSigma \cong
\pr_*H^n\big(\Omega^{\bullet,\rm an}_{\cB/\cM}[\![z]\!],zd + dF\wedge\big)
 \cong R^n\pr_*\big(
\Omega^{\bullet,\rm an}_{\cB/\cM}[\![z]\!],zd + dF\wedge\big).
\]
Note that the second and the third term is supported on $\cV = \cU \cap (\cM\times \{0\})$.
\end{Remark}

\subsubsection{The higher residue pairing on the analytified Brieskorn module}\label{subsubsec:higher_residue}
A version of K.~Saito's higher residue pairing~\cite{SaitoK:higherresidue}
on the completed equivariant Brieskorn module
$\Briequiv(F)\sphat_\bSigma$
was introduced in~\cite[Section~6]{CCIT:MS}
via the asymptotic expansions of oscillatory integrals.
We explain that the definition there can be extended to the
analytified Brieskorn module $\Briequivan(F)_\bSigma$.
Consider the equivariant potential function
(as appeared in Remark \ref{rem:twisted_de_Rham})
\begin{equation*}
F_\T = F_\T(x,q)= F(x,q) - \sum_{i=1}^n \chi_i \log x_i.
\end{equation*}
This is a multi-valued function on $\cY$ with parameter $\chi\in \Lie \T$.
For a fixed $(q,\chi) \in \cM_\T = \cM\times \Lie\T$,
(logarithmic) critical points of $F_\T|_{\pr^{-1}(q)}$ are solution
to the equation
\[
x_i \parfrac{F}{x_i}= \chi_i.
\]
Thus we can regard $\tpr \colon \cB \to \cU$ \eqref{eq:tpr}
as a family\footnote{More precisely, $\cB$ is the union of branches of
critical points that tend to $\tzero$ as $q \to 0=0_\bSigma$.}
of critical points of~$F_\T$.
It follows from the study \cite[Lemma~6.2]{CCIT:MS} of critical points
near~$0_\bSigma$ that
the fibre of the finite morphism $\tpr|_{\cB}
\colon \cB \to \cU$
at a generic point consists of $\dim H_\CR^*(\frX_\bSigma)$ many
reduced points.
We write
\[
\cUss= \big\{(q,\chi) \in \cU \colon
\text{$\tpr^{-1}(q,\chi) \cap \cB$ consists of only reduced points}\big\} \neq \varnothing,
\]
where ``ss'' means semisimplicity. The complement of $\cUss$ in $\cU$ (called \emph{caustic}) is an analytic
subvariety in $\cU$. Let $\phi \cdot \omega \in \Briequivan(F)_\bSigma$ be a section over $\cUss$,
where $\phi =\phi(x,q,\chi) \in \tpr_*(\cOan_{\cB}[\![z]\!])$. For $(q,\chi) \in \cUss$ and a critical point
$p \in \tpr^{-1}(q,\chi)\cap \cB$, we can define the \emph{formal asymptotic expansion}
of the oscillatory integral (see \cite[Section~6.2]{CCIT:MS})
\[
\int_{\Gamma(p)} e^{F_\T/z} \phi(x,q,\chi) \omega
\sim e^{F_\T(p)/z} (-2\pi z)^{n/2} \sum_{n=0}^\infty a_n(q,\chi) z^n
\qquad \text{as $z \to 0$.}
\]
We obtain the right-hand side by expanding the integrand $e^{F_\T/z} \phi(x,q,\chi)$ in Taylor series
at $p$ (with respect to the logarithmic co-ordinates $\log x_1,\dots,\log x_n$) and performing termwise (Gaussian) integration. More precisely, we have
\begin{equation}\label{eq:Asymp}
\sum_{n=0}^\infty a_n(q,\chi) z^n=
\frac{1}{|\bN_{\rm tor}| \sqrt{\det(h_{i,j})}}
\Big[ e^{-\frac{z}{2}\sum_{i,j} h^{i,j} \parfrac{}{s_i}
\parfrac{}{s_j}}
e^{F_\T^{\ge 3}/z} \phi(pe^s,q,\chi)
\Big]_{s=0},
\end{equation}
where $s_1,\dots,s_n$ are the
logarithmic co-ordinates centred at $p$
so that $x_j = x_j(p) e^{s_j}$,
\begin{equation}
\label{eq:Hessian_matrix}
h_{i,j} =
\parfrac{^2 F_\T}{\log x_i \partial \log x_j}(p)
\end{equation}
is the Hessian matrix at $p$, $(h^{i,j})$ are the coefficients
of the matrix inverse to $(h_{i,j})$ and
\[
F^{\ge 3}_\T = \sum_{k\ge 3} \frac{1}{k!}
\sum_{i_1,\dots,i_k}
\parfrac{^k F_\T}{\log x_{i_1} \cdots \partial \log x_{i_k}}(p)
s_{i_1} \cdots s_{i_k}
\]
is the truncated Taylor expansion of $F_\T$ at the critical point $p$.
\begin{Definition}
\label{def:Asymp}
We define $\Asymp_p(\phi\cdot \omega)$ to be the
right-hand side of \eqref{eq:Asymp}.
\end{Definition}

\begin{Remark}\label{rem:Asymp}\quad
\begin{enumerate}\itemsep=0pt
\item[(1)] When the critical point $p$ does not lie in the logarithmic
locus of $\cY$ and $\phi(x,q,\chi)$ is a~polynomial,
the above \emph{formal} asymptotic expansion
makes sense as an \emph{actual} asymptotic expansion:
for this we choose the integration cycle
$\Gamma(p)\subset \cY_q :=\pr^{-1}(q)$ to be a stable manifold
for the Morse function $x\mapsto \Re(F_\T(x,q))$ associated
with $p$ and assume that
$z$ approaches zero from the \emph{negative} real axis.
\item[(2)] More precisely, the above formal asymptotic expansion depends
on the choice of the square root of the Hessian.
This corresponds to the choice of an orientation of the Morse
cycle~$\Gamma(p)$ and a branch of~$(-2 \pi z)^{n/2}$.
\end{enumerate}
\end{Remark}

\begin{Definition}\label{def:higher_residue}
The \emph{higher residue pairing} of two sections
$s_1, s_2 \in \Briequivan(F)_\bSigma$ are defined as:
\[
P(s_1,s_2) (q,\chi) = \sum_{p\in \tpr^{-1}(q,\chi)}
\left[\Asymp_p(s_1)\right]_{z\to -z} \cdot \Asymp_p(s_2),
\]
where $(q,\chi) \in \cUss$.
\end{Definition}

The higher residue pairing gives a bilinear pairing
\[
P \colon \ \Briequivan(F)_\bSigma \bigr|_{\cUss}
\times \Briequivan(F)_\bSigma \bigr|_{\cUss}
\to \cOan_{\cUss}[\![z]\!],
\]
which satisfies the following properties \cite[Proposition~6.8]{CCIT:MS}:
\begin{itemize}\itemsep=0pt
\item[(a)] $P$ is
$\cOan_{\cUss}$-bilinear, non-degenerate, $z$-sesquilinear
and symmetric:
\begin{gather*}
P(f(-z) s_1,s_2) = P(s_1,f(z) s_2) = f(z) P(s_1,s_2),
\\
P(s_2,s_1) = P(s_1,s_2)|_{z\to -z} ,
\end{gather*}
where $f(z) \in \cO_{\cUss}[\![z]\!]$;
\item[(b)] $P$ is $\nabla$-flat, i.e., $d P(s_1,s_2) = P(\nabla s_1, s_2) +
P(s_1,\nabla s_2)$;
\item[(c)] $P$ is homogeneous of degree $-n$ ($n=\dim \frX_\bSigma$), i.e.,
\[
\left(z\parfrac{}{z}+\cE+n\right)P(s_1,s_2) = P(\Gr s_1, s_2) + P(s_1,\Gr s_2),
\]
where $\cE$ is the Euler vector field~\eqref{eq:Euler_B}. \item[(d)] along $z=0$, $P$ equals the Grothendieck residue pairing.
\end{itemize}

\begin{Definition}Suppose that $\cVss := \cUss \cap (\cM \times \{0\})$ is nonempty.
The non-equivariant higher residue pairing
$P \colon \Brian(F)_\bSigma|_{\cVss} \times \Brian(F)_\bSigma|_{\cVss}
\to \cO_{\cVss}[\![z]\!]$ is defined to be the restriction of
the above~$P$ to~$\cVss$.
If $\frX_\bSigma$ is compact, the argument in
\cite[Proposition~3.10]{Iritani:Integral}
shows that $\cVss$ is
an open dense subset of $\cV$, and hence the non-equivariant
higher residue pairing is defined.
\end{Definition}

\begin{Remark}The definition of the higher residue pairing in terms of oscillatory integrals
is originally due to Pham \cite[2\'eme Partie, 4]{Pham:Lefschetz}.
\end{Remark}

\begin{Remark}The higher residue pairing here does not necessarily extend to the caustic~$\cU \setminus \cUss$.
Under mirror symmetry, the higher residue pairing corresponds
to the Poincar\'e pairing \cite[Theorem 6.11]{CCIT:MS}, therefore it has poles
along $\chi=0$ when $\frX_\bSigma$ is noncompact.
When $\frX_\bSigma$ is compact, it extends to a
holomorphic and non-degenerate pairing
in a neighbourhood of $(0_\bSigma,0)$.
\end{Remark}

\subsection{Analytic mirror isomorphism}
\label{subsec:an_mirror_isom}
Using the convergence result \cite[Theorem 7.2]{CCIT:MS}
(which generalizes \cite[Theorem 1.2]{Iritani:coLef}),
we show that the mirror isomorphism in Theorem \ref{thm:mirror_isom}
extends to a neighbourhood of $0=0_\bSigma$
as an isomorphism between the \emph{analytified}
Brieskorn module $\Briequivan(F)_\bSigma$
and the \emph{analytic} quantum D-module.
Let $\cU \subset \cM_\T$ be an open neighbourhood of $(0_\bSigma,0)$
as in Corollary \ref{cor:tpr_locally_free}.
Recall that $\Briequivan(F)_\bSigma$ was defined on $\cU$.

The mirror isomorphism in Theorem \ref{thm:mirror_isom} induces,
via the isomorphism in Proposition \ref{prop:analytification_D_modules},
the following isomorphism
\begin{equation}\label{eq:formal_mirror_isom}
\hMir\colon \ \big(\Briequivan(F)_\bSigma\big)\sphat_{(0_\bSigma,0)}
\cong \mir^*\QDM_\T(\frX_\bSigma) \otimes_{R_\T[z][\![\Laa(\bSigma)_+]\!]}
\hcO_{\cU,(0_\bSigma,0)}[\![z]\!].
\end{equation}
This isomorphism extends to an analytic neighbourhood of $(0_\bSigma,0)$.
We recall the following facts from \cite[Section~7]{CCIT:MS}:
\begin{itemize}\itemsep=0pt
\item[(a)] the structure constants of the big equivariant
quantum product are convergent and analytic
in $q$, $\tau'$ and $\chi$
\cite[Corollary 7.3]{CCIT:MS} (where $q$ and $\tau'$ are parameters of
the K\"ahler moduli space, see Section~\ref{subsec:Kaehler_moduli}, and
$\chi$ is the equivariant parameter);
\item[\namedlabel{item:Mir_matrix}{(b)}]
choose a $\Gr$-homogeneous
basis $\{\Omega_i\}_{i=0}^s$ of the completed
Brieskorn module consisting of algebraic differential forms
(i.e., $\Omega_i \in \C[z][\OO(\bSigma)_+] \omega$)
and let $\{\phi_i\}_{i=0}^s$ be the basis of
$H_\CR^*(\frX_\bSigma)$ as in Section~\ref{subsec:qcoh};
then the matrix entries $M_i^j(q,\chi,z)$
of the mirror isomorphism $\Mir$ given by
\[
\Mir(\Omega_i) = \sum_{j=0}^s M_i^j(q,\chi,z) \phi_j
\]
belong to $\cOan(\cU')[\![z]\!]$ for some open neighbourhood
$\cU'$ of $(0_\bSigma,0)$ in $\cM_\T=\cM\times \Lie \T$
\cite[Theorem~7.1]{CCIT:MS};
\item[(c)] the mirror map $\mir$ is also analytic in a neighbourhood
of $(0_\bSigma,0)\in \cM_\T$ (\emph{ibid.}).
\end{itemize}
By Part~(a), the equivariant quantum D-module $\QDM_\T(\frX_\bSigma)$
(see~\eqref{eq:QDM})
extends to a small analytic neighbourhood $U$ of
$q=\tau'=\chi=0$ in the equivariant K\"ahler moduli space
\[
\big[\cM_{\rm A,\T}(\frX_\bSigma)/\Pic^\st(\frX_\bSigma)\big]
\]
(see~\eqref{eq:equiv_Kaehler_moduli}).
We denote it by
\[
\QDMan_\T(\frX_\bSigma) :=
\big(H^*_{\CR,\T}(\frX_\bSigma) \otimes \cOan_{\tU}[z],
\nabla, \Gr, P\big),
\]
where $\tU$ is the preimage of $U$ in $\cM_{\rm A, \T}(\frX_\bSigma)$;
$H^*_{\CR,\T}(\frX_\bSigma) \otimes \cOan_{\tU}[z]$ is a
$\Pic^\st(\frX_\bSigma)$-equivariant sheaf by the Galois symmetry
in Section~\ref{subsec:Kaehler_moduli}, and we regard it as a sheaf
on the stack $U = \big[\tU/\Pic^\st(\frX_\bSigma)\big]$.
By Part~(b), $\hMir^{-1}(\phi_i)$ extends to a section of
$\Briequivan(F)_\bSigma$ over a small analytic neighbourhood $\cU'\subset \cU$
of $(0_\bSigma,0)\in\cM_\T$.
By Part~(c), by shrinking $\cU'$ if necessary, we have an
analytic mirror map $\mir\colon \cU' \to U$.
We now have the following result.

\begin{Theorem}
\label{thm:an_mirror_isom}
The isomorphism $\hMir$ in \eqref{eq:formal_mirror_isom} extends
to an open neighbourhood
$\cU'\subset \cU$ of $(0_\bSigma,0)\in \cM_\T = \cM \times \Lie \T$
and yields an isomorphism $\Miran$ of
$\cOan_{\cU'}[\![z]\!]$-modules:
\[
\Miran\colon \ \Briequivan(F)_\bSigma\bigr|_{\cU'} \cong
\mir^* \overline{\QDMan_\T}(\frX_\bSigma),
\]
where $\overline{\QDMan_\T}(\frX_\bSigma)$
denotes the $z$-adic completion of $\QDMan_\T(\frX_\bSigma)$, i.e.,
\[
\overline{\QDMan_\T}(\frX_\bSigma) :=
\QDMan_\T(\frX_\bSigma) \otimes_{\cOan_{U}[z]}\cOan_{U}[\![z]\!].
\]
The analytic mirror isomorphism $\Miran$ satisfies
the same properties $(1)$--$(3)$ as in Theorem $\ref{thm:mirror_isom}$
$($see Section~$\ref{subsubsec:higher_residue}$ for the higher residue
pairing on $\Briequivan(F)_\bSigma)$.
\end{Theorem}

\begin{Remark}
\label{rem:choice_of_splitting}
As explained before Theorem \ref{thm:mirror_isom}, in the above theorem,
we choose a splitting $\bN \to \OO^\bSigma$ of the refined fan sequence
\eqref{eq:refined_fanseq}, which simultaneously defines
a partial connection $\nabla$ on $\Briequivan(F)_\bSigma$ and
the equivariant quantum D-module.
\end{Remark}

\begin{Remark}
The analytic mirror theorem above shows that the analytified Brieskorn
module can be further analytified in the $z$-direction
(since $\QDMan_\T(\frX_\bSigma)$ is analytic in $z$).
We can regard this as a solution to the Birkhoff problem (i.e., finding a normal
form of the Gauss--Manin connection)
which has been studied extensively in the construction of K.~Saito's
flat structure \cite{Barannikov:projective, Douai-Sabbah:I, Douai-Sabbah:II, Sabbah:tame, SaitoK:primitiveform, SaitoM:Brieskorn}.
\end{Remark}

\section{Discrepant wall-crossings}
\label{sec:discrepant}
We study the change of quantum cohomology of smooth toric DM stacks
under a ``discrepant'' wall crossing.
We show a decomposition of formal (i.e., completed in the variable $z$)
quantum D-modules under discrepant wall-crossings.
We work in the set-up of Section~\ref{subsec:data} and
fix the data $(\bN,\Pi,S)$ as usual.

\subsection{Discrepant transformation of smooth toric DM stacks}
\label{subsec:discrepant}
We describe birational transformations between smooth toric DM stacks
following \cite{GKZ:discriminants},
\cite[Sections~4 and~5]{Borisov-Horja:FM}, \cite[Section~5.1 and~6.3]{CIJ}
and \cite[Section~3]{Acosta-Shoemaker:toric_birational}.

Recall from Section~\ref{subsubsec:toric_stacks} that the
toric stacks $\frX_\bSigma$ with $\bSigma \in \Fan(S)$
arise as the GIT quotients of $\C^S$ by the torus $\LL_{\C^\times}$.
These toric stacks are birational to each other since they
contain the (stacky) torus $\big[(\C^\times)^S/\LL_{\C^\times}\big]$
as an open dense subset.
We can regard $\LL^\star_\R$ as the space of GIT stability conditions
for the $\LL_{\C^\times}$-action on~$\C^S$;
if we choose a stability condition from the interior of the
maximal cone $\cpl(\bSigma)$ of the secondary fan~$\Xi$
(see Definition~\ref{def:LG}), then the corresponding GIT quotient
is~$\frX_\bSigma$.
We choose two adjacent maximal cones $\cpl(\bSigma_+)$,
$\cpl(\bSigma_-)$ of $\Xi$ which are separated
by a hyperplane wall $W\subset \LL^\star_\R$.
Here we assume that
$W \cap \cpl(\bSigma_+) = W \cap \cpl(\bSigma_-)$ is a common codimension-one
face of $\cpl(\bSigma_\pm)$.
Let $\bw\in \LL$ be a primitive integral vector which is
perpendicular to the wall $W\subset \LL^\star_\R$
and is non-negative on the chamber $\cpl(\bSigma_+)$.
By the definition of $\LL$ (see~\eqref{eq:ext_fanseq}), the vector $\bw\in \LL$
gives rise to a linear relation
\begin{equation}\label{eq:circuit}
\sum_{b\in S} (D_b \cdot \bw) b =0,
\end{equation}
where recall that $D_b = D(e_b^\star)$ (see~\eqref{eq:ext_divseq}).
This linear relation defines a \emph{circuit}
$\{b \in S: D_b \cdot \bw \neq 0\}$ in the terminology of
Gelfand--Kapranov--Zelevinsky \cite{GKZ:discriminants},
where a `circuit' means a minimal linearly dependent set.
The transition between $\bSigma_+$ and $\bSigma_-$
can be described in terms of the circuit
(`modification along a circuit' \cite{GKZ:discriminants}).

\begin{figure}[htbp]\centering
\includegraphics{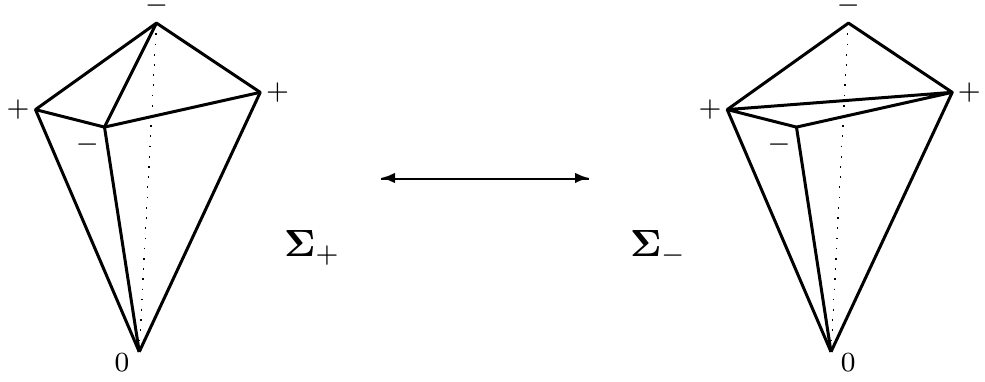}
\caption{Modification along a circuit. The signs $\pm$ mean rays belonging to $M_\pm = \{b \in S \colon \pm D_b \cdot \bw >0\}$.} \label{fig:modification_along_a_circuit}
\end{figure}

For $I\subset S$, we write\footnote{Note that $\sigma_I$ is a closed
cone, whereas $\angle_I$ is a relatively open cone.}
\begin{gather*}
\sigma_I := \sum_{b\in I} \R_{\ge 0} \overline{b} \subset \bN_\R, \qquad
\angle_I := \sum_{b\in I} \R_{>0} D_b \subset \LL^\star_\R.
\end{gather*}
We also set $\sigma_\varnothing = \{0\}$, $\angle_\varnothing = \{0\}$.
Consider the (not necessarily simplicial) fan $\Sigma_0$ on the vector
space $\bN_\R$ given by $\Sigma_0 := \{\sigma_I \colon I \in \scrS_0\}$,
where
\begin{gather*}
\scrS_0 = \{ I \subset S\colon \text{$\angle_{S\setminus I}$ contains the relative interior
of $\cpl(\bSigma_+) \cap \cpl(\bSigma_-)$} \}.
\end{gather*}
The simplicial fans $\Sigma_\pm$ underlying
$\bSigma_\pm$ arise as different subdivisions of $\Sigma_0$.
Set $M_\pm = \{b\in S \colon \allowbreak {\pm}D_b \cdot \bw >0\}$.
We have a decomposition $\scrS_0 = \scrS_0^{\rm simp}
\sqcup \scrS_0^{\rm circ}$ (disjoint union), where
\begin{gather*}
\scrS_0^{\rm simp} =
\{ I \in \scrS_0 \colon M_+ \not\subset I, M_- \not \subset I \}, \\
\scrS_0^{\rm circ} =
\{I \in \scrS_0 \colon M_+ \cup M_- \subset I\},
\end{gather*}
such that $\Sigma_\pm = \{\sigma_I \colon I \in \scrS_\pm\}$ with
\begin{gather*}
\scrS_\pm = \big\{ I \colon I \in \scrS_0^{\rm simp}
\big\} \sqcup
\big\{ I\setminus J\colon I \in \scrS_0^{\rm circ},\,
\varnothing \neq J\subset M_\pm \big\}.
\end{gather*}
The set of rays of $\bSigma_\pm$ is given by
$R_\pm := R(\bSigma_\pm) = \bigcup\limits_{I\in \scrS_\pm} I$.
See \cite[Lemma~5.2]{CIJ}\footnote{In \cite[Section~5.1]{CIJ}, the set of ``anti-cones''
$\{ S \setminus I \colon I \in \scrS_0\}$,
$\{ S \setminus I \colon I \in \scrS_0^{\rm simp}\}$,
$\{ S \setminus I \colon I \in \scrS_0^{\rm circ}\}$
are denoted by $\cA_0$, $\cA_0^{\rm thick}$,
$\cA_0^{\rm thin}$ respectively.}.
Here $I \in \scrS_0^{\rm simp}$ yields a simplicial cone
$\sigma_I\in \Sigma_0$ belonging to both $\Sigma_+$ and $\Sigma_-$,
and $I \in \scrS_0^{\rm circ}$ yields a
(not necessarily simplicial) cone $\sigma_I\in \Sigma_0$
containing the circuit $M_+\cup M_-$;
the cone $\sigma_I$ with $I\in \scrS_0^{\rm circ}$ is
subdivided into simplicial cones $\sigma_{I\setminus \{v\}}$, $v\in M_\pm$
in the fans $\Sigma_\pm$.
See Fig.~\ref{fig:modification_along_a_circuit}.
We also remark that $M_+ \cup M_- \in \scrS_0^{\rm circ}$
so that $(M_+ \cup M_-) \setminus \{v\} \in \scrS_\pm$
for every $v\in M_\pm$. (In particular, $M_\pm \subset R_\mp$.)

Let $\frX_\pm$ denote the toric stack corresponding to
$\bSigma_\pm$. As discussed in \cite[Section~5]{Borisov-Horja:FM},
\cite[Section~6.3]{CIJ}, the toric birational map
$\varphi\colon \frX_+\dasharrow \frX_-$ fits into a commutative diagram
\begin{equation}
\label{eq:roof}
\begin{aligned}
\xymatrix{ & \hfrX \ar[ld]_{f_+} \ar[rd]^{f_-} & \\
\frX_+ \ar@{-->}[rr]^{\varphi} \ar[dr]_{g_+} & & \frX_-, \ar[dl]^{g_-} \\
& X_0 & }
\end{aligned}
\end{equation}
where $X_0$ is the toric variety associated with $\Sigma_0$,
$\hfrX$ is another smooth toric DM stack and $f_\pm
\colon \hfrX \to \frX_\pm$, $g_\pm \colon
\frX_\pm \to X_0$ are projective birational toric
morphisms. Define $\hb\in \bN\cap \Pi$ to be the vector:
\[
\hb := \sum_{D_b\cdot \bw>0} (D_b \cdot \bw) b = - \sum_{D_b\cdot \bw<0}
(D_b\cdot \bw) b.
\]
The smooth toric DM stack $\hfrX$ is given by a stacky fan $\hbSigma$
adapted to $S \cup \{\hb\}$ (in the sense of
Definition \ref{def:stacky_fan_adapted_to_S}):
the set of rays of $\hbSigma$ is $\hR = (R_+ \cap R_-)
\cup \{\hb\}$ where $R_\pm := R(\bSigma_\pm)$;
the fan $\hSigma$ underlying $\hbSigma$ is a simultaneous
subdivision of $\Sigma_+$ and $\Sigma_-$ given by
\[
\hSigma =
\big\{\sigma_I \colon I \in \scrS_0^{\rm simp}\big\} \sqcup
\big\{\sigma_K\colon K = I \setminus (J_+\cup J_-) \cup \{\hb\}, \,
I \in \scrS_0^{\rm circ}, \,
\varnothing \neq J_\pm \subset M_\pm \big\}.
\]
The toric morphisms $f_\pm \colon \hfrX \to \frX_\pm$ are
induced by natural maps $\hbSigma \to \bSigma_\pm$ of stacky fans.
We refer the reader to \cite[Section~6.3]{CIJ} for a description
of $\hfrX$ and $f_\pm$ in terms of GIT quotients.

\begin{Lemma}Let $K_\pm = K_{\frX_\pm}$ denote the canonical class of $\frX_\pm$
and $E \subset \hfrX$ denote the toric divisor of $\hfrX$ corresponding to
the ray $\hb$. Then we have
\[
f_+^\star K_{+} = f_-^\star K_{-} +
\left(\sum_{b\in S} D_b \cdot \bw\right) [E].
\]
\end{Lemma}
\begin{proof}This follows immediately from \cite[Proposition 6.21]{CIJ}.
\end{proof}

In view of the lemma above, we say that the birational transformation $\frX_+ \dasharrow \frX_-$
is \emph{crepant} if $\sum\limits_{b\in S} D_b\cdot \bw = 0$
(i.e., $\sum\limits_{b\in S} D_b$ is on the wall $W$)
and \emph{discrepant} otherwise.
We shall restrict ourselves to the case where the transformation is
discrepant. By exchanging $\frX_+$ and $\frX_-$
if necessary, we may assume:
\begin{Assumption}
\label{assump:discrepancy}
The birational transformation $\frX_+ \dasharrow \frX_-$ satisfies
$f_+^\star K_+ > f_-^\star K_-$, i.e., $\sum\limits_{b\in S} D_b\cdot \bw> 0$.
\end{Assumption}
\begin{Remark}\label{rem:cohomology_drop}
Under the above assumption, the dimension of orbifold cohomology
decreases: $\dim H^*_{\CR}(\frX_+)>\dim H^*_\CR(\frX_-)$.
We can see this from the change of the fans, and
using the fact that $|\bN_{\rm tor}|^{-1}
\dim H^*_\CR(\frX_\bSigma)$ equals the sum of volumes
of simplices spanned by $\{0\} \cup \{\overline{b}\in R(\bSigma)\colon\allowbreak
\overline{b} \in \sigma\}$ over all maximal cones $\sigma$ of $\bSigma$,
where we normalize the volume of the standard simplex to be one
(see, e.g., \cite[Lemma~3.9]{Iritani:Integral}).
\end{Remark}

\begin{Remark}\label{rem:3_types} There are three types of discrepant wall-crossings:
(I) $\frX_+$ and $\frX_-$ are isomorphic in codimension one (``flip''),
(II) the birational map induces a map
(i) $\frX_+ \to X_-$ or (ii)~\mbox{$\frX_- \to X_+$}
contracting a divisor to a toric subvariety,
where $X_\pm$ is the coarse moduli space of $\frX_\pm$
(``discrepant resolution'') and
(III) $X_+ = X_-$ but the stack structures of $\frX_+$ and $\frX_-$ differ along
a~divisor. In terms of stacky fans, we have
\begin{itemize}\itemsep=0pt
\item[(I)] $R_+ = R_-$,
$\sharp M_+ \ge 2$ and $\sharp M_-\ge 2$ (rays are the same);
\item[(II-i)] $R_+ = R_-\sqcup M_-$, $\sharp M_-=1$ and
$\sharp M_+\ge 2$ (removing one ray);
\item[(II-ii)] $R_- = R_+ \sqcup M_+$, $\sharp M_+ = 1$ and $\sharp M_-\ge 2$
(adding one ray);
\item[(III)] $R_+ \setminus R_- = M_-$, $R_- \setminus R_+ =M_+$ and
$\sharp M_- = \sharp M_+ =1$
(replace a ray $b_- \in R_+$ with a~shorter and parallel ray $b_+\in R_-$, where
$M_\pm = \{b_\pm\}$).
\end{itemize}
This is similar to the classification of crepant transformations
given in \cite[Propositions~5.4 and~5.5]{CIJ} and can be shown by a parallel argument.
\end{Remark}

\begin{Example}Let $a_1,\dots,a_k$, $b_1,\dots,b_l$ be positive integers
with \mbox{$a_1+ \cdots + a_k<b_1+\cdots + b_l$}.
Consider the $\C^\times$-action on $\C^{k+l}$
given by the weights $(-a_1,\dots,-a_k, b_1,\dots,b_l)$.
The GIT variation gives a discrepant transformation between
the spaces:
\begin{gather*}
\frX_+ = \text{the total space of $\cO(-a_1)\oplus \cdots \oplus \cO(-a_k)$
over $\PP(b_1,\dots,b_l)$, and} \\
\frX_- = \text{the total space of $\cO(-b_1) \oplus \cdots \oplus \cO(-b_l)$
over $\PP(a_1,\dots,a_k)$.}
\end{gather*}
Following the classification in Remark \ref{rem:3_types}, we have:
(I) if $k, l \ge 2$, this is a flip;
(II-i) if $k=1$ and $l\ge 2$, this is a resolution
$\frX_+ \to X_- = \C^l/\mu_{a_1}$ with positive discrepancy;
(II-ii) if $k\ge 2$ and $l=1$, this is a resolution $\frX_- \to X_+=\C^k/\mu_{b_1}$
with negative discrepancy;
(III) if $k=l=1$, we have $\frX_+ = [\C/\mu_{b_1}]$
and $\frX_- = [\C/\mu_{a_1}]$; the stack structure at the origin changes.
Note that the example in Section~\ref{subsubsec:cyclic} is a special case
of the current example
with $k=2$, $l=1$, $(a_1,a_2,b_1) =(1,1,d)$.
\end{Example}

\begin{Example}A blow-up along a toric subvariety is an example of type (II) discrepant transformation.
\end{Example}

\begin{Example}A root construction~\cite{Cadman:tangency}
along a toric divisor is an example of type~(III) discrepant transformation.
\end{Example}

\subsection{The LG model along a curve} \label{subsec:LG_curve}
Consider the partially compactified LG model
$(\pr\colon \cY \to \cM, F)$ from Section~\ref{subsec:LG}. Recall that~$\cM$ is defined in terms of
the secondary fan $\Xi$ consisting of maximal cones $\cpl(\bSigma)$,
$\bSigma \in \Fan(S)$.
Let $\cC\subset \cM$ denote the 1-dimensional toric substack
corresponding to the codimension-1 cone
$\cpl(\bSigma_+) \cap \cpl(\bSigma_+)$.
The curve $\cC$ lies in the boundary of $\cM$ and connects the large radius limit points $0_{\bSigma_+}$ and $0_{\bSigma_-}$.

We cover $\cC$ by the two open sets $\cM_{\pm} :=
\cM_{\bSigma_\pm} = \big[\Spec \C[\Laa(\bSigma_\pm)_+]/\Pic^\st(\frX_\pm)\big]$;
in these local charts, the embedding $\cC \subset
\cM$ is given by the $\C$-algebra homomorphism:
\begin{equation}\label{eq:curve_embedding}
\C[\Laa(\bSigma_\pm)_+] \to \C\big[q^{\pm\bw/e_\pm}\big],
\qquad
q^\lambda \mapsto
\begin{cases}
q^\lambda, &\text{if $\lambda$ is proportional to $\bw$}, \\
0, & \text{otherwise},
\end{cases}
\end{equation}
where $e_\pm\in \N$ is the smallest common denominator of
$\{c\in \Q \colon c \bw \in \Laa(\bSigma_\pm)\}$.
(Recall that $\Laa(\bSigma_\pm)_+ = \Laa(\bSigma_\pm) \cap
\cpl(\bSigma_\pm)^\vee$,
see Lemma \ref{lem:dual_cone_lattice}.)

\begin{Lemma}
\label{lem:LG_on_curve}
The inverse image $\pr^{-1}(\cC)$ is covered by two charts
\[
\pr^{-1}(\cC \cap \cM_\pm) = \big[\Spec(A_\pm)/\Pic^\st(\frX_\pm)\big],
\]
where $A_\pm$ is the $\C\big[q^{\pm \bw/e_\pm}\big]$-algebra
$\bigoplus\limits_{v\in \bN \cap \Pi} \C\big[q^{\pm \bw/e_\pm}\big] w_v^\pm$
equipped with the product
\[
w_{v_1}^\pm w_{v_2}^\pm =
\begin{cases}
q^{\Psi^\pm(v_1) + \Psi^\pm(v_2) - \Psi^\pm(v_1+v_2)}
w_{v_1+v_2}^\pm, & \text{if $v_1$, $v_2$ lie in the same cone of
$\Sigma_0$}, \\
0 & \text{otherwise},
\end{cases}
\]
where $\Psi^\pm(v_1) + \Psi^\pm(v_2) - \Psi^\pm(v_1+v_2)$ is
proportional to $\bw$ in the first case,
$w_v^\pm$ is the restriction of $u^{(\Psi^{\pm}(\overline{v}),v)}$
to $\pr^{-1}(\cC\cap \cM_\pm)$
and $\Psi^\pm := \Psi^{\bSigma_\pm}$ $($see Notation $\ref{nota:Psi})$.
The two charts are glued by
\[
w_v^- =
\begin{cases}
q^{\Psi^-(v) - \Psi^+(v)}
w_v^+, & \text{if $v$ lies in a cone $\sigma_I$ with
$I\in \scrS_0^{\rm circ}$}, \\
w_v^+, &\text{otherwise}.
\end{cases}
\]
\end{Lemma}
\begin{proof}The space $\pr^{-1}(\cC\cap \cM_\pm)$ is the base change of
$\cY_{\bSigma_\pm} =
\big[\Spec \C[\OO(\bSigma_\pm)_+]/\Pic^\st(\frX_\pm)\big]$
via~\eqref{eq:curve_embedding}.
The conclusion follows from Remark~\ref{rem:flatness_pr}
and the description of cones of~$\Sigma_0$,~$\Sigma_\pm$ in terms of the circuit~\eqref{eq:circuit}.
\end{proof}

\begin{Remark}\label{rem:curve_asymptotics}We describe how the curve $\cC$ looks like in the K\"ahler
moduli space. Using the equation \eqref{eq:curve_embedding} for $\cC$ and the asymptotics of the mirror map
(Remark~\ref{rem:mirror_map_asymptotic}), we find that the image of $\cC$ under the (non-equivariant)
mirror map for~$\frX_+$ is asymptotically close to, near the large radius limit point~$0_{\bSigma_+}$,
\begin{description}\itemsep=0pt
\item[type (I) or (II-i) case:]
the curve given by $\tau' = 0$ and $q^d=0$
for all $d\in \Laa_+^{\bSigma_+} \setminus \Q_{\ge 0} \bw$
(i.e., the curve corresponding to the extremal class
$\bw \in H_2(\frX_+,\Q)$);
\item[type (II-ii) or (III) case:] the curve given by
$\tau' \in \C \frD_{b_+}$ and
$q^d = 0$ for all
$d \in\Laa_+^{\bSigma_+}\setminus \{0\}$,
where $b_+$ is the unique element of $M_+$
and $\frD_{b_+}\in H^{<2}_{\CR}(\frX_+)$
is as in Remark \ref{rem:mirror_map_asymptotic}.
\end{description}
Here we use the notation on the K\"ahler moduli space from Section~\ref{subsec:Kaehler_moduli}
and the classification in Remark~\ref{rem:3_types}.
Similarly, near $0_{\bSigma_-}$, the image of $\cC$
under the mirror map for $\frX_-$ is asymptotically close to
\begin{description}\itemsep=0pt
\item[type (I) or (II-ii) case:]
the curve given by $\tau' = 0$ and $q^d=0$
for all $d\in \Laa_+^{\bSigma_-} \setminus \Q_{\ge 0} (-\bw)$
(i.e., the curve corresponding to the extremal class $-\bw\in H_2(\frX_-,\Q)$);
\item[type (II-i) or (III) case:]
the curve given by $\tau' \in \C \frD_{b_-}$ and
$q^d = 0$ for all $d \in\Laa_+^{\bSigma_-}\setminus \{0\}$,
where $b_-$ is the unique element of $M_-$
and $\frD_{b_-} \in H^{>2}_\CR(\frX_-)$ is as in
Remark \ref{rem:mirror_map_asymptotic}.
\end{description}
The above classes $\frD_{b_\pm}$ are supported on the image
of the exceptional divisor. We note that they can be zero, and
in that case we need to examine the higher-order terms in
the mirror map to see the asymptotic behaviour of $\cC$.
\end{Remark}

\begin{Remark}
Note that $\pr^{-1}(\cC)\subset \cY$ is a possibly
reducible toric substack and its components are in one-to-one
correspondence with maximal cones of the fan $\Sigma_0$.
The LG potential restricted to $\pr^{-1}(\cC)$ is of the form
$F= \sum\limits_{b\in R_+\cup R_-} u_b$
(here $u_b$ with $b\in S\setminus (R_+\cup R_-)$ vanishes
on $\pr^{-1}(\cC)$).
\end{Remark}

\subsection{Decomposition of the Brieskorn module}\label{subsec:decomp_Bri}
In this section, we show that the analytified Brieskorn module
associated with $\bSigma_-$
is a direct summand of that associated with $\bSigma_+$
in a neighbourhood of $\cC$.
For this, we study the family~\eqref{eq:tpr}
\[
\tpr =\left(\pr, x_1\parfrac{F}{x_1},\dots,x_n \parfrac{F}{x_n}\right)
\colon \ \cY \to \cM_\T = \cM \times \Lie \T
\]
over an analytic neighbourhood of the curve $\cC \times \{0\}
\subset \cM \times \Lie \T$,
where $x_1,\dots,x_n$ are co-ordinates along fibres of $\cY \to \cM$
as in Section~\ref{subsec:Bri}.
Recall that this family can be regarded as the
relative critical scheme of $F_\T$ (see Section~\ref{subsubsec:higher_residue}).

The $\C^\times$-action generated by the Euler vector field \eqref{eq:Euler_B}
plays an important role in the fol\-lowing discussion.
Consider the elements $\sum\limits_{b\in S} e_b^\star \in \big(\Z^S\big)^\star$
and $\sum\limits_{b\in S} D_b \in \LL^\star$; they define
$\C^\times$-actions, respectively, on $\cY$ and $\cM$
such that $\pr \colon \cY \to \cM$ is $\C^\times$-equivariant.
In terms of the co-ordinates~$(u_b)_{b\in S}$, the $\C^\times$-action
is given by
\begin{align*}
s\cdot u_b = s u_b \qquad \text{with $s\in \C^\times$},
\end{align*}
and the potential function $F=\sum\limits_{b\in S} u_b$ is of weight 1
with respect to the action.
Introduce the $\C^\times$-action on $\Lie \T$ given by
the scalar multiplication; then the map $\tpr \colon \cY
\to \cM \times \Lie \T$ is $\C^\times$-equivariant.
Let $0_\pm := 0_{\bSigma_\pm} = \{q^{\pm \bw} = 0\} \in \cC$ denote
the large radius limit points of $\frX_\pm$
(see Definition~\ref{def:large_radius_limit}) and
let $\tzero_\pm = \tzero_{\bSigma_\pm}\in \pr^{-1}(\cC)$ denote
the torus-fixed points such that $\pr(\tzero_\pm) = 0_\pm$
as in Section~\ref{subsubsec:analytification_algebras}.
The $\C^\times$-action on the family $\pr^{-1}(\cC) \to \cC$ is given by
(with notation as in Lemma~\ref{lem:LG_on_curve})
\[
s\cdot w^\pm_v = s^{\sum_{b\in S} \Psi^\pm_b(v)} w^\pm_v,
\qquad
s\cdot q^\bw = s^{\sum_{b\in S} D_b \cdot \bw} q^\bw.
\]
By Assumption \ref{assump:discrepancy},
we have that $\lim\limits_{s\to 0} s \cdot x = 0_+$,
$\lim\limits_{s\to \infty} s\cdot x = 0_-$ for every
$x\in \cC \setminus \{0_+,0_-\}$
and
$\lim\limits_{s\to 0} s \cdot y = \tzero_+$ for every
$y\in \pr^{-1}(\cC \setminus \{0_-\})$.

We choose analytic open sets
$\cB_\pm \subset \cY$, $\cU_\pm \subset \cM_\T$
with $\tzero_\pm \in \cB_\pm$, $(0_\pm,0) \in \cU_\pm$
such that the conclusion of Corollary \ref{cor:tpr_locally_free}
holds.
Since $\tpr$ is $\C^\times$-equivariant, even after replacing
$\cB_\pm$ and $\cU_\pm$ with
\[
\bigcup_{s\in \C^\times} s \cdot \cB_\pm,
\qquad \text{and} \qquad
\bigcup_{s\in \C^\times} s \cdot \cU_\pm,
\]
we have that the conclusion of Corollary \ref{cor:tpr_locally_free}
still holds. \emph{We henceforth assume that $\cB_\pm$, $\cU_\pm$
are preserved by the $\C^\times$-action}.
Since every point in $\pr^{-1}(\cC\setminus \{0_-\})$
flows to $\tzero_+$ under the $\C^\times$-action
and $\cB_+$ is an open neighbourhood of $\tzero_+$,
we have that $\pr^{-1}(\cC\setminus \{0_-\}) \subset \cB_+$.
Similarly we have $(\cC \setminus \{0_+,0_-\})
\times \{0\}\subset \cU_\pm$.

\begin{Lemma}
\label{lem:subcover}
There exists an analytic open set $\cU_0$ of
$\cM_\T = \cM\times \Lie\T$ such that
\begin{itemize}\itemsep=0pt
\item[{\rm (1)}] $(\cC\setminus\{0_+,0_-\}) \times \{0\} \subset \cU_0
\subset \cU_+ \cap \cU_-$;
\item[{\rm (2)}]
$\tpr^{-1}(\cU_0) \cap \cB_+ = (\tpr^{-1}(\cU_0) \cap \cB_-)
\sqcup \cR$ for some open set $\cR$ of $\cY$.
\end{itemize}
\end{Lemma}
\begin{proof}
First note that $\cU_+\cap \cU_-$ contains
$(\cC\setminus \{0_+,0_-\})\times \{0\}$.
Since $\tpr\colon \cB_- \to \cU_-$ is proper and
$\cB_- \setminus \cB_+$ is closed in $\cB_-$,
$\tpr(\cB_-\setminus \cB_+)$ is closed in $\cU_-$.
Define $\cU_0:=(\cU_+ \cap \cU_-) \setminus \tpr(\cB_-\setminus \cB_+)$.
This is an open subset of $\cU_+\cap \cU_-$.
By definition we have $\tpr^{-1}(\cU_0) \cap \cB_- \subset
\tpr^{-1}(\cU_0) \cap \cB_+$.
Since we have $\tpr^{-1}((\cC\setminus \{0_+,0_-\}) \times \{0\})
\subset \pr^{-1}(\cC\setminus \{0_+,0_-\})
\subset \cB_+$, we conclude
$(\cC\setminus \{0_+,0_-\})\times \{0\} \subset \cU_0$.
Finally we show that the complement of $\tpr^{-1}(\cU_0) \cap \cB_-$
in $\tpr^{-1}(\cU_0) \cap \cB_+$ is open. Take any point
$y$ from the complement, and choose a compact neighbourhood $K$
of $\tpr(y)$ in $\cU_0$. Since $\tpr\colon \cB_-\to \cU_-$ is proper,
$\tpr^{-1}(K)\cap \cB_-$ is a compact set not containing $y$.
Let $K^\circ$ be the interior of $K$;
then $(\tpr^{-1}(K^\circ) \cap \cB_+) \setminus
(\tpr^{-1}(K)\cap \cB_-)$ is an open neighbourhood of $y$
which does not intersect with $\tpr^{-1}(\cU_0) \cap \cB_-$.
\end{proof}

We consider the analytified equivariant Brieskorn module
$\Briequivan(F)_\pm := \Briequivan(F)_{\bSigma_\pm}$
from Definition \ref{def:analytic_Bri}. Note that it is defined
over the $\cU_\pm$ above, since the only properties
we need in the construction are those in Corollary \ref{cor:tpr_locally_free}.

\begin{Corollary}
\label{cor:direct_summand}
Let $\cU_0,\cR$ be as in Lemma $\ref{lem:subcover}$.
\begin{itemize}\itemsep=0pt
\item[{\rm (1)}]
We have a direct sum decomposition of $\cOan_{\cU_0}$-algebras:
\[
(\tpr_*\cOan_{\cB_+})\big|_{\cU_0} \cong
(\tpr_*\cOan_{\cB_-})\big|_{\cU_0}\oplus (\tpr_*\cOan_\cR).
\]
\item[{\rm (2)}]
We have a direct sum decomposition of $\cOan_{\cU_0}[\![z]\!]$-modules:
\[
\Briequivan(F)_+\big|_{\cU_0}
\cong \Briequivan(F)_-\big|_{\cU_0}
\oplus \tpr_*(\cOan_\cR[\![z]\!]).
\]
Under this decomposition, the Gauss--Manin
connection $\nabla$, the grading operator
$\Gr$ and the higher residue pairing $P$ split into the direct sum.
\end{itemize}
\end{Corollary}

\begin{Definition}
\label{def:rest}
We define $\scrR := \tpr_*(\cOan_\cR[\![z]\!])$.
Since $\tpr\colon \cR \to \cU_0$ is a finite flat morphism,
we can define the $\cOan_{\cU_0}[\![z]\!]$-module structure,
the Gauss--Manin connection $\nabla^\scrR$,
the grading opera\-tor~$\Gr^\scrR$ and
the higher residue pairing~$P^\scrR$ on~$\scrR$
similarly to Definition~\ref{def:analytic_Bri} and
Section~\ref{subsubsec:higher_residue}.
By the same argument as in Proposition~\ref{prop:Briequivan_finite_free},
$\scrR$ is a free $\cOan_{\cU_0}[\![z]\!]$-module
of rank $\dim H^*_\CR(\frX_+) - \dim H^*_\CR(\frX_-)$.
For a generic $(q,\chi) \in \cU_0$, $\tpr^{-1}(q,\chi)\cap \cR$
consists of finitely many reduced points, and the asymptotic expansion
in Definition \ref{def:Asymp} defines an isomorphism
$\Asymp\colon \scrR_{(q,\chi)} \cong \C[\![z]\!]^{\oplus \rank \scrR}$.
\end{Definition}

\begin{Remark}
Since $\rank \tpr_*(\cOan_{\cB_+})
= \dim H_\CR^*(\frX_+) > \dim H_\CR^*(\frX_-)
= \rank \tpr_*(\cOan_{\cB_-})$
(by Corollary \ref{cor:tpr_locally_free} and
Remark \ref{rem:cohomology_drop}),
$\cR$ is never empty under Assumption~\ref{assump:discrepancy}.
\end{Remark}

\begin{figure}[htbp]\centering
\includegraphics{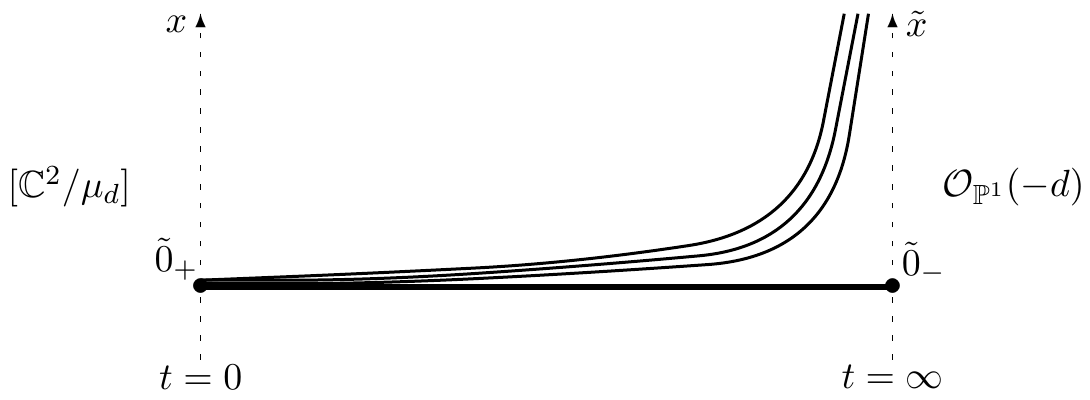}
\caption{Family of critical points. Out of $d$ critical points,
$d-2$ points go to infinity at the large radius limit of
$\cO_{\PP^1}(-d)$ ($d=5$ in the picture).} \label{fig:critical_points_transition}
\end{figure}

\begin{Example} Consider the example of the LG model in Section~\ref{subsubsec:cyclic}.
This corresponds to a
discrepant transformation between $\frX_+ = \big[\C^2/\mu_d\big]$
(of type $\frac{1}{d}(1,1)$) and its minimal resolution
$\frX_- = \cO_{\PP^1}(-d)$. The LG model was given by
the potential function
\[
F = x_2 + \frac{x_1^d}{x_2} + t x_1 = \tx_2+ q \frac{\tx_1^d}{\tx_2}
+ \tx_1,
\]
where $(x_1,x_2,t)$ and $(\tx_1,\tx_2,q)$ are related by
$\tx_1 = t x_1$, $\tx_2 = x_2$, $q=t^{-d}$.
Here $q=t^{-d}$ is a~co-ordinate on $\cM = \PP(1,d)$,
$\{t=0\}$ is the large radius limit point of $\frX_+$
and $\{q=0\}$ is the large radius limit point of $\frX_-$.
On the chart near $t=0$ (associated to $\frX_+$),
the precise domain of definition of the potential
function $F$ is given by (see Section~\ref{subsubsec:cyclic})
\[
\cY_{\bSigma_+}= \big\{(x_1,x_2,y,t) \in \C^4\colon x_2 y =x_1^d\big\},
\]
where $y$ is identified with the term $x_1^d/x_2$ in $F$.
The relative critical scheme of $F$ on this chart is given by
\begin{align*}
\left\{x_1\parfrac{F}{x_1}= x_2\parfrac{F}{x_2} = 0\right\}
& = \left\{
d \frac{x_1^d}{x_2}+ t x_1 = 0, \
x_2 - \frac{x_1^d}{x_2} =0 \right\}
\\
& = \big\{ (x_1,x_2,y,t) \in \C^4\colon
d y+ tx_1 =0, \ x_2 -y = 0, \ x_2 y = x_1^d \big\} \\
& = \left\{\left(x_1, - \frac{tx_1}{d}, -\frac{tx_1}{d},t\right) \in \C^4
\colon x_1^2\big(x_1^{d-2} - t^2/d^2\big) =0\right\}.
\end{align*}
This is schematically depicted in Fig.~\ref{fig:critical_points_transition}.
\end{Example}

\subsection{Comparison of quantum D-modules}\label{subsec:comparison_QDM}
Finally we obtain a comparison result between the quantum D-modules of
$\frX_+$ and $\frX_-$
combining Theorem~\ref{thm:an_mirror_isom} and Corollary~\ref{cor:direct_summand}.

Recall the analytic quantum D-module
$\QDMan_\T(\frX_\pm)$ of $\frX_\pm$
from Section~\ref{subsec:an_mirror_isom};
it is defined over an open neighbourhood $U_\pm$ of the origin in the
equivariant K\"ahler moduli space
\[
\big[\cM_{\rm A,\T}(\frX_\pm)/\Pic^\st(\frX_\pm)\big].
\]
By Theorem~\ref{thm:an_mirror_isom},
we have the analytic mirror isomorphism
\begin{equation}
\label{eq:an_mirror_isom_pm}
\Miran_\pm \colon \ \Briequivan(F)_\pm \big|_{\cU'_\pm}
\cong \mir_\pm^*\overline{\QDMan_\T}(\frX_\pm)
\end{equation}
over an open neighbourhood $\cU'_\pm\subset \cU_\pm$ of $(0_\pm,0)
\in \cM_\T = \cM \times \Lie \T$,
where $\mir_\pm \colon \cU'_\pm \to U_\pm$ denotes the mirror map
and the overline $\overline{\cdots}$ means the $z$-adic completion.
Let us observe that the analytic mirror isomorphism extends to a domain
which is closed under the $\C^\times$-action
(as discussed in Section~\ref{subsec:decomp_Bri}).
We have already seen in Section~\ref{subsec:decomp_Bri}
that $\Briequivan(F)_\pm$ extends to
$\C^\times \cU'_\pm = \bigcup\limits_{s\in \C^\times} s \cdot \cU'_\pm$.
Introduce the $\C^\times$-action on the equivariant K\"ahler moduli
space $[\cM_{\rm A,\T}(\frX_\pm)/\Pic^\st(\frX_\pm)]$
generated by the Euler vector field
\eqref{eq:Euler_A}; since the mirror map $\mir_\pm$ preserves
the Euler vector fields, it can be extended to a unique
$\C^\times$-equivariant map $\mir_\pm \colon \C^\times
\cU'_\pm \to \C^\times U_\pm$.
We may assume (by shrinking $\cU'$ if necessary) that
$\Briequivan(F)_\pm|_{\cU'}$ is generated by $\Gr$-homogeneous
sections $\Omega_0^\pm,\dots,\Omega_s^\pm$ over
$\cOan_{\cU'_\pm}[\![z]\!]$
(see \cite[Theorem 4.26]{CCIT:MS} and
Part~(b) in Section~\ref{subsec:an_mirror_isom}).
Since $\Miran$ intertwines the grading operators as
$\Miran\circ \Gr = \Gr \circ \Miran$,
the sections $\Miran(\Omega_i^\pm)$ extend to $\C^\times \cU'_\pm$,
and the mirror isomorphism $\Miran$ also extends there.

Henceforth we assume that the open set $\cU'_\pm$ (where
the analytic mirror isomorphism \eqref{eq:an_mirror_isom_pm}
is defined) is closed under the $\C^\times$-action.
By the same argument as in Section~\ref{subsec:decomp_Bri},
$\cU'_+\cap \cU'_-$ is an open set containing
$(\cC\setminus \{0_+,0_-\})\times \{0\}\subset \cM_\T$.
Setting $\cU_0' = \cU_0 \cap(\cU'_+ \cap \cU'_-)$
for the open set~$\cU_0$ from Lemma~\ref{lem:subcover},
we obtain the following result.

\begin{Theorem}\label{thm:decomp_QDM} There exist an open subset $\cU_0'$ of $\cM_\T$ containing
$(\cC \setminus \{0_+, 0_-\})\times \{0\}$ and mirror maps
$\mir_\pm \colon \cU_0' \to
[\cM_{\rm A,\T}(\frX_\pm)/\Pic^\st(\frX_\pm)]$
to the equivariant K\"ahler moduli spaces of $\frX_\pm$
such that the following decomposition of
$\cOan_{\cU'_0}[\![z]\!]$-modules holds:
\begin{equation}\label{eq:decomp_QDM}
\mir_+^* \overline{\QDMan_\T}(\frX_+)
\cong
\mir_-^* \overline{\QDMan_\T}(\frX_-)
\oplus \scrR\big|_{\cU'_0},
\end{equation}
where $\overline{\QDMan_\T}(\frX_\pm)$ denotes the $z$-adic completion
of the equivariant quantum D-module of $\frX_\pm$
$($see Section~$\ref{subsec:an_mirror_isom})$ and
$\scrR$ is as in Definition $\ref{def:rest}$.
Under this decomposition, the flat connection, the grading operator and
the pairing split as follows:
\begin{itemize}\itemsep=0pt
\item[{\rm (1)}]
$\mir_+^*\nabla^+ = (\mir_-^*\nabla^- + \alpha \cdot \id)
\oplus \nabla^\scrR$ for some $1$-form
$\alpha\in \Omega^1_{\cM_\T/\Lie \T}$;
\item[{\rm (2)}]
$\mir_+^*\Gr^+ = \mir_-^*\Gr^- \oplus \Gr^\scrR$;
\item[{\rm (3)}]
$\mir_+^*P^+ = \mir_-^*P^- \oplus P^\scrR$,
\end{itemize}
where
$\nabla^\pm$, $\Gr^\pm$, $P^\pm$ denote respectively
the quantum connection, the grading operator~\eqref{eq:grading_A} and
the pairing~\eqref{eq:pairing_A} on the quantum D-module of~$\frX_\pm$.
\end{Theorem}

\begin{Remark}
The one-form $\alpha$ arises from the difference of the splittings
of the extended divisor sequence~\eqref{eq:ext_divseq} over $\Q$,
one chosen for $\frX_+$ and the other for $\frX_-$
(see Remark~\ref{rem:choice_of_splitting}).
It is of the form $\alpha(\xi q\parfrac{}{q}) = \overline{\alpha}(\xi)$
for some $\overline{\alpha} \in \Hom(\LL_\Q^\star, \bM_\Q^\star)$;
in particular $\alpha$ vanishes in the non-equivariant limit.
Recall that the Gauss--Manin connection in the equivariant case
depends on the choice of a splitting (see Section~\ref{subsec:Bri}).
\end{Remark}

\begin{Remark}
By construction, the above decomposition \eqref{eq:decomp_QDM}
preserves the additional structure given by multiplication
by the mirror co-ordinates $x_1,\dots,x_n$.
They correspond to the equivariant shift operators
(Seidel representation), see Remark \ref{rem:Seidel}.
\end{Remark}

\subsection{Comparison of Gromov--Witten theories in all genera}\label{subsec:comp_all_genera}
Using a formula due to Givental \cite{Givental:quadratic, Givental:higher_genus}
and Teleman~\cite{Teleman:semisimple}
for higher genus Gromov--Witten potentials,
we show that the ancestor Gromov--Witten potential of $\frX_+$
is decomposed into the ancestor potential of $\frX_-$
and a product of Witten--Kontsevich tau functions, under
the action of a quantized symplectic operator.
We use Givental's formula for orbifolds
studied by Brini--Cavalieri--Ross \cite{BCR:crepant_open} and Zong~\cite{Zong:Givental_GKM}.

\subsubsection{Ancestor potentials}
Let $\frX$ be a smooth DM stack equipped
with a $\T$-action satisfying the assumptions
in Section~\ref{subsec:qcoh}; we use the notation there.
Consider the forgetful morphism
$p\colon \frX_{g,l+m,d} \to \overline{M}_{g,l+m} \to \overline{M}_{g,l}$
which forgets the map and the last $m$ marked points,
and let
$\overline{\psi}_i\in H^2(\frX_{g,l+m,d})$, $1\le i\le l$ denote the pull-back of
the universal cotangent class $\psi_i\in H^2(\overline{M}_{g,l})$
by $p$.
We define the \emph{ancestor Gromov--Witten invariants}
as the $\T$-equivariant integral
\[
\big\langle \alpha_1 \overline{\psi}^{k_1},\dots,\alpha_l \overline{\psi}^{k_l};
\beta_1,\dots,\beta_m\big\rangle_{g,l+m,d} :=
\int_{[\frX_{g,l+m,d}]_{\rm vir}}
\prod_{i=1}^l \ev_i^*(\alpha_i) \overline{\psi}_i^{k_i}
\cdot \prod_{i=1}^{m} \ev_{i+l}^*(\beta_i),
\]
where $\alpha_i,\beta_j \in H^*_{\CR,\T}(\frX)$.
When the moduli stack $\frX_{g,m+l,d}$ is not proper,
the right-hand side is defined by the virtual localization formula
and lies in $S_\T = \Frac(R_\T)$ as before.
We choose a homogeneous $R_\T$-basis $\{\phi_i\}_{i=0}^N$ of
$H^*_{\CR,\T}(\frX)$ satisfying \eqref{eq:basis_equiv_lift}
and introduce the infinite set
$\by=\{y_k^i\}_{k\ge 0, 0\le i\le s}$
of co-ordinates on $H^*_{\CR,\T}(\frX)[\![z]\!]$ given by
\[
\by \mapsto \by(z) = \sum_{k=0}^\infty \sum_{i=0}^s
y_k^i \phi_i z^k.
\]
The \emph{ancestor potential} of $\frX$ is the following
generating function of the ancestor Gromov--Witten invariants
\[
\scrA_{\frX,\tau} = \exp\left(
\sum_{d\in \Laa_+}
\sum_{\substack{g\ge 0, \, l\ge 0 \\ 2g-2+l>0}}
\sum_{m=0}^\infty
\corr{\by(\overline{\psi}),\dots,\by(\overline{\psi});
\tau,\dots,\tau}_{g,l+m,d} \hbar^{g-1} \frac{Q^d}{l!m!}
\right),
\]
which we regard as a function in the formal neighbourhood of $\by=0$.
Introduce another set of variables
$\bx=\{x_k^i\}_{k\ge 0, 0\le i\le N}$
that are related to $\by$ by the formula
(Dilaton shift)
\[
\bx(z) = \by(z) - \phi_0 z,
\]
where $\bx(z) = \sum\limits_{k=0}^\infty \sum\limits_{i=0}^N x_k^i \phi_i z^k$.
Using the co-ordinate $\bx(z)$, we shall regard $\scrA_{\frX,\tau}$ as a
formal function on the formal neighbourhood of $\bx(z) = -\phi_0 z$ in
$H^*_{\CR,\T}(\frX)[\![z]\!]$.

As in Section~\ref{subsec:Kaehler_moduli}, we can specialize
$Q$ to one in the ancestor potential
by using the divisor equation for $\tau$.
We have
\[
\scrA_{\frX,\tau} = \exp\left(
\sum_{d\in \Laa_+}
\sum_{\substack{g\ge 0, \, l\ge 0 \\ 2g-2+l>0}}
\sum_{m=0}^\infty
\corr{\by(\overline{\psi}),\dots,\by(\overline{\psi});
\tau',\dots,\tau'}_{g,l+m,d} \hbar^{g-1} \frac{e^{\sigma\cdot d}Q^d}{l!m!}
\right),
\]
when $\tau = \sigma + \tau'$ with $\sigma = \sum\limits_{i=1}^r
\tau^i \phi_i \in H^2(\frX)$ and $\tau' = \tau^0\phi_0
+ \sum\limits_{i=r+1}^s \tau^i \phi_i$.
Moreover, the Galois symmetry \eqref{eq:Galois} implies
\[
\scrA_{\frX,g(\xi)\tau}(d g(\xi) \bx) = \scrA_{\frX,\tau}(\bx), \qquad
\xi \in H^2(\frX,\Z),
\]
where $g(\xi)$, $dg(\xi)$ are as in Section~\ref{subsec:Kaehler_moduli}
(they act on the equivariant cohomology $H_{\CR,\T}^*(\frX)$
via the splitting $H_{\CR,\T}^*(\frX) \cong H^*_\CR(\frX) \otimes R_\T$
given by the basis $\{\phi_i\}$).
Therefore the specialization $\scrA_{\frX,\tau}|_{Q=1}$ makes sense as
an element of $S_\T(\!(\hbar)\!)[\![\Laa_+]\!][\![\tau',\by]\!]$ --
we regard it as an $(H^2(\frX,\Z)/\Laa^\star)$-invariant function
of
$(\tau,\bx) \in \cM_{\rm A,\T}(\frX) \times_{\Spec R_\T}
\bH_{\CR,\T}(\frX)[\![z]\!]$,
where $\bH_{\CR,\T}(\frX)[\![z]\!]$ denotes the
infinite-dimensional vector bundle over $\Spec R_\T$ associated with
the $R_\T$-module $H_{\CR,\T}^*(\frX)[\![z]\!]$
(see Remark \ref{rem:equiv_coh}).
Henceforth we assume that $Q$ is specialized to one
in the ancestor potential.

\subsubsection{Givental's quantization formalism}
\label{subsubsec:quantization}
We briefly review the quantization formalism of Givental
\cite{Givental:quadratic}. We also refer to
\cite[Section~3]{Coates-Iritani:convergence},
\cite[Section~5.1]{Coates-Iritani:Fock} for the exposition.

Let $V$ be a finite-dimensional $\C$-vector space equipped
with a symmetric non-degenerate pairing $(\cdot,\cdot)_V$.
Then $V(\!(z)\!)$ has the following symplectic form $\Omega_V$:
\[
\Omega_V(f,g) = \mathop{\operatorname{Res}}\limits_{z=0}(f(-z),g(z))_V dz.
\]
Choosing a basis $\{\phi_i\}$ on $V$,
we let $\bx = \{x_k^i\}_{k\ge 0, 0\le i\le s}$ denote the
co-ordinates on $V[\![z]\!]$ given by
$\bx \mapsto \bx(z) = \sum\limits_{k=0}^\infty \sum\limits_{i=0}^s
x_k^i \phi_i z^k$.

Let $W$ be another $\C$-vector space of the same dimension
equipped with a symmetric non-degenerate pairing $(\cdot,\cdot)_W$.
A $\C[\![z]\!]$-linear operator $U \colon V[\![z]\!]\to W[\![z]\!]$
is said to be \emph{unitary} when it satisfies
\begin{equation}\label{eq:unitarity}
(U(-z) v_1, U(z) v_2)_W = (v_1,v_2)_V.
\end{equation}
A unitary operator $U$ induces a $\C(\!(z)\!)$-linear
symplectic transformation
$U \colon V(\!(z)\!) \to W(\!(z)\!)$.
The quantized operator $\hU$ described below
sends a certain ``tame'' (spelled out below) formal
function $\scrA$ on $V[\![z]\!]$ to a tame formal function
$\hU\scrA$ on $W[\![z]\!]$.
For $\bD \in z V[\![z]\!]$,
let $\Fockan(V,\bD)$ denote the set of
formal power series $\scrA$ of the form
\begin{equation*}
\scrA = \exp\left(\sum_{g=0}^\infty \hbar^{g-1} F^g(\bx)\right)
\end{equation*}
with $F^g$ a formal power series in $\bx - \bD$
(i.e., a function in the formal neighbourhood of $\bx=\bD$
in $V[\![z]\!]$)
satisfying the following tameness condition
\[
\left.
\parfrac{^m F^g}{x_{k_1}^{i_1} \cdots \partial x_{k_m}^{i_m}}
\right|_{\bx = \bD}
= 0 \qquad \text{if $k_1+\dots+k_m > 3g-3+m$}
\]
and $F^0|_{\bx=\bD} = F^1|_{\bx=\bD}=0$
(see \cite[Definition~5.1]{Coates-Iritani:Fock}).
Then $\hU$ defines an operator \cite[Definition~5.7]{Coates-Iritani:Fock}:
\[
\hU \colon \ \Fockan(V,\bD) \to \Fockan(W, U(\bD)).
\]
For $\bs \in z V[\![z]\!]$, let $T_\bs$ denote the operator
acting on tame functions on $V[\![z]\!]$
which shifts the base point $\bD$ by $\bs$
\cite[Definition 5.5]{Coates-Iritani:Fock};
for a tame function $\scrA\in \Fockan(V,\bD)$,
$T_\bs \scrA\in \Fockan(V,\bD+\bs)$ denotes the Taylor expansion
of $e^{-F_1(\bD+\bs)} \scrA$ at the new base point
$\bD+\bs$ (whenever it makes sense).
The tameness condition implies that $T_\bs\scrA$ is well-defined
for all $\scrA \in \Fockan(V,\bD)$
if $\bs \in z^2 V[\![z]\!]$;
under the additional assumption that $\scrA$ is rational
\cite[Definition 5.2]{Coates-Iritani:Fock}, $T_\bs\scrA$ is
well-defined for $\bs \in z V[\![z]\!]$
\cite[Definition 5.5]{Coates-Iritani:Fock}.
All tame functions appearing in this section are rational,
since they are obtained by the Givental formula
\cite[Proposition 3.20]{Coates-Iritani:convergence}.

When the ancestor potential $\scrA_{\frX,\tau}$ is convergent
and analytic in the variable $\tau$, it defines an element of
$\Fockan(\bH_{\CR,\T}(\frX)_\chi,-z \phi_0)$ where
$\bH_{\CR,\T}(\frX)_\chi$ denotes the fibre of
the vector bundle $\bH_{\CR,\T}(\frX) \to \Spec R_\T$
(see Remark \ref{rem:equiv_coh})
at $\chi$ and $\chi$ is the image of
$\tau \in \cM_{\rm A,\T}(\frX)$ in $\Spec R_\T$.

The \emph{Witten--Kontsevich tau-function} is the ancestor potential
for a point. It is a function on the formal neighbourhood
of $-z$ in $\C[\![z]\!]$ given by
\begin{align*}
\scrT & = \exp\left(\sum_{g=0}^\infty \hbar^{g-1}
\sum_{l\colon 2g-2+l>0} \frac{1}{l!}
\sum_{k_1,\dots,k_l\ge 0}
\big\langle\psi^{k_1},\dots,\psi^{k_l}\big\rangle_{g,l} y_{k_1} \cdots y_{k_l}
\right) \\
& = (-x_1)^{-\frac{1}{24}}\exp\left(
\sum_{g=0}^\infty \hbar^{g-1}
\sum_{l\colon 2g-2+l>0} \frac{1}{l!}
\sum_{\substack{k_1,\dots,k_l\ge 0,\\ k_j \neq 1}}
\big\langle \psi^{k_1},\cdots,\psi^{k_l}\big\rangle_{g,l}
\frac{x_{k_1}\cdots x_{k_l}}{(-x_1)^{2g-2+l}}
\right),
\end{align*}
where
\[
\big\langle \psi^{k_1},\dots,\psi^{k_l}\big\rangle_{g,l}
= \int_{\overline{M}_{g,l}} \psi^{k_1} \cdots \psi^{k_l},
\]
$\sum\limits_{k=0}^\infty x_k z^k$ denotes
the co-ordinate on $\C[\![z]\!]$ and $y_k = x_k + \delta_{k,1}$.
We used the dilaton equation \cite[Theorem 8.3.1]{AGV:GW}
in the second line. The Witten--Kontsevich tau-function
defines a rational element of $\Fockan(\C,-z)$.

\subsubsection{A relationship between ancestor potentials}
\label{subsubsec:rel_ancestors}

We fix a homogeneous $R_\T$-basis
$\{\phi_i\}$ of $H^*_{\CR,\T}(\frX_\pm)$ satisfying
\eqref{eq:basis_equiv_lift} as above.
Let $\cU_0'\subset \cM_\T$ be an open set from
Theorem \ref{thm:decomp_QDM}.
As discussed in Section~\ref{subsubsec:higher_residue},
there exists an open dense subset ${\cU'_0}^{\rm ss}$
of $\cU_0'$ such that for each $(q,\chi)\in {\cU'_0}^{\rm ss}$,
$\tpr^{-1}(q,\chi) \cap \cB_+$ consists of $\dim H^*_\CR(\frX_+)$
many reduced points.
Choose a point $(q,\chi) \in {\cU'_0}^{\rm ss}$.
Combining the decomposition
\eqref{eq:decomp_QDM} at $(q,\chi)$
and the isomorphism
$\Asymp\colon \scrR_{(q,\chi)} \cong \C[\![z]\!]^{\oplus \rank \scrR}$
given by the asymptotic expansion
(see Definition~\ref{def:rest}), we obtain a
$\C[\![z]\!]$-linear isomorphism
\[
\tU_{q,\chi} \colon \ \bH_{\CR,\T}(\frX_+)_\chi[\![z]\!]
\xrightarrow{\cong} \bH_{\CR,\T}(\frX_-)_\chi[\![z]\!] \oplus
\C[\![z]\!]^{\oplus \rank \scrR},
\]
where $\bH_{\CR,\T}(\frX_\pm)_\chi$ denotes the
fibre of the vector bundle $\bH_{\CR,\T}(\frX_\pm)
\to \Spec R_\T$ at $\chi$ (see Remark~\ref{rem:equiv_coh}).
The equivariant orbifold Poincar\'e pairing~\eqref{eq:orb_Poincare_pairing}
defines a $\C$-bilinear pairing on $\bH_{\CR,\T}(\frX_\pm)_\chi$
for a generic $\chi$. With respect to these pairings
and the diagonal pairing on $\C[\![z]\!]^{\oplus \rank \scrR}$,
$U = \tU_{q,\chi}$ satisfies the unitarity \eqref{eq:unitarity};
this follows from Theorem \ref{thm:decomp_QDM} and the definition of the
higher residue pairing (see Section~\ref{subsubsec:higher_residue}).
We flip the sign of $z$ and set $
U_{q,\chi} := \tU_{q,\chi}\big|_{z\to -z}$.

When functions $\scrA_i$ on $V_i[\![z]\!]$ $(i=1,\dots,k)$ are given,
we write $\scrA_1\otimes \cdots \otimes \scrA_k$ for the function
on $V_1[\![z]\!] \times \cdots \times V_k[\![z]\!]$
given by $\big(\bx^{(1)},\dots,\bx^{(k)}\big) \mapsto
\prod\limits_{i=1}^k\scrA_i(\bx^{(i)})$.

\begin{Theorem}\label{thm:ancestor}
Let $\scrA_{\pm,\tau}$ denote the ancestor potential
of $\frX_\pm$.
For $(q,\chi) \in {\cU_0'}^{\rm ss}$, we have
\[
T_{\bs} \hU_{q,\chi} \scrA_{+,\mir_+(q,\chi)} =
\scrA_{-,\mir_-(q,\chi)} \otimes
\scrT^{ \otimes\rank \scrR},
\]
where $\bs :=(-z \phi_0, (-z,\dots,-z)) + U_{q,\chi}(z\phi_0)
\in \bH_{\CR,\T}(\frX_-)_\chi[\![z]\!] \times \C[\![z]\!]^{\oplus
\rank \scrR}$.
\end{Theorem}

\begin{Remark}
Implicit in the above theorem is the convergence and
analyticity of $\scrA_{\pm,\tau}$ with respect to $\tau$.
This follows from the Givental formula.
We refer to \cite[Theorem~1.4, Defini\-tion~3.13]{Coates-Iritani:convergence}
for the discussion on the convergence of ancestor potentials.
\end{Remark}

\begin{Remark}
The ancestor potentials of $\frX_\pm$ define sections
of the Fock sheaves \cite{Coates-Iritani:Fock} associated
with the quantum D-modules of $\frX_\pm$, and the
above relationship can be interpreted in this language.
\end{Remark}

\begin{Remark}
We can state the relationship between the Gromov--Witten theories
of~$\frX_+$ and~$\frX_-$
in terms of cohomological field theory (CohFT): \emph{the CohFT of
$\frX_+$ transforms into the product of the CohFT of~$\frX_-$
and the trivial semisimple CohFT under the Givental action
of~$U_{q,\chi}$}. We refer the reader to~\cite{PPZ:relations, Shadrin:BCOV, Teleman:semisimple} for the Givental group action on CohFTs.
\end{Remark}

\subsubsection{Givental's formula and the proof of
Theorem $\ref{thm:ancestor}$}
Teleman's classification theorem \cite{Teleman:semisimple}
implies that
the ancestor potential $\scrA$ of a semisimple cohomological
field theory on a vector space $H$
lies in the orbit of $\scrT^{\otimes N}$ ($N=\dim H$)
with respect to the action of symplectic operators
(via the Givental quantization).
That is, there exists a unitary operator $R \colon \C[\![z]\!]^{\oplus N}
\to H[\![z]\!]$ such that
\[
\scrA = T_{-ze+R(z,\dots, z)}\hR \scrT^{\otimes N},
\]
where $e\in H$ is the identity element of the Frobenius algebra $H$.
In our context, $R$ depends on $\tau\in \bH_{\CR,\T}(\frX)$ and
gives a formal fundamental solution of the quantum connection
in the $\tau$-direction.
It is ambiguous up to the right multiplication by a constant diagonal matrix
\cite[Remarks after Proposition~1.1]{Givental:elliptic}; as
\cite{BCR:crepant_open,Zong:Givental_GKM} did,
the ambiguity can be fixed by Tseng's orbifold quantum
Riemann-Roch theorem \cite{Tseng:QRR}.

We state Givental's formula for a semiprojective toric DM stack
in terms of the mirror LG model.
Let $\frX_\bSigma$ be a smooth toric DM stack
with $\bSigma \in \Fan(S)$ and let $(\pr\colon \cY \to \cM, F)$
denote the mirror LG model from Definition \ref{def:LG}.
Let $\cB\subset \cY$, $\cU\subset \cM_\T :=\cM\times \Lie \T$
be analytic open neighbourhoods of $\tzero_\bSigma$
and $(0_\bSigma,0)$ respectively such that the
conclusion of Corollary \ref{cor:tpr_locally_free} holds, as usual.
Let $\cUss\subset \cU$ denote the open dense subset
consisting of $(q,\chi)\in\cU$ such that
$\tpr^{-1}(q,\chi)\cap \cB$ is reduced
(recall, as discussed in Section~\ref{subsubsec:higher_residue}, that
$\tpr^{-1}(q,\chi)$ consists of relative critical points
of the equivariant potential $F_\T=F - \sum\limits_{i=1}^n \chi_i \log x_i$).
For $(q,\chi) \in \cUss$,
the asymptotic expansion of equivariant oscillatory integrals
(see Definition \ref{def:Asymp})
defines an isomorphism
\[
\Asymp \colon \ \Briequivan(F)_\bSigma\big|_{(q,\chi)} \to
\C[\![z]\!]^{\oplus N}, \qquad
\Asymp := \bigoplus_{p\in \tpr^{-1}(q,\chi)\cap \cB} \Asymp_p,
\]
where $N =\deg(\tpr|_\cB)= \dim H^*_\CR(\frX)$.
Composing the inverse map with the mirror isomorphism
from Theorem~\ref{thm:an_mirror_isom},
we obtain a linear map
\[
\tR_{q,\chi} \colon \ \C[\![z]\!]^{\oplus N}
\xrightarrow{\Asymp^{-1}}
\Briequivan(F)_\bSigma\big|_{(q,\chi)}
\xrightarrow{\Miran}
\overline{\QDMan_\T}(\frX_\bSigma) \big|_{\mir(q,\chi)}
\cong \bH_{\CR,\T}(\frX_\bSigma)_\chi[\![z]\!],
\]
where $\mir \colon \cU \to
[\cM_{\rm A,\T}(\frX_\bSigma)/\Pic^\st(\frX_\bSigma)]$
is the mirror map. We flip the sign of $z$ and
set $R_{q,\chi} := \tR_{q,\chi}|_{z\to -z}$.

\begin{Proposition}[Givental's formula for toric DM stacks]
\label{prop:Givental_formula_toric_stacks}
The ancestor potential of $\frX_\bSigma$ is given by
$\scrA_{\frX_\bSigma,\mir(q,\chi)}
= T_{-z \phi_0+R(z,\dots,z)} \hR \scrT^{\otimes N}$
with $R = R_{q,\chi}$ above.
\end{Proposition}
\begin{proof}
The mirror isomorphism $\Miran$ identifies the Gauss--Manin
connection with the quantum connection, and it was shown
in \cite[Lemma 6.7]{CCIT:MS} that $e^{F_\T(p)/z}\Asymp_p$ gives a
solution to the Gauss--Manin connection
for each critical branch $p$ of $F_\T$.
Let $\bU$ denote the diagonal matrix whose entries are
relative critical values of $F_\T$.
It follows that $\tR_{q,\chi} e^{-\bU/z}$ gives a formal fundamental
solution for $\mir^*\nabla$ in the sense of
\cite[Proposition 1.1]{Givental:elliptic},
\cite[Theorem 5.1]{Zong:Givental_GKM}
(our $R_{q,\chi}$ corresponds to $\Psi R$ in \cite{Givental:elliptic,
Zong:Givental_GKM}; see also
Proposition \ref{prop:formal_decomposition} below).
We need to flip the sign of $z$ because we use the sign convention
for the quantum connection opposite to
\cite{Givental:elliptic, Zong:Givental_GKM}.
It now suffices to check that $R_{q,\chi}$ satisfies
the classical limit condition at the large radius limit $q=0_\bSigma$
given in
\cite[Lemma 6.5, equation (156), Remark~6.6]{BCR:crepant_open},
\cite[Theorem~6.2]{Zong:Givental_GKM}
(which generalize \cite[Theorem~9.1]{Givental:quadratic} to the
orbifold setting).
We can easily check that $R_{q,\chi}$ satisfies this condition by using
the computation in \cite[Proposition 6.9]{CCIT:MS}.
\end{proof}

\begin{Remark}
It has been observed by Givental \cite{Givental:quadratic}
that the operator $R$ can be constructed by equivariant mirrors
(for Fano toric manifolds).
\end{Remark}
\begin{Remark}
Recall from Remark \ref{rem:Asymp} that the asymptotic
expansion is ambiguous up to sign. Therefore $R_{q,\chi}$
is ambiguous up to the right multiplication by a signed
permutation matrix.
The right-hand side of the Givental formula is, however,
independent of the choice (see,
e.g., \cite[Proposition 4.3]{Coates-Iritani:convergence}).
\end{Remark}

\begin{proof}[Proof of Theorem $\ref{thm:ancestor}$]
Let $R_{q,\chi}^\pm$ denote the $R$-operators for
$\frX_\pm$ introduced above.
By the construction of $U_{q,\chi}$ in Section~\ref{subsubsec:rel_ancestors},
we have the following commutative diagram for
$(q,\chi) \in {\cU'_0}^{\rm ss}$ (when we order the critical points
of $F_\T$ appropriately):
\[
\xymatrix{
\bH_{\CR,\T}(\frX_+)_\chi [\![z]\!] \ar[r]^{U_{q,\chi}\phantom{ABCDE}} &
\bH_{\CR,\T}(\frX_-)_\chi[\![z]\!] \oplus \C[\![z]\!]^{\oplus \rank \scrR} \\
\C[\![z]\!]^{\oplus N_+} \ar[u]^{R_{q,\chi}^+}
\ar@{=}[r]
& \C[\![z]\!]^{\oplus N_-} \times
\C[\![z]\!]^{\oplus \rank \scrR}, \ar[u]^{R_{q,\chi}^- \times \id}
}
\]
where $N_\pm = \dim H^*_\CR(\frX_\pm)$.
The conclusion follows by Proposition \ref{prop:Givental_formula_toric_stacks}
and the `chain rule' of the Givental quantization
(see, e.g., \cite[Remark~3.22]{Coates-Iritani:convergence}).
\end{proof}

\section{Formal decomposition and analytic lift}\label{sec:lift}
We discuss the formal decomposition and its analytic lift for
quantum D-modules or Brieskorn modules.
This is known as the Hukuhara--Turrittin theorem for irregular
differential equations.
For mirrors of the small quantum cohomology of
weak Fano compact toric stacks, the analytic lift is described
explicitly in terms of oscillatory integrals.
In this section, we restrict ourselves to the
\emph{non-equivariant} quantum cohomology
and Brieskorn module, and assume that the toric stacks
are compact.

\subsection{Hukuhara--Turrittin type result}
\label{subsec:Hukuhara-Turrittin}
Recall that the non-equivariant quantum connection
(or, equivalently, the non-equivariant \linebreak Gauss--Manin connection) can be
extended in the direction of $z$.
The connection in the $z$-direction is of the form
(see~\eqref{eq:conn_z})
\begin{equation*}
\nabla_{z\parfrac{}{z}} = z\parfrac{}{z} - \frac{1}{z}
(E\star_\tau)+ \mu.
\end{equation*}
When the quantum product $\star_\tau$ is semisimple,
$\nabla$ decomposes over $\C[\![z]\!]$ into the sum of
rank-one connections $d+ d(\bu_i/z)$, $i=1,\dots,N$
where $\bu_1,\dots,\bu_N$ are eigenvalues of $E\star_\tau$.
Moreover this formal decomposition admits an analytic lift over
a certain angular sector: this is an instance of
the \emph{Hukuhara--Turrittin theorem}
(see, e.g., \cite[Theorem~12.3]{Wasow:book},
\cite[Theorem~A]{BJL79}, \cite[Chapter~II, Section~5.d]{Sabbah:isomonodromic}
for more general statement).
We state a version of this theorem
following Hertling--Sevenheck \cite[Section~8]{Hertling-Sevenheck:nilpotent} and
Galkin--Golyshev--Iritani \cite[Section~2.5]{GGI:gammagrass}.
We restrict our attention to the quantum cohomology of toric stacks, but
the discussion in this section can be equally applied to a general
semisimple Frobenius manifold \cite{Dubrovin:Painleve}.

Let $\frX_\bSigma$ be a semiprojective smooth toric DM stack from
Section~\ref{subsubsec:toric_stacks}.
In this section (Section~\ref{subsec:Hukuhara-Turrittin}), we assume that
$\frX_\bSigma$ is compact but not necessarily weak Fano.
Let $\cM_{\rm A}(\frX_\bSigma)$ be
the non-equivariant K\"ahler moduli space of
$\frX_\bSigma$ from Section~\ref{subsec:Kaehler_moduli}.
By the convergence result \cite[Corollary 7.3]{CCIT:MS},
the non-equivariant quantum product is analytic over
an open neighbourhood $V\subset
[\cM_{\rm A}(\frX_\bSigma)/\Pic^\st(\frX_\bSigma)]$
of the large radius limit point.
This defines the (non-equivariant) quantum connection $\nabla$
over $V$
and the \emph{analytic quantum D-module}
(cf.~Sections~\ref{subsec:QDM} and~\ref{subsec:an_mirror_isom}):
\begin{equation}
\label{eq:analytic_QDM}
\QDMan(\frX_\bSigma)
= (H^*_{\CR}(\frX_\bSigma)\otimes \cOan_{\tV}[z], \nabla, P),
\end{equation}
where $\nabla$ is understood as being extended in the $z$-direction
by \eqref{eq:conn_z} and~$\tV$ is the preimage of~$V$ in
$\cM_{\rm A}(\frX_\bSigma)$;
$H^*_\CR(\frX_\bSigma)\otimes \cOan_{\tV}[z]$ is
a $\Pic^\st(\frX_\bSigma)$-equivariant sheaf,
which shall be regarded as a sheaf on the stack
$V=[\tV/\Pic^\st(\frX_\bSigma)]$.
We omitted the grading operator from the data because it can be
recovered from the other data as $\Gr = \nabla_\cE + \nabla_{z\parfrac{}{z}}
+\frac{1}{2} \dim \frX_\bSigma$.
We also write
\begin{equation}
\label{eq:z-completed_QDM}
\overline{\QDMan}(\frX_\bSigma) =
(H^*_{\CR}(\frX_\bSigma)\otimes \cOan_{\tV}[\![z]\!],
\nabla, P)
\end{equation}
for the quantum D-module completed in $z$.
We set
\begin{gather*}
V^\times := V \cap \big(\text{the image
of the natural map
$H^*_\CR(\frX_\bSigma) \to
[\cM_{\rm A}(\frX_\bSigma)/\Pic^\st(\frX_\bSigma)]$}\big), \\
\Vss := \big\{\tau \in V^\times\colon \text{$(H^*_\CR(\frX_\bSigma),\star_\tau)$ is semisimple as a ring}\big\} \subset V^\times.
\end{gather*}
Here recall from Section~\ref{subsec:Kaehler_moduli}
that $\cM_{\rm A}(\frX_\bSigma)$ is
a partial compactification of
$H^*_\CR(\frX_\bSigma)/2\pi\iu \big(\Laa^\bSigma\big)^\star$
and hence we have a natural map $H^*_\CR(\frX_\bSigma)
\to \cM_{\rm A}(\frX_\bSigma) \to
[\cM_{\rm A}(\frX_\bSigma)/\Pic^\st(\frX_\bSigma)]$.
When $\frX_\bSigma$ is compact,
$\Vss$ is open and dense in $V^\times$
\cite[Remark~7.11]{CCIT:MS}.
The following proposition describes a Levelt--Turrittin normal form
\cite[Chapter II, Theorem 5.7]{Sabbah:isomonodromic} of the
quantum connection in the semisimple case.

\begin{Proposition}\label{prop:formal_decomposition}
Let $\tau_0 \in \Vss$ be a semisimple point.
We have the following decomposition in an open neighbourhood
of $\tau_0$:
\[
\hPhi \colon \
\overline{\QDMan}(\frX_\bSigma)
\overset{\cong}{\longrightarrow} \bigoplus_{i=1}^N
\left( \cOan[\![z]\!], d + d(\bu_i/z), P_{\std} \right),
\]
where
$\bu_1,\dots,\bu_N$ are eigenvalues of $E\star_\tau$
with $N = \dim H^*_\CR(\frX_\bSigma)$
and $P_{\std}$ is the pairing given by the
multiplication $P_{\std}(f,g) = f(-z) g(z)$ with
$f,g \in \cOan[\![z]\!]$. Once we fix the ordering of $\bu_1,\dots,\bu_N$,
the isomorphism $\hPhi$ is unique up to a right multiplication
by $\diag(\pm 1,\dots, \pm 1)$.
\end{Proposition}
\begin{proof}
When $\bu_1,\dots,\bu_N$ are mutually distinct, this is discussed in
\cite[Lecture 4]{Dubrovin:Painleve}, \cite[Proposition 1.1]{Givental:elliptic},
\cite[Chapter II, Theorem 5.7]{Sabbah:isomonodromic},
\cite[Proposition 7.2]{Coates-Iritani:Fock}.
When $\bu_1,\dots,\bu_N$ are not mutually distinct (but $\star_\tau$
is still semisimple), the existence of $\hPhi$
was shown by Teleman \cite[Theorem~8.15]{Teleman:semisimple}
(see also \cite[Section~8]{Bridgeland--Toledano-Laredo:multilog}).
That $\hPhi$ depends analytically on $\tau$ (along the locus
where some of $\bu_1,\dots,\bu_N$ coalesce) was discussed
in \cite[Remark 2.5.7]{GGI:gammagrass}.
The pairing fixes the isomorphism~$\hPhi$ up to sign (on each factor),
see the argument in \cite[Proposition~7.2]{Coates-Iritani:Fock}.
\end{proof}

\begin{Remark}\label{rem:generically_tame_semisimple}
In this proposition, we do not require that eigenvalues
of $E\star_\tau$ are mutually distinct.
Note however that $\bu_1,\dots,\bu_N$
form a co-ordinate system in a neighbourhood of a~semisimple
point $\tau_0$ \cite[Lecture~3]{Dubrovin:Painleve}, and thus
they are \emph{generically} mutually distinct.
\end{Remark}

\begin{Remark}\label{rem:normalized_idempotents}
The isomorphism $\hPhi$ modulo $z$ is given by a
\emph{normalized idempotent basis} of
$(H^*_\CR(\frX_\bSigma),\star_\tau)$,
i.e., a basis $\{\Psi_1,\dots,\Psi_N\}$
such that $\Psi_i \star_\tau \Psi_j$ is a non-zero
scalar multiple of $\delta_{i,j} \Psi_i$ and that $(\Psi_i,\Psi_j) = \delta_{i,j}$.
\end{Remark}

In plain words, Proposition \ref{prop:formal_decomposition}
means that the differential equation $\nabla_{z\parfrac{}{z}} f=0$
for a coho\-mo\-lo\-gy-valued function $f=f(z)$ admits a formal
power series solution
\begin{equation}\label{eq:formal_solution}
f = e^{-\bu_i/z}\big(\Psi_{i,0} + \Psi_{i,1} z + \Psi_{i,2} z^2 + \cdots \big)
\end{equation}
with $\Psi_{i,n} \in H^*_\CR(\frX_\bSigma)$,
where $\Psi_{i,0} = \Psi_i$ is the
normalized idempotent in Remark \ref{rem:normalized_idempotents}.
These formal power series are typically divergent.
Over an appropriate angular sector in the $z$-plane, however,
we can find an actual analytic solution whose asymptotic expansion
is given by \eqref{eq:formal_solution};
moreover the actual solution with prescribed asymptotics
is unique if the angle of the sector is bigger than $\pi$.

To state this analytic lift in a sheaf-theoretic language, we follow
Sabbah \cite{Sabbah:isomonodromic}
and introduce a sheaf $\cA$ of ``holomorphic'' functions
on the real blowup of $\C$.
An (open) \emph{sector} is a subset of $\C$ of the form
$\{z\in \C^\times \colon \phi_1 < \arg z < \phi_2, |z|<\delta \}$
for some $\phi_1, \phi_2 \in \R$ and $\delta\in (0,\infty]$.
A~holomorphic function $f(z)$ on the sector $I=\{z\in \C^\times \colon
\phi_1 <\arg z <\phi_2, |z| <\delta \}$ is
said to have the \emph{asymptotic expansion
$f\sim \sum\limits_{k=0}^\infty a_k z^k$
as $z\to 0$ along $I$} if for every $\epsilon>0$ and for every $m\ge 0$,
there exists a constant $C_{\epsilon,m}>0$ such that
\[
\left | f(z) - \sum_{k=0}^m a_k z^k\right| \le C_{\epsilon,m} |z|^{m+1}
\]
for all $z$ with $\phi_1 +\epsilon \le \arg z \le \phi_2 -\epsilon$
and $|z|\le \delta/2$.
\begin{Definition}[{\cite[Chapter II, Section~5.c]{Sabbah:isomonodromic}}]
Let $\tCC:=[0,\infty)\times S^1$ denote the oriented real blowup
of $\C$ at the origin.
This is a smooth manifold with boundary
and is equipped with the map $\pi \colon \tCC \to \C$,
$\pi\big(r,e^{\iu\theta}\big) = r e^{\iu\theta}$.
Let $C^\infty_\tCC$ denote the sheaf of complex-valued
$C^\infty$-functions on $\tCC$ and define $\cA_\tCC$
to be the subsheaf of $C^\infty_\tCC$
of germs annihilated by the Cauchy-Riemann
operator $\overline{z} \parfrac{}{\overline{z}}
= \frac{1}{2}\big(r \parfrac{}{r} + \iu \parfrac{}{\theta}\big)$.
Sections of $\cA_\tCC$ over the open set $\{(r,e^{\iu\theta}) \in \tCC \colon
0\le r<\delta,\, \phi_1<\theta<\phi_2\}$ are precisely those holomorphic
functions $f(z)$ on the sector $\{z \in \C^\times \colon
\phi_1<\arg z<\phi_2,\, |z|<\delta\}$
which admit asymptotic expansions $f(z) \sim \sum\limits_{k=0}^\infty a_k z^k$
along this sector.

The same construction works in families: for a complex
manifold $M$, we similarly
define the sheaf $\cA_{M\times \tCC}$ over $M\times \tCC$
to be the subsheaf of $C^\infty_{M\times \tCC}$ of germs
annihilated by the Cauchy--Riemann operators
$\overline{z}\parfrac{}{\overline{z}}$, $\overline\partial_M$.
\end{Definition}

The asymptotic expansion yields a map
\begin{equation}\label{eq:formalization}
i_\xi^{-1} \cA_{M\times \tCC} \rightarrow \cOan_M[\![z]\!]
\end{equation}
for any $\xi\in S^1$, where $i_\xi \colon M \cong
M\times \{(0,\xi)\} \hookrightarrow
M\times \tCC$ is the inclusion.
This map is known to be surjective (the Borel--Ritt lemma;
see, e.g., \cite{Wasow:book}).
The natural map $\pi\colon M\times \tCC \to M\times \C$ is a~map of ringed spaces, i.e., we have a~map
$\pi^{-1}\cOan_{M\times \C} \to \cA_{M\times \tCC}$ of
sheaves of rings.
In particular, we can define the pull-back of an
$\cOan_{M\times \C}$-module $\cF$ to be
$\pi^*\cF := \pi^{-1}\cF
\otimes_{\pi^{-1}\cOan_{M\times \C}} \cA_{M\times \tCC}$.

For a multi-set $\{\bu_i\}=\{\bu_1,\dots,\bu_N\}$
of complex numbers, we say that
a direction $e^{\iu\phi}\in S^1$ or a phase $\phi$
 is \emph{admissible} for $\{\bu_i\}$ if $e^{\iu\phi}$
is not parallel to any non-zero difference $\bu_i -\bu_j$,
i.e., $e^{\iu\phi} \notin \R(\bu_i-\bu_j)$ for all $i$, $j$.

\begin{Proposition}
\label{prop:Hukuhara-Turrittin}
Let $\tau_0\in \Vss$ be a semisimple point
and let $e^{\iu\phi} \in S^1$ be an admissible direction
for the spectrum $\big\{\bu_1^0,\dots,\bu_N^0\big\}$ of $E\star_{\tau_0}$.
Let $\pi \colon \Vss \times \tCC \to \Vss \times \C$
denote the oriented real blowup along $\Vss \times \{0\}$.
There exist an open neighbourhood $B$ of $\tau_0$ in $\Vss$,
a positive number $\epsilon>0$,
and an isomorphism
\[
\Phi_{\phi} \colon \ \pi^*\QDMan(\frX_\bSigma) \big|_{B\times I_\phi}
\cong \bigoplus_{i=1}^N (\cA_{B\times I_\phi}, d + d(\bu_i/z))
\]
over the sector
$I_\phi=\big\{\big(r,e^{\iu\theta}\big)\colon |\theta -\phi|<\frac{\pi}{2}
+\epsilon\big\}$
such that $\Phi_\phi$ induces, via~\eqref{eq:formalization},
the formal decomposition $\hPhi$ in Proposition~$\ref{prop:formal_decomposition}$.
Here we exclude the data of the pairing from
$\QDMan(\frX_\bSigma)$.
Moreover such a $\Phi_\phi$ is unique. We call $\Phi_\phi$
the \emph{analytic lift} of~$\hPhi$.
\end{Proposition}

\begin{proof}In the case where $\big\{\bu_1^0,\dots,\bu_N^0\big\}$ are pairwise distinct,
similar results are given in \cite[Theorem 12.3]{Wasow:book},
\cite[Theorem A]{BJL79},
\cite[Theorem 4.2]{Dubrovin:Painleve},
\cite[Chapter II, Theorem 5.12]{Sabbah:isomonodromic}.
We closely follow Hertling--Sevenheck
\cite[Lemma 8.3]{Hertling-Sevenheck:nilpotent} for the formulation.
The general case follows from
\cite[Proposition 2.5.1]{GGI:gammagrass}, where
a fundamental solution matrix $Y_\phi(\tau,z)$ for $\nabla$
with the asymptotics $Y_\phi(\tau,z) e^{-\bU/z} \to (\Psi_1,\dots,\Psi_N)$
as $z\to 0$ along the sector $I_\phi$ is constructed for $\tau \in B$.
Here $\bU = \diag[\bu_1,\dots,\bu_N]$ and $\Psi_i$ is the normalized idempotent
in Remark \ref{rem:normalized_idempotents}.
\end{proof}

\begin{Remark}The uniqueness of the analytic lift $\Phi_\phi$ is ensured by
the fact that the angle of the sector $I_\phi$ is bigger than $\pi$:
see \cite[Remark~1.4]{BJL79}. The lift $\Phi_\phi$ depends
on $\tau_0$ and $e^{\iu\phi}$, and it depends continuously
on $\big(\tau_0,e^{\iu\phi}\big)$ unless $\big(\tau_0,e^{\iu\phi}\big)$
crosses the locus where $e^{\iu\phi}$ is non-admissible for
the spectrum of~$E\star_{\tau_0}$.
\end{Remark}

\begin{Remark}\label{rem:pairing_analytic_lift}
The analytic lift $\Phi_\phi$ preserves the pairing in the following sense:
consider the analytic lift $\Phi_{\phi+\pi}$ associated with
the opposite direction $-e^{\iu\phi}$, then
\[
P(s_-,s_+) (x,z) = \sum_{i=1}^N \Phi^i_{\phi+\pi}(s_-)(x,-z)
\Phi^i_{\phi}(s_+)(x,z)
\]
for sections $s_-$, $s_+$ of $\pi^*\QDMan(\frX_\bSigma)$, respectively,
over $B\times I_{\phi+\pi}$, $B\times I_\phi$.
This follows from the fact that the asymptotic expansions of both sides
coincide by Proposition~\ref{prop:Hukuhara-Turrittin}, and
that the pairings are flat.
\end{Remark}

\begin{Remark}For a thorough discussion on the isomonodromic deformation theory of
irregular differential equations (of Poincar\'e rank 1)
with coalescing eigenvalues $\bu_1,\dots,\bu_N$, see Cotti--Dubrovin--Guzzetti~\cite{Cotti-Dubrovin-Guzzetti}.
\end{Remark}

\subsection{Asymptotic basis and marked reflection system}\label{subsec:MRS}
The sectorial decomposition in Proposition~\ref{prop:Hukuhara-Turrittin}
gives rise to a linear algebraic data which we call
the \emph{marked reflection system} \cite[Section~4.3]{GGI:gammagrass};
this notion is equivalent to the central connection matrix
together with canonical co-ordinates in Dubrovin's theory
\cite[Lecture~4]{Dubrovin:Painleve}.
We briefly review it for our later purposes.

Introduce a pairing $[\cdot,\cdot)$ on the $\C$-vector space
$H^*_\CR(\frX_\bSigma)$ by
\[
[\alpha,\beta) := \frac{1}{(2\pi)^n}
\big(e^{\pi\iu\mu} e^{-\pi\iu c_1(\frX_\bSigma)} \alpha, \beta\big),
\]
where $(\cdot,\cdot)$ is the orbifold Poincar\'e pairing~\eqref{eq:orb_Poincare_pairing}.

Recall the fundamental solution $L(\tau,z) z^{-\mu} z^{c_1(\frX_\bSigma)}$
of the quantum connection introduced in Section~\ref{subsec:integral}.
The map $\alpha \mapsto (2\pi)^{-n/2}L(\tau,z) z^{-\mu}
z^{c_1(\frX_\bSigma)}
\alpha$ intertwines the pairing $[\cdot,\cdot)$ on $H^*_\CR(\frX_\bSigma)$
with the orbifold Poincar\'e pairing in the sense that
\[
(s_1(\tau,e^{-\iu\pi}z), s_2(\tau,z)) = [\alpha_1,\alpha_2),
\]
when $s_i(\tau,z)= (2\pi)^{-n/2} L(\tau,z) z^{-\mu} z^{c_1(\frX_\bSigma)}
\alpha_i$ (where $n= \dim \frX_\bSigma$ as usual),
see \cite[(20)]{Iritani:Integral}.
The pairing $[\cdot,\cdot)$
is also related to the Euler pairing $\chi(\cdot,\cdot)$ on
the $K$-group by \eqref{eq:Euler_Poincare}, i.e.,
\[
\chi(V_1,V_2) = [\alpha_1,\alpha_2),
\]
when $\alpha_i = \hGamma_{\frX_\bSigma} \cup (2\pi\iu)^{\deg_0/2}
\inv^*\tch(V_i)$ and $V_i \in K(\frX_\bSigma)$.

Let $\tau_0\in \Vss$ be a semisimple point and let~$\phi$ be an admissible
phase for the eigenvalues of~$E\star_{\tau_0}$.
Let $\Phi_\phi$ be the sectorial decomposition
associated with $\tau_0$ and $\phi$
as in Proposition \ref{prop:Hukuhara-Turrittin}:
\[
\Phi_{\phi} \colon \ \pi^*\QDMan(\frX_\bSigma) \Big|_{B\times I_\phi}
\cong \bigoplus_{i=1}^N
(\cA_{B\times I_\phi}, d + d(\bu_i/z)).
\]
We choose a base point $\tau_\star \in V^\times$
corresponding to a real cohomology class and choose
a path connecting $\tau_\star$ and $\tau_0$ in $V^\times$
(see Remark~\ref{rem:branch_L} for the choice of a base point).
Let $s_i$ be the flat section of $\QDMan(\frX_\bSigma)$
on a neighbourhood of $\big(\tau_0,e^{\iu\phi}\big)$
satisfying $\Phi_\phi(s_i) = e^{-\bu_i/z} e_i$,
where $e_i$ is the $i$th standard basis of $\cA_{B\times I_\phi}^{\oplus N}$.
Then we have vectors $v_i\in H^*_\CR(\frX_\bSigma)$
such that
\[
s_i(\tau,z) = (2\pi)^{-n/2} L(\tau,z)z^{-\mu} z^{c_1(\frX_\bSigma)} v_i,
\]
where the determination of the fundamental solution is given by
$\arg z= \phi$ at $z= e^{\iu\phi}$ and the chosen path
(see Remark \ref{rem:branch_L}).
The \emph{marked reflection system} associated with $\tau_0$,
$\phi$ (and the path from $\tau_\star$ to $\tau_0$) is
a tuple
$\big(H^*_\CR(\frX_\bSigma), [\cdot,\cdot), \{v_1,\dots,v_N\},
m, e^{\iu\phi}\big)$, where
\begin{itemize}\itemsep=0pt
\item $\{v_1,\dots,v_N\}$ is the basis of $H^*_\CR(\frX_\bSigma)$
defined as above, called the \emph{asymptotic basis};
\item $m\colon \{v_1,\dots,v_N\} \to \C$ is the map
given by $m(v_i) = \bu_i$, called the \emph{marking}.
\end{itemize}
The asymptotic basis satisfies the following semiorthogonality
condition \cite[Proposi\-tion~2.6.4, Section~4.4]{GGI:gammagrass}
\begin{equation}\label{eq:semiorthogonal}
[v_i,v_j) =
\begin{cases}
1, & \text{if $i=j$}, \\
0, & \text{if $i\neq j$ and
$\Im\big(e^{-\iu\phi}\bu_i\big) \le \Im\big(e^{-\iu\phi}\bu_j\big)$}.
\end{cases}
\end{equation}
Because of the ambiguity of $\hPhi$ in
Proposition~\ref{prop:formal_decomposition},
the asymptotic basis $\{v_1,\dots,v_n\}$ is not
a priori ordered and each $v_i$ is determined up to sign
$v_i \to \pm v_i$.
On the other hand, given a phase~$\phi$,
we can order $\{v_i\}$ in such a way that
$\Im\big(e^{-\iu\phi} \bu_1\big) \ge \Im\big(e^{-\iu\phi} \bu_2\big) \ge
\cdots \ge \Im\big(e^{\iu\phi}\bu_N\big)$; then the Gram matrix
$([v_i,v_j))$ is upper-triangular with diagonal entries all equal
to one, and gives the Stokes matrix of the quantum connection
at the irregular singular point $z=0$
(see \cite[Lecture~4]{Dubrovin:Painleve},
\cite[Proposition~2.6.4]{GGI:gammagrass}).

When $\tau_0$ and $\phi$ vary and cross the locus where the
corresponding eigenvalues
$\{\bu_i\}$ become non-admissible for~$\phi$, the corresponding
marked reflection system undergoes mutation.
We refer to \cite[Section~4.2]{GGI:gammagrass} for the full details
of the deformation theory of marked reflection systems.
We illustrate an example of mutation in Fig.~\ref{fig:mutation};
the figure describes a typical procedure where~$\{\bu_i\}$ varies
in the configuration space of $N$ points in $\C$ and
crosses the wall of non-admissible configurations.
In the picture, we drew the half-ray $\bu_j+\R_{\ge 0} e^{\iu\phi}$
from each $\bu_j$ to show the direction $e^{\iu\phi}$.
Suppose that the asymptotic basis $\{v_1,\dots,v_N\}$
is marked by $\{\bu_1,\dots,\bu_N\}$ in the leftmost picture,
i.e., $\bu_j= m(v_j)$. We assume that the basis is ordered
so that $\Im\big(e^{-\iu\phi}\bu_1\big) > \Im\big(e^{-\iu\phi}\bu_2\big)
>\cdots > \Im\big(e^{-\iu\phi}\bu_N\big)$.
After passing through the non-admissible configuration in the
middle picture, the vector~$v_i$ marked by~$\bu_i$ is transformed
into
\begin{equation} \label{eq:right_mutation}
v_i' = v_i - [v_i, v_{i+1}) v_{i+1}
\end{equation}
and the other vectors remain the same
(i.e., the marking is given by $v_i'\mapsto \bu_i$
and $v_j \mapsto \bu_j$ for $j\neq i$ in the
rightmost picture).
This is called the \emph{right mutation}
of $v_i$ with respect to $v_{i+1}$.
The inverse procedure is the \emph{left mutation}
of $v_i'$ with respect to $v_{i+1}$:
\begin{equation}
\label{eq:left_mutation}
v_i = v_i' - [v_{i+1}, v_i') v_{i+1}.
\end{equation}
The two operations \eqref{eq:right_mutation},
\eqref{eq:left_mutation} are inverse to each other
because of the semiorthogonality condition \eqref{eq:semiorthogonal}.
\begin{Remark}
The result \cite[Theorem 4.11]{Iritani:Integral} implies that
the asymptotic basis $\{v_1,\dots,v_N\}$
for the quantum cohomology of weak-Fano, compact toric stacks
$\frX_\bSigma$ is of the form
\[
v_i = \hGamma_{\frX_\bSigma} \cup (2\pi\iu)^{\deg_0/2} \inv^* \tch(V_i)
\]
for some classes $V_i\in K(\frX_\bSigma)$ in the $K$-group,
i.e., the corresponding flat section $s_i$ lies in the $\hGamma$-integral
structure (see also Proposition~\ref{prop:lift_str_sheaf_+}).
The Gamma conjecture II \cite[Section~4.6]{GGI:gammagrass}
(recently proved by~\cite{Fang-Zhou} for Fano toric manifolds) says
that $V_i$ comes from a full exceptional collection in the derived
category of coherent sheaves. In this situation, mutation of
asymptotic basis corresponds to that of full exceptional collections.
\end{Remark}
\begin{figure}[htbp]\centering
\includegraphics{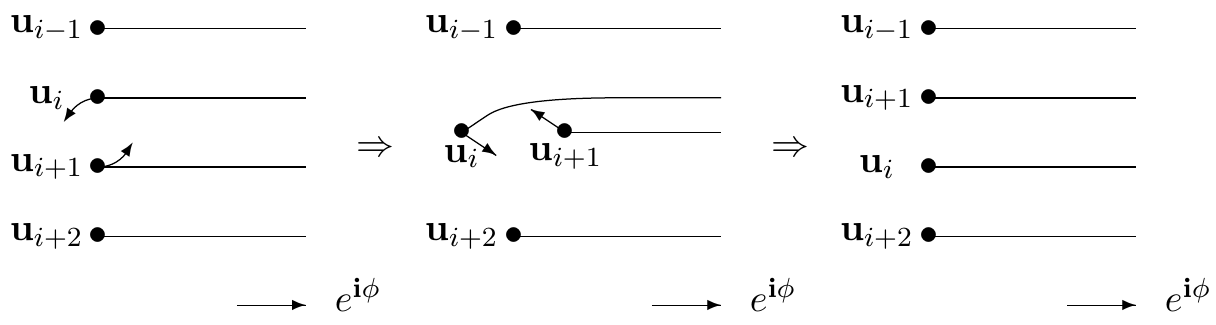}
\caption{Right mutation: the ``wall-crossing'' from the left picture to the right one in the configuration space yields a right mutation of $v_i$.}\label{fig:mutation}
\end{figure}

\subsection{Sectorial decomposition of the Brieskorn module}
\label{subsec:sectorial_decomp_Bri}
In this section, we describe the
Hukuhara--Turrittin sectorial decomposition
(Propositions \ref{prop:formal_decomposition} and~\ref{prop:Hukuhara-Turrittin}) explicitly
for the non-equivariant Brieskorn modules.
First we observe that the
analytified Brieskorn module $\Brian(F)_\bSigma$
(Definition \ref{def:analytic_Bri})
admits a formal decomposition over the locus where $F$ has
only non-degenerate critical points, via the formal asymptotic
expansion in Section~\ref{subsubsec:higher_residue}.
When a toric stack $\frX_\bSigma$ is \emph{weak Fano},
the Brieskorn module $\Bri(F)$ over the small quantum
cohomology locus of $\frX_\bSigma$
has the expected rank equal to $\dim H^*_\CR(\frX_\bSigma)$
and gives a fully analytic (i.e., analytic both in the $\cM$ direction
and in the $z$-direction) D-module mirror to the small
quantum D-module of $\frX_\bSigma$.
In this case, we can describe the analytic lift of the formal decomposition
using oscillatory integrals.

\subsubsection{Formal decomposition of the Brieskorn module}
\label{subsubsec:formal_decomp_Bri}
Let $\frX_\bSigma$ be a toric stack
with $\bSigma \in \Fan(S)$ and let $(\pr\colon \cY \to \cM, F)$
denote the LG model from Definition \ref{def:LG}.
Let $\cB\subset \cY$, $\cU\subset \cM_\T=\cM\times \Lie\T$
be open neighbourhoods of (respectively)~$\tzero_\bSigma$
and $(0_\bSigma,0)$
as in Corollary~\ref{cor:tpr_locally_free}. Recall that
the analytified equivariant Brieskorn module $\Briequivan(F)_\bSigma$
is defined over $\cU$ and its
non-equivariant version $\Brian(F)_\bSigma$
is defined over $\cV = \cU \cap (\cM\times \{0\})$
(see Definition \ref{def:analytic_Bri}).
When $\frX_\bSigma$ is compact, the family of relative critical points
of the LG potential $F$
\[
\Cr := \left\{p \in \cB \colon
x_1\parfrac{F}{x_1}(p) = \cdots = x_n \parfrac{F}{x_n}(p) =0
\right\} \xrightarrow{\pr} \cV
\]
is a finite morphism whose generic fibre is reduced
(see \cite[Proposition~3.10]{Iritani:Integral}), i.e., there exists
an open dense subset $\cVss$ of $\cV$ such that for any $q\in \cVss$,
$F|_{\cB \cap \pr^{-1}(q)}$ has only non-degenerate
critical points.
We set $\Crss := \Cr \cap \pr^{-1}(\cVss)$.
By definition, $\pr \colon \Crss \to \cVss$ is a finite
covering. The following result describes a formal decomposition
parallel to Proposition~\ref{prop:formal_decomposition}
for the analytified Brieskorn module.

\begin{Proposition}\label{prop:formal_decomp_Bri}
The formal asymptotic expansion $\Asymp$ in Definition~$\ref{def:Asymp}$ defines an isomorphism
\[
\Asymp\colon \ \left. \Brian(F)_\bSigma \right|_{\cVss}
\cong \pr_*\left( \cOan_{\Crss}[\![z]\!] \otimes
\ori \right),
\]
where $\ori$ is the $\mu_2$-local system over $\Crss$
defined by the monodromy of the square root $\sqrt{\det(h_{i,j})}$
of the logarithmic Hessian of $F$ $($see \eqref{eq:Hessian_matrix}
for $h_{i,j})$.
This map identifies the Gauss--Manin
connection $\nabla$ with $d+ d(F|_{\Crss}/z)\wedge$ and
the higher residue pairing with the diagonal pairing:
\[
P_{\std}(f,g)(x)
= \sum_{p \in \pr^{-1}(x) \cap \Crss}f(p,-z) g(p,z),
\]
where $f = f(x,z), g= g(x,z) \in \pr_*(\cOan_{\Crss}[\![z]\!])$.
\end{Proposition}
\begin{proof}As we remarked in Remark \ref{rem:Asymp}, we need to choose
a square root of the Hessian when defining the formal
asymptotic expansion; hence we need the
$\mu_2$-local system~$\ori$.
Along $z=0$, the map $\Asymp$ is given by
$\phi \cdot \omega \mapsto \phi|_{\Crss}
/(|\bN_{\rm tor}|\sqrt{\det(h_{i,j})})$.
On the other hand,
\begin{align*}
\left. \Brian(F)_\bSigma\right|_{z=0}
& \cong \pr_*
\left(\cOan_\cB[\![z]\!]/(\chi_1,\dots,\chi_n,z) \cOan_\cB[\![z]\!]
\right)
\\
& \cong \pr_*\left( \cOan_\cB\Big/
\left( x_1 \parfrac{F}{x_1}, \dots, x_n \parfrac{F}{x_n}\right)
\right)
\cong \pr_* \cOan_{\Cr},
\end{align*}
where $\chi_i$ acts on $\cOan_\cB[\![z]\!]$ as in~\eqref{eq:chi_action_on_Bri}. Hence $\Asymp$ is an isomorphism along $z=0$. Since both sides are locally free $\cO_{\cVss}[\![z]\!]$-modules, $\Asymp$ is an isomorphism. That $\Asymp$ identifies the connections follows from \cite[Lemma~6.7]{CCIT:MS}. That $\Asymp$ identifies the pairings is obvious from Definition~\ref{def:higher_residue}.
\end{proof}

\begin{Remark}The eigenvalues $\bu_1,\dots,\bu_N$ of $E\star_\tau$ in Proposition~\ref{prop:formal_decomposition} correspond to critical values
of $F$.
\end{Remark}

\subsubsection{Brieskorn module over the small quantum cohomology locus}\label{subsubsec:Bri_sm}
To discuss an analytic lift of the above formal decomposition, we restrict the Brieskorn module to the small quantum cohomology locus. We assume that $\frX_\bSigma$ is a \emph{weak Fano} (i.e., $-K_{\frX_\bSigma}$ is nef) and \emph{compact} toric stack. Furthermore, we assume that
\begin{equation}
\label{eq:generation}
\text{$S_-:=S \cap \Delta =\{b\in S \colon \overline{b} \in \Delta\}$
generates $\bN$ as a group},
\end{equation}
where $\Delta\subset \bN_\R$
denotes the convex hull of ray vectors
$\{\overline{b}\colon b\in R(\bSigma)\}$.
The compactness implies that $\Delta$
contains the origin in its interior and
the weak-Fano condition implies that
all rays $b\in R(\bSigma)$ lie in the boundary of $\Delta$.
By replacing $S$ with $S_- = S \cap \Delta$ in
the construction of the (partially compactified) LG model
in Section~\ref{subsec:LG}, we obtain an LG model
\begin{equation}\label{eq:LG_sm}
\begin{CD}
\cYsm @>{F}>> \C, \\
@V{\pr}VV @. \\
\cMsm @.
\end{CD}
\end{equation}
which we call \emph{the LG model mirror to the small
quantum cohomology of $\frX_\bSigma$}.
We also call $\cMsm$ the \emph{small quantum
cohomology locus} of $\frX_\bSigma$.
Under the mirror map, $\cMsm$
maps to $H^{\le 2}_\CR(\frX_\bSigma)$.

\begin{Lemma}
\label{lem:pullback}
The total space $\cYsm$ is a closed toric substack
of $\cY$ corresponding to the cone $(\R_{\ge 0})^{S\setminus S_-}$
of $\tXi$; similarly $\cMsm$ is a closed toric substack
of $\cM$ corresponding to the cone
$D\big((\R_{\ge 0})^{S\setminus S_-}\big)\in \Xi$.
Moreover, we have the pull-back diagram:
\begin{equation}\label{eq:cMsm_pullback}
\begin{aligned}
\xymatrix{
\cYsm \ar@{^(->}[r] \ar[d] & \cY \ar[d]\\
\cMsm \ar@{^(->}[r] & \cM.
}
\end{aligned}
\end{equation}
\end{Lemma}
\begin{proof}
Let $\tXi_-$ denote the fan defining $\cYsm$.
Recall that maximal cones of $\tXi$ are in one-to-one
correspondence with stacky fans adapted
to $S$; likewise, maximal cones of $\tXi_-$ are in one-to-one
correspondence with stacky fans adapted to $S_-$.
Thus the set of maximal cones of $\tXi_-$ can be identified
with a subset of the set of maximal cones of $\tXi$.
We can see that this subset consists of maximal cones
$\CPL_+(\bSigma)$ of $\tXi$ that contain
$(\R_{\ge 0})^{S\setminus S_-}$ as a face; moreover
the corresponding maximal cone of $\tXi_-$ is given by
the image of $\CPL_+(\bSigma)$ under the
projection $\big(\R^S\big)^\star \to \big(\R^{S_-}\big)^\star$.
Therefore $\tXi_-$ is a fan obtained as the star of the
cone $(\R_{\ge 0})^{S\setminus S_-}$ in~$\tXi$.
This shows the first statement.
A similar argument shows that $\cMsm$ is
a closed toric substack corresponding to
$D\big((\R_{\ge 0})^{S\setminus S_-}\big)$.
To see the pull-back diagram, we recall the description
of the uniformizing chart in~\eqref{eq:LG_local_chart}.
When $\bSigma$ is adapted to~$S_-$, $G(\bSigma)$ contains~$S\setminus S_-$. In the local chart associated with~$\bSigma$,
the diagram~\eqref{eq:cMsm_pullback} is of the form
\[
\xymatrix{
\Spec \C\big[\OO^\bSigma_+\big] \times
\C^{G(\bSigma) \setminus (S\setminus S_-)}
\ar@{^(->}[r] \ar[d] &
\Spec \C\big[\OO^\bSigma_+\big] \times \C^{G(\bSigma)}
\ar[d] \\
\Spec \C\big[\Laa^\bSigma_+\big] \times \C^{G(\bSigma)
\setminus (S\setminus S_-)}
\ar@{^(->}[r] &
\Spec \C\big[\Laa^\bSigma_+\big] \times \C^{G(\bSigma)},
}
\]
which is clearly a pull-back diagram.
\end{proof}

\begin{Remark}
The small quantum cohomology locus $\cMsm
\subset \cM$ depends on the choice of $\bSigma\in \Fan(S)$.
\end{Remark}

\begin{Remark}The condition \eqref{eq:generation} ensures that
Assumption \ref{assump:S_generates_N} holds for
$S_- = S\cap \Delta$.
This condition is, however, not necessary at this point;
we can define $\cYsm$ (or $\cMsm$) as the substack
corresponding to $(\R_{\ge 0})^{S\setminus S_-}$
(resp.~to $D\big((\R_{\ge 0})^{S\setminus S_-}\big)$).
We shall need this condition later when we apply the
results on the $\hGamma$-integral structure for toric stacks
in~\cite{Iritani:Integral} (see Section~3.1.4 \emph{ibid} where
the same assumption was made).
\end{Remark}

We observe that the non-equivariant
Brieskorn module $\Bri(F)$ over $\cMsm$
already has the expected rank; hence
the completion and the analytification studied in
Sections~\ref{subsec:completion}--\ref{subsec:analytified_Bri}
are unnecessary over $\cMsm$ in the weak-Fano case.

\begin{Proposition}
\label{prop:Bri_weak_Fano}
Let $\frX_\bSigma$ be a weak Fano compact toric stack
satisfying \eqref{eq:generation}.
There exists a Zariski-open subset
$\cMsmlf$ of $\cMsm$ containing $0_\bSigma$
such that the non-equivariant Brieskorn module $\Bri(F)|_{\cMsmlf}$
is a locally free $\cO_{\cMsmlf\times \C_z}$-module of
rank $\dim H^*_\CR(\frX_\bSigma)$.
\end{Proposition}
\begin{proof}
This essentially follows from a result of
Mann-Reichelt \cite[Theorem 4.10]{Mann-Reichelt} on the GKZ system;
we give a proof of a more general statement (including the equivariant
case and without assumption \eqref{eq:generation})
in Appendix \ref{append:Bri_weak_Fano}.
\end{proof}

Let $\cV\subset \cM$ denote the base of the
analytified Brieskorn module $\Brian(F)_\bSigma$ as before.
This is an analytic open neighbourhood of $0_\bSigma$.

\begin{Proposition}\label{prop:Bri_weak_Fano_cor}\quad
\begin{enumerate}\itemsep=0pt
\item[$(1)$] Set $\cVsm = \cV \cap \cMsmlf$. The natural map
\[
\Bri(F) \otimes_{\cOan_{\cVsm\times \C_z}}
\cOan_{\cVsm}[\![z]\!]
\to \Brian(F)_\bSigma
\]
is an isomorphism.
\item[$(2)$] The mirror isomorphism $\Mir|_{\chi=0}$
in Proposition~$\ref{prop:nonequiv_completion}$ extends to an isomorphism
\[ \Bri(F)|_{{\cVsm}'} \cong \mir^*\QDMan(\frX)\] over
an open neighbourhood ${\cVsm}' \subset \cVsm$
of $0_\bSigma$, where
$\mir \colon {\cVsm}' \to V$ is the analytic mirror map
$($see Section~$\ref{subsec:an_mirror_isom}$
for the convergence of the mirror map$)$.
\end{enumerate}
\end{Proposition}
\begin{proof}
For Part (1), it suffices to show that the map is an isomorphism
along $z=0$.
As discussed in the proof of Proposition \ref{prop:formal_decomp_Bri},
the analytified Brieskorn module along $z=0$ is isomorphic to
$\pr_*\cOan_{\Cr}$. Thus the natural map in (1) is surjective
along $z=0$; it is an isomorphism since the ranks are the same.
Part (2) follows from the convergence of the $I$-function in
the weak-Fano case,
see, e.g., \cite[Proposition 4.8]{Iritani:Integral}.
\end{proof}

\subsubsection{Analytic lift of the formal decomposition}
\label{subsubsec:analytic_lift}
We continue to assume that $\frX_\bSigma$ is a compact,
weak-Fano toric stack and that the condition \eqref{eq:generation} holds.
Consider the LG model \eqref{eq:LG_sm} over $\cMsm$.
As a toric stack, $\cMsm$ contains the open dense torus orbit\footnote{The following discussion works if we replace $\cMsmtimes$ with
the slightly bigger subspace $\{q\in \cMsm_\bSigma\colon\allowbreak
\text{$q^\lambda \neq 0$ for all $\lambda
\in \Laa^\bSigma_+$}\}$: this bigger space parametrizes
Laurent polynomials with Newton polytope $\Delta$.}
\[
\cMsmtimes = (\LL')^\star \otimes \C^\times \subset \cMsm,
\]
where $\LL'=\Ker(\Z^{S_-} \to \bN)$ is the lattice appearing
in the leftmost term of the extended fan sequence~\eqref{eq:ext_fanseq}
with $S$ replaced with $S_-=S\cap \Delta$.
The fibre of $\pr\colon \cYsm\to \cMsm$ at a point $q\in \cMsmtimes$
is isomorphic to $\Hom(\bN,\C^\times)$
(i.e., the fibre of the uncompactified LG model, see Section~\ref{subsec:LG}).
In fact, since $\Delta$ contains the origin in its interior, we have
a linear relation $\sum\limits_{b\in S_-} \lambda_b b =0$
with $\lambda_b \in \Z_{>0}$;
this shows that $\prod\limits_{b\in S_-} u_b^{\lambda_b} = q^\lambda \neq 0$
on the fibre at $q\in \cMsmtimes$ and
in particular that $u_b\neq 0$ for all $b\in S_-$.

The locally freeness of $\Bri(F)|_{\cMsmlf}$ in Proposition~\ref{prop:Bri_weak_Fano} implies
(by the restriction to $z=0$)
that the family of relative critical points of $F$
\begin{equation}\label{eq:Crsm}
\Crsm := \left\{p \in \cYsm \colon x_1 \parfrac{F}{x_1}(p) = \cdots
= x_n \parfrac{F}{x_n}(p) = 0 \right\} \xrightarrow{\pr} \cMsm
\end{equation}
is finite flat of degree $N=\dim H^*_\CR(\frX_\bSigma)$
over the base $\cMsmlf$.
The generic fibre of the family $\Crsm\to \cMsm$ is reduced
\cite[Proposition 3.10]{Iritani:Integral}, and thus
\[
\cMsmss = \big\{q\in \cMsmtimes \cap \cMsmlf \colon
\text{$F|_{\pr^{-1}(q)}$ has only non-degenerate critical points}\big\}
\]
is a non-empty Zariski open subset, where ``ss'' means
semisimplicity.

For $q\in \cMsmss$, let $\{c_1,\dots c_N\}$
be the set of (mutually distinct) critical points of
$F_q :=F|_{\pr^{-1}(q)}$, and
let $\bu_i = F_q(c_i)$ be the critical value.
For an admissible phase $\phi$ for $\{\bu_1,\dots,\bu_N\}$,
let $\Gamma_i^\phi \subset \pr^{-1}(q)$ denote the Lefschetz
thimble of $F_q$ emanating from the critical point $c_i$
whose image under $F_q$ is the half-line $\bu_i - \R_{\ge 0} e^{\iu\phi}$;
it is given as the stable manifold of the Morse function
$\Re\big(e^{-\iu\phi}F_q\big)\colon \pr^{-1}(q) \to \R$:
\begin{equation}\label{eq:Lefschetz}
\Gamma_i^\phi = \big\{x \in \pr^{-1}(q)\colon
\lim\limits_{t\to \infty} \varphi_t(x) = c_i \big\} \cong \R^n,
\end{equation}
where $\varphi_t$ is the upward gradient flow\footnote{The gradient flow of $\Re\big(e^{-\iu\phi}F_q\big)$ equals the Hamiltonian flow of $\Im\big(e^{-\iu\phi} F_q\big)$ and thus preserves $\Im\big(e^{-\iu\phi}F_q\big)$.}
of $\Re\big(e^{-\iu\phi} F_q\big)$ with respect to the complete K\"ahler
metric $\frac{\iu}{2}\sum\limits_{j=1}^n d \log x_j \wedge d \overline{\log x_j}$
on $\pr^{-1}(q)$.
The cycles $\Gamma_1^\phi,\dots,\Gamma_N^\phi$
form a basis of the relative homology
$H_n\big(\pr^{-1}(q), \{\Re(e^{-\iu\phi}F_q)\le -M\};\Z\big)$
for sufficiently large~$M$,
see \cite[Section~3.3.1]{Iritani:Integral} for more details
(see also Section~\ref{subsubsec:Kouchnirenko_Morse}).
\begin{figure}[htbp]\centering
\includegraphics{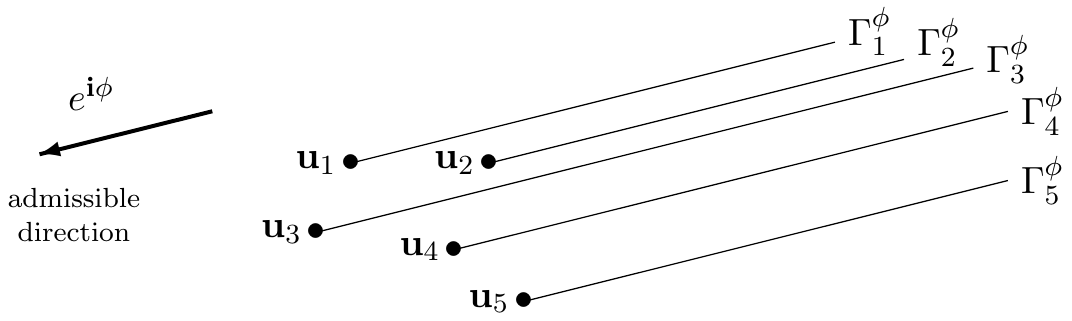}
\caption{The images of the Lefschetz thimbles $\Gamma_i^\phi$ by $F_q = F|_{\pr^{-1}(q)}$.}
\end{figure}

\begin{Remark}
For a non-admissible phase $\phi$, the Lefschetz thimble $\Gamma_i^\phi$
is not always defined because
$\Gamma_i^\phi$ may hit other critical points $c_j$
with $\bu_j \in \bu_i - \R_{>0} e^{\iu\phi}$.
On the other hand, when some of the critical values $\bu_1,\dots,\bu_N$ coalesce
at $q_0\in \cMsmss$
and $\phi$ is an admissible phase
for the critical values of $F_{q_0}$, the Lefschetz thimbles
$\Gamma_1^\phi,\dots, \Gamma_N^\phi$ are well-defined
in a neighbourhood of $q=q_0$ despite the possibility that
$\phi$ can be non-admissible at a nearby point.
This is because different Lefschetz thimbles associated with
the same critical value do not intersect each other, and
no non-trivial Picard--Lefschetz transformations occur
among these thimbles around $q=q_0$.
\end{Remark}

\begin{Proposition}\label{prop:analytic_lift_Bri}
Let $q_0$ be a point in $\cMsmss$. Choose an admissible phase $\phi$
for the critical values $\{\bu_{1,0},\dots,\bu_{N,0}\}$ of $F_{q_0}$.
Choose a sufficiently small open neighbourhood $B \subset \cMsmss$
of~$q_0$
and a sufficiently small number $\epsilon>0$ such that
$e^{-\iu\phi'} (\bu_i - \bu_j) \notin \R$ whenever $q \in B$,
$|\phi-\phi'|<\epsilon$ and $\bu_{i,0} \neq \bu_{j,0}$.
Let $\pi\colon \cMsmss \times \tCC \to \cMsmss \times \C$ denote
the oriented real blowup along $\cMsmss \times \{0\}$.
Define the map
\[
\Phi_\phi \colon \
\pi^*\Bri(F)\big|_{B \times I_\phi} \to
\cA_{B\times I_\phi}^{\oplus N}
\]
by
\begin{equation}\label{eq:lift_by_oscillatory}
\Phi_\phi(s) (q,z)=
\left( (-2\pi z)^{-n/2} e^{-\bu_i/z}\int_{\Gamma_i^{\phi+\delta}}
e^{F_q/z} s \right)_{i=1}^N, \qquad
(q,z) \in B\times I_\phi,
\end{equation}
where $I_\phi = \big\{\big(r,e^{\iu\theta}\big) \colon |\theta - \phi| < \frac{\pi}{2} +
\epsilon\big\}$ and
we choose $\delta \in (-\epsilon,\epsilon)$
depending on the argument of~$z$ so that the integral converges.
Then $\Phi_\phi$ is an isomorphism that identifies the Gauss--Manin
connection $\nabla$ with $\bigoplus\limits_{i=1}^N (d + d(\bu_i/z))$ and
induces the formal decomposition in Proposition~$\ref{prop:formal_decomp_Bri}$
$($combined with Proposition~$\ref{prop:Bri_weak_Fano_cor})$
if $B \subset \cVss$:
\[
\Asymp \colon \
\Bri(F)\otimes_{\cOan_{B\times \C_z}}
\cOan_B[\![z]\!] \cong \cOan_B[\![z]\!]^{\oplus N}.
\]
\end{Proposition}
\begin{proof} First observe that the oscillatory integral $\Phi_\phi(s)$ converges
for a suitable choice of $\delta\in (-\epsilon,\epsilon)$, and
does not depend on $\delta$ as far as it converges.
If $|\arg(z)-\phi|<\frac{\pi}{2}$, we can choose $\delta =0$
because $\Re(F_q/z) \to -\infty$ in the end of $\Gamma_i^\phi$;
if not we can choose a suitable $\delta$ so that
$\Re(F_q/z)\to -\infty$ in the end of $\Gamma_i^{\phi+\delta}$.
The fact that $\Phi_\phi$ identifies the connections
follows from the definition of the Gauss--Manin connection, see
\cite[equation~(54), Lemma~3.15]{Iritani:Integral}.
Note that the shift of the connection $\nabla_{z\parfrac{}{z}}$ by $n/2$ in
Remark~\ref{rem:conn_z} is compensated by the prefactor
$(-2\pi z)^{-n/2}$.
The last statement follows from the
fact that $\Asymp_{c_i}(s)$ gives the asymptotic expansion
of the $i$th component of $\Phi_\phi(s)$ along the sector
$I_\phi$.
\end{proof}

\begin{Remark} The choice of an orientation of
$\Gamma_i^{\phi+\delta}$ and the choice of
a branch of $(-2\pi z)^{-n/2}$ in \eqref{eq:lift_by_oscillatory}
together give rise to a section
of the $\mu_2$-local system $\ori$ in
Proposition \ref{prop:formal_decomp_Bri}.
\end{Remark}

\section{Functoriality under toric birational morphisms}\label{sec:functoriality}
We study the analytic lift of the formal decomposition of
the quantum D-modules in Theo\-rem~\ref{thm:decomp_QDM}
in the case where the birational map
$\frX_+ \dasharrow \frX_-$ extends to a \emph{morphism}.
We show that the analytic lift associated with a certain
deformation parameter $\tau_+$ and a phase is induced by the pull-back between
the $K$-groups via the $\hGamma$-integral structure.
Moreover, the sectorial decomposition of the quantum D-module of $\frX_+$
at some $\tau_+$ corresponds to
an Orlov-type semiorthogonal decomposition of the $K$-group.
We assume that both $\frX_+$ and $\frX_-$ are compact weak-Fano
smooth toric DM stacks and restrict ourselves to the non-equivariant quantum D-modules.

\subsection{Notation and assumption}\label{subsec:functoriality_nota}
Consider a discrepant transformation $\frX_+ \dasharrow \frX_-$
arising from a codimension-one wall crossing as in Section~\ref{subsec:discrepant}.
Let $\bSigma_\pm$ be the stacky fan of $\frX_\pm$.
We assume that $\frX_+ \dasharrow \frX_-$ extends to a~birational morphism $\varphi \colon \frX_+\to \frX_-$. In this case, the common blowup $\hfrX$ in~\eqref{eq:roof}
is isomorphic to~$\frX_+$, and~$\varphi$ is necessarily a type~(II-i) or~(III) discrepant transformation
in the classification of Remark~\ref{rem:3_types}, i.e., $\varphi$
is a divisorial contraction or a root construction.
We further assume that
\begin{itemize}\itemsep=0pt
\item $\frX_+$ and $\frX_-$ are compact weak Fano toric stacks;
we write $\Delta_\pm \subset \bN_\R$ for the fan polytopes
of $\frX_\pm$; they are convex polytopes
containing the origin in their interiors and we have
$\Delta_- \subset \Delta_+$;
\item $S_- := S \cap \Delta_-=\{b \in \bN \colon \overline{b} \in \Delta_-\}$
generates $\bN$ over $\Z$.
\end{itemize}
Here the fan polytope $\Delta_\pm$ means the convex hull of ray vectors of the stacky fan $\bSigma_\pm$
and $S$ is a finite subset of $\bN$ as in Section~\ref{subsec:data} such that both $\bSigma_+$ and $\bSigma_-$ are adapted to $S$ in the sense of Definition~\ref{def:stacky_fan_adapted_to_S}.
We need these assumptions so that we can apply the results~\cite{Iritani:Integral} on the $\hGamma$-integral structure for toric stacks, where the same assumptions were made (see Section~3.1.4 \emph{ibid}).

\begin{figure}[ht]\centering
\includegraphics{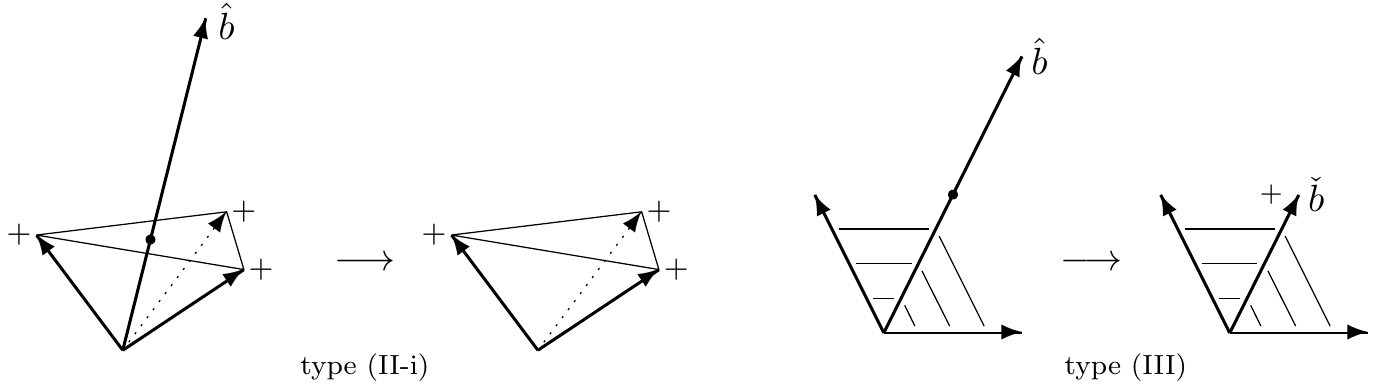}
\caption{The change of fans associated with $\frX_+ \to \frX_-$.
The sign $+$ means ray vectors from $M_+\subset R_-=R(\bSigma_-)$
and $\hb$ is a positive integral linear combination of
$\{v\colon v\in M_+\}$.} \label{fig:fan_morphism}
\end{figure}

As in Section~\ref{subsec:discrepant},
let $W$ denote the hyperplane wall between the maximal cones
$\cpl(\bSigma_+)$ and $\cpl(\bSigma_-)$
of the secondary fan $\Xi$,
and let $\bw\in \LL$ denote the primitive normal vector
of the wall $W$ pointing towards $\cpl(\bSigma_+)$.
Set $M_\pm = \{ b\in S\colon \pm D_b\cdot \bw >0\}$ as before.
When the wall-crossing induces a morphism
$\frX_+ \to \frX_-$,
$M_-$ is a singleton $\{\hb\}$ with
$D_{\hb} \cdot \bw = -1$ and the corresponding circuit is (see \eqref{eq:circuit})
\[
\hb = \sum_{b\in M_+} k_b b \qquad \text{with $k_b := D_b \cdot \bw$}.
\]
Assumption \ref{assump:discrepancy} gives
$\sum\limits_{b\in M_+} k_b>1$.
In the type~(II-i) case, we have $\sharp M_+\ge 2$ and
the stacky fan~$\bSigma_+$ is obtained from~$\bSigma_-$
by adding the new ray~$\hb$;
the cone $\sigma_{M_+}=\sum\limits_{b\in M_+} \R_{\ge 0} \overline{b}$
of $\bSigma_-$ is subdivided into cones
$\sigma_{M_+\cup \{\hb\} \setminus \{v\}}$ with $v\in M_+$.
In the type~(III) case, $M_+$ is also a singleton~$\{\chb\}$ and
$\bSigma_+$ is obtained from $\bSigma_-$
by replacing $\chb$ with $\hb = k_{\chb} \chb$:
see Fig.~\ref{fig:fan_morphism}.
We write $R_\pm = R(\bSigma_\pm)$ for the set of rays.
Then $R_+ = R_- \sqcup \{\hb\}$ in the type (II-i) case and
$R_+ \sqcup \{\chb\} = R_- \sqcup \{\hb\}$ in the type (III) case,
where $\sqcup$ denotes disjoint union.
We also note that $M_+ \subset R_-$.
For simplicity of notation, we assume that $S$ is chosen to be a
minimal extension of $S_-$:
\begin{itemize}\itemsep=0pt
\item $S = S_- \cup \{\hb\}$.
\end{itemize}
We do not lose any generality by this assumption:
the base $\cM$ of the LG model for a larger $S$
always contains the locus corresponding to $S_- \cup \{\hb\}$.

The smooth toric DM stacks $\frX_+$, $\frX_-$
are the GIT quotients of $\C^S$.
The toric birational morphism $\varphi \colon \frX_+ \to \frX_-$
is induced by the self-map
\begin{equation}
\label{eq:varphi_coord}
\big(z_\hb, (z_v)_{v\in S\setminus \{\hb\}}\big)
\longmapsto
\big(1, \big(z_\hb^{k_v} z_v\big)_{v\in S\setminus \{\hb\}}\big)
\end{equation}
on $\C^S \cong \C^{\{\hb\}} \times \C^{S\setminus \{\hb\}}$,
where note that $k_v \ge 0$ for $v\neq \hb$.
It is easy to check that this sends the stable locus
for $\cpl(\bSigma_+)$ to the stable locus for $\cpl(\bSigma_-)$.
The map $\varphi$ contracts the divisor $\{z_{\hb} = 0\}$
onto the toric substack $Z= \bigcap\limits_{v\in M_+} \{z_v=0\}$.
In the type (II-i) case, $\varphi$ is a weighted blowup
along the codimension $\ge 2$ substack~$Z$;
in the type (III) case where $M_+$ is a singleton~$\{\chb\}$,
$\varphi$ exhibits $\frX_+$ as a root stack of $\frX_-$
with respect to the divisor $Z=\{z_{\chb} =0\}$.

We write $(\pr\colon \cY\to \cM, F)$, $(\pr\colon \cYsm \to \cMsm, F)$
for the LG models associated with $S$ and $S_-=S\cap \Delta_-$
respectively (see Definition \ref{def:LG}).
These two LG models are related by the pull-back
(Lemma \ref{lem:pullback}):
\[
\xymatrix{
\cYsm \ar@{^(->}[r] \ar[d] & \cY \ar[d] \ar[r]^{F} & \C \\
\cMsm \ar@{^(->}[r] & \cM. &
}
\]
We write $\Bri(F)$, $\Brism(F)$ for the non-equivariant Brieskorn modules
(Definition \ref{def:Brieskorn})
associated with the LG models $(\pr\colon \cY\to \cM, F)$,
$(\pr\colon \cYsm\to\cMsm, F)$ respectively.
On the affine chart associated with $\bSigma_-$,
$\cMsm$ is cut out from $\cM$
by the equation $t_\hb =0$, where
$t_\hb = q^{\delta_\hb^{\bSigma_-}}$
is the co-ordinate introduced in Section~\ref{subsec:coord_localchart_LG}.
We have $t_{\hb}= q^{-\bw}$ since
$\delta_\hb^{\bSigma_-} =
e_{\hb} - \Psi_-(\hb) = e_{\hb} - \sum\limits_{b\in M_+} k_b e_b
= -\bw\in \LL$,
where $\Psi_- = \Psi^{\bSigma_-}$ (see Notation \ref{nota:Psi}).

In this section, we fix an isomorphism
$\bN \cong \Z^n \times \bN_{\rm tor}$
and a splitting $\varsigma \colon \bN\to \OO^{\bSigma_-}$ of
the refined fan sequence \eqref{eq:refined_fanseq} for $\bSigma_-$
and use the associated co-ordinates $x_1,\dots,x_n$ along
the fibres of the LG model as introduced
in Section~\ref{subsec:coord_localchart_LG}.

\begin{Caution}
Under our assumption $S=S_- \cup \{\hb\}$,
$\cM$ is the small quantum cohomology locus
$($see $Section~\ref{subsubsec:Bri_sm})$ of $\frX_+$
and $\cMsm$ is
the small quantum cohomology locus of $\frX_-$.
\end{Caution}

\subsection{Critical points along a curve}
\label{subsec:crit_curve}
Consider the 1-dimensional toric substack $\cC$ of $\cM$
corresponding to $W\cap \cpl(\bSigma_+)=
W\cap \cpl(\bSigma_-)$.
Recall that uniformizing affine charts of $\cC$ and $\pr^{-1}(\cC)$ have been
described explicitly in Section~\ref{subsec:LG_curve}:
$q^{-\bw/e} = t_\hb^{1/e}$ gives a rational co-ordinate
of $\cC \cap \cM_{\bSigma_-}$,
where $e=e_-\in \N$ is the smallest common denominator
of $\{c\in \Q\colon c\bw \in \Laa(\bSigma_-)\}$, and
$w_v = w_v^- = u^{(\Psi_-(v),v)}$ with $v\in \bN$ generate
the co-ordinate ring $A = \bigoplus\limits_{v\in\bN} \C[t_\hb^{1/e}] w_v$
of $\pr^{-1}(\cC) \cap \cY_{\bSigma_-}$
(see Lemma~\ref{lem:LG_on_curve}).
Recall also that
$F = \sum\limits_{b\in R_+ \cup R_-} u_b$ on $\pr^{-1}(\cC)$.
We study the family of relative critical points
\begin{equation}
\label{eq:relative_critical}
x_1\parfrac{F}{x_1} = \cdots = x_n \parfrac{F}{x_n} =0
\end{equation}
over the curve $\cC$.

\begin{Proposition}
\label{prop:critv_on_curve}
A relative critical point of the LG potential $F$
over $\cC$ is given by an assignment of a complex
number $w_v$ to every $v\in \bN$
satisfying $w_0=1$,
\begin{align*}
w_{v_1} w_{v_2}=w_{v_1+v_2}
\qquad \text{if $\overline{v_1},\overline{v_2} \in \sigma_{M_+}$} \qquad
\text{and} \qquad
w_v = \begin{cases}
0, & \text{if $\overline{v} \notin \sigma_{M_+}$}, \\
k_v \gamma, & \text{if $v\in M_+$},
\end{cases}
\end{align*}
where $k_v = D_v \cdot \bw$ for $v\in S$ and
\[
\gamma=0 \qquad \text{or} \qquad
\gamma = \left(-\frac{1}{Kt_\hb}\right)^{1/J}
\qquad \text{when $t_\hb \neq 0$}
\]
with $J:= \big(\sum\limits_{b\in M_+} k_b\big)-1>0$ and $K:= \prod\limits_{b\in M_+} k_b^{k_b}$.
The corresponding critical value of $F$ is given by $J \gamma$.
\end{Proposition}
\begin{proof}
Take a relative critical point and let $w_v$ denote the value of the function $w_v$ at that point.
We have $w_0=1$. The equation \eqref{eq:relative_critical} for relative critical points reads
\begin{equation}\label{eq:balance}
\sum_{b\in R_+ \cup R_-} u_b \overline{b} = 0 \qquad
\text{in $\bN_\R$.}
\end{equation}
Note that $R_+ \cup R_- = R_- \sqcup \{\hb\}$,
$u_b= w_b$ for $b\in R_-$
and $u_\hb = t_\hb w_\hb$.
By Lemma~\ref{lem:LG_on_curve}, we have
\[
w_v \cdot w_{v'} =
\begin{cases}
w_{v+v'}, & \text{if $\overline{v}$ and $\overline{v'}$ belong to
a common cone of $\Sigma_-= \Sigma_0$}, \\
0, & \text{otherwise},
\end{cases}
\]
where note that $\Sigma_0$ in Lemma \ref{lem:LG_on_curve}
coincides with the underlying fan $\Sigma_-$ of $\bSigma_-$
in the current setting.
Therefore there exists a cone $\sigma \in \Sigma_-$
such that $\{\overline{v}\colon w_v \neq 0\} = \bN \cap \sigma$.
Then equation~\eqref{eq:balance} implies that
$(R_-\cup \{\hb\}) \cap \sigma$ is linearly dependent.
Therefore we have either $(R_-\cup \{\hb\}) \cap \sigma=\varnothing$
or $\overline{\hb} \in \sigma$.
In the former case, we have $\sigma=0$; this corresponds
to the case where $\gamma=0$ in the proposition.
In the latter case, we have $\sigma \supset \sigma_{M_+}$;
the linear relation $\hb = \sum\limits_{b\in M_+} k_b b$
together with~\eqref{eq:balance} implies
\[
\begin{cases}
w_b + k_b t_\hb w_{\hb} = 0
& \text{for $b\in M_+$}, \\
w_b =0 & \text{for $b\in (R_- \setminus M_+) \cap \sigma$},
\end{cases}
\qquad \text{and}
\qquad
w_{\hb} = \prod_{b\in M_+} w_{b}^{k_b}.
\]
Since $w_b \neq 0$ for $\overline{b} \in \sigma$, we have
$(R_-\setminus M_+) \cap \sigma = \varnothing$
and thus $\sigma = \sigma_{M_+}$.
Setting $\gamma = w_b/k_b = -t_\hb w_\hb$ (with $b\in M_+$),
we find that $\gamma$ satisfies
\[
-\gamma/t_\hb = \prod_{b\in M_+}(k_b \gamma)^{k_b}
\ \Longleftrightarrow \
-1/(K t_\hb) = \gamma^J.
\]
This proves the proposition.
\end{proof}

\begin{Remark}
Proposition {\rm \ref{prop:critv_on_curve}} implies that
the critical values of $F$ along $\cC$ do not belong to $\R_{>0}$
when $t_\hb >0$. This phenomenon is closely related to
the Conjecture $\cO$
{\rm \cite[Conjecture 3.1.2]{GGI:gammagrass}}.
\end{Remark}

\subsection{Identifying $\cO$}
\label{subsec:identifying_O}
Recall from Section~\ref{subsec:integral} that
the $\hGamma$-integral structure identifies the $K$-group of
a smooth DM stack with a lattice of flat sections of the quantum D-module.
In this section, we introduce a `positive real' Lefschetz thimble $\Gamma_\R$
defined along a `positive real' locus $\cM_\R$ (or $\cMsm_\R$).
The flat section $\frs_\cO$ given by the structure sheaf $\cO$
corresponds to $\Gamma_\R$
under mirror symmetry, and thus spans a component of
the analytic lift of the formal decomposition
(as discussed in
Propositions~\ref{prop:Hukuhara-Turrittin} and~\ref{prop:analytic_lift_Bri})
associated with the so-called conifold point.
The content in this section is essentially an adaptation of the main result of~\cite{Iritani:Integral} to the current setting.

\subsubsection[The structure sheaf of $\frX_+$]{The structure sheaf of $\boldsymbol{\frX_+}$}\label{subsubsec:str_sheaf_+}

By definition (see Section~\ref{subsec:localchart_LG}),
$\cM$ and $\cY$ contain open dense tori
$\cMtimes=\LL^\star \otimes \C^\times$,
$\cYtimes=(\C^\times)^S$ respectively.
We define the positive real loci $\cM_\R \subset \cMtimes$,
$\cY_\R \subset \cYtimes$ to be
\begin{gather*}
\cM_{\R} := \LL^\star \otimes \R_{>0}, \qquad
\cY_\R := (\R_{>0})^S.
\end{gather*}
We write $\cY_q$ for the fibre of $\pr\colon \cY \to \cM$ at $q\in \cM$;
and $\Gamma_\R=\Gamma_\R(q)$ for the fibre of $\cY_\R \to \cM_\R$ at $q\in \cM_\R$:
\[
\Gamma_\R(q) = \cY_q \cap \cY_\R \cong \Hom(\bN,\R_{>0}) \cong
(\R_{>0})^n.
\]
Consider the restriction of $F_q = F|_{\cY_q}$
to the real positive locus $\Gamma_\R$. Then
\begin{itemize}\itemsep=0pt
\item it is a strictly
convex function since the Hessian
$\parfrac{^2 F_q}{\log x_i\partial \log x_j} = \sum\limits_{b\in S}
u_b b_i b_j$ is positive definite on $\Gamma_\R$ (where $b_i\in \Z$
is the $i$th entry of $\overline{b}\in \overline{\bN}\cong \Z^n$);
\item it is proper and bounded from below
since $0$ is in the interior of the
convex hull $\Delta_+$ of $\{\overline{b}\colon b\in S\}$.
\end{itemize}
Therefore, $F_q|_{\Gamma_\R}$ attains a global minimum
at a unique critical point $\crit_\R = \crit_\R(q) \in \Gamma_\R(q)$; the point
$\crit_\R$ is called
the \emph{conifold point} \cite{Galkin:conifold,GGI:gammagrass}.
Since $\Gamma_\R$ is preserved by the gradient flow of~$\Re(F_q)$ (with respect to the
K\"ahler metric $\frac{\iu}{2} \sum\limits_{i=1}^n
d\log x_i \wedge d\overline{\log x_i}$)
on $\cY_q=\pr^{-1}(q)$, we have the following:
\begin{Lemma}\label{lem:real_Lefschetz}
The positive real locus $\Gamma_\R$ of $\cY_q$ is the
Lefschetz thimble~\eqref{eq:Lefschetz} of $F_q$ associated with the conifold point $\crit_\R$ and the phase $\phi=\pi$.
\end{Lemma}

By Proposition~\ref{prop:Bri_weak_Fano_cor}, there exist
an analytic neighbourhood $\cV_+$ of $0_+ := 0_{\bSigma_+}
\in \cM$, a mirror map $\mir_+ \colon \cV_+ \to
[\cM_{\rm A}(\frX_+)/\Pic^\st(\frX_+)]$
and an isomorphism
\[
\Mir_+ \colon \ \Bri(F)|_{\cV_+} \cong \mir_+^* \QDMan(\frX_+),
\]
where $\QDMan(\frX_+)$ denotes the analytic quantum D-module as in~\eqref{eq:analytic_QDM}.
The following result says that $\Gamma_\R(q)$ corresponds to the flat section
$\frs_\cO$ under mirror symmetry:

\begin{Theorem}[{\cite[Theorems 4.11 and 4.14, Section~4.3.1]{Iritani:Integral}}]\label{thm:main_intpaper}
Let $P(\alpha,\beta) = (\alpha(-z),\beta(z))$ denote the pairing~\eqref{eq:pairing_A} of the quantum D-module induced by the orbifold Poincar\'e pairing.
We write $\Omega=\Omega_{q,z}$ for a local section of~$\Bri(F)$.
\begin{itemize}\itemsep=0pt
\item[{\rm (1)}] For $q\in \cV_+\cap \cM_\R$, we have an isomorphism
\[
H_n(\cY_q, \{\Re(-F_q)\ll 0\};\Z) \cong K(\frX_+), \qquad \Gamma \mapsto V(\Gamma),
\]
which varies locally constantly in $q$ such that
\[
(2\pi z)^{-n/2} \int_{\Gamma} e^{-F_q/z} \Omega_{q,-z}
= P\left( \Mir_+(\Omega), \frs_{V(\Gamma)}(\mir_+(q),z) \right)
\]
for $\Gamma \in H_n(\cY_q, \{\Re(-F_q)\ll 0\};\Z)$,
$\Omega \in \Bri(F)$ and $z>0$,
where $\frs_V(\tau,z)$ is the flat section of the
$\hGamma$-integral structure $($see Definition $\ref{def:s})$.
\item[{\rm (2)}] In {\rm (1)},
$V(\Gamma_\R(q))$ is given by the structure sheaf $\cO$ of $\frX_+$.
\end{itemize}
In other words, the integral structure of the Brieskorn module
$\Bri(F)$ dual to the lattice
\[
H_n(\cY_q, \{\Re(F_q/z)\ll 0\};\Z)
\] corresponds to the $\hGamma$-integral structure of the quantum D-module under the mirror isomorphism.
\end{Theorem}

\begin{Remark}[{\cite[Section~4.3]{Iritani:Integral}}] \label{rem:intersection_Euler}
Since the mirror isomorphism $\Mir_+$ preserves the pairing,
the map $\Gamma \mapsto V(\Gamma)$ above
satisfies
\[
(-1)^{n(n-1)/2}\#\big(e^{-\pi\iu} \Gamma_1 \cdot \Gamma_2\big) = \chi(V(\Gamma_1),V(\Gamma_2))
\]
for $\Gamma_1,\Gamma_2 \in H_n(\cY_q,\{\Re(-F_q)\ll 0\};\Z)$,
where
$\chi(V_1,V_2)$ is the Euler pairing,
\[
e^{-\pi\iu} \Gamma_1\in H_n\big(\cY_{q_*}, \{\Re(F_{q_*})\ll 0\};\Z\big)
\] is the parallel translate
of $\Gamma_1$ in the local system
\[
\bigcup_{\theta \in [0,\pi]}
H_n\big(\cY_{q_*}, \big\{\Re\big({-}e^{\iu\theta}F_{q_*}\big) \ll 0\big\};\Z\big)
\]
from $\theta = 0$ to $\theta =\pi$, and $\#\big(e^{-\pi\iu}\Gamma_1\cdot \Gamma_2\big)$
denotes the algebraic intersection number.
Here we use the fact that the higher residue pairing corresponds
to the intersection pairing on relative homology
(see \cite[Sections~3.3.1--3.3.2]{Iritani:Integral}); note however
that the sign factor $(-1)^{n(n-1)/2}$ was missing in \cite{Iritani:Integral},
see \cite[footnote (16)]{Iritani:periods} for the correction.
\end{Remark}

\begin{Remark}[cf.~Remark \ref{rem:branch_L}]
\label{rem:standard_determination}
We need to specify a branch of the
multi-valued section $\frs_V$
in the above theorem. We have a standard
choice for the branch of $\frs_V(\mir_+(q),z)$ with
$V\in K(\frX_+)$ when $q\in \cM_\R$, and the above
theorem holds for this choice.
By the argument preceding \cite[Proposition 4.8]{Iritani:Integral},
the fundamental solution $L(\mir(q),z)$ can be obtained from
the $I$-function via the Birkhoff factorization; the $I$-function
has a standard determination on the positive real locus
(by requiring $\log q_a\in \R$ in the formula
\cite[(59)]{Iritani:Integral} of the $I$-function).
We also have a standard determination of $z^{-\mu}
z^{c_1(\frX)}$ given by $\log z \in \R$ for $z>0$.
Hence we obtain a standard identification of the space of
flat sections of $\Bri(F)$ over $(q,z)\in (\cM_\R \cap \cV_+)\times \R_{>0}$
with the $K$-group $K(\frX_+)\otimes \C$.
\end{Remark}

Introduce the following subsets of $\cV_+ \subset \cM$:
\begin{gather*}
\cVss_+ :=\{q\in \cV_+ \cap \cMtimes\colon \text{$F_q = F|_{\cY_q}$ has only non-degenerate critical points} \}, \\
\cVss_{+,\R} := \cVss_+\cap \cM_\R.
\end{gather*}
Note that $\cVss_+$ is the intersection of $\cV_+$ and
a non-empty Zariski-open subset of $\cM$ by the discussion
in Section~\ref{subsubsec:analytic_lift}; hence
$\cVss_+$ is open dense in $\cV_+$, and $\cVss_{+,\R}$ is
open dense in $\cV_+\cap \cM_\R$.

Choose $q_0\in \cVss_{+,\R}$ and let
$\crit_1(q), \dots,\crit_{N_+}(q)$
denote all branches of critical points of $F_q$ near $q=q_0$,
where $N_+ = \dim H^*_\CR(\frX_+)$.
We may assume that $\crit_1(q_0)$ is the conifold point $\crit_\R(q_0)$.
Let $\bu_i(q) = F_q(\crit_i(q))$ be the corresponding critical value.
Suppose that $\phi\in \R$ is an admissible phase for
$\{\bu_1(q_0),\dots,\bu_{N_+}(q_0)\}$.
By Proposition \ref{prop:analytic_lift_Bri},
there exist an open neighbourhood~$B$ of~$q_0$ in~$\cVss_+$,
a sector $I_\phi =
\big\{\big(r,e^{\iu\theta}\big) \in \tCC\colon
|\theta-\phi|<\frac{\pi}{2}+\epsilon\big\}$
(with small $\epsilon>0$) and an isomorphism (analytic lift)
\[
\Phi^+_\phi \colon \ \pi^* \Bri(F)\big|_{B\times I_\phi}
\cong \bigoplus_{i=1}^{N_+}
\left (\cA_{B\times I_\phi},
d + d(\bu_i(q)/z) \right)
\]
that induces the formal decomposition
$\Asymp\colon
\Bri(F) \otimes_{\cOan_{B\times \C_z}} \cOan_B[\![z]\!] \cong
\cOan_B[\![z]\!]^{\oplus N_+}$, where
$\pi \colon \cVss_+ \times \tCC \to \cVss_+ \times \C$ is the oriented
real blow-up. Composing this with
the mirror isomorphism, we also get the analytic lift
of the formal decomposition of the
quantum D-module (cf.~Proposition \ref{prop:Hukuhara-Turrittin}):
\[
\tPhi_\phi^+ \colon \
\pi^* \mir_+^*\QDMan(\frX_+)\big |_{B\times I_\phi}
\cong
\bigoplus_{i=1}^{N_+}
\left(\cA_{B\times I_\phi}, d+d(\bu_i(q)/z) \right),
\]
where $\tPhi_\phi^+=\Phi_\phi^+ \circ \Mir_+^{-1}$.

\begin{Proposition}\label{prop:lift_str_sheaf_+}
For $q_0 \in \cVss_{+,\R}$,
there exists $\alpha_0\in (0,\pi/2)$ such that the following holds.
For every admissible phase $\phi \in (-\alpha_0,\alpha_0)$
for $\{\bu_1(q_0),\dots,\bu_{N_+}(q_0)\}$,
there exists a basis $\{V_i\}_{i=1}^{N_+}$ of $K(\frX_+)$
with $V_1$ being the structure sheaf $\cO$ of $\frX_+$,
such that the corresponding
flat section $s_i = \frs_{V_i}(\mir_+(q),z)$ satisfies
$\tPhi^+_\phi\big(e^{\bu_i(q)/z} s_i\big) = e_i$,
where~$e_i$ denotes the $i$th standard basis of~$\cA_{B\times I_\phi}^{\oplus N_+}$.
\end{Proposition}
\begin{proof}Let $\Gamma^\phi_i(q)$ denote the Lefschetz thimble \eqref{eq:Lefschetz} of $F_q$
associated with the critical point $\crit_i(q)$ and phase~$\phi$.
By Lemma~\ref{lem:real_Lefschetz}, $\Gamma_1^\pi(q_0)
=\Gamma_\R(q_0)$.
Since there are no critical points of~$F_{q_0}$ on~$\Gamma_\R(q_0)$
other than the conifold point $\crit_\R(q_0)$,
$\Gamma^{\pi+\phi}_1(q)$ varies continuously
in a neighbourhood of $(q,\phi) = (q_0,0)$.
Let $\alpha_0\in (0,\pi/2)$ be such that $\Gamma_1^{\pi+\phi}(q)$
depends continuously on $(q,\phi)$ as $(q,\phi)$ varies
in a neighbourhood of $\{q_0\} \times (-\alpha_0,\alpha_0)$
in $\cVss_+ \times \R$.

Choose an admissible phase $\phi \in (-\alpha_0,\alpha_0)$
for $\{\bu_1(q_0),\dots,\bu_{N_+}(q_0)\}$ and let
$\Phi^+_\phi$, $\tPhi^+_\phi$ be the associated analytic lifts
over a neighbourhood $B\times I_\phi$ of $\big(q_0,\big(0,e^{\iu\phi}\big)\big)$
as above.
Let $\varphi_i$ denote a section of
$\pi^*\mir_+^*\QDMan(\frX_+)|_{B\times I_\phi}$
such that $\tPhi_\phi^+(\varphi_i) = e_i$.
Take a local section $\Omega = \Omega_{q,z}$ of $\Bri(F)$ near
$q=q_0$.
By the definition~\eqref{eq:lift_by_oscillatory} of $\Phi_\phi^+$,
for $q\in B$ and $z\in\C^\times$ with
$|\arg(z)-\phi|<\pi/2$,
we have\footnote{Recall from Proposition \ref{prop:analytic_lift_Bri}
that $B$ is chosen sufficiently small so that $\Gamma_i^{\phi+\pi}(q)$ deforms continuously as $q$ varies in $B$.}
\begin{align*}
(2\pi z)^{-n/2} e^{\bu_i(q)/z}
\int_{\Gamma_i^{\phi+\pi}(q)} e^{-F_q/z} \Omega_{q,-z} &=
\Phi_{\phi+\pi}^i(\Omega)(q,-z) \\
& = P_\std(\Phi_{\phi+\pi}(\Omega), e_i)(q,z) \\
& = P_\std\big(\tPhi_{\phi+\pi}(\Mir_+(\Omega)),
\tPhi_{\phi}(\varphi_i)\big)(q,z) \\
&= P(\Mir_+(\Omega),\varphi_i)(q,z),
\end{align*}
where $P_\std$ is the diagonal pairing as in
Proposition \ref{prop:formal_decomp_Bri}
and we used Remark \ref{rem:pairing_analytic_lift} in the last step.
By Theorem~\ref{thm:main_intpaper}(1), the left-hand side
of the above equation equals
\[
e^{\bu_i(q)/z} P(\Mir_+(\Omega), \frs_{V_i}(\mir_+(q),z))
\qquad
\text{with $V_i = V\big(\Gamma^{\phi+\pi}_i(q)\big)$}.
\]
Since the above equation holds for all $\Omega$, we have $\varphi_i = e^{\bu_i(q)/z}\frs_{V_i}(\mir_+(q),z)$.

We now set $i=1$. By the choice of $\alpha_0$ and $\phi$, the cycle $\Gamma_1^{\phi+\pi}(q_0)$ (and hence
$\Gamma_1^{\phi+\pi}(q)$ with $q\in B$) is a continuous deformation of $\Gamma_1^\pi(q_0) = \Gamma_\R(q_0)$.
Thus we get $V_1 = V_1(\Gamma_\R(q_0)) = \cO$ by Theorem \ref{thm:main_intpaper}(2).
\end{proof}

\begin{Remark}\label{rem:unambiguous_component} Recall from Propositions \ref{prop:formal_decomposition} and~\ref{prop:formal_decomp_Bri} that the formal decomposition and its analytic lift $\Phi^+_\phi$ are ambiguous up to multiplication by $\diag(\pm 1,\dots,\pm 1)$, and this ambiguity is fixed once we give a local trivialization of the $\mu_2$-local system~$\ori$.
We note that there is a~standard trivialization of $\ori$ at the
conifold point $\crit_\R$ since the Hessian of $F_q$
at $\crit_\R$ is positive-definite; hence the component of the
analytic lift $\Phi^+_\phi$ corresponding to $\crit_\R(q)$
(which is the first component) is unambiguous.
\end{Remark}

\begin{Remark} Proposition \ref{prop:lift_str_sheaf_+}
says that $\frs_\cO$ is characterized by the exponential asymptotics
$\frs_\cO \sim e^{-\bu_1(q)/z} \Psi_\R(q)$ as $z\to 0$
along the sector $\arg z \in (-\frac{\pi}{2}-\alpha_0,
\frac{\pi}{2}+\alpha_0)$
when $q$ lies in a neighbourhood of the positive real locus, where
$\Psi_\R(q)$ is the normalized idempotent
(see Remark~\ref{rem:normalized_idempotents})
corresponding to the conifold point $\crit_\R(q)$.
In terms of the marked reflection system in Section~\ref{subsec:MRS},
it also says that $\hGamma_{\frX_+} \cup (2\pi\iu)^{\deg_0/2} \inv^*
\tch(V_i^+)$, $1\le i\le N_+$ give the asymptotic basis at~$q_0$ and~$\phi$.
\end{Remark}

\begin{Remark} At various places in Section~\ref{sec:functoriality}, we work around
for the fact that $u_1(q)+\R_{\ge 0}$ may contain
other critical values; the argument would become much
simpler if otherwise.
\end{Remark}

\subsubsection{The structure sheaf of $\frX_-$}
\label{subsubsec:str_sheaf_-}
We repeat the same discussion for $\frX_-$.
The difference is that we consider the analytic lift
over a region $\cV_-$ which protrudes from the small
quantum cohomology locus $\cMsm$ of $\frX_-$.

Let $\cMsmtimes:=\LL'^\star \otimes \C^\times$,
$\cYsmtimes:=(\C^\times)^{S_-}$ denote the open dense tori
in $\cMsm$ and $\cYsm$ respectively and set:
\begin{gather*}
\cMsm_\R := \LL'^\star \otimes \R_{>0} \subset
\cMsmtimes, \\
\cYsm_\R := (\R_{>0})^{S_-} \subset
\cYsmtimes, \\
\Gamma_\R = \Gamma_\R(q) := \text{the fibre of $\cYsm_\R \to \cMsm_\R$
at $q\in \cMsm_\R$ (as defined before)},
\end{gather*}
where $\LL'=\Ker\big(\Z^{S_-}\to \bN\big)$
is the lattice as in Section~\ref{subsubsec:analytic_lift}.
For $q\in \cMsm_\R$, $\Gamma_\R(q)$ is
the Lefschetz thimble of $F_q$ associated
with the conifold point $\crit_\R(q)$ and the phase $\pi$
as in Lemma~\ref{lem:real_Lefschetz}.
The analytified Brieskorn module $\Brian(F)_\bSigma$
is defined over an analytic open neighbourhood $\cV_-$ of
$0_-:=0_{\bSigma_-}$
in $\cM$ (see Definition~\ref{def:analytic_Bri}).
By the non-equivariant limit of Theorem~\ref{thm:an_mirror_isom},
by shrinking $\cV_-$ if necessary, we have a mirror map
$\mir_- \colon \cV_- \to
[\cM_{\rm A}(\frX_-)/\Pic^\st(\frX_-)]$ and a~mirror isomorphism
\begin{equation}\label{eq:Mir_-_completed}
\Mir_- \colon \ \Brian(F)_\bSigma \cong
\mir_-^*\overline{\QDMan}(\frX_-).
\end{equation}
By Proposition~\ref{prop:Bri_weak_Fano_cor},
by shrinking $\cV_-$ further if necessary,
this mirror isomorphism can be lifted to a fully analytic mirror isomorphism
over $\cVsm_-=\cV_-\cap \cMsm$
\begin{equation}\label{eq:Mir_-}
\Mir_- \colon \ \Brism(F)|_{\cVsm_-} \cong \mir_-^*\QDMan(\frX_-).
\end{equation}
\begin{Theorem}[{\cite[Theorems 4.11 and 4.14, Section~4.3.1]{Iritani:Integral}}]
\label{thm:main_intpaper_2}
The conclusions of Theorem $\ref{thm:main_intpaper}$ hold true
even when $\frX_+$, $\cV_+$, $\Mir_+$, $\mir_+$, $\cM_\R$
there are replaced with
$\frX_-$, $\cV_-$, $\Mir_-$, $\mir_-$, $\cMsm_\R$ respectively.
\end{Theorem}

By the construction of $\Brian(F)_{\bSigma_-}$
in Section~\ref{subsec:analytified_Bri},
we have an open neighbourhood $\cB_-$
of $\tzero_- :=\tzero_{\bSigma_-} \in \cY$ such that
the family
\[
\Cr_- := \cB_- \cap \tpr^{-1}(\cV_-\times \{0\}) \xrightarrow{\pr} \cV_-
\]
of relative critical points (of $F$) in $\cB_-$
is a finite flat morphism of degree
$N_- = \dim H^*_\CR(\frX_-)$ (see also the discussion
in Section~\ref{subsubsec:formal_decomp_Bri}),
where $\tpr$ is the map in \eqref{eq:tpr}.
Along the small quantum cohomology locus $\cVsm_- =\cV_-\cap \cMsm$,
this family contains all relative critical points, i.e., $\cB_- \cap \tpr^{-1}
(\cVsm_-\times \{0\}) = \tpr^{-1}(\cVsm_-\times \{0\})$
(since the Brieskorn module $\Brism(F)$ has the rank $N_-$
by Proposition \ref{prop:Bri_weak_Fano}).
We define
\begin{gather*}
\cVss_- := \{q\in \cV_- \colon
\text{the fibre of $\Cr_- \to \cV_-$ at $q\in \cV_-$ is reduced} \},\\
\cVsmss_- := \cVss_- \cap \cMsm, \\
\cVsmss_{-,\R} := \cVsmss_- \cap \cMsm_\R,
\end{gather*}
where $\cVss_-$ is open dense in $\cV_-$, $\cVsmss_-$ is open dense in
$\cVsm_-$ and $\cVsmss_{-,\R}$ is open dense in $\cVsm_- \cap \cMsm_\R$.

Choose $q_0 \in \cVsmss_{-,\R}$.
Let $\crit_1(q),\dots,\crit_{N_-}(q)$ be all branches of critical points
of~$F_q$ contained in~$\cB_-$ and defined in a neighbourhood of
$q=q_0$ in $\cVss_-$. We may assume that
$\crit_1(q_0) = \crit_\R(q_0)$ and
write $\bu_i(q) = F_q(\crit_i(q))$ as before.
By combining the mirror isomorphism~\eqref{eq:Mir_-_completed}
and the formal decomposition of the analytified Brieskorn module
in Proposition~\ref{prop:formal_decomp_Bri}, we obtain
a~formal decomposition
\[
\hPhi \colon \ \mir_-^*\overline{\QDMan}(\frX_-) \cong
\bigoplus_{i=1}^{N_-}
\left(\cOan[\![z]\!], d+d(\bu_i/z)\right)
\]
over a neighbourhood of $q_0$ in $\cVss_-$.
By Proposition~\ref{prop:Hukuhara-Turrittin},
for an admissible phase $\phi$ for $\{\bu_1(q_0),\dots,\allowbreak \bu_{N_-}(q_0)\}$,
we have a connected open neighbourhood $B$ of $q_0$ in $\cVss_-$,
a sector
$I_\phi = \big\{\big(r,e^{\iu\theta}\big)\colon\allowbreak |\theta -\phi|<\frac{\pi}{2}+\epsilon\big\}$
(with small $\epsilon>0$) and an analytic lift of $\hPhi$:
\[
\tPhi^-_\phi \colon \ \pi^*\mir_-^*\QDMan(\frX_-) \Big|_{B\times I_\phi}
\cong \bigoplus_{i=1}^{N_-}\left
(\cA_{B\times I_\phi}, d + d(\bu_i/z) \right),
\]
where $\pi \colon \cVss_- \times \tCC \to \cVss_- \times \C$
is the oriented real blow-up.
By the uniqueness of the analytic lift,
this coincides with the analytic lift of the Brieskorn module
from Proposition~\ref{prop:analytic_lift_Bri}
over the small quantum cohomology locus $B\cap \cMsm$, via
the mirror isomorphism \eqref{eq:Mir_-}.
Each component of~$\hPhi$,~$\tPhi^-_\phi$ is ambiguous up to sign,
but recall from Remark~\ref{rem:unambiguous_component}
that the sign of the \emph{first} component
(corresponding to the conifold point~$\crit_\R(q)$)
is determined canonically.
We have the following result for $\frX_-$
parallel to Proposition~\ref{prop:lift_str_sheaf_+}:

\begin{Proposition}\label{prop:lift_str_sheaf_-}
For $q_0\in \cVsmss_{-,\R}$, there exists
$\alpha_0\in (0,\pi/2)$ such that the following holds.
For every admissible phase $\phi \in (-\alpha_0,\alpha_0)$
for $\{\bu_1(q_0),\dots,\bu_{N_-}(q_0)\}$,
there exists a basis $\{V_i \}_{i=1}^{N_-}$ of
$K(\frX_-)$ with $V_1$ being the structure sheaf $\cO$ of $\frX_-$,
such that the corresponding
flat section $s_i = \frs_{V_i}(\mir_-(q),z)$ satisfies
$\tPhi^-_\phi(e^{\bu_i(q)/z} s_i) = e_i$.
\end{Proposition}
\begin{proof}
The same argument as Proposition \ref{prop:lift_str_sheaf_-}
(using Theorem \ref{thm:main_intpaper_2} in place of
Theorem \ref{thm:main_intpaper})
shows that there exist $V_i \in K(\frX_-)$, $i=1,\dots,N_-$
with $V_1 = \cO$ such that
$\tPhi^-_\phi\big(e^{\bu_i(q)/z}s_i\big) = e_i$ over $B \cap \cMsm$,
where $s_i = \frs_{V_i}(\mir_-(q),z)$.
It follows from the flatness of $s_i$ and $e^{-\bu_i(q)/z} e_i$
that $\tPhi^-_\phi\big(e^{\bu_i(q)/z} s_i\big) =e_i$ holds
over the whole~$B$.
\end{proof}

\subsection{Inclusion of the local systems of Lefschetz thimbles}\label{subsec:inclusion}
The Brieskorn module over the small quantum cohomology locus
is underlain by a local system of Lefschetz thimbles. In this section,
we observe an inclusion of the local system over $\cMsm$
in a neighbourhood of $0_- =0_{\bSigma_-}$ (mirror to $\frX_-$)
to the local system over $\cM$ (mirror to $\frX_+$)
under a slide in the `positive real' direction.
The inclusion shall be identified with the pull-back $\varphi^* \colon
K(\frX_-) \to K(\frX_+)$ in $K$-theory in Section~\ref{subsec:functoriality}.

\subsubsection{Convergent and divergent critical branches}
\label{subsubsec:conv_div}
Let $\cV_\pm$ be (sufficiently small) analytic open neighbourhoods
of $0_\pm\in \cM$ as in
the previous Section~\ref{subsec:identifying_O}.
Recall the $\C^\times$-action on $\pr\colon \cY\to \cM$
generated by the Euler vector field
considered in Sections~\ref{subsec:decomp_Bri}--\ref{subsec:comparison_QDM}.
As discussed there, we may assume that
$\cV_\pm$ is $\C^\times$-invariant, because the mirror map $\mir_\pm$,
the mirror isomorphism $\Mir_\pm$ and the analytified Brieskorn
module can be extended to the orbit $\C^\times \cV_\pm$
($\cV_\pm$ is the intersection of
$\cU_\pm\subset \cM\times \Lie \T$ in Section~\ref{subsec:comparison_QDM}
with $\cM\times \{0\}$).
Let $\Cr_\pm \to \cV_\pm$ denote the (finite, flat) family of relative critical points
over $\cV_\pm$:
\begin{gather*}
\Cr_+ = \left\{p \in \pr^{-1}(\cV_+)
\colon x_i\parfrac{F}{x_i}(p) = 0 \, (\forall\, i) \right\}
= \tpr^{-1}(\cV_+ \times \{0\}), \\
\Cr_- = \left\{p \in \cB_- \cap \pr^{-1}(\cV_-) \colon
x_i \parfrac{F}{x_i}(p) = 0 \, (\forall\, i) \right\}
= \tpr^{-1}(\cV_-\times \{0\}) \cap \cB_-,
\end{gather*}
where $\cB_-$ is the subset of $\cY$ appearing in the construction of
$\Brian(F)_{\bSigma_-}$ in Section~\ref{subsec:analytified_Bri}
and $\tpr$ is the map in~\eqref{eq:tpr}.
Since all the
relative critical points over $\cV_+$ are contained in $\cB_+$,
we do not need to take the intersection with $\cB_+$ in the first formula.
By Lemma \ref{lem:subcover}, we have an open subset
$\cV_0 = \cU_0 \cap (\cM\times \{0\})
\subset \cV_+ \cap \cV_-$ containing $\cC \setminus \{0_+,0_-\}$
such that the ramified covering $\Cr_+|_{\cV_0}\to \cV_0$
decomposes as:
\[
\Cr_+|_{\cV_0} = \Cr_-|_{\cV_0} \sqcup \cD,
\]
where $\cD\subset \Cr_+$ is an open subset giving
a subcover of $\Cr_+|_{\cV_0}$.
By taking the $\C^\times$-orbit, we may assume
that $\cV_0$ is also $\C^\times$-invariant.
The subcover $\cD \to \cV_0$ consists
of branches of critical points that diverge at $0_-$.
We call critical points corresponding to $\cD$ \emph{divergent}
and those corresponding to $\Cr_-$ \emph{convergent}.
Among the critical points over $\cC$ described in
Proposition~\ref{prop:critv_on_curve}, those corresponding to
$\gamma\neq 0$ are divergent.

\subsubsection[Local co-ordinate system around $0_-=0_{\bSigma_-}$]{Local co-ordinate system around $\boldsymbol{0_-=0_{\bSigma_-}}$}\label{subsubsec:local_around_0-}

Recall from Proposition \ref{prop:chart_LG}
that the local chart of $\cM$ around $0_-$
is given by
\[
\cM_- = \big[\tcM_- /\Pic^\st(\frX_-)\big] \qquad
\text{with} \quad
\tcM_- = \Spec (\C[\Laa(\bSigma_-)_+])
\]
Set $R_- = R(\bSigma_-)$.
By the decomposition \eqref{eq:cone_monoid_decomp},
we have $\Laa(\bSigma_-)_+
\cong \Laa^{\bSigma_-}_+ \times (\Z_{\ge 0})^{S\setminus R_-}$.
Let $\Laa'(\bSigma_-)_+ \cong \Laa^{\bSigma_-}_+\times
(\Z_{\ge 0})^{S_-\setminus R_-}$ denote the monoid
that we obtain from $\Laa(\bSigma_-)_+$
by replacing~$S$ with~$S_-$.
Then we can decompose $\tcM_-$ as
\begin{equation}\label{eq:chart_two_factors}
\tcM_- =
\Spec (\C[\Laa'(\bSigma_-)_+]) \times \C^{\{\hb\}}.
\end{equation}
We denote by $(q,t) = (q, t_\hb)$ a point on $\tcM_-$,
where $q\in \Spec(\C[\Laa'(\bSigma_-)_+])$.
Note that the local chart of $\cMsm$ around $0_-$ is the
substack $\{t=0\}$ of $\cM_-$:
\[
\cMsm_- = \big[
\tcMsm_- /\Pic^\st(\frX_-)\big]
\qquad \text{with} \quad
\tcMsm_- = \Spec(\C[\Laa'(\bSigma_-)_+]).
\]
Recall that $\Pic^\st(\frX_-)$ acts on the chart $\tcM_-$~\eqref{eq:chart_two_factors} by the age pairing~\eqref{eq:Picst_action}; since $t_\hb = q^{-\bw}$
and $\bw \in \LL$, $\Pic^\st(\frX_-)$ acts trivially on
the last co-ordinate~$t_\hb$.

On the chart \eqref{eq:chart_two_factors},
the $\C^\times$-action generated by the Euler vector field
is given by
\begin{equation}\label{eq:Euler_action}
s\cdot (q, t) = \big(s\cdot q, s^{-J} t\big),
\end{equation}
where
$J= \big(\sum\limits_{b\in M_+} D_b\cdot \bw\big)-1$
(as in Proposition \ref{prop:critv_on_curve}).
Note that the $\C^\times$-weight of the co-ordinate~$q^\lambda$ with $\lambda \in \Laa'(\bSigma_-)_+$ is non-negative by Lemma~\ref{lem:non-negative_grading}.

\subsubsection[Sliding out $\cMsm$]{Sliding out $\boldsymbol{\cMsm}$}\label{subsubsec:slide}
Let $\cC\subset \cM$ be the toric curve
connecting $0_-$ and $0_+$ as in Section~\ref{subsec:crit_curve};
in the above chart $\tcM_-$, $\cC$~consists of points $(0,t)$.
By Proposition~\ref{prop:critv_on_curve},
we can choose a closed polydisc $\tD
\subset \{t_\hb=1\} \subset \tcM_-$ of radius $0<\varepsilon<1$
centred at $(0,1)\in \cC$ such that
\begin{itemize}\itemsep=0pt
\item $\tD=\tDsm \times \{1\}$ with
$\tDsm \subset \tcMsm_- = \Spec(\C[\Laa'(\bSigma_-)_+])$
a neighbourhood of $0_-$
given by
\[
\tDsm := \big\{q\colon \big|q^\lambda\big|\le \varepsilon, \, \forall\, \lambda \in
\Laa'(\bSigma_-)_+
\setminus \{0\} \big \};
\]
\item $\D:=\tD/\Pic^\st(\frX_-)$
is contained in $\cV_0$;
\item $\Dsm :=\tDsm/\Pic^\st(\frX_-)$ is contained in
$\cV_-$ (when regarded as a subset of $\cMsm_-$);
\item $F\big(\Cr_- \cap \pr^{-1}(\D)\big)
\subset B_{\rho_0}(0)$;
\item $F\big(\cD\cap \pr^{-1}(\D)\big)
\subset \bigcup\limits_{k=1}^{J} B_{\rho_1}\big(3\rho_0 e^{(2k-1)\pi\iu/J}\big)$,
\end{itemize}
where $\rho_0 := \frac{1}{3} K^{-1/J}$,
$\rho_1 := \min\big(1,3\sin\big(\frac{\pi}{2J}\big)\big) \rho_0$,
$B_\rho(z)$ denotes the open disc of radius~$\rho$
centred at $z\in \C$,
and $K$, $J$ are as in Proposition~\ref{prop:critv_on_curve}.
The last two conditions imply that convergent critical values over $\D$
are contained in $\{|\bu|<\rho_0\}$ and divergent
critical values over $\D$ are contained in
$\{|\bu| > 2\rho_0\}$
and away from the sector $-\frac{\pi}{2J}\le \arg(\bu)
\le \frac{\pi}{2J}$ with vertex at the origin,
see Fig.~\ref{fig:conv_div_critv}.
See also Fig.~\ref{fig:moduli_space}.

\begin{figure}[th!]\centering
\includegraphics[scale=0.92]{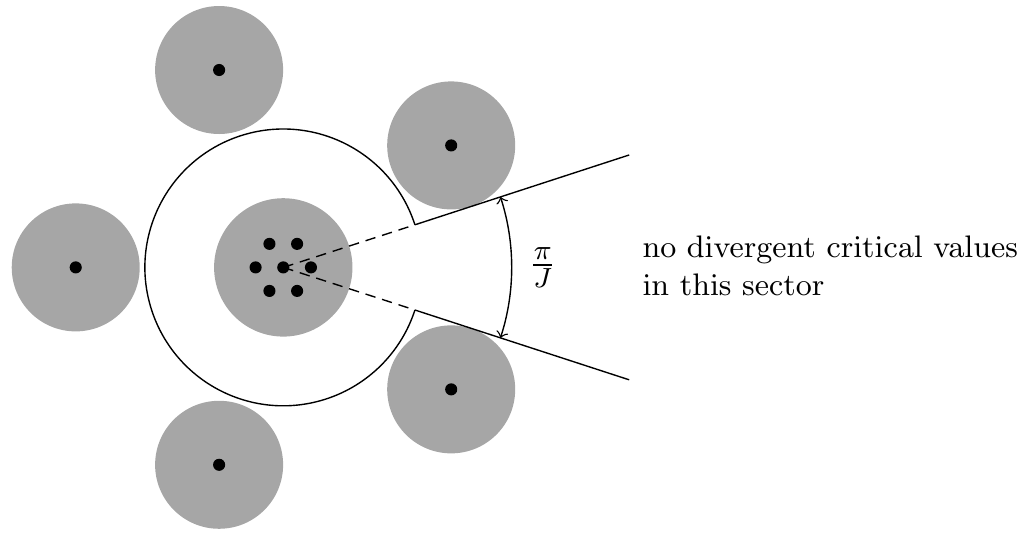}
\caption{Distribution of critical values over $\D$ for $J=5$:
convergent critical values are in the shaded disc of radius $\rho_0$
in the centre; divergent critical values are in the `satellite' discs.} \label{fig:conv_div_critv}
\end{figure}

\begin{figure}[th!]\centering
\includegraphics[scale=0.92]{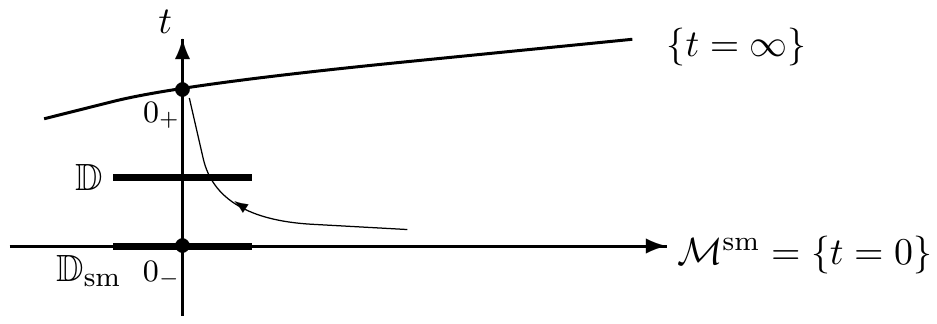}
\caption{The base space $\cM$ of the LG model and the discs $\D$, $\Dsm$:
the vertical $t$-axis is the curve $\cC$ connecting $0_+$ and $0_-$;
the thin curve denotes the negative Euler flow (the flow as $s\to 0$ in~\eqref{eq:Euler_action}).}\label{fig:moduli_space}
\end{figure}

\begin{Definition}\label{def:slide}
Define the \emph{sliding map} $\slide_t \colon \Dsm \to \cM_-$
with $0\le t\le 1$ by
$\slide_t(q) = (q, t)$.
Note that the map $\tilde{\slide}_t \colon
\tDsm \to \tcM_-$, $q\mapsto (q,t)$ between
the uniformizing charts is
$\Pic^\st(\frX_-)$-equivariant and thus descends to the map $\slide_t$.
Note also that $\slide_0 = \id_{\Dsm}$.
\end{Definition}
\begin{Lemma}\label{lem:slide}\quad
\begin{enumerate}\itemsep=0pt
\item[{\rm (1)}] For $0<t\le 1$ and $q\in \Dsm^\times :=
\Dsm\cap \cMsmtimes$, we have $\slide_t(q)\in \cMtimes$.
\item[{\rm (2)}] For $0<t\le 1$, we have $\Image(\slide_t) \subset
t^{-1/J} \cdot \D \subset \cV_0$, where $t^{-1/J}\cdot(-)$ denotes
the $\C^\times$-action generated by the Euler vector field.
\item[{\rm (3)}]
There exists a constant $\rho_2>0$ such that
convergent critical values over
$\Image(\slide_t)$ are contained in $B_{\rho_2}(0) \cap t^{-1/J}\cdot B_{\rho_0}(0)$ and
divergent critical values over $\Image(\slide_t)$ are contained in
$\bigcup\limits_{k=1}^J t^{-1/J} \cdot
B_{\rho_1}\big(3\rho_0 e^{(2k-1)\pi\iu/J}\big)$.
\end{enumerate}
\end{Lemma}

\begin{proof}
Part (1) is obvious. By \eqref{eq:Euler_action},
$t^{1/J}\cdot \Image(\slide_t)$
consists of points $\big(t^{1/J}\cdot q, 1\big)$ with
$q\in \tDsm$. Since $0<t\le 1$ and the $\C^\times$-weights
on $\C[\Laa'(\bSigma_-)_+]$ are non-negative,
$\big|\big(t^{1/J}\cdot q\big)^\lambda\big|\le \big|q^\lambda\big|\le \varepsilon$ for $q\in \tDsm$
and $\lambda \in \Laa'(\bSigma_-)_+$.
Hence $t^{1/J} \cdot \Image(\slide_t) \subset \D$.
Part (2) follows.
Part (2) and the fact that the action of $t^{-1/J}$
scales the critical values by $t^{-1/J}$ imply that
convergent critical values over $\Image(\slide_t)$ are contained
in $t^{-1/J} \cdot B_{\rho_0}(0)$ and
divergent ones over $\Image(\slide_t)$ are contained in
$\bigcup\limits_{k=1}^{J} t^{-1/J} \cdot
B_{\rho_1}\big(3\rho_0 e^{(2k-1)\pi\iu/J}\big)$.
Because the convergent critical branches form a proper family
$\Cr_- \to \cV_-$, the corresponding critical values are
uniformly bounded in a neighbourhood of $\Dsm=\Image(\slide_0)$.
Part~(3) follows.
\end{proof}

\subsubsection{Kouchnirenko's condition and a local system of Lefschetz thimbles} \label{subsubsec:Kouchnirenko_Morse}
Let $(q,t) = (q,t_\hb)$ be the local co-ordinates
on $\cM_-$ in Section~\ref{subsubsec:local_around_0-}.
By abuse of notation, we mean by~$q$ (resp.~$(q,t)$) either a point on
the uniformizing chart~$\tcMsm_-$ (resp.~$\tcM_-$)
or its image in~$\cMsm_-$ (resp.~$\cM_-$)
depending on the context.
We write $\cY_{q,t} := \pr^{-1}(q,t) \subset \cY$,
$\cYsm_q := \pr^{-1}(q) \subset \cYsm$ for the fibres
of the LG model; note that $\cYsm_q = \cY_{q,0}$.
The LG potential
$F_{q,t}=F|_{\cY_{q,t}}$ with $q \in \cMsmtimes$
is a Laurent polynomial of the form (cf.~\eqref{eq:LG_pot_co-ordinates})
\begin{equation}
\label{eq:LG_pot_q_t}
F_{q,t} = \sum_{b\in S} u_b = \left(\sum_{b\in S_-} q^{\ell_b} x^b
\right)
+ t q^{\hat\ell} x^{\hb}.
\end{equation}

Following Kouchnirenko \cite[D\'{e}finition 1.19]{Kouchnirenko:Newton},
we make the following definition:
\begin{Definition}[\cite{Kouchnirenko:Newton}]
Let $T\subset \bN$ be a finite set and let
$f(x) = \sum\limits_{b\in T}c_b x^b$ be a Laurent polynomial function
on $\Hom(\bN,\C^\times)$. Suppose that $c_b \neq 0$ for all
$b\in T$ and that the Newton polytope $\Delta$ of~$f$, i.e., the
convex hull of $\{\overline{b}\colon b\in T\}$
in $\bN_\R$, contains the origin in its interior.
We say that $f$ is (Newton) \emph{non-degenerate} if for every proper face
$\diamondsuit$ of $\Delta$ (of any dimension),
$f_\diamondsuit(x) := \sum\limits_{b \in T \cap \diamondsuit} c_b x^b$ has no
critical points on $\Hom(\bN,\C^\times)$.
\end{Definition}

We are interested in the case where $T$ is $S$ or $S_-$, and
$f$ is $F_{q,t}$ with $(q,t)\in \cMtimes$ or $F_{q,0}$ with
$q\in \cMsmtimes$. The compactness of $\frX_{\pm}$
implies that the Newton polytope contains the origin in its interior
in these cases.
We define
\begin{gather*}
\cMsmnd := \{q \in \cMsmtimes\colon
\text{$F_{q,0}$ is Newton non-degenerate} \}, \\
\cMnd := \{(q,t) \in \cMtimes\colon
\text{$F_{q,t}$ is Newton non-degenerate} \}.
\end{gather*}

\begin{Proposition}\label{prop:non_deg_locus}\quad
\begin{enumerate}\itemsep=0pt
\item[{\rm (1)}] $\cMsmnd$ is Zariski-open and dense in $\cMsmtimes$.
\item[{\rm (2)}] There exists a neighbourhood $B$ of $0_-$ in $\cMsm$ such that
$\cMsmnd$ contains $B \cap \cMsmtimes$.
\item[{\rm (3)}] The same conclusions as {\rm (1)}, {\rm (2)}
hold for $\cMnd$ when we replace $\cMsmtimes$ with $\cMtimes$
and $0_-$ with $0_+$.
\end{enumerate}
\end{Proposition}
\begin{proof}
Part (1) is due to Kouchnirenko
\cite[Th\'eor\`eme 6.1]{Kouchnirenko:Newton}.
Part (2) follows from the proof of
\cite[Lemma 3.8]{Iritani:Integral}; there we stated and
proved a similar result using a particular co-ordinate system
$\C^r \to \cMsm_-$,
but the argument works verbatim for the possibly singular base
$\cMsm_-$.
\end{proof}

\begin{Remark}
\label{rem:Vpm_contained_in_nd_locus}
In view of Proposition \ref{prop:non_deg_locus}(2),
after shrinking the open sets $\cV_\pm$ in Section~\ref{subsubsec:conv_div}
if necessary, we may assume that
$\cV_+ \cap \cMtimes \subset \cMnd$ and
$\cV_- \cap \cMsmtimes
\subset \cMsmnd$.
This is possible because
$\cV_\pm$ was given as the $\C^\times$-orbit of an arbitrarily
small neighbourhood of $0_\pm$, and the non-degenerate loci
$\cMsmnd$, $\cMnd$ are $\C^\times$-invariant.
\end{Remark}

We define a function $H(x)$ of $(x_1,\dots,x_n)\in (\C^\times)^n$ as
\begin{equation}\label{eq:H}
H(x) :=\sqrt{\sum_{b\in S_-} |x^b|^2}.
\end{equation}
Since the convex hull $\Delta_-$ of $\{\overline{b}\colon b\in S_-\}$
contains the origin in its interior,
$H(x)$ is proper and bounded from below.
By the splitting $\varsigma$ chosen at the end of Section~\ref{subsec:functoriality_nota},
we can regard $x_1,\dots,x_n$ and $H(x)$
as functions on the preimage of
$\{(q,t)\colon q \in \cMsmtimes\}\subset \cM_-$ under
$\pr \colon \cY \to \cM$.

\begin{Proposition}\label{prop:strong_tame}\quad
\begin{enumerate}\itemsep=0pt
\item[{\rm (1)}] For every compact set $K\subset \cMsmnd$, there exist a compact
set $B \subset \pr^{-1}(K)$ and $\epsilon>0$ such that
\[
\|dF_{q,0}(x)\|\ge \epsilon H(x)
\qquad
\text{for all $q\in K$ and $x \in \cYsm_q\setminus B$},
\]
where $\|\cdot \|$ is the norm with respect to the
K\"ahler metric $\frac{\iu}{2}\sum\limits_{j=1}^n d\log x_j \wedge
d\overline{\log x_j}$, i.e., $\|dF_{q,0}(x)\| = \big(\sum\limits_{j=1}^n
|\partial F_{q,0}/\partial \log x_j|^2\big)^{1/2}$.
\item[{\rm (2)}] The same estimate holds for $\|dF_{q,t}\|$
by replacing $\cMsmnd$ with $\cMnd$, $\cYsm$ with $\cY$,
and $H(x)$ with $\big(\sum\limits_{b\in S} |x^b|^2\big)^{1/2}$.
\end{enumerate}
\end{Proposition}
\begin{proof}
This is a refinement of~\cite[Lemma 3.11]{Iritani:Integral},
which says that $\|dF_{q,0}(x)\|$ is proper
on $\pr^{-1}(K)$. In fact, a slight modification of the argument
there yields a proof of the proposition.

It suffices to show that there exists $\epsilon>0$ such that
$\big\{x\in \pr^{-1}(K)\colon \|dF_{q,0}(x)\|\le \epsilon H(x)\big\}$
is compact.
Suppose on the contrary that $\big\{x\in \pr^{-1}(K)\colon
\|dF_{q,0}(x)\| \le H(x)/k\big\}$ is non-compact for all
$k\ge 1$. Then we can find $q(k) \in K$ and
$x(k) \in \cYsm_{q(k)}$ such that $\|dF_{q(k),0}(x(k))\| \le H(x(k))/k$
and $H(x(k)) \ge k$.
Passing to a subsequence
we may assume that $q(k)$ converges in $K$;
we may also assume that we can label elements of $S_-$ as
$\{b(1),b(2),\dots,b(m)\}$ so that $\big|x(k)^{b(1)}\big| \ge \big|x(k)^{b(2)}\big| \ge
\cdots \ge \big|x(k)^{b(m)}\big|$ holds for all $k$.
Since $H(x(k)) \le \sqrt{m} \big|x(k)^{b(1)}\big|$, we have
$\lim\limits_{k\to \infty} \big|x(k)^{b(1)}\big| =\infty$ and
\[
\lim_{k\to \infty} \frac{\|dF_{q(k),0}(x(k))\|}{|x(k)^{b(1)}|} =0.
\]
Now the argument after the second displayed equation in
\cite[Section~A.2]{Iritani:Integral} yields a contradiction.
The proposition follows.
\end{proof}

\begin{Remark}
Proposition \ref{prop:strong_tame} shows that
the relative critical set of $F$ is proper over the non-degenerate
loci $\cMsmnd$, $\cMnd$.
\end{Remark}

\begin{Corollary}\label{cor:locally_trivial}\quad
\begin{enumerate}\itemsep=0pt
\item[{\rm (1)}] The family $\big\{F_{q,0}^{-1}(\bu) \subset \cYsm_q\big\}_{q,\bu}$ of affine varieties
is a locally trivial family of $C^\infty$ manifolds
over $\{(q,\bu)\colon q\in \cMsmnd, \text{$\bu$ is a regular value of $F_{q,0}$}\}$.
\item[{\rm (2)}] The same conclusion holds for $F_{q,t}^{-1}(\bu)\subset \cY_{q,t}$
when we replace $\cMsmnd$ with $\cMnd$.
\end{enumerate}
\end{Corollary}

\begin{proof}
Take $(q_0,\bu_0)$ such that $q_0\in\cMsmnd$ and
$\bu_0$ is a regular value of $F_{q_0,0}$.
Choose a~sufficiently small co-ordinate neighbourhood
$\big(B; q^1,\dots,q^r, \bu\big)$ of $(q_0,\bu_0)$
which does not intersect the discriminant locus.
The ambient family $\bigcup\limits_{(q,\bu)\in B}
\cYsm_q$ is trivial over $B$, and is identified with
$B \times \Hom(\bN,\C^\times)$ through the
co-ordinates $x_1,\dots,x_n$.
It suffices to shows that the co-ordinate
vector fields\footnote{Here we identify (real) vector fields of
a complex manifold with $(1,0)$ vector fields.}
$\alpha \partial/\partial q^i$, $\alpha \partial/\partial \bu$
with $\alpha \in \C$ on $B$
can be lifted to integrable vector fields tangent to the family
$\bigcup\limits_{(q,u)\in B} F_{q,0}^{-1}(\bu)$.
Lifts of $\alpha \partial/\partial q^i$, $\alpha \partial/\partial \bu$ are
given under the trivialization $\bigcup\limits_{(q,\bu) \in B}
\cYsm_q \cong B \times \Hom(\bN,\C^\times)$ by
\[
\left(\alpha \parfrac{}{q^i},
-\alpha \parfrac{F_{q,0}}{q^i} \frac{\grad F_{q,0}}{\|dF_{q,0}\|^2}
\right), \qquad
\left(\alpha \parfrac{}{\bu},
\alpha\frac{\grad F_{q,0} }{\|dF_{q,0}\|^2}
\right)
\]
where
\[
\grad F_{q,0} := \sum_{j=1}^n \overline{\parfrac{F_{q,0}}{\log x_j}}
\parfrac{}{\log x_j}.
\]
Since the potential $F_{q,0}$ is of the form \eqref{eq:LG_pot_q_t},
we have $|\partial F_{q,0}/\partial q^i|\le C \cdot H(x)$ for some
constant $C>0$ over $B$.
Thus the estimate in Proposition~\ref{prop:strong_tame}
shows that these lifts are bounded on
the family $\bigcup\limits_{(q,\bu)\in B} F_{q,0}^{-1}(\bu)$.
Therefore the flows of these vector fields exist as long as
the corresponding flows on the base $B$ exists.
\end{proof}

Proposition~\ref{prop:strong_tame} implies that the improper
function $\Re(F_{q,0}(x))$
satisfies the so-called Palais--Smale condition\footnote
{$\|d F_{q,0}(x)\|$ is bounded away
from zero outside a neighbourhood of the critical set.} when
$q\in \cMsmnd$,
and hence the usual Morse theory can be applied to it.
It follows (see \cite[Section~3.3.1]{Iritani:Integral})
that the relative homology group
$H_n\big(\cYsm_q, \{x \colon \Re(F_{q,0}(x)) \ge M\};\Z\big)$
is a free $\Z$-module of rank $\dim H^*_\CR(\frX_-)$
when $M$ is large enough so that all critical values of $F_{q,0}$ are
contained in $\{\Re(z)< M\}$.
This group is independent of the choice of
sufficiently large $M$, and we denote it by
$H_n\big(\cYsm_q,\{x\colon \Re(F_{q,0}(x)) \gg 0\};\Z\big)$.
By the local triviality in Corollary~\ref{cor:locally_trivial},
we have
\begin{align*}
\Lefsm_q := H_n\big(\cYsm_q, \{x\colon \Re( F_{q,0}(x)) \gg 0 \} ; \Z\big)
\cong H_n\big(\cYsm_q, F_{q,0}^{-1}(\bu);\Z\big)
\end{align*}
for any $\bu>0$ such that all critical values of $F_{q,0}$
are contained in $\{z\colon \Re(z)<\bu\}$, and this forms a local
system over $\cMsmnd$.
Similarly, the relative homology groups
\begin{equation}\label{eq:Lef}
\Lef_{q,t} :=
H_n\big(\cY_{q,t}, \{x\colon \Re(F_{q,t}(x)) \gg 0 \} ; \Z\big)
\cong H_n\big(\cY_{q,t}, F_{q,t}^{-1}(\bu);\Z\big)
\end{equation}
form a local system of rank $\dim H^*_\CR(\frX_+)$
over $\cMnd$, where $\bu>0$ is such that all critical values of $F_{q,t}$
are contained in $\{z\colon \Re(z)<\bu\}$.
These two local systems have different ranks and we will relate them below.

\subsubsection{Inclusion of the local systems: statement and proof}
Let $\slide_t$ be the sliding map in Definition~\ref{def:slide}
and set $\Dsm^\times := \Dsm \cap \cMsmtimes$.
Let $\slide_{>0}$ denote the map $(0,1] \times \Dsm^\times
\to \cMnd$ given by $\slide_{>0}(t,q) =\slide_t(q)$.
Since $(0,1]$ is contractible, we have
\[
\slide_{>0}^{-1} \Lef \cong \ttp^{-1}\big(\big(\slide_t^{-1}\Lef\big)|_{\Dsm^\times}\big)
\qquad \forall\, t\in (0,1],
\]
where $\ttp \colon (0,1]\times\Dsm^\times \to \Dsm^\times$
is the projection to the second factor.
Let $\ttj \colon (0,1]\times \Dsm^\times \rightarrow
[0,1]\times \Dsm^\times$ and
$\tti \colon \Dsm^\times \cong \{0\} \times \Dsm^\times
\to [0,1]\times \Dsm^\times$ denote the inclusions.
Then we have
\[
\tti^{-1} \ttj_* \slide_{>0}^{-1} \Lef \cong
\tti^{-1} \ttj_* \ttp^{-1} \big(\big(\slide_t^{-1}\Lef\big)|_{\Dsm^\times}\big)
\cong
\big(\slide_t^{-1} \Lef\big)|_{\Dsm^\times}.
\]

\begin{Theorem}\label{thm:inclusion}We have an inclusion of the local systems:
\[
\iota \colon \ \Lefsm|_{\Dsm^\times} \to
\tti^{-1} \ttj_* \slide_{>0}^{-1} \Lef \cong
\big(\slide_t^{-1}\Lef\big) |_{\Dsm^\times}, \qquad t\in (0,1].
\]
The map $\iota$ maps the positive real
Lefschetz thimble $($introduced in Section~$\ref{subsec:identifying_O})$
to the positive real one, i.e., $\iota(\Gamma_\R(q))
= \Gamma_\R(\slide_t(q))$ when $q \in \Dsm^\times\cap\cMsm_\R$.
\end{Theorem}

\begin{Remark}By Remark \ref{rem:Vpm_contained_in_nd_locus},
we may assume that $\cV_+\cap \cMtimes
\subset \cMnd$ and $\cV_- \cap \cMsmtimes \subset \cMsmnd$.
Then by the choice of $\Dsm$, $\D$ in Section~\ref{subsubsec:slide}
and Lemma~\ref{lem:slide},
we have that $\Dsm^\times \subset \cMsmnd$ and
$\slide_t(\Dsm^\times) \subset \cV_0 \cap \cMtimes
\subset \cV_+ \cap \cMtimes \subset \cMnd$ for $t>0$.
Hence $\Lefsm$ and $\Lef$ are local systems over
$\Dsm^\times$ and $\Image(\slide_t)$ (with $0<t\le 1$) respectively,
and the statement of Theorem \ref{thm:inclusion} makes sense.
\end{Remark}

For $(q,t) \in \cM_-$ and $\eta>0$, we set
\[
A_{q,t}(\eta) := \cY_{q,t} \cap \{H(x) \le \eta\},
\]
where $H(x)$ is as in \eqref{eq:H}.
When $\eta> \min_{x\in \cY_{q,t}} H(x)$,
$A_{q,t}(\eta)$ is a compact region
such that the inclusion $A_{q,t}(\eta) \hookrightarrow \cY_{q,t}$ is
a homotopy equivalence.
Indeed, via the $\Hom\big(\bN,S^1\big)$-action, we have
\[
A_{q,t}(\eta) \cong
\Hom\big(\bN,S^1\big) \times \big\{(x_1,\dots,x_n) \in
(\R_{>0})^n\colon H(x) \le \eta\big\}
\]
and the second factor is contractible since $H|_{(\R_{>0})^n}$
is strictly convex.
For any $(q,\bu)\in \cMsmtimes\times \C$, the real algebraic function~$H(x)^2$ restricted to $F_{q,0}^{-1}(\bu)$ has finitely many critical values
by \cite[Corollary~2.8]{Milnor:singular}, and thus there exists
$\eta_0>0$ such that $F_{q,0}^{-1}(\bu)$ and $\partial A_{q,0}(\eta)$ intersect
transversally for all $\eta$ with $\eta\ge \eta_0$.
We will show in the following lemma that
such an $\eta_0$ can be chosen independently
of $(q,\bu)$ as far as $(q,\bu)$ varies in a compact subset of
$\cMsmnd\times \C$.
We note that Nemethi--Zaharia \cite{Nemethi-Zaharia:Milnor}
and Parusinski~\cite{Parusinski} have obtained
analogous results for a single polynomial function
on~$\C^n$ (and the proof is similar).

\begin{Lemma}\label{lem:M-tame}\quad
\begin{enumerate}\itemsep=0pt
\item[{\rm (1)}] For a compact subset $K\subset \cMsmnd \times \C$,
there exists $\eta_0$ such that for all $\eta\ge \eta_0$
and all $(q,\bu) \in K$,
$F_{q,0}^{-1}(\bu)$ and $\partial A_{q,0}(\eta)$ intersect transversally.
\item[{\rm (2)}] The same result on the transversality of
$F_{q,t}^{-1}(\bu)$ and $\partial A_{q,t}(\eta)$ holds when we
replace $\cMsmnd$ with $\cMnd$.
\end{enumerate}
\end{Lemma}

The proof of Lemma \ref{lem:M-tame}
will be given in Appendix \ref{append:M-tame}.
In the situation of Lemma \ref{lem:M-tame}(1),
$A_{q,0}(\eta) \cap F_{q,0}^{-1}(\bu)$ is a deformation
retract of $F_{q,0}^{-1}(\bu)$ for $\eta \ge \eta_0$;
we can see this using the Morse flow for the function~$H(x)$ on $F_{q,0}^{-1}(\bu)$.
In particular, the inclusion of pairs
\[
\big(A_{q,0}(\eta), F_{q,0}^{-1}(\bu)\cap A_{q,0}(\eta)\big) \rightarrow
\big(\cYsm_q, F_{q,0}^{-1}(\bu)\big)
\]
induces an isomorphism of relative homology.
This fact will be used in the following proof.

\begin{proof}[Proof of Theorem $\ref{thm:inclusion}$]
We shall construct an inclusion of
local systems:
\[
\iota \colon \
\Lefsm|_{\Dsm^\times} \to \tti^{-1} \ttj_* \slide_{>0}^{-1} \Lef.
\]
Note that the stalk of $\tti^{-1} \ttj_*\slide_{>0}^{-1}\Lef$
at $q_0\in \Dsm^\times$ consists of a Gauss--Manin flat
family of relative homology classes in
$\Lef_{\slide_t(q)}$ with $t>0$ sufficiently small
and $q$ in a~small contractible neighbourhood of $q_0$
in $\cMsmnd$. The construction of $\iota$ will be
done in the following 5 steps.

\proofstep{{\rm (1)}
Construction of $\iota$ on the stalk at $q_0\in \Dsm^\times$}
Let $\rho_2>0$ be the constant in Lemma \ref{lem:slide}(3)
and let $\bu_0>\rho_2$ be such that all critical values of $F_{q_0,0}$
are contained in $\{\bu\colon \Re(\bu) < \bu_0\}$.
By Lemma~\ref{lem:M-tame}, there exists $\eta_0>0$
such that $F_{q,0}^{-1}(\bu) \pitchfork \partial A_{q,0}(\eta)$
for all $\eta\ge \eta_0$ and $(q,\bu)$ in a~neighbourhood
of $(q_0,\bu_0)$. Then by the remark preceding the proof,
the inclusion induces an isomorphism of relative homology
(we use $\Z$ coefficients unless otherwise mentioned):
\begin{equation}
\label{eq:iota_first}
H_n\big(A_{q_0,0}(\eta_0), F_{q_0,0}^{-1}(\bu_0)
\cap A_{q_0,0}(\eta_0)\big)
\xrightarrow{\cong}
H_n\big(\cYsm_{q_0}, F_{q_0,0}^{-1}(\bu_0)\big) \cong \Lefsm_{q_0}.
\end{equation}
Since $\{\partial A_{q,t}(\eta_0)\}_{q,t}$ is a proper family,
$F_{q,t}^{-1}(\bu_0)$ and $\partial A_{q,t}(\eta_0) $
intersect transversally for $(q,t)$ in a sufficiently small contractible
neighbourhood $B$ of $(q_0,0)$ in $\cM_-$.
The Ehresmann fibration theorem
implies that $F^{-1}_{q,t}(\bu_0)\cap A_{q,t}(\eta_0)$
is a trivial family of $C^\infty$-manifolds (with boundary)
when $(q,t)$ varies in~$B$.
Therefore, whenever $\slide_t(q)$ lies in $B$, the inclusion
of pairs
\[
\big(A_{\slide_t(q)}(\eta_0), F^{-1}_{\slide_t(q)}(\bu_0)
\cap A_{\slide_t(q)}(\eta_0)\big)
\hookrightarrow
\big(A_B(\eta_0), F^{-1}(\bu_0) \cap A_B(\eta_0) \big)
\]
induces an isomorphism of relative homology, where we set
$A_B(\eta_0) := \bigcup\limits_{(q,t) \in B} A_{q,t}(\eta_0)$.
Thus we obtain a Gauss--Manin flat isomorphism
\begin{equation}\label{eq:iota_second}
\begin{aligned}
\xymatrix@C=-100pt{
 & H_n\big(A_B(\eta_0), F^{-1}(\bu_0) \cap A_B(\eta_0)\big) & \\
H_n\big(A_{q_0,0}(\eta_0), F_{q_0,0}^{-1}(\bu_0) \cap A_{q_0,0}(\eta_0)\big)
\phantom{ABCDE}
\ar[ur]^(.4){\cong} & & \ar[ul]_(.4){\cong}
\phantom{ABCDE}
H_n\big(A_{\slide_t(q)}(\eta_0), F_{\slide_t(q)}^{-1}(\bu_0)
\cap A_{\slide_t(q)}(\eta_0)\big)
}
\end{aligned}
\end{equation}
for $t>0$ sufficiently small and
$q$ in a small neighbourhood of $q_0$.
Composing \eqref{eq:iota_first}, \eqref{eq:iota_second}
and the natural map induced by the inclusion:
\[
H_n\big(A_{\slide_t(q)}(\eta_0), F_{\slide_t(q)}^{-1}(\bu_0)
\cap A_{\slide_t(q)}(\eta_0)\big)
\longrightarrow
H_n\big(\cY_{\slide_t(q)}, F_{\slide_t(q)}^{-1}(\bu_0)\big)
\]
we obtain
\begin{equation}\label{eq:iota_comp}
\Lefsm_{q_0} \longrightarrow
H_n\big(\cY_{\slide_t(q)}, F_{\slide_t(q)}^{-1}(\bu_0)\big).
\end{equation}
Now recall from Lemma~\ref{lem:slide}(3) that
there are no critical values of $F_{\slide_t(q)}$
in the region
$\big\{\bu\in \C\colon \allowbreak |\bu|\ge \rho_2, \, \arg(\bu)
\in \big[{-}\frac{\pi}{2J}, \frac{\pi}{2J}\big]\big\}$.
Since $\bu_0> \rho_2$, the parallel transportation along a straight path
defines a canonical isomorphism
\[
H_n\big(\cY_{\slide_t(q)}, F_{\slide_t(q)}^{-1}(\bu_0)\big)
\cong
H_n\big(\cY_{\slide_t(q)}, F_{\slide_t(q)}^{-1}(\bu_1)\big)
\]
for all $\bu_1> \bu_0$. Thus~\eqref{eq:iota_comp} gives
a map $\Lefsm_{q_0} \to \Lef_{\slide_t(q)}$
for $t>0$ sufficiently small and $q$ in a~small neighbourhood of $q_0$.
This map is Gauss--Manin flat as~$(t,q)$ varies and defines
the map~$\iota_{q_0}$ on the stalks.

\proofstep{{\rm (2)} Independence of the choice of $\eta_0$, $\bu_0$}
We show that the map $\iota_{q_0}$ is independent of
the choices made. That replacing $\eta_0$ with a bigger $\eta_1>\eta_0$
does not change the map $\iota_{q_0}$ follows from
the commutative diagram:
\begin{multline*}
\xymatrix@C=10pt{
\big(\cYsm_{q_0}, F_{q_0,0}^{-1}(\bu_0) \big)
& \big(A_{q_0,0}(\eta_0), F_{q_0,0}^{-1}(\bu_0) \cap A_{q_0,0}(\eta_0)\big)
\ar[r] \ar[d] \ar[l] &
\big(A_B(\eta_0), F^{-1}(\bu_0) \cap A_B(\eta_0) \big) \ar[d]
& \ar[l]
\\
& \ar[r] \big(A_{q_0,0}(\eta_1), F_{q_0,0}^{-1}(\bu_0) \cap A_{q_0,0}(\eta_1)\big)
\ar[ul]
&
\big(A_B(\eta_1), F^{-1}(\bu_0) \cap A_B(\eta_1) \big) & \ar[l]
}
\\
\xymatrix{
& \big(A_{\slide_t(q)}(\eta_0), F^{-1}_{\slide_t(q)}(\bu_0)
\cap A_{\slide_t(q)}(\eta_0)\big)
\ar[l] \ar[d] \ar[rd] & \\
&
\big(A_{\slide_t(q)}(\eta_1), F^{-1}_{\slide_t(q)}(\bu_0)
\cap A_{\slide_t(q)}(\eta_1)\big)
\ar[l] \ar[r] & \big(\cY_{\slide_t(q)}, F_{\slide_t(q)}^{-1}(\bu_0)\big).
}
\end{multline*}
Next we show that the map is independent of $\bu_0$.
Suppose that we choose $\bu_1>\bu_0$ in place of~$\bu_0$.
By Lemma \ref{lem:M-tame} again, we can choose
$\eta_0>0$ such that
$F_{q,0}^{-1}(\bu) \pitchfork \partial A_{q,0}(\eta)$
for all $\eta\ge \eta_0$, all $q$ in a neighbourhood of $q_0$
and all $\bu$ in the interval $[\bu_0,\bu_1]$.
Then the family of $C^\infty$-manifolds
$F_{q,t}^{-1}(\bu) \cap A_{q,t}(\eta_0)$ is trivial
as $\bu$ varies in $[\bu_0,\bu_1]$ and $(q,t)$ varies in a
contractible neighbourhood $B$
of $(q_0,0)$.
The various maps in the construction of $\iota_{q_0}$
can be compared with the following sequence of maps:
\begin{gather*}
\big(\cYsm_{q_0} \times [\bu_0,\bu_1],
F^{-1}_{q_0,0}[\bu_0,\bu_1]\big) \\
\qquad{}\leftarrow
\big(A_{q_0,0}(\eta_0) \times [\bu_0,\bu_1],
F^{-1}_{q_0,0}[\bu_0,\bu_1] \cap
(A_{q_0,0}(\eta_0) \times [\bu_0,\bu_1]) \big) \\
\qquad{} \rightarrow
\big(A_B(\eta_0) \times [\bu_0,\bu_1],
F^{-1}[\bu_0,\bu_1] \cap (A_B(\eta_0) \times [\bu_0,\bu_1])
\big) \\
\qquad{} \leftarrow
\big( A_{\slide_t(q)}(\eta_0) \times [\bu_0,\bu_1],
F_{\slide_t(q)}^{-1}[\bu_0,\bu_1] \cap
(A_{\slide_t(q)}(\eta_0) \times [\bu_0, \bu_1])
\big) \\
\qquad{} \rightarrow
\big( \cY_{\slide_t(q)} \times [\bu_0, \bu_1],
F^{-1}_{\slide_t(q)} [\bu_0,\bu_1]
\big),
\end{gather*}
where $F_{q,t}^{-1}[\bu_0,\bu_1]$ is regarded as a subset of
$\cY_{q,t} \times [\bu_0,\bu_1]$ via $x \mapsto (x, F_{q,t}(x))$
and the maps between the first four pairs induce isomorphisms
in relative homology. It follows that the different choices $\bu_0$, $\bu_1$
give the same map~$\iota_{q_0}$.

\proofstep{{\rm (3)} Gauss--Manin flatness of $\iota$}
This is obvious from the construction in (1); note that $q_0$
there can vary in a small open subset of $\Dsm^\times$.

\proofstep{{\rm (4)} Injectivity of $\iota$.}
Choose $q_0\in \Dsm^\times\cap \cVsmss_-$,
where $\cVsmss_-$ is an open dense subset of $\cVsm_-$
appearing in Section~\ref{subsubsec:str_sheaf_-}, so that
all critical points $\crit_1,\dots,\crit_{N_-}$
of $F_{q_0,0}$ are non-degenerate.
It suffices to show that $\iota_{q_0}$ is injective.
Let $\bu_0>0$ be as in (1).
In this case, a basis of $\Lefsm_{q_0}$ is given by Lefschetz thimbles:
setting $\bu_i = F_{q_0,0}(\crit_i)$ with $1\le i\le N_-$
and choosing a system of mutually non-intersecting paths
$\gamma_i$ connecting $\bu_i$ and $\bu_0$,
we have finite Lefschetz thimbles
$\Gamma_i \cong D^n$ emanating from the critical point $\crit_i$ and
fibred over the path $\gamma_i$; then $\Gamma_1,\dots,\Gamma_{N_-}$
define relative cycles in the pair $\big(\cYsm_{q_0}, F_{q_0,0}^{-1}(\bu_0)\big)$ and
form a basis of $\Lefsm_{q_0}$. We choose $\eta_0>0$ big enough
so that $\Gamma_i$ are contained in $A_{q_0,0}(\eta_0)$
and that the conditions in (1) are satisfied.
The critical points $\crit_i$ belong to convergent critical branches,
and as such, vary continuously in a neighbourhood of $(q_0,0)$
in $\cM_-$; the path $\gamma_i$ can be also continuously
deformed to a family of paths $\gamma_i(q,t)$
connecting the critical value $F_{q,t}(\crit_i)$ and $\bu_0$.
Due to the compactness, the finite thimble
$\Gamma_i$ can be continuously deformed\footnote
{Away from the critical point, we use a symplectic connection
to deform the thimble; around the critical point we use
a family version of the Morse lemma.}
to a finite thimble $\Gamma_i(q,t)\subset A_{q,t}(\eta_0)$
for~$F_{q,t}$ fibred over~$\gamma_i(q,t)$
when $(q,t)$ is sufficiently close to $(q_0,0)$;
$\Gamma_i(q,t)$ gives a relative cycle of
$\big(A_{q,t}(\eta_0), F_{q,t}^{-1}(\bu_0) \cap A_{q,t}(\eta_0)\big)$.
As relative cycles in
$\big(\cY_{\slide_t(q)}, F_{\slide_t(q)}^{-1}(\bu_0)\big)$,
$\Gamma_1(\slide_t(q)),\dots,\Gamma_{N_-}(\slide_t(q))$
form part of a basis of the $n$th relative homology
(corresponding to the convergent critical points)
when $t>0$ is sufficiently small and~$q$ is sufficiently
close to~$q_0$.
Thus $\iota_{q_0}$ sends a basis to part of a basis, and
is injective.

\proofstep{{\rm (5)} That $\iota$ maps $\Gamma_\R$ to $\Gamma_\R$
along $\Dsm^\times \cap \cMsm_\R$}
The positive real Lefschetz thimble~$\Gamma_\R$
(see Section~\ref{subsec:identifying_O}) defines a global section of
$\Lefsm$ over $\Dsm^\times \cap \cMsm_\R$
and a global section of $\slide_{>0}^{-1}\Lef$ over
$(0,1]\times (\Dsm^\times \cap \cMsm_\R)$.
It is obvious from the construction that the map
$\iota_{q_0}$ with $q_0\in \Dsm^\times \cap \cMsm_\R$
sends $\Gamma_\R(q_0)$ to the germ in
$\big(\tti^{-1} \ttj_*\slide_{>0}^{-1} \Lef\big)_{q_0}$
given by $\{\Gamma_\R(\slide_t(q))\}_{t>0,q}$.
\end{proof}

\subsection{Functoriality}\label{subsec:functoriality}
Let $\cV_0$ be an open neighbourhood of $\cC\setminus \{0_+,0_-\}$
as in Section~\ref{subsubsec:conv_div}. We consider the open dense subset $\cVss_0$ of $\cV_0$:
\[
\cVss_0:=\{q \in \cV_0\cap \cMtimes\colon
\text{$F_q=F|_{\pr^{-1}(q)}$ has only non-degenerate critical points}
\}.
\]
Take $q_0\in \cVss_0$ and choose an admissible phase $\phi$
for the critical values of $F_{q_0}$.
By composing the formal decomposition of the analytified Brieskorn module
in Proposition \ref{prop:formal_decomp_Bri} and
(the non-equivariant version of) the mirror isomorphism
in Theorem~\ref{thm:an_mirror_isom},
we get formal decompositions of the pulled-back quantum D-modules
over an analytic open neighbourhood $B$ of $q_0$:
\[
\mir_\pm^*\overline\QDMan(\frX_\pm)\big|_B \cong
\bigoplus_{i=1}^{N_\pm} (\cO_B[\![z]\!], d+d(\bu_i/z),P_\std),
\]
where $N_\pm= \dim H^*_\CR(\frX_\pm)$ and $\bu_i$ are
relative critical values of $F$.
By Proposition~\ref{prop:Hukuhara-Turrittin}, by shrinking~$B$ if necessary,
we have analytic lifts of these formal decompositions
over $B\times I_\phi$, where
$I_\phi = \big\{\big(r,e^{\iu\theta}\big) \in\tCC\colon
|\theta-\phi|<\frac{\pi}{2} + \epsilon\big\}$ (for some small $\epsilon>0$):
\begin{equation}\label{eq:analytic_lifts_pm}
\pi^*\mir_\pm^*\QDMan(\frX_\pm)\big|_{B\times I_\phi}
\cong \bigoplus_{i=1}^{N_\pm}
(\cA_{B\times I_\phi}, d+d(\bu_i/z)),
\end{equation}
where $\pi\colon \cMtimes \times \tCC \to \cMtimes \times \C$
denotes the oriented real blowup along $\cM\times \{0\}$.
Summands of this sectorial decomposition are indexed by relative
critical points of~$F$; those for $\frX_+$ are indexed by
all critical points and those for $\frX_-$ are indexed by the subset
of convergent critical points.
Combining these two decompositions, we get a sectorial decomposition:
\begin{equation}\label{eq:analytic_lift_decomp}
\pi^*\mir_+^*\QDMan(\frX_+) \big|_{B\times I_\phi}
\cong \pi^*\mir_-^*\QDMan(\frX_-)\big|_{B\times I_\phi}
\oplus
\cA_{B\times I_\phi}^{\oplus (N_+-N_-)}.
\end{equation}
This gives an analytic lift of the formal decomposition
in Theorem~\ref{thm:decomp_QDM}.
We consider the inclusion
\begin{equation}
\label{eq:analytic_lift_incl}
\pi^*\mir_-^*\QDMan(\frX_-)\big|_{B\times I_\phi}
\hookrightarrow
\pi^*\mir_+^*\QDMan(\frX_+) \big|_{B\times I_\phi}
\end{equation}
induced by \eqref{eq:analytic_lift_decomp}.
Note that the maps
\eqref{eq:analytic_lift_decomp}, \eqref{eq:analytic_lift_incl} are
associated with a semisimple point $q_0$ and
an admissible direction $e^{\iu\phi}$.

\begin{Theorem}
\label{thm:functoriality}
Let $\varphi\colon \frX_+ \to \frX_-$ be a toric
birational morphism as in Section~$\ref{subsec:functoriality_nota}$.
For each point $q_*\in \cVsmss_{-,\R} \cap \intDsmtimes
\subset \cMsm_\R$,
there exists a contractible open neighbourhood $W_*$ of $q_*$
in $\cVsmss_{-,\R}\cap \intDsmtimes$
and positive numbers
$t_*\in (0,1)$, $\alpha_0\in \big(0,\frac{\pi}{2J}\big)$ such
that the following holds.
For each $q_0 \in \bigcup\limits_{0<t<t_*} \slide_t(W_*) \cap \cVss_0$
and each $\phi\in (-\alpha_0,\alpha_0)$ that is admissible for
the critical values of~$F_{q_0}$,
the inclusion~\eqref{eq:analytic_lift_incl} associated with $q_0$
and $e^{\iu\phi}$ is induced by the pull-back in $K$-theory
via the $\hGamma$-integral structure, i.e., it sends $\frs_V$ to
$\frs_{\varphi^*V}$ for $V\in K(\frX_-)$.
\end{Theorem}
Recall from Section~\ref{subsubsec:str_sheaf_-} that
$\cVsmss_{-,\R}$ is the intersection of the semisimple locus
$\cVss_-$ of~$\cV_-$ and the real positive locus $\cMsm_\R$
of $\cMsm$. Also, $\intDsmtimes$ denotes
the interior of $\Dsm^\times=\Dsm\cap \cMsmtimes$ as a~subset of~$\cMsm$.
Recall that $\slide_t$ is the sliding map from Definition~\ref{def:slide}
and $J\ge 1$ is the natural number in Proposition~\ref{prop:critv_on_curve}.
Note that $\bigcup\limits_{0<t<t_*}\slide_t(W_*) \cap \cVss_0$ is non-empty
because $\cVss_0 \cap \cM_\R$ is open dense in $\cV_0\cap \cM_\R$
and $\bigcup\limits_{0<t<t_*} \slide_t(W_*)$ is an open subset of
$\cV_0\cap \cM_\R$ by Lemma \ref{lem:slide}(2).

\begin{Remark}As we explained in Remark~\ref{rem:standard_determination},
we have a standard identification between $K$-classes of $\frX_+$ (resp.~$\frX_-$)
and flat sections over $(\cM_\R \cap \cV_+) \times \R_{>0}$
(resp.~$(\cMsm_\R \cap \cV_-)\times \R_{>0}$).
We use these identifications in the above theorem;
we also use the analytic continuation
along the sliding homotopy $\slide_t$ for $\frX_-$.
\end{Remark}

\begin{Remark}By the duality of the analytic lift discussed in
Remark \ref{rem:pairing_analytic_lift}, if the inclusion \eqref{eq:analytic_lift_incl} corresponds to the pull-back $\varphi^* \colon K(\frX_-) \to K(\frX_+)$ in $K$-theory, the projection
associated with the opposite direction $-e^{\iu\phi}= e^{\iu (\phi+\pi)}$
\[
\pi^*\mir_+^*\QDMan(\frX_+)\big|_{B\times I_{\phi+\pi}}
\twoheadrightarrow
\pi^*\mir_-^*\QDMan(\frX_-)\big|_{B\times I_{\phi+\pi}}
\]
corresponds to the push-forward $\varphi_* \colon K(\frX_+) \to K(\frX_-)$
in $K$-theory.
\end{Remark}

\begin{Lemma}
\label{lem:diag_pullback}
Let $q_*$ be a point in $\cMsm_\R \cap \Dsm^\times$.
We have a commutative diagram
\begin{equation}
\label{eq:diag_K_Lef_pullback}
\begin{aligned}
\xymatrix{
\Lefsm_{q_*} \ar[r]^(0.25)\iota \ar[d]^{\cong} &
\big(\tti^{-1}\ttj_* \slide_{>0}^{-1} \Lef\big)_{q_*} \cong
\big(\slide_t^{-1}\Lef\big)_{q_*}
 \ar[d]^{\cong} \\
K(\frX_-) \ar[r]^{\varphi^*} & K(\frX_+)
}
\end{aligned}
\qquad \text{with $t\in (0,1]$},
\end{equation}
where the vertical arrows are induced by the isomorphisms
in Theorems $\ref{thm:main_intpaper}$ and~$\ref{thm:main_intpaper_2}$,
the top horizontal arrow is the inclusion in Theorem~$\ref{thm:inclusion}$.
\end{Lemma}

\begin{proof}Note that the left vertical arrow is induced, via the $\hGamma$-integral structure,
by the isomorphism of local systems over $\Dsm^\times$:
\begin{equation}\label{eq:Lefsm_QDM-}
\Lefsm\otimes \C\big|_{\Dsm^\times}
\cong \big(\mir_-^*\QDMan(\frX_-)|_{\Dsm^\times \times \{1\}} \big)^\nabla
\end{equation}
sending a class $[\Gamma] \in H_n(\cY_q, \{\Re(F_q)\gg 0\})$
to a flat section $s_\Gamma$ of $\mir_-^*\QDMan(\frX_-)|_{\Dsm^\times \times\{1\}}$
satisfying
\[
\frac{1}{(2\pi)^{n/2}} \int_\Gamma
e^{-F_q} \Omega_{q,-1} = P(\Mir(\Omega), s_\Gamma)
\]
for every local section $\Omega_{q,z}$ of $\Brism(F)$.
Here $(\cdots)^\nabla$ denotes the local system defined
by the kernel of $\nabla$.
Also, the right vertical arrow in~\eqref{eq:diag_K_Lef_pullback}
is induced by the isomorphism of local systems
over $\Image(\slide_{>0})$:
\begin{equation}\label{eq:Lef_QDM+}
\Lef\otimes \C\big|_{\Image(\slide_{>0})} \cong
\big(\mir_+^*\QDMan(\frX_+)
|_{\Image(\slide_{>0})\times \{1\}}\big)^\nabla,
\end{equation}
which is defined similarly.
We conclude the lemma by studying the monodromy
of these local systems.
Let $\LL'=\Ker\big(\Z^{S_-}\to \bN\big)$ be as in Section~\ref{subsubsec:analytic_lift}.
The inclusion $\Dsm^\times \subset \cMsmtimes = (\LL')^\star
\otimes \C^\times$
induces an isomorphism
$\pi_1(\Dsm^\times) \cong (\LL')^\star$;
an element $\xi \in (\LL')^\star$ corresponds to
the loop $[0,2\pi]\ni \theta \mapsto
e^{\iu\theta \xi} \cdot q_*$ based at $q_*\in \Dsm^\times$.
We claim that the monodromy of $\Lefsm$ around the loop
$e^{\iu\theta\xi} \cdot q_*$
corresponds to tensoring by $L_\xi^{-1}$ in $K(\frX_-)$
on the $\hGamma$-integral structure
under the isomorphism \eqref{eq:Lefsm_QDM-},
where $L_\xi$ denotes the orbi-line bundle as in
Section~\ref{subsubsec:Pic_2nd}.
We also claim that the monodromy of $\Lef$ around
$\slide_t(e^{\iu\theta\xi} \cdot q_*) = (e^{\iu\theta\xi} \cdot q_*,t)$
corresponds to tensoring by $\varphi^*L_\xi^{-1}$ in $K(\frX_+)$
under the isomorphism \eqref{eq:Lef_QDM+}.
These two claims prove the lemma.
In fact, since the map $\iota$ is monodromy-equivariant,
the composition
\begin{equation}
\label{eq:iota_in_K}
K(\frX_-) \cong \Lefsm_{q_*}
\overset{\iota}{\longrightarrow} \big(\slide_t^{-1}\Lef\big)_{q_*} \cong K(\frX_+)
\end{equation}
identifies tensoring by $L_\xi$ with tensoring by $\varphi^*L_\xi$;
also this map \eqref{eq:iota_in_K}
sends the structure sheaf to the structure sheaf
since $\iota$ sends the positive real Lefschetz thimble
$\Gamma_\R(q_*)$ to the positive-real one
$\Gamma_\R(\slide_t(q_*))$ (see Theorem \ref{thm:inclusion});
since $K(\frX_-)$ is generated by line bundles \cite{Borisov-Horja:K},
we conclude that the map \eqref{eq:iota_in_K}
equals $\varphi^*$.

It remains to prove the above two claims.
Let $[\xi]$ denote the class in $\Pic(\frX_-)$ of $L_\xi$.
By \cite[(61)]{Iritani:Integral}, the mirror map satisfies
\[
\mir_-\big(e^{2\pi \iu \xi} \cdot q_*\big) = g(-[\xi]) \mir_-(q_*)
\]
(see Section~\ref{subsec:Kaehler_moduli} for $g(-[\xi])$)
and hence the flat section $\frs_V(\mir_-(q_*),z=1)$
is analytically continued,
along the path $\theta \mapsto e^{\iu\theta\xi} \cdot q_*$,
to $\frs_V(g(-[\xi])\mir_-(q_*),1)$, which is identified with
$dg(-[\xi])^{-1} \frs_V(g(-[\xi])\mir_-(q_*),1) =
\frs_{V\otimes L_\xi^{-1}}(\mir_-(q_*),1)$ by the Galois action
in Section~\ref{subsec:QDM}; here we
used the formula \eqref{eq:Galois_line_tensor}.
This proves the first claim.
To prove the second claim, it suffices to show that
the homotopy class $\hxi \in \LL^\star \cong \pi_1(\cMtimes)$
of the loop $\slide_t(e^{\iu\theta\xi} \cdot q_*)$
defines a line bundle $L_\hxi$ isomorphic to $\varphi^*L_\xi$.
Observe that we have a direct sum decomposition
\begin{equation}
\label{eq:LL_LL'}
\LL = \LL' \oplus \Z \delta_\hb^{\bSigma_-},
\end{equation}
where $\delta_\hb^{\bSigma_-} = e_\hb - \Psi_-(\hb) = -\bw$
is an element introduced in Section~\ref{subsec:ext_refined_fanseq}
(see also Section~\ref{subsec:functoriality_nota})
which corresponds to the co-ordinate $t=t_\hb$.
The class $\hxi\in \LL^\star$ is a unique
lift of $\xi\in (\LL')^\star$ such that
$\hxi \cdot \delta_\hb^{\bSigma_-}=0$.
The map \eqref{eq:varphi_coord} inducing the birational
morphism $\varphi\colon \frX_+ \to \frX_-$
is equivariant with respect to the homomorphism
$\LL\otimes \C^\times \to \LL'\otimes \C^\times$,
$\exp(\lambda) \mapsto \exp(\lambda + (D_\hb \cdot \lambda) \bw)$
(where $\lambda \in \LL_\C$), which is dual to the above lift
$(\LL')^\star \to \LL^\star$, $\xi \mapsto \hxi$.
In view of the definition of $L_\hxi$ in Section~\ref{subsubsec:Pic_2nd},
we see that $L_\hxi \cong \varphi^*L_\xi$.
The lemma is proved.
\end{proof}

We show that $\alpha_0$ appearing in Proposition~\ref{prop:lift_str_sheaf_+} can be chosen independently of $q_0\in \cVss_{+,\R}$ if~$q_0$ is close to a given point $q_*$ in $\cVsmss_{-,\R}
\cap \intDsmtimes$.

\begin{Lemma}\label{lem:uniform_alpha0} Let $q_*$ be a point in $\cVsmss_{-,\R} \cap \intDsmtimes$.
There exist a contractible open neighbourhood~$W_*$ of~$q_*$ in
$\cVsmss_{-,\R}\cap \intDsmtimes$ and positive numbers $t_*\in (0,1)$, $\alpha_0\in \big(0,\frac{\pi}{2J}\big)$ such that the conclusion of Proposition~$\ref{prop:lift_str_sheaf_+}$
holds for all $q_0\in \bigcup\limits_{0<t<t_*} \slide_t(W_*)
\cap \cVss_0 \subset \cVss_{+,\R}$ with this same $\alpha_0$.
\end{Lemma}

\begin{proof}Let $\crit_1(q),\dots,\crit_{N_-}(q)$ be branches of ``convergent''
critical points of $F_q$ contained in~$\cB_-$ and defined in a
neighbourhood of $q=q_*$ in $\cVss_-$.
We assume that $\crit_1(q_*) = \crit_\R(q_*)$ and
write $\bu_i(q) = F_q(\crit_i(q))$ as before.
We may also assume that
$\bu_i(q_*)\in \bu_1(q_*)+\R_{\ge 0}$ if and only if $1\le i\le k$.
Choose $\alpha_0\in \big(0,\frac{\pi}{2J}\big)$ so that the closed sector
$I_*:= \bu_1(q_*) + \big\{r e^{\iu\theta}\colon r\ge 0, |\theta|\le \alpha_0\big\}$
with vertex at $\bu_1(q_*)$
does not contain critical values other than $\bu_1(q_*),\dots,\bu_k(q_*)$,
see Fig.~\ref{fig:alpha_star}.
We can find a contractible open neighbourhood $W_*$ of $q_*$
in $\cVsmss_{-,\R}\cap \intDsmtimes$,
$0<t_*<1$, and $\epsilon>0$ such that
for all $q\in \bigcup\limits_{0\le t<t_*}\slide_t(W_*)$,
$\bu_i(q) \in -\epsilon + I_*$ if and only if $1\le i\le k$.
Note that $W_* \subset \cVsmss_- \subset \cMsmnd$
and $\bigcup\limits_{0<t<t_*}\slide_t(W_*) \subset \cV_0 \subset \cMnd$
by Lemma~\ref{lem:slide} and Remark~\ref{rem:Vpm_contained_in_nd_locus}.

\begin{figure}[ht]\centering
\includegraphics{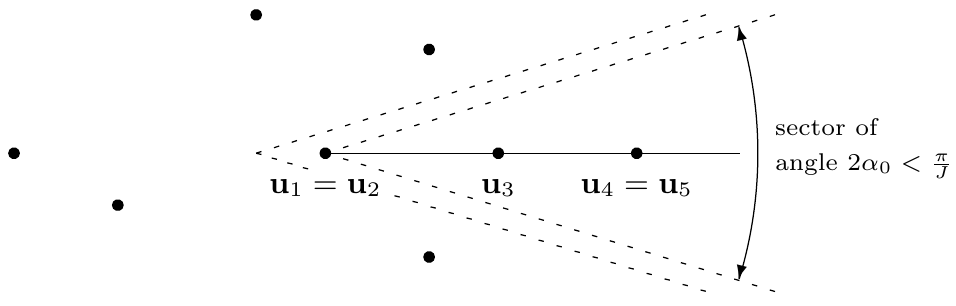}
\caption{Critical values at $q_*$: only $\bu_1,\dots,\bu_k$ (critical values on the half-line $\bu_1+\R_{\ge 0}$) lie in the sector $-\alpha_0\le \arg(\bu-\bu_1) \le \alpha_0$ ($k=5$).}\label{fig:alpha_star}
\end{figure}

We also have divergent critical points over
$\bigcup\limits_{0<t<t_*} \slide_t(W_*)$; let
$\crit_j(q)$, $j=N_- +1,\dots,N_+$ denote
divergent critical branches over $\bigcup\limits_{0<t<t_*} \slide_t(W_*)$
(they can be ramified and may not be single-valued).
By taking smaller $t_*>0$ if necessary, by Lemma \ref{lem:slide}(3),
we may assume that the divergent critical values $\bu_j(q)
= F_q(\crit_j(q))$, $N_-+1 \le j\le N_+$ do not lie in
the sector $-\epsilon + I_*$ as long as
$q \in \bigcup\limits_{0<t<t_*}\slide_t(W_*)$.

Let $\Gamma^\phi_i(q)$ denote the Lefschetz thimble
\eqref{eq:Lefschetz} of $F_q$ associated with
the critical point $\crit_i(q)$ and the phase $\phi$.
Here $q$ can be either in $W_*$
or in $\bigcup\limits_{0<t<t_*}\slide_t(W_*)$~-- we have $1\le i\le N_-$
in the former case and $1\le i\le N_+$ in the latter case.
In view of the proof of Proposition \ref{prop:lift_str_sheaf_+},
it suffices to show that the relative homology class
\[
\big[\Gamma_1^{\pi+\phi}(q)\big] \in H_n\big(\cY_q, \{\Re(F_q) \gg 0\}; \Z\big)
\]
represented by the thimble $\Gamma_1^{\pi+\phi}(q)$
is constant as $(q,\phi)$ varies in $\bigcup\limits_{0<t<t_*}\slide_t(W_*)
\times (-\alpha_0,\alpha_0)$.
The homology class $\big[\Gamma_1^{\pi+\phi}(q)\big]$ can jump
by the Picard--Lefschetz transformation
as $(q,\phi)$ varies (see, e.g., \cite[Chapter~I]{AGV:Singularity_II}).
By our choice of $W_*$, $t_*$ and~$\alpha_0$, it suffices to check that the
intersection numbers of vanishing cycles at~$\crit_1(q)$
and $\crit_i(q)$ (with $2\le i\le k$) are zero at some
$q\in \bigcup\limits_{0<t<t_*}\slide_t(W_*)$; here we
consider vanishing cycles associated with
paths from $\bu_i(q)$ to a base point $\bu_0\gg 0$
\emph{inside} the sector $-\epsilon + I_*$.
At $(q,\phi) = (q_*,0)$, the Lefschetz thimble
$\Gamma^\pi_1(q_*) = \Gamma_\R(q_*)$ lying over
$\bu_1(q_*) + \R_{\ge 0}$ does not contain
the critical points $\crit_i(q_*)$, $2\le i\le k$,
and therefore the vanishing cycles at $\crit_1(q_*)$ and
$\crit_i(q_*)$ (with $2\le i\le k$) do not intersect.
For $1\le i\le k$, choose a path $\gamma_i$
starting from $\bu_i(q_*)$ and ending at $\bu_0$
which avoids other critical points,
and let $\Gamma_i$ denote the finite Lefschetz thimble
(with boundary in $F_{q_*}^{-1}(\bu_0)$)
associated with $\crit_i(q_*)$ and the path $\gamma_i$.
We choose a sufficiently big $\eta_0>0$ such that
$\Gamma_i$'s are contained in $A_{q_*}(\eta_0)$ and
that $\partial A_{q_*}(\eta_0) \pitchfork F_{q_*}^{-1}(\bu_0)$.
Arguing as in Parts~(1),~(4)
in the proof of Theorem~\ref{thm:inclusion}, we see that
$\Gamma_i$ can be continuously deformed to
a relative cycle $\Gamma_i(q)$ of
$\big(A_q(\eta_0), A_q(\eta_0)\cap F_q^{-1}(\bu_0)\big)$
as $q$ varies in a small neighbourhood of $q_*$ in $\cVss_-$.
This shows that the vanishing cycles $\partial \Gamma_1(q)$ and
$\partial \Gamma_i(q)$ (with $2\le i\le k$) have zero intersection
number. The lemma is proved.
\end{proof}

\begin{proof}[Proof of Theorem $\ref{thm:functoriality}$]
Choose $q_* \in \cVsmss_{-,\R}\cap \intDsmtimes$.
Let $W_* \subset \cVsmss_{-,\R} \cap \intDsmtimes$,
$t_*\in (0,1)$, $\alpha_0\in \big(0,\frac{\pi}{2J}\big)$
as in Lemma~\ref{lem:uniform_alpha0}.
By taking smaller $\alpha_0$ if necessary,
we may assume that the conclusion of Proposition~\ref{prop:lift_str_sheaf_-}
holds for the same $\alpha_0$ at $q_0 = q_*$.
As in the proof of Lemma~\ref{lem:uniform_alpha0},
let $\crit_1(q),\dots,\crit_{N_-}(q)$ denote the convergent
critical branches over a neighbourhood of $q_*$ in $\cVss_-$ such that
$\crit_1(q)$ is the conifold point when
$q\in \cM_\R \cup \cMsm_\R$.
Also, let $\crit_{N_-+1}(q),\dots,\crit_{N_+}(q)$
denote divergent critical branches over a neighbourhood
of $q_*$, defined away from $\cMsm$
(they can be ramified and may not be single-valued).
We write $\bu_i(q) = F_q(\crit_i(q))$ for the corresponding critical value.
Let $\phi_*\in (-\alpha_0,\alpha_0)$
be an admissible phase for $\{\bu_1(q_*),\dots, \bu_{N_-}(q_*)\}$.
Proposition~\ref{prop:lift_str_sheaf_-} gives an analytic decomposition
\[
\tPhi_{\phi_*}^- \colon \ \pi^*\mir_-^*\QDM(\frX_-)
\big|_{B' \times I_{\phi_*}}
\cong \bigoplus_{i=1}^{N_-} (\cA_{B'\times I_{\phi_*}},
d + d(\bu_i/z) ),
\]
where $B'$ is an open neighbourhood of $q_*$ in
$\cVss_-$, $I_{\phi_*} = \big\{\big(r,e^{\iu\theta}\big)\colon
|\theta-\phi_*| <\frac{\pi}{2} + \epsilon \big\}$ (for some $\epsilon>0$)
and $\pi \colon \cVss_- \times \tCC \to \cVss_-\times \C$ is the
oriented real blowup, with the following property:
there exists a basis $V_1^-,\dots,V_{N_-}^-$ of $K(\frX_-)$
such that $V_1^-$ is the structure sheaf of $\frX_-$ and that
$\tPhi^-_\phi\big(e^{-\bu_j(q)/z} s^-_j\big) = e_j$,
where $s^-_j$ is the flat section $\frs_{V_j^-}(\mir_-(q),z)$ associated
with~$V_j^-$.
Choose $\phi_1,\phi_2 \in \big(\phi_*-\frac{\epsilon}{2}, \phi_* +
\frac{\epsilon}{2}\big)$ (where $\epsilon$ is the one
appearing in $I_{\phi_*}$) such that $\phi_1<\phi_2$
and that all phases in $[\phi_1,\phi_2]$ are admissible
for $\{\bu_1(q_*),\dots,\bu_{N_-}(q_*)\}$.
By shrinking $W_*$ and taking smaller $t_*>0$ if necessary, we may assume that
\begin{itemize}\itemsep=0pt
\item[(a)] $\bigcup\limits_{0\le t<t_*} \slide_t(W_*) \subset B'$;
\item[(b)] $\bu_i(q) - \bu_j(q) \notin e^{\iu\theta}\R$ for all
$q\in \bigcup\limits_{0\le t<t_*} \slide_t(W_*)$, $\theta\in [\phi_1,\phi_2]$,
$i,j\in \{1,2,\dots,N_-\}$ provided that $\bu_i(q_*) \neq \bu_j(q_*)$;
\item[(c)] $\bigcup\limits_{1\le i\le N_-} \bu_i(q) + \big\{r e^{\iu\theta}\colon
r\ge 0, |\theta|\le \alpha_0\big\}$ does not contain divergent
critical values $\bu_j(q)$ with $N_-+1 \le j\le N_+$ for
$q\in \bigcup\limits_{0<t<t_*}\slide_t(W_*)$.
\end{itemize}
The third point (c) is possible by Lemma~\ref{lem:slide}(3) and
the fact that $\alpha_0<\frac{\pi}{2J}$.

Take $q_0 \in \bigcup\limits_{0<t<t_*}\slide_t(W_*) \cap \cVss_0$
and a phase $\phi\in (-\alpha_0,\alpha_0)$ which is
admissible for $\{\bu_1(q_0),\dots,\allowbreak \bu_{N_+}(q_0)\}$.
First we prove the conclusion of the theorem
\emph{assuming that $\phi$ lies in $(\phi_1,\phi_2)
\subset (-\alpha_0,\alpha_0)$}.
The discussion for any admissible $\phi\in (-\alpha_0,\alpha_0)$
will be postponed until the last paragraph of the proof.
Since $\phi\in (-\alpha_0,\alpha_0)$,
Lemma~\ref{lem:uniform_alpha0} (Proposition~\ref{prop:lift_str_sheaf_+})
gives an analytic decomposition $\tPhi^+_\phi$
\begin{gather*}
\tPhi^+_\phi \colon \ \pi^*\mir_{+}^*\QDM(\frX_+)
\big|_{B\times I_\phi}
\cong \bigoplus_{i=1}^{N_+} (\cA_{B\times I_\phi}, d+d(\bu_i/z)),
\end{gather*}
where $B$ is a neighbourhood of $q_0$ in $\cVss_-$
and $I_\phi$ is a sector of the form
$\big\{\big(r,e^{\iu\theta}\big) \in \tCC\colon |\theta -\phi|
< \frac{\pi}{2}+\delta\big\}$ (for some $0<\delta\le \frac{\epsilon}{2}$)
with the following property:
there exists a basis $V_1^+,\dots,V_{N_+}^+$ of~$K(\frX_+)$
such that~$V_1^+$ is the structure sheaf of~$\frX_+$ and
that $\tPhi^+_\phi\big(e^{-\bu_j(q)/z} s_j^+\big) = e_j$,
where~$s_j^+$ is the flat section $\frs_{V_j^+}(\mir_+(q),z)$ associated
with~$V_j^+$.
We may assume that $B\subset B'$ by the condition~(a) above.
We have $I_\phi \subset I_{\phi_*}$
because $\phi \in (\phi_1,\phi_2) \subset
\big(\phi_*-\frac{\epsilon}{2}, \phi_*+\frac{\epsilon}{2}\big)$.
Therefore~$\tPhi^-_{\phi_*}$ gives the analytic lift associated
with the point~$q_0$ and the phase~$\phi$.
In particular, the decomposition~\eqref{eq:analytic_lift_decomp}
is induced by $\tPhi^+_\phi$ and $\tPhi^-_{\phi_*}$,
and the inclusion~\eqref{eq:analytic_lift_incl} is induced by
the map $\kappa \colon
K(\frX_-) \to K(\frX_+)$ sending $V_i^-$ to $V_i^+$
for $1\le i\le N_-$. It now suffices to show that $V_i^+= \varphi^* V_i^-$.
We already know that this is true for $i=1$.

Recall from the proof of Propositions \ref{prop:lift_str_sheaf_+},
\ref{prop:lift_str_sheaf_-} that $V_i^+$ (resp.~$V_i^-$)
corresponds to the Lefschetz thimble $\Gamma_i^{\pi+\phi}(q_0)$ of
$F_{q_0}$ (resp.~the Lefschetz thimble $\Gamma_i^{\pi+\phi}(q_*)$
of $F_{q_*}$)
under the isomorphism in Theorem \ref{thm:main_intpaper}
(resp.~Theorem \ref{thm:main_intpaper_2}).
We claim that the map sending $\Gamma_i^{\pi+\phi}(q_*)
\in \Lefsm_{q_*}$
to $\Gamma_i^{\pi+\phi}(q_0) \in \Lef_{q_0}$ coincides
with the map $\iota$ from Theorem \ref{thm:inclusion}.
Here we regard $\Gamma_i^{\pi+\phi}(q_0)$ as a germ
of $\tti^{-1}\ttj_* \slide_{>0}^{-1} \Lef$ at $q_*$
by extending it to a Gauss--Manin flat section over
$\bigcup\limits_{0<t<t_*} \slide_t(W_*)$.

Fix a sufficiently large $\bu_0\gg 0$.
Under the isomorphism
\[
H_n\big(\cY_{q_*}, \{\Re(F_{q_*})\gg 0\}\big)
\cong
H_n\big(\cY_{q_*}, F_{q_*}^{-1}(\bu_0)\big),
\]
the class of
$\Gamma_i^{\pi+\phi}(q_*)$
corresponds to the class of a finite Lefschetz thimble
fibred over a bent ray as shown in Fig.~\ref{fig:moving_thimbles}(ii).
We move $q_*$ along the sliding map.
For sufficiently small $t>0$, these finite Lefschetz thimbles
can be continuously deformed to finite thimbles for $F_{\slide_t(q_*)}$
with boundary in $F_{\slide_t(q_*)}^{-1}(\bu_0)$,
as discussed in Parts~(1) and~(4) in the proof of Theorem~\ref{thm:inclusion}.
By straightening these paths in the direction $e^{\iu\phi}$ again,
we see that these relative cycles in $\big(\cY_{\slide_t(q_*)},
F_{\slide_t(q_*)}^{-1}(\bu_0)\big)$ correspond to the Lefschetz
thimbles $\Gamma_i^{\pi+\phi}(\slide_t(q_*))$ over the
straight ray $\bu_i(\slide_t(q_*)) + e^{\iu\phi}\R_{\ge 0}$.
Therefore, the map $\iota$ sends $\Gamma_i^{\pi+\phi}(q_*)$
to $\Gamma_i^{\pi+\phi}(\slide_t(q_*))$ for sufficiently
small $t>0$.
The assumptions (b), (c) above ensure that the homology
class of $\Gamma_i^{\pi+\phi}(q)$ stays constant as
$q$ varies inside $\bigcup\limits_{0<t<t_*}\slide_t(W_*)$;
the Picard--Lefschetz transformation can only arise
from the intersection of vanishing cycles at $\crit_i(q)$ and $\crit_j(q)$ with
$i\neq j$, $\bu_i(q_*) = \bu_j(q_*)$, $i,j \in \{1,\dots,N_-\}$,
but these intersection numbers vanish since the vanishing cycles
do not intersect at $q_*$. This proves the claim.

\begin{figure}[htb]\centering
\includegraphics{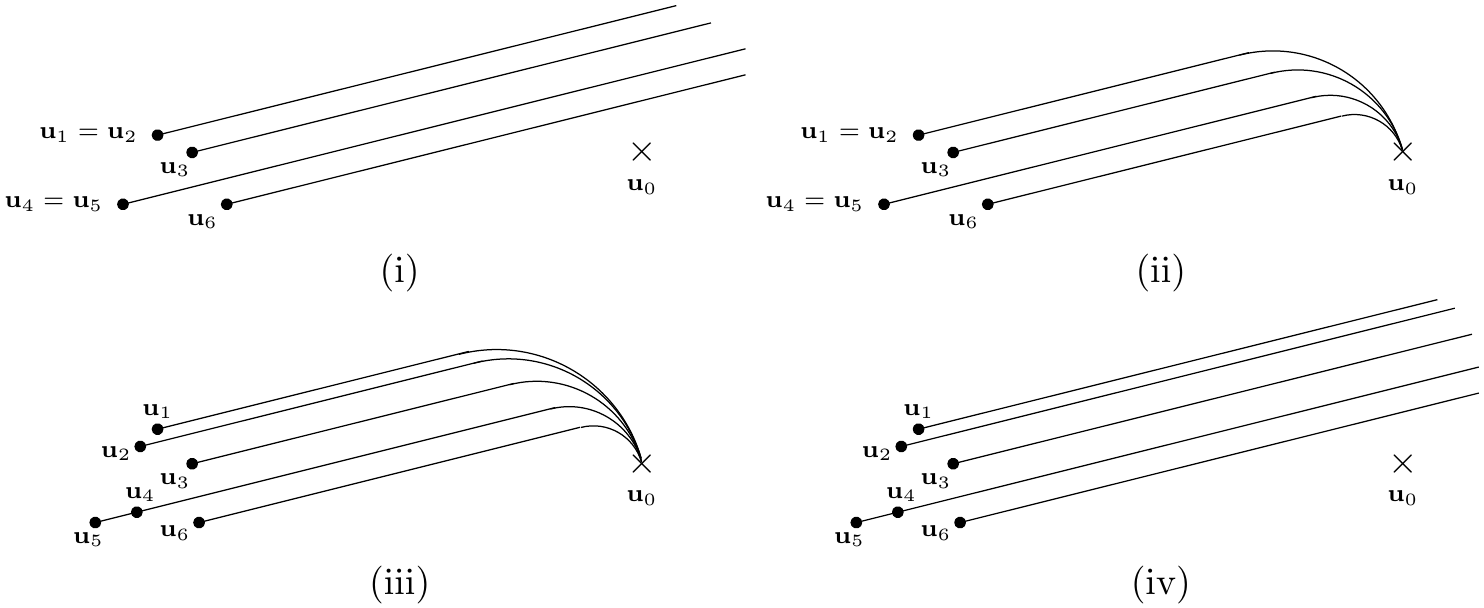}
\caption{Deforming Lefschetz thimbles: (i)~semi-infinite
Lefschetz thimbles fibred over the ray $\bu_j(q_*) + \R_{\ge 0} e^{\iu\phi}$;
(ii)~bent rays passing through $\bu_0$;
(iii)~moving from $q_*$ to a nearby point in
$\bigcup\limits_{0<t<t_*} \slide_t(W_*)$;
(iv)~straightening the paths again. In (iii) and (iv), the paths from
$\bu_4$ and $\bu_5$ overlaps, but it does not matter since the
relevant vanishing cycles do not intersect.}\label{fig:moving_thimbles}
\end{figure}

The above discussion gives a commutative diagram
\begin{equation*}
\begin{aligned}
\xymatrix{
\Lefsm_{q_*} \ar[r]^(0.25)\iota \ar[d]^{\cong} &
\big(\tti^{-1}\ttj_* \slide_{>0}^{-1} \Lef\big)_{q_*} \cong
\big(\slide_t^{-1}\Lef\big)_{q_*}
 \ar[d]^{\cong} \\
K(\frX_-) \ar[r]^{\kappa} & K(\frX_+),
}
\end{aligned}
\end{equation*}
where the vertical arrows are induced by the isomorphisms
in Theorems \ref{thm:main_intpaper} and \ref{thm:main_intpaper_2},
the top horizontal arrow is the inclusion in Theorem \ref{thm:inclusion}
and the bottom arrow $\kappa$
sends $V_i^-$ to $V_i^+$ for $1\le i\le N_-$.
Comparing this with the commutative diagram in Lemma \ref{lem:diag_pullback},
we conclude that $\kappa = \varphi^*$ as required.

Finally, we explain that the conclusion of the theorem
holds for every admissible phase $\phi\in (-\alpha_0,\alpha_0)$
for $\{\bu_1(q_0),\dots,\bu_{N_+}(q_0)\}$.
As we reviewed in Section~\ref{subsec:MRS},
the analytic decompositions~\eqref{eq:analytic_lifts_pm}
for $\frX_\pm$ change by mutation as the phase
$\phi$ varies; see \cite[Sections~2.6, 4.2 and~4.3]{GGI:gammagrass}.
In the case at hand, the analytic decompositions
are given by the basis $\{s_j^\pm\}$ of
flat sections associated with the $K$-classes $\{V_j^\pm\}$;
they give rise to an asymptotic basis in the sense of Section~\ref{subsec:MRS}.
Their all possible mutations are completely determined by the configuration
$\{\bu_1(q_0),\dots,\bu_{N_\pm}(q_0)\}$ of the critical values
and the Euler pairings $\chi(V_i^\pm,V_j^\pm)$.
We have $\chi(V_i^+,V_j^+) = \chi(V_i^-, V_j^-)$
for $1\le i,j\le N_-$ since $V_i^+ = \varphi^*V_i^-$
and $\varphi$ is birational.
The condition (c) above ensures that the basis
$\{s_1^-,\dots,s_{N_-}^-\}$ of flat sections for $\frX_-$
and the part $\{s_1^+,\dots,s_{N_-}^+\}$
of the basis of flat sections for $\frX_+$ undergo
the same mutation when $\phi$ varies in $(-\alpha_0,\alpha_0)$.
Hence the relation $V_i^+ = \varphi^*V_i^-$ is preserved
under mutation. The theorem is proved.
\end{proof}

\begin{Remark}We hope that we can analyze general discrepant transformations
$\frX_+ \leftarrow \hfrX \rightarrow \frX_-$ by using the
functoriality under blowups (Theorem~\ref{thm:functoriality})
twice. In the case of crepant toric wall crossings,
we can see how a Fourier--Mukai transformation (as discussed in \cite{Borisov-Horja:FM, CIJ})
arises from this result, see the slides from~\cite{Iritani:MIT}.
\end{Remark}

\subsection{Orlov's decomposition and analytic lift} \label{subsec:Orlov_lift}

In the previous section, we have observed that
the Lefschetz thimbles associated with
`convergent' critical points (at some point $q_0\in \cVss_0 \cap \cM_\R$
and for some phase $\phi$) correspond to $K$-classes from
$\varphi^*(K(\frX_-))\subset K(\frX_+)$.
In this section, we see that the remaining `divergent' critical points correspond
to $K$-classes supported on the exceptional divisor of $\varphi \colon
\frX_+ \to \frX_-$.
We will see the correspondence between
Orlov's semiorthogonal decomposition
of the derived category $D^b(\frX_+)$ and the analytic
decomposition \eqref{eq:analytic_lift_decomp} of the quantum D-module of $\frX_+$.

\subsubsection{Decomposition of the relative homology mirror to Orlov's decomposition}
Recall from Lemma \ref{lem:slide}(3) that convergent
critical values over $\slide_t(\Dsm)$ are contained in $B_{\rho_2}(0)$ and
divergent critical values over $\slide_t(\Dsm)$ are
contained in the $J$ `satellite' discs $\frB_k(t):= t^{-1/J} \cdot
B_{\rho_1}\big(3\rho_0 e^{(-2k-1)\pi\iu/J}\big)$, $k\in \Z/J\Z$
(see Fig.~\ref{fig:conv_div_critv}).
We choose $t_1>0$ small enough so that
$\overline{B_{\rho_2}(0)}+\R_{\ge 0} $ does not intersect
any other satellite discs $\overline{\frB_k(t)}$ when $t$ is real and
$0<t\le t_1$. Then, if $|\phi|$ is sufficiently small, we have that
\[
\cS_{\rm conv} := \overline{B_{\rho_2}(0)} + \R_{\ge 0} e^{\iu\phi}
\]
does not intersect with any $\overline{\frB_k(t)}$ when $0<t\le t_1$.
We choose a number $h\in \{0,1,2,\dots, J\}$ and
let $\cS_k\subset \C$, $k=-h,-h+1,\dots,J-h-1$ be mutually disjoint
strip regions as in Fig.~\ref{fig:connecting_domains};
each $\cS_k$ is a closed region disjoint from $\cS_{\rm conv}$,
emanating from $\overline{\frB_k(t)}$,
going around the origin by the angle $(2k-1)\pi/J+\phi$ anticlockwise,
and extending straight toward the direction $e^{\iu\phi}$ near the end.
These regions $\cS_{\rm conv}$, $\cS_k$ depend on $t$, $\phi$
and are defined when $0<t\le t_1$ and $|\phi|$ is sufficiently small.

\begin{figure}[ht]\centering
\includegraphics{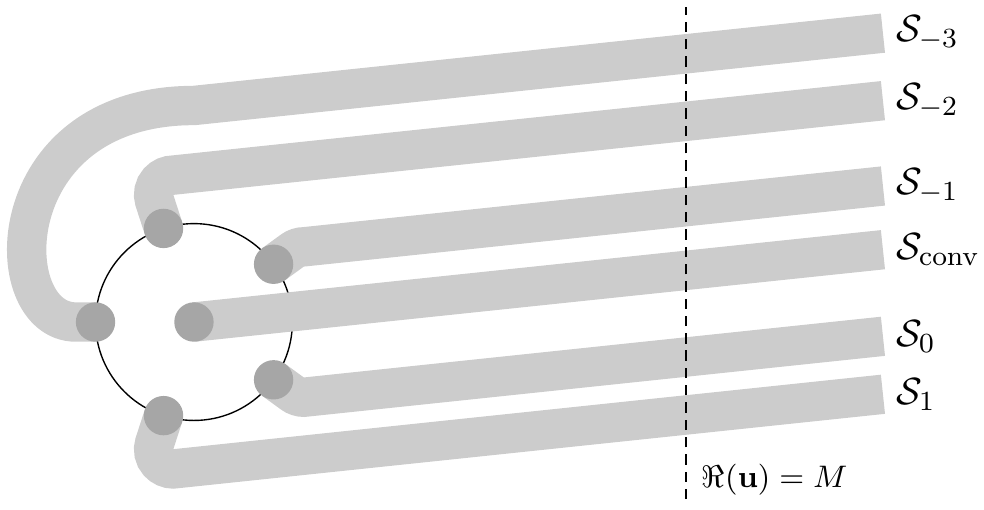}
\caption{Strip regions $\cS_{\rm conv}$ and $\cS_k$. In this picture, $J=5$, $h=3$ and $\phi=\frac{\pi}{30}$.}\label{fig:connecting_domains}
\end{figure}

For sufficiently large $M\gg 0$,
$\bigcup\limits_{k=-h}^{J-h-1} \cS_k \cup \cS_{\rm conv}
\cup \{\Re(\bu)\ge M\}$ is a strong deformation retract of $\C$.
The fibration given by $F_{q,t}$ is smoothly trivial outside this region
by Corollary \ref{cor:locally_trivial}, and hence we get
the decomposition of the relative homology $\Lef_{q,t}$ (see \eqref{eq:Lef}):
\begin{equation}
\label{eq:Orlov_mirror}
\Lef_{q,t} =
\Lef^{(-h)}_{q,t} \oplus \cdots \oplus \Lef^{(-1)}_{q,t}
\oplus
\Lef_{q,t}^{\rm conv}
\oplus
\Lef_{q,t}^{(0)} \oplus \cdots \oplus \Lef_{q,t}^{(J-h-1)}
\end{equation}
for $q\in \Dsm^\times$ and $0<t\le t_1$, where
we use the notation as in Sections~\ref{subsubsec:local_around_0-}--\ref{subsubsec:Kouchnirenko_Morse} and
\begin{gather*}
\Lef_{q,t}^{\rm conv} =
H_n\big(F_{q,t}^{-1}(\cS_{\rm conv}),
F_{q,t}^{-1}(\cS_{\rm conv})
\cap \{x\colon \Re(F_{q,t}(x))\ge M\}; \Z\big), \\
\Lef_{q,t}^{(k)} =
H_n\big(F_{q,t}^{-1}(\cS_k),
F_{q,t}^{-1}(\cS_k) \cap
\{x\colon \Re(F_{q,t}(x))\ge M\};
\Z\big).
\end{gather*}
This decomposition is independent of $\phi$
(with $|\phi|$ sufficiently small)
and is preserved by the Gauss--Manin connection.
We say that a direct sum decomposition $A=A_1\oplus \cdots \oplus A_m$
is \emph{semiorthogonal} with respect to a (not necessarily
symmetric) pairing
$[\cdot,\cdot)$ on $A$ if $[a_i,a_j) = 0$ for $a_i\in A_i$,
$a_j\in A_j$ whenever $i>j$.
Note that the decomposition \eqref{eq:Orlov_mirror} is semiorthogonal
with respect to the pairing $\#\big(e^{-\pi\iu}\Gamma_1\cdot \Gamma_2\big)$
discussed in Remark~\ref{rem:intersection_Euler}.

Let $E\subset \frX_+$ be the exceptional divisor of the
morphism $\varphi \colon \frX_+ \to \frX_-$; this
is the toric divisor corresponding to $\hb\in R_+$.
Let $Z = \varphi(E)$ denote the image of $E$;
this is a toric substack corresponding to the cone
$\sigma_{M_+}$ of $\bSigma_-$ (see~\eqref{eq:varphi_coord}).
We write $\varphi_E \colon E\to Z$ for the restriction of
$\varphi$ to $E$ and $i_E \colon E\to \frX_+$ for the inclusion.

\begin{Theorem}\label{thm:decomp_Lef_K}
For $q\in \Dsm^\times \cap \cMsm_\R$ and $0<t\le t_1$,
the decomposition~\eqref{eq:Orlov_mirror} of the relative
homology corresponds to the decomposition of the $K$-group
\begin{equation}\label{eq:SOD_K}
K(\frX_+) = K(Z)_{-h} \oplus \cdots \oplus K(Z)_{-1} \oplus
\varphi^*K(\frX_-) \oplus K(Z)_0 \oplus
\cdots \oplus K(Z)_{J-h-1}
\end{equation}
under the isomorphism
$\Lef_{q,t} \cong K(\frX_+)$
in Theorem $\ref{thm:main_intpaper}(1)$,
where $K(Z)_k$ denotes the subgroup
$\cO(-kE) \otimes i_{E*}\varphi_E^* K(Z)
\subset K(\frX_+)$.
\end{Theorem}
\begin{Remark}
The decomposition \eqref{eq:SOD_K} is semiorthogonal
with respect to the Euler pairing by Remark \ref{rem:intersection_Euler}
and the above theorem.
In view of this,
we expect that $D^b(\frX_+)$ admits a~semiorthogonal decomposition:
\[
D^b(\frX_+) = \big\langle D^b(Z)_{-h},\cdots, D^b(Z)_{-1},\varphi^* D^b(\frX_-), D^b(Z)_0, \dots, D^b(Z)_{J-h-1} \big\rangle,
\]
where $D^b(Z)_j$ denotes the full subcategory of $D^b(\frX_+)$
which is the image of the functor
$\cO(-j E) \otimes i_{E*} \varphi_E^*\colon
D^b(Z) \to D^b(\frX_+)$.
When $\frX_+$ is the blowup of a smooth \emph{variety}
$\frX_-$ along a smooth centre~$Z$, this is the semiorthogonal decomposition
proved by Orlov \cite[Theorem~4.3]{Orlov:proj};
Orlov stated the result for $h=J$, but the other cases $0\le h\le J-1$ follow
from this case by using the fact that $D^b(Z)_{-h} =
\bS\big(D^b(Z)_{J-h}\big)$, where $\bS$ is the Serre functor for $D^b(\frX_+)$.
\end{Remark}
\begin{proof}[Proof of Theorem $\ref{thm:decomp_Lef_K}$]
Observe that the subgroup $\Lef^{\rm conv}_{q,t}\subset \Lef_{q,t}$
is the image of the map $\iota \colon \Lefsm_q\to
\big(\slide_t^{-1}\Lef\big)_q$ in Theorem~\ref{thm:inclusion}
(e.g., by considering the case $\phi=0$).
Hence it corresponds to~$\varphi^*(K(\frX_-)) \subset K(\frX_+)$
by Lemma~\ref{lem:diag_pullback}.
It suffices to show that
$\Lef_{q,t}^{(k)}$ corresponds to~$K(Z)_k$ for $-h\le k\le J-h-1$.
Let $K^{(k)}\subset K(\frX_+)$ denote the subgroup
corresponding to $\Lef_{q,t}^{(k)}\subset \Lef_{q,t}$.
It suffices to show that
\begin{itemize}\itemsep=0pt
\item[(a)] $K(Z)_k \subset K^{(k)}$;
\item[(b)] the decomposition~\eqref{eq:SOD_K} of the $K$-group holds true.
\end{itemize}
As discussed, over $\slide_t(\Dsm)$, convergent critical values
are contained in $B_{\rho_2}(0)$ and divergent critical values are
contained in the $J$ satellite discs $\frB_k(t)$, $k\in \Z/J\Z$;
these statement still hold for a non-zero complex number~$t$ with $|t|\le 1$
by naturally extending the definition of~$\slide_t$ and~$\frB_k(t)$
to $t\in \C$ (see the proof of Lemma~\ref{lem:slide}).
Consider the class $[\Gamma_\R] \in \Lef_{q,t}$ of the positive
real Lefschetz thimble. It corresponds to the class $[\cO] \in K(\frX_+)$
of the structure sheaf.
We study the monodromy action on $\Lef_{q,t}$ along the loop
$[0,2\pi]\ni \theta \mapsto \big(q,e^{\iu\theta}t\big)$.
The monodromy corresponds to
the tensor product by $\cO(- E)$ on $K(\frX_+)$;
this follows from the same argument as in the proof of Lemma~\ref{lem:diag_pullback}
and the fact that this loop corresponds to the element of $\LL^\star$
given by the second projection in the decomposition~\eqref{eq:LL_LL'}
(which equals \mbox{$D_\hb\in \LL^\star$}).
Under this monodromy, $[\Gamma_\R]$ undergoes the Picard--Lefschetz
transformation. Among the satellite discs, only $\frB_{-1}(t)$ passes through
the strip region $\cS_{\rm conv}$ as $t$ rotates;
here $\frB_{-1}(t)$ becomes $\frB_{-1}\big(e^{2\pi\iu} t\big) = \frB_0(t)$
after rotation. Thus by the Picard--Lefschetz formula,
the monodromy transform of $[\Gamma_\R]$ equals
$[\Gamma_\R] - \alpha_0$ for some $\alpha_0 \in \Lef_{q,t}^{(0)}$.
Then we find that $\alpha_0$ corresponds to $[\cO] - [\cO(-E)] = [\cO_E]
\in K(\frX_+)$ and in particular that it is non-zero.
Similarly, by considering the monodromy around the inverse loop
$\theta \mapsto (q,e^{-\iu\theta} t)$, we find an element
$\beta_1 \in \Lef_{q,t}^{(-1)}$ that corresponds to
$[\cO(E)] - [\cO] = [\cO_E(E)]\in K(\frX_+)$.
By considering further monodromy actions on $\alpha_0,\beta_1$,
we find elements $\alpha_i \in \Lef_{q,t}^{(i)}$ with $0\le i\le J-h-1$
and $\beta_j \in \Lef_{q,t}^{(-j)}$ with $1\le j\le h$ such that
$\alpha_i$ corresponds to $[\cO_E(-i E)]$ and that
$\beta_j$ corresponds to $[\cO_E(jE)]$.
Hence we have
\begin{equation}
\label{eq:str_sheaf_E}
[\cO_E(-kE)] = \cO(-kE) \otimes [\cO_E] \in K^{(k)}
\end{equation}
for $-h\le k \le J-h-1$.
Note that the decomposition \eqref{eq:Orlov_mirror} is invariant
under monodromy in $q\in \Dsm^\times$.
As discussed in the proof of Lemma \ref{lem:diag_pullback},
the monodromy around loops in $\Dsm^\times$ corresponds
to the tensoring $\varphi^*L$ with $L\in \Pic(\frX_-)$.
Since $K(\frX_-)$ is generated by line bundles, we
conclude from \eqref{eq:str_sheaf_E} that
$\cO(-kE) \otimes \cO_E \otimes \varphi^*V \in K^{(k)}$
for all $V\in K(\frX_-)$. Part (a) follows from the fact that
$K(Z)$ is also generated by line bundles and
that the natural map $\Pic(\frX_-) \to \Pic(Z)$
is surjective.

It remains to prove Part~(b). In view of part (a) and
the decomposition~\eqref{eq:Orlov_mirror},
it suffices to show that $K(\frX_+)$ is generated by
the classes of
$\cO$ and
$\cO_E(-kE)$, $-h\le k\le J-h-1$ as a $K(\frX_-)$-module, where
the module structure is given by $\varphi^*\colon K(\frX_-)
\to K(\frX_+)$.
Let $L_b = L_{-D_b}\in K(\frX_+)$ denote the class of the line bundle
associated with $-D_b \in \LL^\star$ for $b\in S$
(see Section~\ref{subsubsec:Pic_2nd} for~$L_{-D_b}$);
$1-L_b$ is the
class of the structure sheaf of the toric divisor associated with~$b$
when $b\in R_+$.
Similarly, let $L_b^-\in K(\frX_-)$ denote the class of the line
bundle associated with $-D_b\in (\LL')^\star$ for $b\in S_-$.
We also set $L:= L_{\hb} = [\cO(-E)] \in K(\frX_+)$.
Recall from the proof of Lemma \ref{lem:diag_pullback} that
the splitting $(\LL')^\star\to \LL^\star$, $\xi\mapsto \hxi$
induced by the decomposition \eqref{eq:LL_LL'} corresponds to
the pull-back of line bundles, i.e., $\varphi^*L_{\xi} \cong L_{\hxi}$.
Since the splitting $(\LL')^\star \to \LL^\star$ sends
$D_b$ to $D_b + k_b D_\hb$
(with $k_b = -D_b\cdot \delta_\hb^{\bSigma_-}$ as before),
we have
\begin{equation}
\label{eq:pullback_Lb-}
\varphi^*L_b^- = L_b \cdot L^{k_b}.
\end{equation}
On the other hand, since the intersection of the toric divisors
corresponding to rays $b\in M_+$ is empty in $\frX_+$,
we have the following relation in $K(\frX_+)$
(see \cite{Borisov-Horja:K}):
\begin{equation}
\label{eq:M+_relation}
\prod_{b\in M_+} (1-L_b) = 0.
\end{equation}
Combining the two relations \eqref{eq:pullback_Lb-}, \eqref{eq:M+_relation}, we obtain
\begin{equation}\label{eq:relation_L}
\prod_{b\in M_+} \big(L^{k_b}- \varphi^*L_b^-\big) = 0.
\end{equation}
The left-hand side of the relation is a monic polynomial in~$L$
of degree $J+1 = \sum\limits_{b\in M_+}k_b$ with coefficients
in $K(\frX_-)$
whose constant term is an invertible element in $K(\frX_-)$.
Note that $\Pic(\frX_+)$ is generated by $\varphi^*\Pic(\frX_-)$
and $L^\pm$; therefore any element in $K(\frX_+)$ can be
written as a Laurent polynomial of $L$ with coefficients in
$K(\frX_-)$. Using the above relation \eqref{eq:relation_L},
we find that every element of $K(\frX_+)$ can be written in the form
\[
\sum_{k=-h}^{J-h} a_k \cdot L^k \qquad \text{with $a_k \in K(\frX_-)$},
\]
which can also be rewritten as a $K(\frX_-)$-linear combination
of $1$ and $(1-L) L^k=[\cO_E(-kE)]$ with $-h\le k\le J-h-1$. Part (b) follows. The theorem is proved.
\end{proof}

\subsubsection{Orlov's decomposition as a sectorial decomposition of the quantum D-module}
Next we see that this Orlov-type decomposition actually
arises from a sectorial decomposition of the quantum D-module of~$\frX_+$~-- as appears in Proposition~\ref{prop:Hukuhara-Turrittin}~--
associated with \emph{some} $\tau_+\in H^*_\CR(\frX_+)$;
the point $\tau_+$ can be very far from the large radius limit
point and is not explicit in general (see however Example~\ref{exa:blowup_P4}).

The analytic quantum D-module~\eqref{eq:analytic_QDM} is originally defined in a neighbourhood of the large radius limit point. In the following theorem, we consider analytic continuation of the quantum D-module along certain paths in $H^*_\CR(\frX_\pm)$; we choose a real class $\tau_{\star,\pm} \in H^*_\CR(\frX_\pm;\R)$ which is sufficiently close to the large radius limit point as a base point (see Remark~\ref{rem:branch_L}).

\begin{Theorem}\label{thm:sectorial_Orlov}There exist paths from $\tau_{\star,\pm}$ to $\tau_\pm \in H^*_\CR(\frX_\pm)$ and analytic continuation of the quantum D-modules $\QDMan(\frX_\pm)$ along these paths with the following properties.
\begin{itemize}\itemsep=0pt
\item[{\rm (1)}] The eigenvalues of the Euler multiplication $E\star_{\tau_+}$
at $\tau_+$ have mutually distinct imaginary parts, and
we have the following sectorial decomposition
as in Proposition $\ref{prop:Hukuhara-Turrittin}$
\[
\Phi^+ \colon \ \pi^*\QDMan(\frX_+)\big|_{W \times I} \cong \bigoplus_{i=1}^{N_+}
(\cA_{W\times I}, d + d(\bu_i/z))
\]
over an open neighbourhood $W$ of $\tau_+$ in $H_\CR^*(\frX_+)$,
where $\pi\colon H^*_\CR(\frX_+)\times \tCC \to H^*_\CR(\frX_+)
\times \C$ is the oriented real blow-up,
$I= \big\{\big(z,e^{\iu\theta}\big)\in \tCC\colon |\theta|<\frac{\pi}{2}+\epsilon\big\}$
$($for some $\epsilon>0)$ and $\bu_1,\dots,\bu_{N_+}$
are the eigenvalues of the Euler multiplication.

\item[{\rm (2)}] Moreover, there exists a holomorphic submersion
$f\colon W \to H^*_\CR(\frX_-)$ with $f(\tau_+) = \tau_-$
such that we have the similar sectorial decomposition for $\frX_-$
\[
\Phi^- \colon \ \pi^*f^*\QDMan(\frX_-)\big|_{W\times I}
\cong \bigoplus_{i=1}^{N_-} (\cA_{W\times I}, d + d(\bu_i/z)),
\]
where $W$ and $I$ are the same as Part~$(1)$ and
$\bu_1,\dots,\bu_{N_-}$ are the pull-backs along~$f$
of the eigenvalues of the Euler multiplication in the quantum
cohomology of $\frX_-$ which form a subset of
$\{\bu_1,\dots,\bu_{N_+}\}$ in Part~$(1)$.

\item[{\rm (3)}] The eigenvalues $\{\bu_1,\dots,\bu_{N_+}\}$ in Part~$(1)$
are divided into $J+1$ groups $\bigsqcup\limits_{k=-h}^{J-h-1} \big\{\bu^{(k)}_i\colon 1\le i\le r\big\}
\sqcup \{\bu_1,\dots,\bu_{N_-}\}$ satisfying
\[
\Im\big(\bu^{(-h)}_{i_1}\big) > \cdots > \Im\big(\bu^{(-1)}_{i_h}\big)>
\Im(\bu_j) > \Im\big(\bu^{(0)}_{i_{h+1}}\big) > \cdots >
\Im\big(\bu^{(J-h-1)}_{i_J}\big)
\]
for all $i_{k} \in\{1,\dots,r\}$ $($with $1\le k\le J)$
and $j\in \{1,\dots,N_-\}$.

\item[{\rm (4)}] There exist $K$-classes
$V_i^\pm\in K(\frX_\pm)$, $1\le i\le N_\pm$ such that
$\Phi^{\pm}\big(e^{\bu_i/z}s_i^\pm\big) = e_i$, where $s_i^\pm$
is the flat section associated with $V_i^\pm$
via the $\hGamma$-integral structure $($Definition~$\ref{def:s})$
$($which is analytically continued from $\tau_{\star,\pm}$
through the specified paths$)$
and~$e_i$ denotes the $i$th standard basis.
Moreover, we have
\begin{itemize}\itemsep=0pt
\item[$({\rm a})$] $V_i^-$, $1\le i\le N_-$ form a basis of $K(\frX_-)$;
\item[$({\rm b})$] $\big\{ V_i^+\colon
\text{$\bu_i = \bu^{(k)}_j$ for some $j$}\big\}$
gives a basis of $K(Z)_k\subset K(\frX_+)$;
\item[$({\rm c})$] $V_i^+ = \varphi^*V_i^-$ for
$1\le i\le N_-$.
\end{itemize}
\end{itemize}
In particular, we have a sectorial decomposition
\begin{gather*}
\pi^*\QDMan(\frX_+)\big|_{W \times I}\! \cong
\scrRan_{-h} \oplus \cdots \oplus \scrRan_{-1} \oplus
\pi^*f^*\QDMan(\frX_-)\big|_{W\times I} \oplus
\scrRan_{0}\oplus \cdots \oplus \scrRan_{J-h-1},
\end{gather*}
which corresponds to the decomposition~\eqref{eq:SOD_K}
of the $K$-group via the $\hGamma$-integral structure,
where $\scrRan_k = \bigoplus\limits_{i=1}^{r_k}
\big(\cA_{W\times I}, d + d\big(\bu_i^{(k)}/z\big)\big)$.
\end{Theorem}

\begin{proof}Throughout Section~\ref{sec:functoriality},
we have assumed that $S_- = S\cap \Delta_-$ and
$S= S_- \cup \{\hb\}$ so that~$\cM$ is the small quantum
cohomology locus of~$\frX_+$.
The construction of the sectorial decompositions~\eqref{eq:analytic_lifts_pm},
however, does not require this assumption
because $S$ can be arbitrarily large in the discussion of
Sections~\ref{sec:discrepant}--\ref{sec:lift}.
We choose a large enough finite set $\hS\subset \bN\cap \Pi$
containing $S$ such that
the corresponding mirror map $\mir_+$ is generically
submersive (this is possible by the argument in \cite[Section~7.4]{CCIT:MS}).
Let $\hcM\supset \cM$ denote the base of the LG model
corresponding to $\hS$.
The discussion at the beginning of Section~\ref{subsec:functoriality}
yields, for any $q_0\in \cVss_0$ and
an admissible phase $\phi$ for the critical values of $F_{q_0}$,
the following
decompositions
(similarly to~\eqref{eq:analytic_lifts_pm})
\begin{align}
\label{eq:analytic_lifts_decomp_large}
\pi^*\mir_\pm^*\QDMan(\frX_\pm)\big|_{\hB\times I_\phi}
& \cong \bigoplus_{i=1}^{N_\pm}
(\cA_{\hB\times I_\phi}, d+d(\bu_i/z))
\end{align}
over a neighbourhood $\hB$ of $q_0$ in $\hcM$.
Here $(\bu_1,\dots,\bu_{N_+})$ are the eigenvalues of $E\star_\tau$
in the quantum cohomology of $\frX_+$ at $\tau = \mir_+(q)$,
and first $N_-$ $(\bu_1,\dots,\bu_{N_-})$ of them are the eigenvalues
of $E\star_\tau$ in the quantum cohomology of $\frX_-$ at
$\tau = \mir_-(q)$, where $q$ varies in $\hB$.
By the choice of $\hS$, the eigenvalues $\bu_1,\dots,\bu_{N_+}$ are
pairwise distinct at generic points in $\hB$
(see Remark \ref{rem:generically_tame_semisimple}).

Suppose now that $q_0$ is a point from
$\bigcup\limits_{0<t<t_1}\slide_t\big(\cVsmss_{-,\R}\cap \intDsmtimes\big)
\cap \cVss_0$
and~$|\phi|$ is sufficiently small so that the conclusion of
Theorem~\ref{thm:functoriality} holds; then
there exist $K$-classes $V_i^{\circ \pm}\in K(\frX_\pm)$ such that the
associated flat sections $s_i^\pm$ (via the $\hGamma$-integral
structure) correspond to~$e^{-\bu_i/z} e_i$
under the decomposition~\eqref{eq:analytic_lifts_decomp_large}
and that $V_i^{\circ+} = \varphi^* V_i^{\circ-}$ for
$1\le i\le N_-$.
Choose a point $q^\circ \in \hB$ such that the corresponding
eigenvalues $\bu_1^\circ,\dots,\bu_{N_+}^\circ$
are pairwise distinct, and set $\tau^\circ_\pm = \mir_\pm(q^\circ)$;
we regard $\tau^\circ_\pm$ as elements of $H_\CR^*(\frX_\pm)$
(rather than their images in $[\cM_{\rm A}(\frX_\pm)/\Pic^\st(\frX_\pm)]$).
Let $C_N$ denote the configuration space of distinct $N$ points
in $\C$:
\[
C_N := \big\{(\bu_1,\dots,\bu_N)\in \C^N\colon
\text{$\bu_i \neq \bu_j$ if $i\neq j$}\big\}/\frS_N.
\]
Since the eigenvalues of $E\star_\tau$ form a local co-ordinate
system on $H^*(\frX_\pm)$ near a semisimple point,
we can identify a neighbourhood of
$\big(\bu_1^\circ,\dots,\bu_{N_\pm}^\circ\big)$ in $C_{N_\pm}$
with a neighbourhood of $\tau_\pm^\circ$ in $H^*_{\CR}(\frX_\pm)$.
Let $\tC_N$ denote the universal cover of $C_N$.
By isomonodromic deformation
\cite[Theorem 4.7]{Dubrovin:Painleve},
there exist analytic hypersurfaces\footnote
{In the original version of the present paper, we did not delete
the hypersurfaces $H_\pm$ from $\tC_{N_\pm}$. The need
for the deletion was pointed out by one of the referees. } $H_\pm \subset \tC_{N_\pm}$
such that the quantum connection in a neighbourhood of
$\tau_\pm^\circ \in H^*_\CR(\frX_\pm)$
can be extended to a meromorphic flat connection $\nabla$ on the trivial bundle
$H^*_\CR(\frX_\pm) \times \big(\tC_{N_\pm}^\circ \times \C\big)
\to \tC_{N_\pm}^\circ \times \C$ of the form
\[
\nabla = d + \frac{1}{z} A +\left(- \frac{1}{z^2} U + \frac{1}{z} V \right) dz,
\]
where $\tC_{N_\pm}^\circ := \tC_{N_\pm} \setminus H_\pm$,
$A$ is an $\End(H^*_\CR(\frX_\pm))$-valued 1-form
on $\tC_{N_\pm}^\circ$,
$U$ and $V$ are \linebreak $\End(H^*_\CR(\frX_\pm))$-valued
functions on $\tC_{N_\pm}^\circ$ and
$A$, $U$, $V$ are independent of~$z$.
Here the eigenva\-lues of $U$ give the co-ordinates
$(\bu_1,\dots,\bu_{N_\pm})$ on $\tC_{N_\pm}$.
Moreover, by taking bigger hypersurfaces~$H_\pm$
if necessary, we may assume that this defines a Frobenius manifold structure
on $\tC^\circ_{N_\pm}$.
By choosing a basis $\{\phi_i\}$ of $H^*_\CR(\frX_\pm)$,
we can determine the flat vector fields $\parfrac{}{\tau^i}$
by $A_{\partial/\partial \tau^i} 1 = \phi_i$. In particular,
we have a flat co-ordinate system
$\big(\tC^\circ_{N_\pm}\big)\sptilde \to H^*_\CR(\frX_\pm)$ on
the universal cover $\big(\tC^\circ_{N_\pm}\big)\sptilde$ which
equals $\tau^\circ_\pm$ at $\big(\bu^\circ_1,\dots,\bu^\circ_{N_\pm}\big)$.
We also have a submersion $g \colon \tC_{N_+} \to \tC_{N_-}$
sending $(\bu_1,\dots,\bu_{N_+})$ to $(\bu_1,\dots,\bu_{N_-})$;
this induces a submersion $g \colon \tC_{N_+}^{\circ\circ} \to \tC_{N_-}^\circ$,
where $\tC_{N_+}^{\circ\circ} := \tC_{N_+} \setminus \big(H_+ \cup g^{-1}(H_-)\big)\subset \tC_{N_+}^\circ$.

We may assume that $q^\circ$ is sufficiently close to $q_0$
so that $\big\{\bu_1^\circ,\dots,\bu_{N_+}^\circ\big\}$ are contained
in $B_{\rho_2}(0) \sqcup \bigsqcup\limits_{k\in \Z/J\Z} \frB_k(t)$,
where $0<t<t_1$ is the number such that $q_0\in \Image(\slide_t)$.
Moving $\bu_i^\circ$
along the strip regions $\cS_{\rm conv}\cup \cS_{-h}\cup
\cdots\cup \cS_{J-h-1}$,
we can connect $\big(\bu_1^\circ,\dots,\bu_{N_+}^\circ\big)$
by a continuous path $\gamma$ inside $\tC_{N_+}^{\circ\circ}$
with a point $\big(\bu^\Diamond_1,\dots,\bu^\Diamond_{N_+}\big)
\in \tC_{N_+}^{\circ\circ}$
having mutually distinct imaginary parts: see Fig.~\ref{fig:moving_u}.
The numbers $\bu^\Diamond_1,\dots,\bu^\Diamond_{N_+}$
are divided into $J+1$
groups as in Part~(3) of the statement, depending on which strip
regions they belong to.
Let $\tau_+\in H^*_\CR(\frX_+)$ be the analytically continued
(along the path $\gamma$)
flat co-ordinate at the point $\big(\bu^\Diamond_1,\dots,\bu^\Diamond_{N_+}\big)\in
\tC_{N_+}^{\circ\circ}$
and $\tau_-\in H^*_\CR(\frX_-)$ be the analytically continued
(along the path $g(\gamma)$)
flat co-ordinate at the
point $g\big(\bu^\Diamond_1,\dots,\bu^\Diamond_{N_+}\big)=
\big(\bu^\Diamond_1,\dots,\bu^\Diamond_{N_-}\big) \in \tC_{N_-}^\circ$.
By the above construction, the quantum D-modules of $\frX_\pm$
are analytically continued to $\tau_\pm$ along certain paths, and
the submersion~$g$ induces, when written in flat co-ordinates,
a submersion~$f$ from a neighbourhood of $\tau_+$ in $H^*_\CR(\frX_+)$ to a~neighbourhood of $\tau_-$
in $H^*_\CR(\frX_-)$ with $f(\tau_+) = \tau_-$.

\begin{figure}[htbp]\centering
\includegraphics{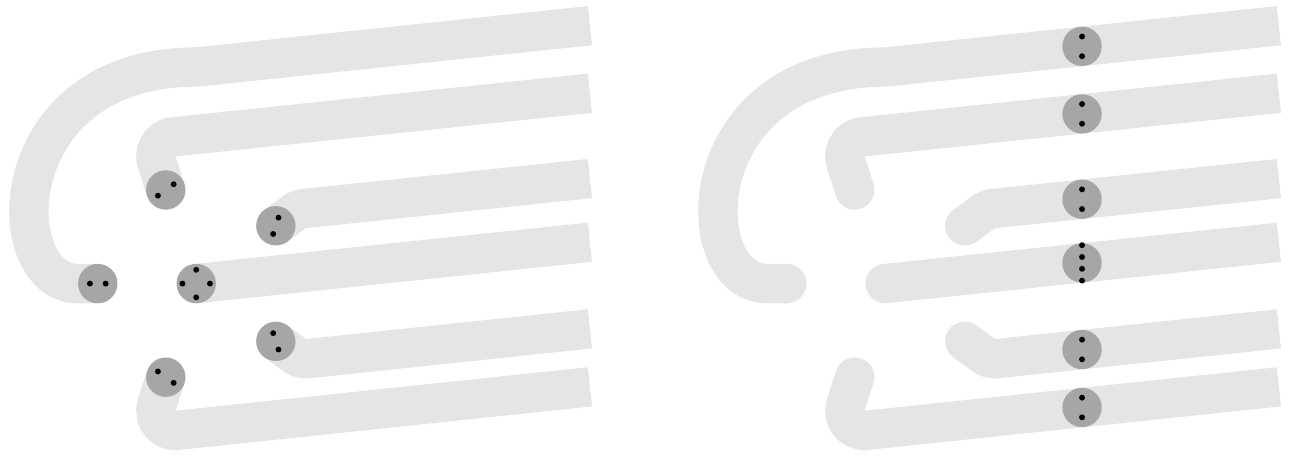}
\caption{A path from $\big(\bu_1^\circ,\dots,\bu_{N_+}^\circ\big)$
to $\big(\bu_1^\Diamond,\dots,\bu_{N_+}^\Diamond\big)$:
we move the left configuration $\big(\bu_i^\circ\big)$ to the
right one $\big(\bu_i^\Diamond\big)$ along the strip regions.} \label{fig:moving_u}
\end{figure}

We claim that the statement of the theorem holds for
$\tau_\pm$, $f$ and a sufficiently small neighbourhood $W$
of $\tau_+$. Parts (1), (2) follow from the construction
and the Hukuhara--Turrittin decomposition
(Proposition \ref{prop:Hukuhara-Turrittin}); note that $0$
is an admissible phase for $\big(\bu_1^\Diamond,\dots,\bu_{N_\pm}^\Diamond\big)$.
Part (3) has been already achieved.
We show Part~(4).
Recall that the decomposition~\eqref{eq:analytic_lifts_decomp_large}
corresponds to the basis $\big\{V_i^{\circ\pm}\big\}$ of $K(\frX_\pm)$
under the $\hGamma$-integral structure and it gives rise to
the asymptotic basis (see Section~\ref{subsec:MRS})
associated with $\tau_\pm^\circ$ and $\phi$.
As we move $(\bu_1,\dots,\bu_{N_\pm})$
along the path~$\gamma$ (or~$g(\gamma)$) and change
the phase from $\phi$ to zero, this
asymptotic basis undergoes mutation as explained in Section~\ref{subsec:MRS}.
The basis $\big\{V_i^\pm\big\}\subset K(\frX_\pm)$ in Part~(4)
arises from $\big\{V_i^{\circ\pm}\big\}$ by this mutation.
Part~(4-a) is obvious from this construction.
By the same argument as in the last paragraph of the proof
of Theorem~\ref{thm:functoriality}, we see that the relation
$V_i^{\circ +} = \varphi^* V_i^{\circ-}$ is preserved by
mutation, and we get $V_i^+ = \varphi^* V_i^-$ for
$1\le i\le N_-$ as required in Part~(4-c).
Let $\Gamma_i^{\pi+\phi}$ denote the Lefschetz thimbles~\eqref{eq:Lefschetz} of $F_{q_0}$ corresponding to the $i$th component of the decomposition~\eqref{eq:analytic_lifts_decomp_large}
(recall that components in~\eqref{eq:analytic_lifts_decomp_large}
are naturally indexed by critical points of $F_{q_0}$;
see Proposition~\ref{prop:analytic_lift_Bri}).
As discussed in the proof of Proposition~\ref{prop:lift_str_sheaf_+},
$V_i^{\circ+}$ equals the $K$-class $V\big(\Gamma_i^{\pi+\phi}\big)$ associated
with $\Gamma_i^{\pi+\phi}$ under the correspondence in
Theorem \ref{thm:main_intpaper}.
Note by Remark \ref{rem:intersection_Euler} that
\[
\chi\big(V_i^{\circ+},V_j^{\circ+}\big) =
(-1)^{n(n-1)/2}\#\big(e^{-\pi\iu} \Gamma_i^{\pi+\phi}
\cdot \Gamma_j^{\pi+\phi}\big).
\]
Therefore, mutation of $\big\{V_i^{\circ+}\big\}$ corresponds to
Picard--Lefschetz transformation (see \cite[Chapter~I]{AGV:Singularity_II})
of thimbles $\big\{\Gamma_i^{\pi+\phi}\big\}$; here
we consider Picard--Lefschetz transformation
with $q_0\in \cM$ fixed.
We transform $\big\{\Gamma_i^{\pi+\phi}\big\}$ by the
sequence of Picard--Lefschetz transformations
corresponding to the sequence of mutations that
$\big\{V_i^{\circ+}\big\}$ undergoes and obtain
a new basis $\{\Gamma_i\}$ of $\Lef_{q_0}$.
The basis~$\{\Gamma_i\}$ is the union of bases
of $\Lef^{(k)}_{q_0}$ and $\Lef^{\rm conv}_{q_0}$;
$\big\{\Gamma_i\colon \text{$\bu_i = \bu_j^{(k)}$ for some $j$}\big\}$
gives a~basis of~$\Lef_{q_0}^{(k)}$.
Since we have $V_i^+ = V(\Gamma_i)$, Part~(4-b) follows
from Theorem~\ref{thm:decomp_Lef_K}. The theorem is proved.
\end{proof}

\begin{Example}\label{exa:blowup_P4} We give an example where the Orlov-type decomposition occurs
at an explicit~$\tau_+$. Set $X_- := \PP^4$ and
let $X_+$ be the blowup of $X_-$
along a line $Z=\PP^1 \subset X_-$.
Both $X_+$ and $X_-$ are Fano, and their small quantum
cohomologies are defined over polynomial rings.
Planes in $X_-=\PP^4$ containing the line~$Z$
are parametrized by~$\PP^2$, and hence
we have a natural projection $X_+ \to \PP^2$.
Thus $X_+$ is a $\PP^2$-bundle
over $\PP^2$:
$X_+ \cong \PP_{\PP^2}(\cO\oplus \cO\oplus \cO(-1))$.
Let $p_1$ be the pull-back of the ample class in $H^2\big(\PP^2;\Z\big)$
and let $p_2=\varphi^*(p)$ be the pull-back of the ample class~$p$ in $H^2(X_-;\Z)$.
They form a nef basis of $H^2(X_+;\Z)$.
The class of the exceptional divisor is given by $[E] = p_1-p_2$.
The uncompactified LG mirror of $X_+$ is given by the
family of Laurent polynomials
\[
F_{q_1,q_2} = x_1+x_2+x_3+x_4 + \frac{q_1q_2}{x_1x_2x_3x_4} +
\frac{1}{q_1} x_1x_2 x_3
\]
parametrized by $(q_1,q_2)\in (\C^\times)^2$.
The large radius limit (LRL) point for $X_+$ corresponds to
$q_1=q_2=0$.
The mirror map for $X_+$ is trivial: $
\mir_+(q_1,q_2) = p_1 \log q_1 + p_2 \log q_2$.
On the other hand, the uncompactified LG mirror of $X_-$ is given by
the family
\[
F_{q,t} = x_1+x_2+x_3+x_4+ \frac{q}{x_1x_2x_3x_4} + t x_1x_2x_3,
\]
where the LRL point for $X_-$ is $q=t=0$.
The two families are related by the change of co-ordinates
$q= q_1q_2$, $t=q_1^{-1}$.
The small quantum cohomology locus for $X_-$ is $t=0$.
It is not easy to find a closed formula for the mirror map
of $X_-$, but we know that it has
the asymptotic form $\mir_-(q,t) \sim p \log q + t \cdot p^3 $ as
$(q,t) \to (0,0)$
(see Remark~\ref{rem:mirror_map_asymptotic})
and that $\mir_-(q,0) = p \log q$.
The Euler vector field gives a grading
$\deg q_1 = 2$, $\deg q_2 =3$ and
$\deg t = -2$, $\deg q = 5$ (in complex unit).
Define a dimensionless parameter
$\lambda := q_1^{-\frac{3}{2}} q_2 = t^{\frac{5}{2}}q$.
Critical points/values of $F_{q_1,q_2} = F_{q,t}$ are given by
\[
x_1=x_2=x_3= t^{-\frac{1}{2}} x, \qquad x_4 = t^{-\frac{1}{2}}
\big(x+ x^3\big), \qquad F_{q,t}(\crit) = t^{-\frac{1}{2}}\big(5x + 3x^3\big),
\]
where $x$ is a root of $x^5\big(x^2+1\big)^2 = \lambda$. The discriminant locus (where $x$ has multiple roots) is $\lambda =0$ or $\pm 400 \sqrt{5}\iu/3^9$.
Therefore the quantum cohomology is semisimple over
the positive real locus $q_1>0$, $q_2>0$. The critical values have the asymptotics
\begin{equation}\label{eq:criticalvalue_lambda_0}
F_{q,t}(\crit) =
\begin{cases}
 t^{-\frac{1}{2}}
\big(5 \lambda^{\frac{1}{5}}+ \lambda^{\frac{3}{5}}
+ O(\lambda)\big), \\
 t^{-\frac{1}{2}}
\big(2\iu \pm 2\sqrt{\iu} \lambda^{\frac{1}{2}}
+ O(\lambda)\big), \\
 t^{-\frac{1}{2}}
\big({-}2\iu \pm 2\sqrt{-\iu} \lambda^{\frac{1}{2}}
+O(\lambda) \big)
\end{cases}
=
\begin{cases}
 5 q^{\frac{1}{5}} + t q^{\frac{3}{5}} + O\big(t^2\big), \\
 2\iu t^{-\frac{1}{2}}
\pm 2\sqrt{\iu} q^{\frac{1}{2}} t^{\frac{3}{4}} + O\big(t^2\big),
\\
 -2\iu t^{-\frac{1}{2}}
\pm 2\sqrt{-\iu} q^{\frac{1}{2}} t^{\frac{3}{4}} + O\big(t^2\big).
\end{cases}
\end{equation}
as $\lambda \to 0$ (or $t\to 0$ with $q$ fixed)
and
\begin{align*}
F_{q,t}(\crit) = t^{-\frac{1}{2}}
\big(3 \lambda^{\frac{1}{3}}
+ 3 \lambda^{\frac{1}{9}}+O(\lambda^{-\frac{1}{9}}) \big)
= 3 q_2^{\frac{1}{3}} + 3 q_1^{\frac{1}{3}} q_2^{\frac{1}{9}}
+ O\big(q_1^{\frac{2}{3}}\big)
\end{align*}
as $\lambda \to \infty$ (or $q_1\to 0$ with $q_2$ fixed).
The first line of~\eqref{eq:criticalvalue_lambda_0} gives
5 `convergent' critical values (depending on the choice of~$q^{\frac{1}{5}}$)
and the next two lines give 4 `divergent' critical values
around $t=0$. They have mutually distinct imaginary parts
when $t>0$ is sufficiently close to zero and $q>0$.
We will identify the corresponding elements in the $K$-group
for the phase $\phi =0$.

\begin{figure}[htb]\centering
\includegraphics{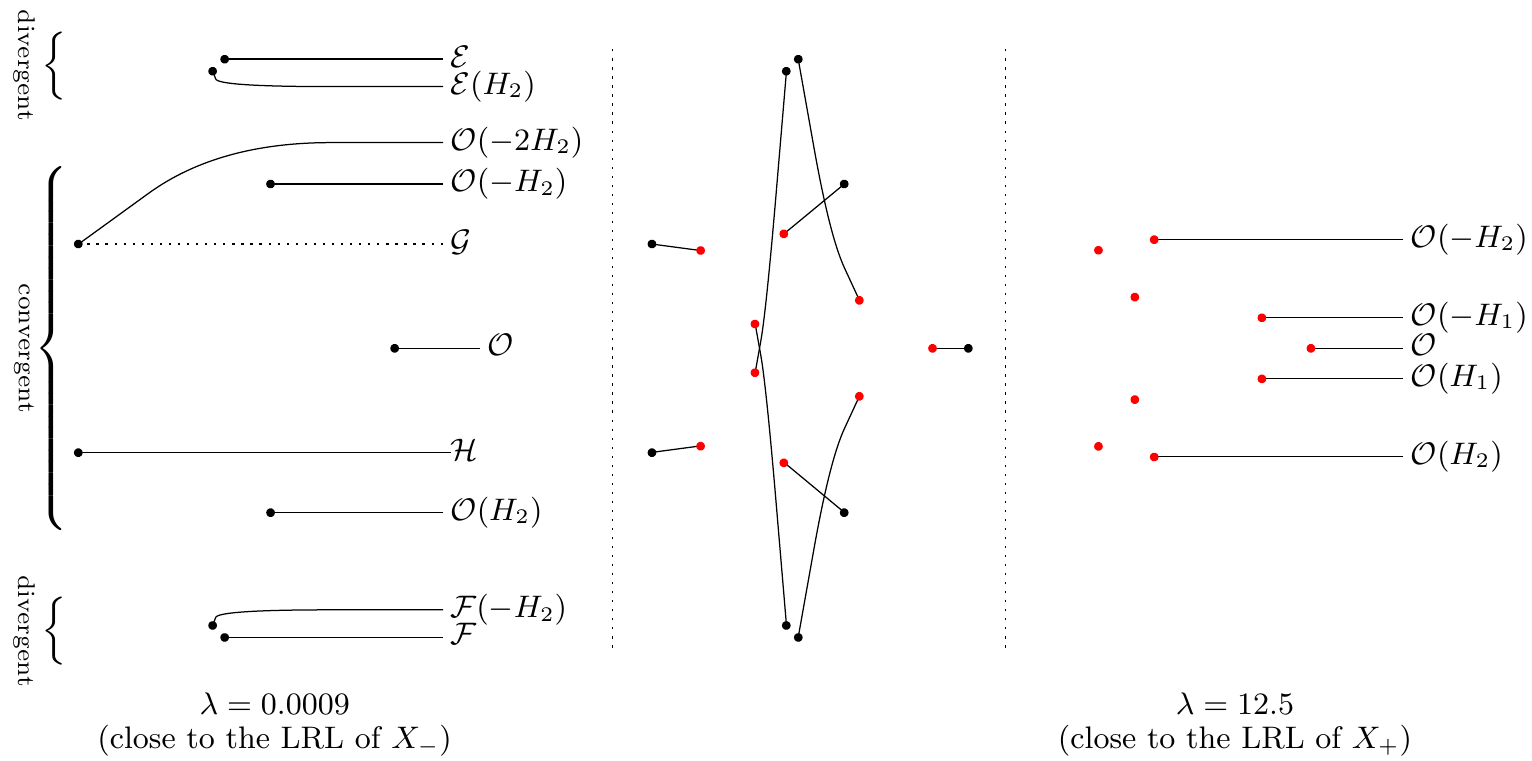}
\caption{Move of critical values: the trajectories of the critical values
are shown in the middle picture; $\cE$ is
the left mutation of $\cO(-H_1)$ with respect to $\cO(-H_2)$;
similarly $\cF$ is the right mutation of $\cO(H_1)$ with respect to
$\cO(H_2)$; $\cG$ is the right mutation of $\cO(-2H_2)$
with respect to $\cO(-H_2)$.}\label{fig:move_critical_values}
\end{figure}

We consider a path where the parameter $\lambda\in\R$ increases
from $0$ to $\infty$.
It is convenient to take $t=t(\lambda) = \lambda^{\frac{2}{3}}
+ \lambda^{\frac{2}{5}}$ so that $(q_1,q_2) \sim
\big(\lambda^{-\frac{3}{2}}, 1\big)$ as $\lambda \to \infty$ and
$(q,t) \sim \big(1,\lambda^{\frac{2}{5}}\big)$ as $\lambda \to 0$.
The move of the critical values is shown in
Fig.~\ref{fig:move_critical_values}.
When $\lambda$ is large, say $\lambda = 12.5$, we can determine
the mirror partners (in the $K$-group) of some Lefschetz thimbles
by using the monodromy action on the $(q_1,q_2)$-space:
see the rightmost picture of Fig.~\ref{fig:move_critical_values}.
Here $H_1$, $H_2$ denote the
pull-backs of the hyperplanes on~$\PP^2$ and~$\PP^4$
respectively (so that~$p_i$ is the Poincar\'e dual of~$H_i$).
A numerical calculation on computer shows that, as
$\lambda$ decreases from $\lambda = 12.5$ to $\lambda = 0.0009$,
the thimble corresponding to $\cO(-H_1)$ undergoes
a Picard--Lefschetz transformation with respect to
that corresponding to~$\cO(-H_2)$.
As we saw in the proof of Theorem~\ref{thm:sectorial_Orlov},
this corresponds to mutation in $K$-group.
After (left) mutation, $\cO(-H_1)$ becomes
\[
\cE = \Cone\big(
\Hom^\bullet (\cO(-H_2),\cO(-H_1) )\otimes \cO(-H_2)
\to \cO(-H_1)\big) = \cO_E(E-H_2).
\]
Considering the monodromy action near $t=0$ and performing
further mutation, we find that the exceptional collection
(adapted to the decomposition \eqref{eq:SOD_K} with $J=2$, $h=1$)
\begin{gather*}
\cE = \cO_E(E-H_2), \quad \cO_E(E), \quad \cO(-H_2), \quad
\cG=\cO(-2H_2) \otimes \varphi^*T_{\PP^4}[-1], \quad \cO, \\ \mathcal{H}=\cO(2H_2)\otimes \varphi^*\Omega^1_{\PP^4},\quad \cO(H_2),\quad \cO_E[-1], \quad
\cF= \cO_E(H_2)[-1]
\end{gather*}
corresponds to the Lefschetz thimbles
$\Gamma_1^\pi,\dots,\Gamma_9^\pi$~\eqref{eq:Lefschetz} ordered
in such a way that the imaginary parts of the
corresponding critical values decrease
$\Im(F(\crit_1))>\Im(F(\crit_2)) >\cdots >\Im(F(\crit_9))$.
These are the classes $V_i^+$ in Theorem~\ref{thm:sectorial_Orlov}(4)
in this case.
Note that the corresponding sectorial decomposition occurs at a point
$\tau_+ = [E] \log t + p_2 \log q\in H^2(X_+)$, $q, t>0$
such that $t$ is sufficiently small
(in the leftmost picture of Fig.~\ref{fig:move_critical_values},
$t\approx 0.0698$, $q\approx 0.698$).
\end{Example}

\begin{Remark}We can use Theorem \ref{thm:sectorial_Orlov}
to prove Gamma conjecture~II \cite[Section~4.6]{GGI:gammagrass}
in some cases. By applying Theorem~\ref{thm:sectorial_Orlov} to the case
where $Z$ is a (non-stacky) point, we know that
if the Gamma conjecture~II holds for a weak-Fano compact
toric stack $\frX_-$, then it also holds for a weighted blowup $\frX_+$ of $\frX_-$
at a non-stacky torus-fixed point (as long as $\frX_+$ is weak-Fano).
\end{Remark}

\section{Conjecture and discussion}\label{subsec:conj_discussion}

\subsection{General conjecture}\label{subsec:conjecture}
In view of our main results (Theorems~\ref{thm:decomp_QDM},
\ref{thm:ancestor},
\ref{thm:functoriality}, \ref{thm:decomp_Lef_K}, and~\ref{thm:sectorial_Orlov}), we conjecture the following phenomena
for more general discrepant birational transformations.
Suppose that we have a birational
transformation $\varphi \colon \frX_+ \dasharrow \frX_-$
between smooth (not necessarily toric)
DM stacks which fits into the diagram
\[
\xymatrix{
& \hfrX \ar[ld]_{f_+}\ar[rd]^{f_-} & \\
\frX_+ \ar@{-->}[rr]^{\varphi}& & \frX_-
}
\]
with $\hfrX$ smooth and $f_\pm$ projective birational morphisms,
such that $f_+^*K_{\frX_+} - f_-^*K_{\frX_-}$ is an
effective divisor.
We assume that the coarse moduli spaces of $\frX_\pm$
are projective, but some of the discussions below can be
also adapted to non-compact cases.

We choose a base point $\tau_{\star,\pm} \in H^*_\CR(\frX_\pm;\R)$
which is real and sufficiently close to the large radius limit point.
\begin{Conjecture}[formal decomposition]\label{conj:formal}
There exist paths from $\tau_{\star,\pm}$ to
$\tau_\pm \in H^*_\CR(\frX_\pm)$
and a~holomorphic map $f$ from a~neighbourhood $W$ of $\tau_+$
in $H^*_\CR(\frX_+)$ to $H^*_\CR(\frX_-)$ with $f(\tau_+) = \tau_-$
such that the quantum D-modules $\QDMan(\frX_\pm)$ are analytically
continued along these paths, and that
we have the decomposition of the quantum D-modules
completed in $z$ $($see \eqref{eq:z-completed_QDM}$)$:
\begin{equation}
\label{eq:formal_decomp_conj}
\overline{\QDMan}(\frX_+)|_W \xrightarrow{\cong}
f^*\overline{\QDMan}(\frX_-)|_W \oplus \scrR.
\end{equation}
Here $\scrR$ is a locally free $\cO_W[\![z]\!]$-module equipped
with a meromorphic flat connection $\nabla^\scrR$
and a $z$-sesquilinear pairing $P^\scrR$,
and the formal decomposition respects both the connection and the
pairing $($in particular it is orthogonal
with respect to the pairing \eqref{eq:pairing_A}$)$.
\end{Conjecture}

Building on the above conjecture, we can also state a conjecture
comparing higher genus Gromov--Witten theories. This involves
quantization of the formal decomposition \eqref{eq:formal_decomp_conj}.
See Section~\ref{subsec:comp_all_genera} for the notation.

\begin{Conjecture}[comparison in all genera]
\label{conj:all_genera}
The ancestor potentials $\scrA_{\pm,\tau}$ of $\frX_\pm$
can be analytically continued along the paths from $\tau_{\star,\pm}$
to $\tau_\pm$ in Conjecture $\ref{conj:formal}$.
There exist a family $\scrA'_\tau \in \Fockan(V_\tau,\bD_\tau)$
of tame functions such that
\[
T_{\bs} \hU_\tau \scrA_{+,\tau} = \scrA_{-,f(\tau)} \otimes \scrA'_\tau
\]
for $\tau\in W$, where $\bs =(-z\phi_0,\bD_\tau) + U_\tau(z \phi_0)$,
$V_\tau\subset \scrR_\tau$ is a $\C$-vector subspace
such that $\scrR_\tau = V_\tau[\![z]\!]$ and that $P^\scrR$ restricts
to the $\C$-valued pairing $V_\tau \times V_\tau \to \C$,
$\bD_\tau \in z \scrR_\tau$
and $U_\tau \colon H^*_\CR(\frX_+)[\![z]\!] \cong H^*_\CR(\frX_-)[\![z]\!]
\oplus V_\tau[\![z]\!]$ is the unitary
isomorphism obtained from \eqref{eq:formal_decomp_conj}
by restricting to $\tau$ and flipping the sign of $z$.
\end{Conjecture}
In this conjecture we implicitly assume that the action of the operator $T_\bs$ on
$\hU_\tau \scrA_{+,\tau}$ is well-defined;
this holds if $\scrA_{+,\tau}$ is rational (see Section~\ref{subsubsec:quantization}).
Note also that the space $\Fockan(V_\tau,\bD_\tau)$ itself
depends only on $\scrR_\tau$ and $\bD_\tau$ and does not
depend on the choice of $V_\tau$, but that $\scrA'_\tau$ depends on
the choice of $V_\tau$.

Next we state a conjecture relating the analytic lift of
the formal decomposition \eqref{eq:formal_decomp_conj}
with a~semiorthogonal decomposition of the $K$-group.
Recall from~\eqref{eq:conn_z} that the action of $-\nabla_{z^2\partial_z}$
on $\QDMan(\frX_\pm)|_{z=0}$ equals the Euler multiplication
$E\star_\tau$. In particular, the formal decomposition~\eqref{eq:formal_decomp_conj} implies that
$E\star_\tau$ on $H_\CR^*(\frX_+)$ is conjugate to $E\star_{f(\tau)}\oplus \big({-}\nabla^\scrR_{z^2\partial_z}\big)$
on $H^*_\CR(\frX_-) \oplus \scrR|_{z=0}$.

\begin{Conjecture}[analytic lift]\label{conj:lift}
We can arrange $\tau_\pm$ and the paths from $\tau_{\star,\pm}$ to $\tau_\pm$ in Conjecture~$\ref{conj:formal}$ so that $($in addition to Conjecture~$\ref{conj:formal})$
the following holds.
\begin{itemize}\itemsep=0pt
\item[{\rm (1)}] We have a decomposition
$\big(\scrR,\nabla^\scrR,P^\scrR\big) = \big(\scrR_1,\nabla^{\scrR_1},P^{\scrR_1}\big)
\oplus \big(\scrR_2,\nabla^{\scrR_2},P^{\scrR_2}\big)$, where
$\scrR_i$ is a locally free $\cO_W[\![z]\!]$-module
equipped with a flat connection $\nabla^{\scrR_i}$
and a $z$-sesquilinear pairing $P^{\scrR_i}$.
There exist a phase $\phi\in \R$ and real numbers $l_2<l_1$ such that
\begin{itemize}\itemsep=0pt
\item[{\rm (a)}] all eigenvalues $\bu$ of $\big({-}\nabla^{\scrR_1}_{z^2\partial_z}\big)$
on $\scrR_1|_{z=0}$ satisfy $\Im\big(e^{-\iu\phi} \bu\big)>l_1$,
\item[{\rm (b)}] all eigenvalues $\bu$ of $\big({-}\nabla^{\scrR_2}_{z^2\partial_z}\big)$
on $\scrR_2|_{z=0}$ satisfy $\Im\big(e^{-\iu\phi} \bu\big)<l_2$ and
\item[{\rm (c)}] all eigenvalues $\bu$ of $E\star_{f(\tau)}$ on $H^*_\CR(\frX_-)$
satisfy $l_2<\Im\big(e^{-\iu\phi}\bu\big)<l_1$.
\end{itemize}

\item[{\rm (2)}] Moreover, we have a sector $I_\phi
= \big\{\big(r,e^{\iu\theta}\big) \in\tCC\colon |\theta-\phi|<\frac{\pi}{2}+\epsilon\big\}$
with some $\epsilon>0$ and an analytic decomposition
\begin{equation}\label{eq:analytic_lift_conj}
\pi^*\QDMan(\frX_+)\big|_{W\times I_\phi}
\cong \scrRan_1 \oplus
\pi^*f^*\QDMan(\frX_-) \big|_{W\times I_\phi} \oplus
\scrRan_2,
\end{equation}
where $\pi\colon W\times \tCC \to W\times \C$ is the oriented
real blow-up and $\scrRan_i$ is a locally free $\cA_{W\times I_\phi}$-module
equipped with a flat connection, such that
$\scrR_i \cong \scrRan_i\otimes_{\cA_{W\times I_\phi}}\cO_W[\![z]\!]$,
$i=1,2$
and that the formal decomposition \eqref{eq:formal_decomp_conj}
is induced by \eqref{eq:analytic_lift_conj}.

\item[{\rm (3)}] Via the $\hGamma$-integral structure for $\frX_+$ and $\frX_-$,
the decomposition \eqref{eq:analytic_lift_conj} is induced by a~semi\-orthogonal decomposition of the topological $K$-groups:
\begin{equation} \label{eq:SOD_K_conj}
K(\frX_+) \cong K_1 \oplus K(\frX_-) \oplus K_2
\end{equation}
such that the associated inclusion $K(\frX_-)\hookrightarrow K(\frX_+)$
respects the Euler pairing.
\end{itemize}
\end{Conjecture}

\begin{Remark}We expect that the semiorthogonal decomposition~\eqref{eq:SOD_K_conj} arises naturally from geo\-metry. In our setting, for example, we could hope that there is a semiorthogonal decomposition of the bounded derived category
of $\frX_+$ of the form (see \cite[Conjecture~4.3.7]{BFK:VGIT}):
\begin{equation}\label{eq:SOD_conj}
D^b(\frX_+) = \big\langle \scrC_1, D^b(\frX_-), \scrC_2\big\rangle.
\end{equation}
When $\frX_\pm$ are toric stacks arising from the variation of
GIT quotients as in Section~\ref{sec:discrepant},
Ballard--Favero--Katzarkov \cite[Theorem~5.2.1]{BFK:VGIT}
showed that such a~semiorthogonal decomposition exists
(see also \cite{Halpern-Leistner:GIT, Kawamata:log_crepant}).
A semiorthogonal summand of $D^b(\frX_+)$ gives a~$K$-motive (see \cite[Section~4]{Gorchinskiy-Orlov:geometric_phantom}),
which in turn defines a direct summand of the topological $K$-group
of $\frX_+$\footnote{I thank Sergey Galkin for this remark and
pointing to the reference~\cite{Gorchinskiy-Orlov:geometric_phantom}.}.
We expect that~\eqref{eq:SOD_K_conj} arises from~\eqref{eq:SOD_conj} in this way (cf.~\cite{Kuznetsov:HH_SOD} for the discussion on Hochschild homology). On the other hand, in view of the deformation invariance of Gromov--Witten invariants, it is more natural to state conjectures in terms of topological $K$-groups instead of derived categories.
\end{Remark}

\subsection{Functoriality and Riemann--Hilbert problem}\label{subsec:functoriality_RH}
We discuss how to recover the quantum cohomology of $\frX_-$
from the quantum cohomology of $\frX_+$,
assuming Conjecture~\ref{conj:lift}.
This involves solving a Riemann--Hilbert boundary value problem.
\begin{Proposition}\label{prop:RH}
Assume that the quantum D-modules $\QDMan(\frX_\pm)$
are of exponential type.
Suppose that we are given the analytic continuation of
the quantum D-module $\QDMan(\frX_+)$
to a~neighbourhood $W$ of $\tau_+\in H^*_\CR(\frX_+)$
and its formal decomposition
\[
\overline{\QDMan}(\frX_+)|_W = \scrQ\oplus \scrR
\]
corresponding to the decomposition~\eqref{eq:formal_decomp_conj}
for which Conjecture~$\ref{conj:lift}$ holds. Then we can recover the map $f$ in Conjecture~$\ref{conj:lift}$
and the quantum D-module $f^*\QDMan(\frX_-)$ of $\frX_-$ $($trivialized
as a~vector bundle over $W\times \C)$ together with an isomorphism
$f^*\overline{\QDMan}(\frX_-)|_W \cong \scrQ$.
\end{Proposition}

Before giving a proof of this proposition, we review the exponential type assumption
(see \cite[Definition~2.12]{KKP:Hodge}). This was originally introduced by Hertling--Sevenheck \cite[Definition~8.1]{Hertling-Sevenheck:nilpotent}
under the name ``require no ramification''. We also review a mutation system of Sanda--Shamoto \cite[Definition~2.30]{Sanda-Shamoto}.

The quantum D-module of $\frX = \frX_\pm$ is of exponential type
if for each $\tau\in H^*_\CR(\frX)$,
we have a formal decomposition of the
quantum D-module $\QDMan(\frX)_\tau :=
\QDMan(\frX)|_{\{\tau\}\times \C}$
(see \cite[Lemma~8.2]{Hertling-Sevenheck:nilpotent}):
\begin{equation} \label{eq:formal_decomp_QDM_general}
\overline{\QDMan}(\frX)_{\tau}
\cong \bigoplus_{\bu\in \Spec(E\star_\tau)} \big( e^{\bu/z}
\otimes \cF_\bu \big)\otimes_{\C\{z\}} \C[\![z]\!],
\end{equation}
where $\Spec(E\star_\tau)$ denotes the set of
(mutually distinct) eigenvalues of $E\star_\tau$,
$e^{\bu/z}$ denotes the rank one connection
$(\C\{z\}, d+d(\bu/z))$ and
$\cF_\bu$ is a free $\C\{z\}$-module equipped with
a regular singular meromorphic connection.
The decomposition is (automatically) orthogonal with respect to
the pairing $P$, and induces a $z$-sesquilinear pairing $P_\bu$
on each $\cF_\bu$;
$\cF_\bu$ is called the \emph{regular singular piece} in
\cite[Lemma 8.2]{Hertling-Sevenheck:nilpotent}.
Moreover, the Hukuhara--Turrittin theorem
(see \cite[Lemma 8.3]{Hertling-Sevenheck:nilpotent} in this context)
implies that,
for each phase $\phi$ admissible for $\Spec(E\star_\tau)$,
the formal decomposition admits a unique analytic lift
over the sector $I_\phi = \big\{\big(r,e^{\iu\theta}\big)\in \tCC \colon
|\theta -\phi|<\frac{\pi}{2}+\epsilon\big\}$ (for some $\epsilon>0$)
\begin{equation}\label{eq:analytic_lift_QDM_general}
\pi^*\left(\QDMan(\frX)_\tau\right)\big|_{I_\phi}
\cong \bigoplus_{\bu\in \Spec(E\star_\tau)}
\pi^*(e^{\bu/z} \otimes \cF_\bu)\big|_{I_\phi},
\end{equation}
where $\pi\colon \tCC\to \C$ is the oriented real blowup.
In Section~\ref{subsec:Hukuhara-Turrittin},
we discussed in details the special case of these decompositions
where the quantum cohomology is semisimple.

The analytic germ at $z=0$
of the quantum D-module $\QDMan(\frX)_\tau$
can be determined by the formal decomposition~\eqref{eq:formal_decomp_QDM_general} and the Stokes data.
Sanda--Shamoto \cite[Definition~2.30]{Sanda-Shamoto}
encoded the Stokes data in linear-algebraic data
which they called a \emph{mutation system}.
This can be viewed as a generalization of a marked reflection system
in Section~\ref{subsec:MRS}.
Let $V$ denote the space of flat sections of $\QDMan(\frX)_\tau$
over the sector $I_\phi^\times=
\big\{z\in \C^\times\colon |\arg z -\phi|<\frac{\pi}{2}+\epsilon\big\}$.
We define a pairing $[\cdot,\cdot)$ on $V$ by (cf.~Section~\ref{subsec:MRS})
\[
[s_1,s_2) = P\big(s_1\big(e^{-\pi\iu} z\big), s_2(z)\big),
\]
where $s_1\big(e^{-\pi\iu}z\big)$ denotes the analytic continuation
of $s_1(z)$ along the path $[0,\pi]\ni \theta \mapsto e^{-\iu\theta} z$.
The mutation system (see \cite[Proposition~2.5, Section~2.7]{Sanda-Shamoto})
for $\QDMan(\frX)_\tau$ associated with the admissible
phase $\phi$ is given by the data\footnote{We omitted the data of a labelling $\tau$ in
\cite[Definition~2.30]{Sanda-Shamoto} since it can be recovered from the phase~$\phi$.}
\begin{itemize}\itemsep=0pt
\item the tuple $(V,[\cdot,\cdot))$ of a vector space and a pairing;
\item the decomposition $V \cong \bigoplus\limits_{\bu\in \Spec(E\star_\tau)}
V_\bu$ induced by the analytic lift~\eqref{eq:analytic_lift_QDM_general},
where $V_\bu$ is the space of flat sections of $e^{\bu/z}\otimes \cF_\bu$
over the sector $I_\phi^\times$
\end{itemize}
satisfying the semiorthogonality:
\[
[v_1,v_2)=0 \qquad \text{if $v_1\in V_{\bu_1}$, $v_2\in V_{\bu_2}$
and $\Im\big(e^{-\iu\phi}\bu_1\big) < \Im\big(e^{-\iu\phi}\bu_2\big)$}.
\]
The pairing $[\cdot,\cdot)$ restricted to $V_\bu$ is induced by
the pairing $P_\bu$ on $\cF_\bu$
\cite[Lemma~8.4]{Hertling-Sevenheck:nilpotent}.
It was shown \cite[Proposition~2.5, Section~2.5]{Sanda-Shamoto}
that a mutation system is equivalent to Stokes data
(or a~Stokes filtered local system)
equipped with a pairing.
By the Riemann--Hilbert correspondence (see \cite[Section~2.7]{Sanda-Shamoto}
in this context), the formal structure~\eqref{eq:formal_decomp_QDM_general}
and the mutation system together recover the analytic
germ at $z=0$ of $\QDM(\frX)_\tau$.

\begin{proof}[Proof of Proposition~\ref{prop:RH}]
Fix $\tau \in W$. We may assume that~$\phi$
in Conjecture~\ref{conj:lift} is admissible for the set
$\Spec(E\star_\tau)$ of eigenvalues of $E\star_\tau$
on $H^*_\CR(\frX_+)$, by perturbing $\phi$ if necessary.
Under the assumption that $\QDMan(\frX_+)_\tau$ is of exponential type,
the summands $\scrQ,\scrR_1,\scrR_2$ of $\overline{\QDMan}(\frX_+)_\tau$
admit decompositions similar to \eqref{eq:formal_decomp_QDM_general}.
We write $\Spec(\scrQ)$ (resp.~$\Spec(\scrR_i)$) for the set of
the eigenvalues of the operators $-\nabla_{z^2\partial_z}$
on $\scrQ|_{z=0}$ (resp.~on $\scrR_i|_{z=0}$).
By Conjecture \ref{conj:lift}(1), the sets
$\Spec(\scrQ)$, $\Spec(\scrR_1)$, $\Spec(\scrR_2)$
are mutually distinct; therefore
$\scrQ$, $\scrR_1$, $\scrR_2$ are partial sums of the right-hand
side of the formal decomposition \eqref{eq:formal_decomp_QDM_general}
for $\frX=\frX_+$.
Writing
$V \cong \bigoplus\limits_{\bu\in \Spec(E\star_\tau)} V_\bu$
for the mutation system of $\QDMan(\frX_+)_\tau$,
we can decompose $V$ as
\begin{equation}
\label{eq:decomp_mutation_system}
V \cong V_{\scrR_1} \oplus V_{\scrQ} \oplus V_{\scrR_2}
\end{equation}
with $V_\scrQ = \bigoplus\limits_{\bu\in \Spec(\scrQ)} V_\bu$,
$V_{\scrR_i} = \bigoplus\limits_{\bu\in \Spec(\scrR_i)} V_\bu$.
By Conjecture \ref{conj:lift}(2), the decomposition
\[
V_\scrQ = \bigoplus_{\bu\in \Spec(\scrQ)} V_\bu
\]
gives the mutation system for $\QDMan(\frX_-)_{f(\tau)}$.
The $\hGamma$-integral structure identifies
$(V,[\cdot,\cdot))$ with $(K(\frX_+)\otimes \C,\chi)$
and the decomposition~\eqref{eq:decomp_mutation_system}
is identified with that~\eqref{eq:SOD_K_conj} of $K(\frX_+)$
by Conjecture~\ref{conj:lift}(3). Then we obtain an isomorphism
\begin{equation}\label{eq:VQ_K}
\Phi\colon V_\scrQ \cong K(\frX_-)\otimes \C.
\end{equation}
Since the inclusion $K(\frX_-) \to K(\frX_+)$ respects
the Euler pairing, we see that the restriction of the pairing $[\cdot,\cdot)$
on~$V$ to~$V_\scrQ$ coincides with the pairing on~$V_\scrQ$
as a mutation system for~$\QDMan(\frX_-)_{f(\tau)}$.
Therefore the mutation system for $\QDMan(\frX_-)_{f(\tau)}$
can be recovered as a~summand of the mutation system
for $\QDMan(\frX_+)_\tau$.
Thus we recover the analytic germ at $z=0$ of
$\QDMan(\frX_-)_{f(\tau)}$ by the Riemann--Hilbert correspondence.
We write $\scrQan$ for the germ so reconstructed.

We extend $\scrQan$ to the trivial vector bundle
$\QDMan(\frX_-)_{f(\tau)}$ over $\C_z$
and construct the fundamental solution
$L(f(\tau),z)$ for~$\frX_-$ (see Section~\ref{subsec:integral})
and the map~$f$.
Consider the trivial bundle
$Q^{(\infty)} := H^*(\frX_-)\times \big(\PP^1\setminus \{0\}\big)\to \big(\PP^1\setminus \{0\}\big)$
equipped with the meromorphic connection
\[
\nabla^{(\infty)}_{z\partial_z} = z\parfrac{}{z}
- \frac{c_1(\frX_-)\cup}{z}+ \mu.
\]
This has $z^{-\mu}z^{c_1(\frX)}$ as a fundamental solution
and the facts recalled in Section~\ref{subsec:integral} imply that
the quantum connection on $\{f(\tau)\}\times \C_z$
is gauge equivalent to $\nabla^{(\infty)}$
via the gauge transformation by $L(f(\tau),z)$
(which is regular and the identity at $z=\infty$).
We glue this trivial bundle $Q^{(\infty)}$
with the germ $\scrQan$ of vector bundle at $z=0$
to get a vector bundle $\hQ$ over $\PP^1$.
The gluing is given by sending a flat section $s\in V_\scrQ$
over the sector $I_\phi^\times$ to
the flat section for $\nabla^{(\infty)}$:
\[
(2\pi)^{-n/2}z^{-\mu} z^{c_1(\frX_-)}
\big( \hGamma_{\frX_-}\cdot (2\pi\iu)^{\deg_0/2} \inv^*\tch(\Phi(s))\big)
\]
with $n=\dim\frX_-$, where $\Phi$ is the isomorphism in
\eqref{eq:VQ_K}.
In view of the definition of the $\hGamma$-integral
structure, this glued bundle $\hQ$ must be isomorphic to the trivial
extension of $\QDMan(\frX_-)_{f(\tau)}$ to $\PP^1$
(with respect to the given trivialization).
In particular, $\hQ$ is trivial, and the identification
$Q^{(\infty)}|_{\infty} \cong H^*_\CR(\frX_-)$ at infinity
induces a
global trivialization $\hQ \cong H^*_\CR(\frX_-)\times \PP^1$.
The trivial bundle $\hQ$ equipped with a meromorphic connection
gives the quantum D-module $\QDMan(\frX_-)_{f(\tau)}$.
Moreover, the isomorphism $\hQ|_{\PP^1\setminus \{0\}}
\cong Q^{(\infty)}$ written in the respective trivializations
gives the fundamental solution $L(f(\tau),z)$ as an
$\End(H^*_\CR(\frX_-))$-valued function.
Varying $\tau$ in $W$, we recover the full quantum connection
for $f^*\QDMan(\frX_-)$ from $L(f(\tau),z)$.
We also recover $f(\tau)$ from the expansion
\[
L(f(\tau),z)^{-1} 1 = 1 + \frac{f(\tau)}{z} + O\big(z^{-2}\big).
\]
The proposition is proved.
\end{proof}

\begin{Remark}Choosing fundamental solutions for regular singular pieces,
we can formulate the above reconstruction in terms of a~Riemann--Hilbert boundary value problem for the triple $(Y,Y^-,L)$ of fundamental solutions, where $Y$ is a~fundamental solution on the sector $I_\phi^\times$ with prescribed asymptotics, $Y^-$ is a fundamental
solution on the opposite sector $I_{\phi+\pi}^\times$
with prescribed asymptotics and~$L$ is the fundamental solution around $z=\infty$.
In the semisimple case, this was explained in details by
Dubrovin \cite[Lecture~4]{Dubrovin:Painleve}
and the method there applies to the situation of our main
Theorem~\ref{thm:sectorial_Orlov}. In the semisimple case,
the formal structure~\eqref{eq:formal_decomp_QDM_general}
is determined only by the eigenvalues of $E\star_\tau$,
and therefore the asymptotic basis (or, if any, the corresponding
exceptional collection) reconstructs the quantum D-module.
\end{Remark}
\begin{Remark}When a candidate formal decomposition $\overline{\QDMan}(\frX_+)_\tau
\cong \scrR_1\oplus \scrQ \oplus \scrR_2$ and a
decomposition~\eqref{eq:SOD_K_conj} of the $K$-group
(corresponding to the analytic lift via the $\hGamma$-integral structure)
are given, what is non-trivial in the above reconstruction
is the triviality of the glued bundle $\hQ$.
\end{Remark}
\begin{Remark}The fundamental solution $L(\tau,z)$ in Section~\ref{subsec:integral}
is called a \emph{calibration} in the theory of Frobenius manifolds.
A calibration of a Frobenius manifold is not unique in general and
its ambiguity was discussed in \cite[Lemma~4.1]{Dubrovin:Painleve}.
The above procedure recovers, not only (the germ of) the quantum
cohomology Frobenius manifold of $\frX_-$ (if~$f$ is submersive), but also its calibration.\footnote{I thank Vasily Golyshev for asking me a question which led me to this observation.}
Since the $\hGamma$-integral structure was normalized at the large radius limit point, it somehow `remembers' the limit point.
\end{Remark}

\begin{Remark}Conjecture \ref{conj:lift} is closely related to Dubrovin's conjecture
\cite[Conjecture~4.2.2]{Dubrovin:ICM}
(or Gamma conjecture \cite[Section~4.6]{GGI:gammagrass} or
Dubrovin-type conjecture
\cite[Definition 5.2]{Sanda-Shamoto}).
It can be viewed as a relative version of these conjectures.
\end{Remark}

\appendix
\section{The Brieskorn module in the weak Fano case}\label{append:Bri_weak_Fano}

We discuss the coherence and the locally-freeness
of the equivariant Brieskorn module $\Briequiv(F)$ mirror to
the small quantum cohomology of a weak-Fano smooth toric DM stack $\frX$,
near the large radius limit point $0_\bSigma$.
We stated this in Proposition \ref{prop:Bri_weak_Fano}.
We can describe the Brieskorn module
as a certain variant of the GKZ system (see \cite[Section~5.2]{CCIT:MS}).
The coherence of the relevant GKZ system near the large radius limit point
has been discussed by the author \cite[Proposition~4.4]{Iritani:Integral}
(on a neighbourhood of $0_\bSigma$ with the logarithmic locus deleted),
Reichelt--Sevenheck \cite[Theorem~3.7]{Reichelt-Sevenheck:logFrob}
(on a neighbourhood of $0_\bSigma$ for toric manifolds) and
Mann--Reichelt \cite[Theorem~4.10]{Mann-Reichelt}
(similarly for toric orbifolds).
We adapt the argument of~\cite{Mann-Reichelt} to our Brieskorn module
with minor modifications.
Note that we impose only minimal assumptions on~$\frX$ and~$S$;
we do not assume compactness of $\frX$
nor a generation condition for $S$.
We also deal with the equivariant case. In this sense our result
is slightly more general than \cite{Mann-Reichelt}.

Let $\frX=\frX_\bSigma$
be a semiprojective smooth toric DM stack from Section~\ref{subsubsec:toric_stacks}
which is weak Fano, i.e., $-K_\frX$ is nef.
Let $\Delta\subset \bN_\R$ denote the fan polytope, that is, the
union of simplices spanned by $\{0\} \cup \{b
\in R(\bSigma)\colon \overline{b} \in \sigma\}$ over all maximal cones $\sigma$
of $\bSigma$. The weak Fano assumption implies that
$\Delta$ is a convex polytope.
We assume that the set $S$ (which is used to construct
the LG mirror) is contained in the fan polytope:
\begin{equation}
\label{eq:S_is_in_Delta}
S \subset \bN \cap \Delta.
\end{equation}
This means that the base space $\cM$ of the LG model
corresponds to the \emph{small quantum cohomology locus}
of $\frX$ in Section~\ref{subsec:sectorial_decomp_Bri}.
The definition of the Brieskorn module on
the affine chart $\Spec \C[\Laa(\bSigma)_+]$ does not
rely on Assumption \ref{assump:S_generates_N} and
we do not need this assumption in the following discussion.

The equivariant Brieskorn module $\Briequiv(F)$ on the affine chart
$\Spec \C[\Laa(\bSigma)_+]$
is given by the module
\[
\sfG := \C[\OO(\bSigma)_+][z].
\]
Recall from Section~\ref{subsec:Bri} that, by choosing
a splitting $\LL_\C^\star \to \big(\C^S\big)^\star$, $\xi \mapsto
\hxi$ of \eqref{eq:ext_divseq}, $\sfG$ has the
structure of a module over
the ring of differential operators:
\[
\sfD := R_\T[\Laa(\bSigma)_+][z] \left\langle z \xi q \parfrac{}{q}\colon \xi\in \LL^\star_\C
\right\rangle,
\]
where $\chi_i\in R_\T$ acts by
$z x_i \parfrac{}{x_i} + x_i\parfrac{F}{x_i}$
and $z \xi q\parfrac{}{q}$ acts by
$z \hxi u \parfrac{}{u} + \hxi u\parfrac{F}{u}$.
For convenience, we choose co-ordinates
$q_1,\dots,q_m\in
\C[\Laa(\bSigma)]$ corresponding to a $\Z$-basis of $\Laa(\bSigma)$
and write $\theta_i = q_i \parfrac{}{q_i}$;
then we have
\[
\sfD = R_\T[\Laa(\bSigma)_+][z]\langle
z\theta_1,\dots,z\theta_m\rangle.
\]
Let $\scrG$ and $\scrD$ denote the sheaves on
\[
B := \Spec(R_\T[\Laa(\bSigma)_+][z])
= \Spec(\C[\Laa(\bSigma)_+]) \times \Lie\T \times \C_z
\]
corresponding to $\sfG$ and $\sfD$ respectively.
The main result in the appendix is the following.
\begin{Proposition}\label{prop:weakFano_Bri_locallyfree}
Suppose, as above, that $\frX$ is a weak Fano semiprojective toric stack
and that~$S$ satisfies~\eqref{eq:S_is_in_Delta}.
There exists a Zariski open subset $U$ of $\Spec(\C[\Laa(\bSigma)_+])$
containing the large radius limit point $0_\bSigma$ such that
$\scrG|_{U\times \Lie\T\times \C_z}$ is a locally free
coherent $\cO_{U\times \Lie \T \times \C_z}$-module of rank $\dim H^*_\CR(\frX)$.
\end{Proposition}

\subsection{Generators and relations}
By Remark \ref{rem:flatness_pr}, the equivariant Brieskorn module
$\sfG$ is generated by
$w_v := u^{(\Psi^\bSigma(\overline{v}),v)}$ with $v\in \bN \cap \Pi$
as a $\C[\Laa(\bSigma)_+][z]$-module.
Relations among these generators as a $\sfD$-module
are given as follows.

\begin{Lemma}
\label{lem:relation_M}
For $b\in S$, we write $e_b^\star = \widehat{D}_b+ \chi(b)$ with
$\chi(b) \in \bM$ $($recall that $D_b = D(e_b^\star) \in \LL^\star)$.
We have the following relation in $\sfG$:
\begin{gather*}
\left( \prod_{b\in S} \prod_{c=0}^{\nu_b-1}
\left(z D_b q\parfrac{}{q}+ \chi(b) -\big(\Psi^\bSigma_b(v)+c\big) z\right)
\right) w_v\\
\qquad{} = \left( \prod_{b\in S} u_b^{\nu_b} \right) w_v
= q^{\Psi(v)+\sum_{b\in S} \nu_b e_b - \Psi(v+ \sum_{b\in S}
\nu_b b)} w_{v+\sum_{b\in S} \nu_b b}
\end{gather*}
for every $(\nu_b)_{b\in S} \in (\Z_{\ge 0})^S$
and $v\in \bN \cap \Pi$.
\end{Lemma}
\begin{proof}
By definition, $z D_b q\parfrac{}{q} + \chi(b)$ acts on
$\sfG$ by $z u_b \parfrac{}{u_b} + u_b$.
The first equality follows from this
and $u_b\parfrac{}{u_b} w_v= \Psi_b(v) w_v$.
The second equality is just by definition,
see Section~\ref{subsec:coord_localchart_LG}.
\end{proof}
When $b\in R(\bSigma)$ and $v\in \bN\cap \Pi$ lie
in the same cone of $\Sigma$, we have a relation
\[
w_{v+b} =
\left(zD_b q\parfrac{}{q} + \chi(b) - z \Psi_b^\bSigma(v)\right)
w_v
\]
in $\sfG$ by Lemma \ref{lem:relation_M}.
From this we can see that $\sfG$ is generated
by finitely many $w_{v}$ as a $\sfD$-module.
For example, the set $\{w_v\colon v\in \Bx(\bSigma)\}$ generates $\sfG$.
By Lemma~\ref{lem:relation_M},
$w_v$ is annihilated by $P_{v,\lambda}\in \sfD$
\begin{gather}
P_{v,\lambda} :=
\prod_{b\in S\colon \lambda_b>0} \prod_{c=0}^{\lambda_b-1}
\left(z D_b q\parfrac{}{q}+ \chi(b) -\big(\Psi^\bSigma_b(v)+c\big) z\right) \nonumber\\
\hphantom{P_{v,\lambda} :=}{}- q^\lambda \prod_{b\in S\colon \lambda_b<0} \prod_{c=0}^{-\lambda_b-1}
\left(z D_b q\parfrac{}{q}+ \chi(b) -\big(\Psi^\bSigma_b(v)+c\big) z\right)\label{eq:relation_in_M}
\end{gather}
for any $\lambda \in \LL \cap \hNE(\frX_\bSigma) \subset \Z^S$,
cf.~{\cite[Section~5.1]{CCIT:MS}}.

\subsection{Characteristic variety and coherence}

Define an increasing filtration $\cF_l(\scrD)$ of $\scrD$ by
the rank of differential operators, i.e., $\cF_l(\scrD)$ consists
of differential operators of the form
\[
\sum_{k_1+\cdots + k_m \le l} a_k(q,\chi,z) (z \theta_1)^{k_1} \cdots (z\theta_m)^{k_m}.
\]
Choose generators $w_{v_1},\dots,w_{v_k}$ of $\sfG$
as a $\sfD$-module and introduce a filtration on $\scrG$ by
\[
\cF_l(\scrG) := \sum_{i=1}^k \cF_l(\scrD) w_{v_i}.
\]
An easy argument shows that if $\gr_\cF(\scrG)|_V$
is finitely generated as an $\cO_V$-module
for an open set $V\subset B$, then
$\scrG|_V$ is also finitely generated as an $\cO_V$-module.
We shall show that $\gr_\cF(\scrG)$ is finitely generated
on a neighbourhood of $\{0_\bSigma\}\times \Lie \T\times \C_z$.

The associated graded module $\gr_\cF(\scrG)$ is a
$\gr_\cF(\scrD)$-module generated by $w_{v_1},\dots,
w_{v_k}$.
We have $\gr_\cF(\scrD) = \cO_B[\xi_1,\dots,\xi_m]$,
where $\xi_i$ denotes the image of $z \theta_i \in \cF_1(\scrD)$.
The \emph{characteristic variety} $\Ch(\scrG)$ of $\scrG$ is defined
to be the support of $\gr_\cF(\scrG)$ as an
$\cO_B[\xi_1,\dots,\xi_m]$-module. It is a~closed subset
of $B\times \C^m$ invariant under
the dilation $(\xi_1,\dots,\xi_m)\mapsto (\lambda \xi_1,\dots,
\lambda \xi_m)$ with $\lambda \in \C^\times$.
Then $\gr_\cF(\scrG)|_V$ is a finitely generated $\cO_V$-module
over the following set $V$:
\begin{equation}\label{eq:smooth_locus}
V = \{x \in B\colon (x,\xi) \in \Ch(\scrG) \Longrightarrow \xi=0\}.
\end{equation}
Note that $\Ch(\scrG)$ induces a closed subset
$C\subset B \times \PP^{m-1}$
and $V$ is the complement of $\pi(C)$,
where $\pi \colon B\times \PP^{m-1} \to B$ is the projection;
thus $V$ is Zariski-open.

For a differential operator $a\in \cF_l(\scrD)\setminus \cF_{l-1}(\scrD)$,
its \emph{principal symbol} $\sigma(a)$ is the image of~$a$ in~$\gr_\cF(\scrD)$; explicitly
\begin{gather*}
\sigma(a) := \sum_{k_1+\cdots + k_m =l}
a_{k}(q,\chi,z) \xi_1^{k_1} \cdots \xi_m^{k_m}
\qquad
\text{if $a= \sum\limits_{k_1+\dots+k_m \le l} a_k(q,\chi,z) (z\theta_1)^{k_1}
\cdots (z\theta_m)^{k_m}$.}
\end{gather*}
The principal symbol of the relation $P_{v,\lambda}$ in~\eqref{eq:relation_in_M} is given by
\[
\sigma(P_{v,\lambda}) =
\begin{cases}
\displaystyle \prod_{b\in S\colon \lambda_b >0} D_b(\xi)^{\lambda_b} -
q^\lambda \prod_{b\in S\colon \lambda_b<0} D_b(\xi)^{-\lambda_b}
& \displaystyle\text{if} \ \sum\limits_{b\in S} \lambda_b = 0, \\
\displaystyle \prod_{b\in S\colon \lambda_b >0} D_b(\xi)^{\lambda_b}
& \displaystyle\text{if} \ \sum\limits_{b\in S} \lambda_b >0, \\
\displaystyle - q^\lambda \prod_{b\in S\colon \lambda_b<0} D_b(\xi)^{-\lambda_b}
& \displaystyle\text{if} \ \sum\limits_{b\in S}\lambda_b <0,
\end{cases}
\]
where $D_b(\xi) := \sigma\big(z D_b q\parfrac{}{q}\big)$
is a linear form in $\xi_1,\dots,\xi_m$.
Because $\sigma(P_{v,\lambda})$ is independent of $v$,
it is an annihilator of $\gr_\cF(\scrG)$.
Therefore $\Ch(\scrG)$ is contained
in the closed subset of $B \times \C^m$
defined by $\sigma(P_{v,\lambda}) = 0$ for all
$\lambda \in \LL \cap \hNE(\frX)$.

\begin{Lemma}
There exists a Zariski-open subset $U$ of $\Spec \C[\Laa(\bSigma)_+]$
containing $0_\bSigma$ such that $U \times \Lie \T \times \C_z$
is contained in the locus $V$ in \eqref{eq:smooth_locus}.
\end{Lemma}
\begin{proof}
Note that $\sigma(P_{v,\lambda})$ does not depend on $(\chi,z)
\in \Lie \T\times \C_z$.
Therefore it suffices to show that $\xi =0$ if $\xi \in \C^m$ satisfies
\begin{equation}
\label{eq:symbol_vanish_at_0}
\sigma(P_{v,\lambda})\big|_{q=0_\bSigma}= 0 \qquad
\text{for all $\lambda \in \LL \cap \hNE(\frX)$.}
\end{equation}

Suppose that $\xi\in \C^m$ satisfies~\eqref{eq:symbol_vanish_at_0}.
We first show that there exists a cone $\sigma \in \Sigma$
such that $\{b \in S\colon D_b(\xi) \neq 0\} = R(\bSigma) \cap \sigma$.
Let $\{b_1,\dots,b_s\}$ be the set of $b\in S$ such that $D_b(\xi) \neq 0$.
The relative interior of the convex hull of $\{b_1,\dots,b_s\}$ intersects
with the relative interior of some cone $\sigma$ of $\Sigma$.
Hence we get a relation of the form
\[
\sum_{i=1}^s c_i \overline{b}_i - \sum_{b\in R(\bSigma) \cap \sigma}
f_b \overline{b} = 0
\]
for some $c_i> 0$ satisfying $\sum\limits_{i=1}^s c_i =1$
and some $f_b>0$.
The convexity of $\Delta$ together with \eqref{eq:S_is_in_Delta}
implies that $\sum\limits_{b\in R(\bSigma) \cap
\sigma} f_b \le 1$.
We may further assume that $c_i$ and $f_b$ are rational
numbers. Then for some positive integer $l$,
\[
\lambda:= \sum_{i=1}^s l c_i e_{b_i} -
\sum_{b \in R(\bSigma) \cap \sigma}
l f_b e_b \in \Z^S
\]
belong to $\LL \cap \hNE(\frX)$ (recall the definition
of $\hNE(\frX)$ around \eqref{eq:tC_bSigma_sigma}).
Note that
\[ \sum\limits_{b\in S} \lambda_b
= l \left(\sum\limits_{i=1}^s c_i - \sum\limits_{b\in R(\bSigma) \cap \sigma} f_b\right) \ge 0.\]
Therefore, if $\lambda \neq 0$, we have the relation
\[
0= \sigma(P_{v,\lambda})\big|_{q=0_\bSigma} =
\prod_{i=1}^s D_{b_i}(\xi)^{l c_i}.
\]
This contradicts the fact that $D_{b_1}(\xi) \neq 0,\dots,
D_{b_s}(\xi) \neq 0$. Hence $\lambda=0$ and
we conclude that $\{b_1,\dots,b_s\} = R(\bSigma)\cap \sigma$.

Now we have $D_b(\xi) = 0$ for all $b\notin R(\bSigma) \cap \sigma$.
Since $\{D_b\in\LL_\C^\star\colon b\notin R(\Sigma) \cap \sigma\}$ spans
$\LL^\star_\C$, we have $\xi=0$. The lemma is proved.
\end{proof}

\begin{Corollary}
There exists a Zariski-open neighbourhood $U$ of $0_\bSigma$
in $\Spec \C[\Laa(\bSigma)_+]$ such that
$\scrG|_{U\times \Lie \T\times \C_z}$ is a coherent
$\cO_{U\times \Lie\T \times \C_z}$-module.
\end{Corollary}

\subsection{Locally freeness and rank}
We complete the proof of Proposition \ref{prop:weakFano_Bri_locallyfree}.
Recall from Theorem \ref{thm:mirror_isom} that the completed
equivariant Brieskorn module is isomorphic to the quantum
D-module of $\frX$. Thus we have
\[
\sfG/\frakm_\bSigma \sfG \cong \widehat{\sfG}
/\frakm_\bSigma \widehat{\sfG} \cong H^*_{\CR,\T}(\frX)[z],
\]
where $\frakm_\bSigma \subset \C[\Laa(\bSigma)_+]$ denotes
the maximal ideal corresponding to $0_\bSigma$
and $\widehat{\sfG}$ denotes the $\frakm_\bSigma$-adic
completion as discussed
in Section~\ref{subsec:completion}. This implies that
the restriction $\scrG|_{\{0_\bSigma\} \times \Lie \T \times \C_z}$
is free of rank $\dim H^*_\CR(\frX)$.
Hence by coherence, $\scrG$ is generated by
$\dim H^*_\CR(\frX)$ many sections
in a neighbourhood of $\{0_\bSigma\} \times \Lie \T \times \C_z$.
On the other hand, the localization map ``$\Loc$'' appearing
in \cite[Definition 4.17]{CCIT:MS} gives $\dim H^*_\CR(\frX)$
many linearly independent solutions of $\scrG$ for a generic $(\chi,z)
\in \Lie \T \times \C_z$; the weak Fano condition ensures
the convergence of the power series solution $\Loc$.
This implies that $\scrG$ is locally
free of rank $\dim H_\CR^*(\frX)$
in a neighbourhood of $\{0_\bSigma\} \times
\Lie \T \times \C_z$.

To see that $\scrG$ is locally free on an open set of the
form $U \times \Lie \T \times \C_z$ (for some open
set $U \subset \Spec \C[\Laa(\bSigma)_+]$ containing $0_\bSigma$),
we use the grading operator \eqref{eq:grading_B}.
The grading operator makes $\scrG$ a $\C^\times$-equivariant sheaf;
the induced $\C^\times$-action on the base is given by
the same grading operator on $\C[\Laa(\bSigma)_+][z]$ and
the weight one $\C^\times$-action on $\Lie \T$.

\begin{Lemma}[cf.~{\cite[Section~3.1.4]{Iritani:Integral}}]
\label{lem:non-negative_grading}
When $\frX$ is weak Fano, $\C[\OO(\bSigma)_+]$ is non-negatively
graded with respect to the grading operator~\eqref{eq:grading_B}.
\end{Lemma}
\begin{proof}It suffices to show that $\sum\limits_{b\in S} \lambda_b \ge 0$ for any
$(\lambda,v) \in \OO(\bSigma)_+$. By the definition~\eqref{eq:O+Laa+} of~$\OO(\bSigma)_+$, it
suffices to show that $\sum\limits_{b\in S} \lambda_b \ge 0$ for all
maximal cones $\sigma\in \Sigma$ and
$\lambda \in \tC_{\bSigma,\sigma}$
(see~\eqref{eq:tC_bSigma_sigma}).
Define a linear function $h \colon \bN_\R\to \R$ by $h(\overline{b}) = 1$
for all $b\in R(\bSigma) \cap \sigma$.
Then the weak Fano condition (i.e., the convexity of $\Delta$) together with
\eqref{eq:S_is_in_Delta} implies that $h(\overline{b}) \le 1$
for all $b\in S$. Hence, for $\lambda \in \tC_{\bSigma,\sigma}$, we have
\[
0\le h(\beta(\lambda)) = \sum_{b\in R(\bSigma) \cap \sigma}
\lambda_b + \sum_{b\notin R(\bSigma)\cap \sigma}
\lambda_b h(\overline{b}) \le \sum_{b\in S} \lambda_b.
\]
This proves the lemma.
\end{proof}

In particular, $\C[\Laa(\bSigma)_+]$ is non-negatively graded.
Because the locus where $\scrG$ is locally free is preserved by
the $\C^\times$-action and contains a neighbourhood of
$(0_\bSigma,0,0) \in B$, it follows that $\scrG$ is locally free
on an open set of the form $U\times \Lie \T \times \C_z$.
The proof of Proposition \ref{prop:weakFano_Bri_locallyfree}
is now complete.

\begin{Remark}The generic rank of the GKZ system has been studied by many people,
notably by Gelfand--Kapranov--Zelevinsky~\cite{GKZ:hypergeom},
Adolphson \cite{Adolphson}, Matusevich--Miller--Walther~\cite{MMW}
and has been identified with the volume\footnote{The volume is normalized so that
the standard simplex has volume~1. When we allow $\bN$ to have torsions,
the generic rank is $|\bN_{\rm tor}|\times \vol(\Delta)$.}
of $\Delta$ (when $\chi$ is not special).
Over the open torus $(\C^\times)^m$ contained in
$\Spec \C[\Laa(\bSigma)_+]$,
the Brieskorn module in this paper corresponds to
the better-behaved GKZ system of Borisov--Horja
\cite{Borisov-Horja:bbGKZ}; they showed that the
generic rank of the better-behaved GKZ system
equals $\vol(\Delta)$ (independently of $\chi$).
\end{Remark}

\section{Proof of Lemma \ref{lem:M-tame}} \label{append:M-tame}

We only prove Part (1) of the lemma; the argument for Part~(2)
is the same.
It is easy to see that $F_{q,0}^{-1}(\bu)$ and $\partial A_{q,0}(\eta)$
intersect transversally at $x\in F_{q,0}^{-1}(\bu)
\cap \partial A_{q,0}(\eta)$ if and only if $\grad F_{q,0}(x)$
and $\grad H(x)$ are linearly independent over $\C$,
where we set
\[
\grad f(x) = \left(\overline{x_i \parfrac{f}{x_i}}\right)_{i=1}^n.
\]
Suppose that the lemma is not true. Then we can find sequences
$(q(k), \bu(k))\in K$
and $x(k) \in \cYsm_{q(k)}$ such that the following holds:
\begin{itemize}\itemsep=0pt
\item $\eta_k = H(x(k)) \to \infty$ as $k\to \infty$;
\item $F_{q(k),0}(x(k)) = \bu(k)$;
\item $\grad F_{q(k),0}(x(k))$ and $\grad H(x(k))$ are linearly dependent
over $\C$, i.e., there exists $\alpha_k\in \C$ such that $\grad F_{q(k),0}(x(k)) =
\alpha_k \grad H(x(k))$.
\end{itemize}
Define an $\R^n$-valued function $\bv(x)$ by
\[
\bv(x) = \sum_{b\in S_-} |x^b|^2 \overline{b}.
\]
Then we have
\[
\grad H(x) = \frac{\bv(x)}{2 H(x)}.
\]
Writing $\bv \cdot \bw = \sum\limits_{i=1}^n \bv_i \bw_i$ for
the $\C$-bilinear scalar product, the third condition above can be written as
\[
\grad F_{q(k),0}(x(k)) = \frac{\grad F_{q(k),0}(x(k)) \cdot \bv(x(k))}
{\|\bv(x(k))\|^2}
\bv(x(k)).
\]
By the curve selection lemma in \cite[Lemma 2]{Nemethi-Zaharia:Milnor},
we can find a real analytic curve $(0,\epsilon) \ni s\mapsto (q(s), \bu(s), x(s))$
admitting a Laurent expansion at $s=0$
such that for $0<s<\epsilon$,
$(q(s),\bu(s)) \in K$,
$F_{q(s),0}(x(s)) = \bu(s)$,
\begin{equation}
\label{eq:linear_dependence_s}
\grad F_{q(s),0}(x(s)) = \frac{\grad F_{q(s),0}(x(s)) \cdot \bv(x(s))}
{\|\bv(x(s)\|^2} \bv(x(s)),
\end{equation}
and $\lim\limits_{s\to +0} H(x(s)) =\infty$.
Since $K$ is compact, $(q(0),\bu(0)) :=
\lim\limits_{s\to +0} (q(s),\bu(s))$ exists in $K$.
We write $x_i(s) = s^{\xi_i} a_i (1+ O(s))$ with $a_i\neq 0$
and $\xi_i \in \Z$. Then the leading term for $H(x(s))^2$ is
\begin{equation}\label{eq:H_leading}
H(x(s))^2 = \left( \sum_{b\in \sigma}
|a|^{2b}\right) s^{-2m} + \text{higher order terms},
\end{equation}
where we set $\sigma := \{b\in S_-\colon \xi \cdot \overline{b} = -m\}$
with $\xi = (\xi_1,\dots,\xi_n)$ and
$-m := \min\{\xi\cdot \overline{b} \colon b\in S_-\}$, and
$|a|^{2b} =\prod\limits_{i=1}^n |a_i|^{2b_i}$.
Note that elements of
$\sigma$ spans a face of the polytope~$\Delta_-$.
Since $\lim\limits_{s\to +0} H(x(s)) =\infty$, we have $m>0$.
We calculate:
\begin{gather*}
\bv(s) = \left( \sum_{b\in \sigma} |a|^{2b} \overline{b}\right) s^{-2m} +
\text{higher order terms}, \\
\bv(s) \cdot \frac{d \log x(s)}{ds}
 = - m \left(\sum_{b\in \sigma} |a|^{2b}\right)
s^{-2m-1} + \text{higher order terms}, \\
\overline{\grad F_{q(s),0}(x(s))} = \sum_{b\in S_-}
q(s)^{\ell_b} x(s)^b \overline{b}
= \left( \sum_{b\in \sigma} a^b q(0)^{\ell_b} \overline{b} \right) s^{-m}
+ \text{higher order terms}.
\end{gather*}
By differentiating the equality
$\bu(s) = F_{q(s),0}(x(s))$ in $s$, we get
\begin{gather}
\frac{du(s)}{ds} - \sum_{i=1}^n \parfrac{F_{q(s),0}(x(s))}{q^i}
\frac{dq^i(s)}{ds}
= \overline{\grad F_{q(s),0}(x(s))} \cdot \frac{d\log x(s)}{ds}\nonumber \\
\qquad{}= \left(\overline{\grad F_{q(s),0}(x(s))} \cdot
\frac{\bv(x(s))}{\|\bv(x(s)\|^2} \right)
\cdot
\left(\bv(x(s)) \cdot \frac{d\log x(s)}{ds}\right),\label{eq:diff_in_s}
\end{gather}
where $\{q^i\}$ denotes a local co-ordinate system on $\cMsmtimes$
and we used~\eqref{eq:linear_dependence_s} in the second line.
We compare the leading order terms as $s\to +0$.
By the above calculation, the right-hand side of~\eqref{eq:diff_in_s}
has the leading term
\begin{equation}
\label{eq:leading_lhs}
c
\left( \sum_{b\in \sigma} a^b q(0)^{\ell_b} \overline{b} \right)
\cdot
\left( \sum_{b\in \sigma}|a|^{2b} \overline{b}\right) s^{-m-1}
\end{equation}
with
$c = - m \sum\limits_{b\in \sigma}|a|^{2b} \neq 0$.
On the other hand, because $q(s)$, $\bu(s)$ are regular at $s=0$ and
\[
\left |\parfrac{F_{q,0}(x(s))}{q^i}\right |
= \left |\sum_{b\in S} \parfrac{q(s)^{\ell_b}}{q^i} x^b\right |
\le C \cdot H(x(s)) \qquad \text{for some $C>0$},
\]
the left-hand side of \eqref{eq:diff_in_s} has poles of order
at most $m$ by \eqref{eq:H_leading}.
Hence the quantity \eqref{eq:leading_lhs} must
vanish. This together with \eqref{eq:linear_dependence_s} implies:
\begin{gather*}
\overline{\grad F_{q(s),0}(x(s))} =
\frac{\overline{\grad F_{q(s),0}(x(s))} \cdot \bv(x(s))}
{\|\bv(x(s)\|^2} \bv(x(s)) =O\big(s^{-m+1}\big).
\end{gather*}
On the other hand, Proposition~\ref{prop:strong_tame}
and~\eqref{eq:H_leading} give an estimate of the form
\[
\|\grad F_{q(s),0}(x(s))\| \ge \epsilon_1 H(x(s)) \ge \epsilon_2 s^{-m}
\]
for sufficiently small $s>0$ (for some $\epsilon_1, \epsilon_2>0$).
These two estimates contradict each other. Lemma~\ref{lem:M-tame} is proved.

\subsection*{Acknowledgements}
I would like to thank Pedro Acosta, Arend Bayer, Andrea Brini, Tom Coates,
Alessio Corti,
Sergey Galkin,
Vasily Golyshev,
Eduardo Gonz\'alez,
Claus Hertling,
Yuki Hirano,
Paul Horja,
Yunfeng Jiang,
Yuan-Pin Lee,
Chiu-Chu Melissa Liu,
Wanmin Liu,
Thomas Reichelt,
Yongbin Ruan,
Kyoji Saito,
Fumihiko Sanda,
Christian Sevenheck,
Yota Shamoto,
Mark Shoemaker,
Takuro Mochizuki,
Mauricio Romo,
Hsian-Hua Tseng,
Chris Woodward
for many insightful discussions and explanations.
I also thank the anonymous referees for their careful reading and helpful comments.
This work is supported by
EPSRC grant EP/E022162/1,
and JSPS Kakenhi Grants Number
22740042, 23224002, 24224001, 25400069, 26610008,
16K05127, 16H06335, 16H06337 and 17H06127.
Part of this work was done while I was in residence at
the Mathematical Sciences Research Institute
in Berkeley, California, during the Spring semester of 2018
and the stay was supported
by the National Science Foundation under
Grant No.~DMS-1440140. I~thank the organizers and
participants of the programme for many stimulating discussions.

\LastPageEnding


\addcontentsline{toc}{section}{References}
\begin{thebibliography}{100}
\footnotesize\itemsep=0pt

\bibitem{AGV1}
Abramovich D., Graber T., Vistoli A., Algebraic orbifold quantum products, in
 Orbifolds in Mathematics and Physics ({M}adison, {WI}, 2001),
 \textit{Contemp. Math.}, Vol.~310, \href{https://doi.org/10.1090/conm/310/05397}{Amer. Math. Soc.}, Providence, RI, 2002, 1--24, \href{https://arxiv.org/abs/math.AG/0112004}{arXiv:math.AG/0112004}.

\bibitem{AGV:GW}
Abramovich D., Graber T., Vistoli A., Gromov--{W}itten theory of
 {D}eligne--{M}umford stacks, \href{https://doi.org/10.1353/ajm.0.0017}{\textit{Amer.~J. Math.}} \textbf{130} (2008), 1337--1398, \href{https://arxiv.org/abs/math.AG/0603151}{arXiv:math.AG/0603151}.

\bibitem{Acosta:asymp}
Acosta P., Asymptotic expansion and the {LG}/({F}ano, general type)
 correspondence, \href{https://arxiv.org/abs/1411.4162}{arXiv:1411.4162}.

\bibitem{Acosta-Shoemaker:blowup_LG}
Acosta P., Shoemaker M., Quantum cohomology of toric blowups and
 {L}andau--{G}inzburg correspondences, \href{https://doi.org/10.14231/AG-2018-008}{\textit{Algebr. Geom.}} \textbf{5}
 (2018), 239--263, \href{https://arxiv.org/abs/1504.04396}{arXiv:1504.04396}.

\bibitem{Acosta-Shoemaker:toric_birational}
Acosta P., Shoemaker M., Gromov--{W}itten theory of toric birational
 transformations, \href{https://doi.org/10.1093/imrn/rnz001}{\textit{Int. Math. Res. Not.}}, {t}o appear,
 \href{https://arxiv.org/abs/1604.03491}{arXiv:1604.03491}.

\bibitem{Adolphson}
Adolphson A., Hypergeometric functions and rings generated by monomials,
 \href{https://doi.org/10.1215/S0012-7094-94-07313-4}{\textit{Duke Math.~J.}} \textbf{73} (1994), 269--290.

\bibitem{AGV:Singularity_II}
Arnol'd V.I., Guse\u{\i}n-Zade S.M., Varchenko A.N., Singularities of
 differentiable maps, {V}ol.~{II}, Monodromy and asymptotics of integrals,
 \textit{Monographs in Mathematics}, Vol.~83, \href{https://doi.org/10.1007/978-1-4612-3940-6}{Birkh\"{a}user Boston}, Inc.,
 Boston, MA, 1988.

\bibitem{BDFKK:Mori_nonFano}
Ballard M., Diemer C., Favero D., Katzarkov L., Kerr G., The {M}ori program and
 non-{F}ano toric homological mirror symmetry, \href{https://doi.org/10.1090/S0002-9947-2015-06541-6}{\textit{Trans. Amer. Math.
 Soc.}} \textbf{367} (2015), 8933--8974, \href{https://arxiv.org/abs/1302.0803}{arXiv:1302.0803}.

\bibitem{BFK:VGIT}
Ballard M., Favero D., Katzarkov L., Variation of geometric invariant theory
 quotients and derived categories, \href{https://doi.org/10.1515/crelle-2015-0096}{\textit{J.~Reine Angew. Math.}} \textbf{746}
 (2019), 235--303, \href{https://arxiv.org/abs/1203.6643}{arXiv:1203.6643}.

\bibitem{BJL79}
Balser W., Jurkat W.B., Lutz D.A., Birkhoff invariants and {S}tokes'
 multipliers for meromorphic linear differential equations, \href{https://doi.org/10.1016/0022-247X(79)90217-8}{\textit{J.~Math.
 Anal. Appl.}} \textbf{71} (1979), 48--94.

\bibitem{Barannikov:projective}
Barannikov S., Semi-infinite {H}odge structure and mirror symmetry for
 projective spaces, \href{https://arxiv.org/abs/math.AG/0010157}{arXiv:math.AG/0010157}.

\bibitem{Bayer:semisimple}
Bayer A., Semisimple quantum cohomology and blowups, \href{https://doi.org/10.1155/S1073792804140907}{\textit{Int. Math. Res.
 Not.}} \textbf{2004} (2004), 2069--2083, \href{https://arxiv.org/abs/math.AG/0403260}{arXiv:math.AG/0403260}.

\bibitem{Bonda-Orlov:SOD}
Bondal A., Orlov D., Semiorthogonal decomposition for algebraic varieties,
 \href{https://arxiv.org/abs/alg-geom/9506012}{arXiv:alg-geom/9506012}.

\bibitem{BCS}
Borisov L.A., Chen L., Smith G.G., The orbifold {C}how ring of toric
 {D}eligne--{M}umford stacks, \href{https://doi.org/10.1090/S0894-0347-04-00471-0}{\textit{J.~Amer. Math. Soc.}} \textbf{18} (2005),
 193--215, \href{https://arxiv.org/abs/math.AG/0309229}{arXiv:math.AG/0309229}.

\bibitem{Borisov-Horja:FM}
Borisov L.A., Horja R.P., Mellin--{B}arnes integrals as {F}ourier--{M}ukai
 transforms, \href{https://doi.org/10.1016/j.aim.2006.01.011}{\textit{Adv. Math.}} \textbf{207} (2006), 876--927,
 \href{https://arxiv.org/abs/math.AG/0510486}{arXiv:math.AG/0510486}.

\bibitem{Borisov-Horja:K}
Borisov L.A., Horja R.P., On the {$K$}-theory of smooth toric {DM} stacks, in
 Snowbird Lectures on String Geometry, \textit{Contemp. Math.}, Vol.~401,
 \href{https://doi.org/10.1090/conm/401/07551}{Amer. Math. Soc.}, Providence, RI, 2006, 21--42, \href{https://arxiv.org/abs/math.AG/0503277}{arXiv:math.AG/0503277}.

\bibitem{Borisov-Horja:bbGKZ}
Borisov L.A., Horja R.P., On the better behaved version of the {GKZ}
 hypergeometric system, \href{https://doi.org/10.1007/s00208-013-0913-6}{\textit{Math. Ann.}} \textbf{357} (2013), 585--603,
 \href{https://arxiv.org/abs/1011.5720}{arXiv:1011.5720}.

\bibitem{Bridgeland--Toledano-Laredo:multilog}
Bridgeland T., Toledano~Laredo V., Stokes factors and multilogarithms,
 \href{https://doi.org/10.1515/crelle-2012-0046}{\textit{J.~Reine Angew. Math.}} \textbf{682} (2013), 89--128,
 \href{https://arxiv.org/abs/1006.4623}{arXiv:1006.4623}.

\bibitem{BCR:crepant_open}
Brini A., Cavalieri R., Ross D., Crepant resolutions and open strings,
 \href{https://doi.org/10.1515/crelle-2017-0011}{\textit{J.~Reine Angew. Math.}} \textbf{755} (2019), 191--245,
 \href{https://arxiv.org/abs/1309.4438}{arXiv:1309.4438}.

\bibitem{Bryan-Graber}
Bryan J., Graber T., The crepant resolution conjecture, in Algebraic
 Geometry~-- {S}eattle 2005, {P}art~1, \textit{Proc. Sympos. Pure Math.},
 Vol.~80, \href{https://doi.org/10.1090/pspum/080.1/2483931}{Amer. Math. Soc.}, Providence, RI, 2009, 23--42,
 \href{https://arxiv.org/abs/math.AG/0610129}{arXiv:math.AG/0610129}.

\bibitem{Cadman:tangency}
Cadman C., Using stacks to impose tangency conditions on curves,
 \href{https://doi.org/10.1353/ajm.2007.0007}{\textit{Amer.~J. Math.}} \textbf{129} (2007), 405--427,
 \href{https://arxiv.org/abs/math.AG/0312349}{arXiv:math.AG/0312349}.

\bibitem{Charest-Woodward:Floer_flip}
Charest F., Woodward C.T., Floer cohomology and flips, \href{https://arxiv.org/abs/1508.01573}{arXiv:1508.01573}.

\bibitem{Chen-Ruan:orbGW}
Chen W., Ruan Y., Orbifold {G}romov--{W}itten theory, in Orbifolds in
 Mathematics and Physics ({M}adison, {WI}, 2001), \textit{Contemp. Math.},
 Vol.~310, \href{https://doi.org/10.1090/conm/310/05398}{Amer. Math. Soc.}, Providence, RI, 2002, 25--85,
 \href{https://arxiv.org/abs/math.AG/0103156}{arXiv:math.AG/0103156}.

\bibitem{Chen-Ruan:new_coh}
Chen W., Ruan Y., A new cohomology theory of orbifold, \href{https://doi.org/10.1007/s00220-004-1089-4}{\textit{Comm. Math.
 Phys.}} \textbf{248} (2004), 1--31, \href{https://arxiv.org/abs/math.AG/0004129}{arXiv:math.AG/0004129}.

\bibitem{Clingempeel-LeFloch-Romo}
Clingempeel J., Le~Floch B., Romo M., Brane transport in anomalous (2,2) models
 and localization, \href{https://arxiv.org/abs/1811.12385}{arXiv:1811.12385}.

\bibitem{CCIT:mirrorthm}
Coates T., Corti A., Iritani H., Tseng H.-H., A mirror theorem for toric stacks,
 \href{https://doi.org/10.1112/S0010437X15007356}{\textit{Compos. Math.}} \textbf{151} (2015), 1878--1912, \href{https://arxiv.org/abs/1310.4163}{arXiv:1310.4163}.

\bibitem{CCIT:MS}
Coates T., Corti A., Iritani H., Tseng H.-H., Hodge-theoretic mirror symmetry
 for toric stacks, \href{https://doi.org/10.4310/jdg/1577502022}{\textit{J.~Diffe\-rential Geom.}} \textbf{114} (2020),
 41--115, \href{https://arxiv.org/abs/1606.07254}{arXiv:1606.07254}.

\bibitem{Coates-Givental}
Coates T., Givental A., Quantum {R}iemann--{R}och, {L}efschetz and {S}erre,
 \href{https://doi.org/10.4007/annals.2007.165.15}{\textit{Ann. of Math.}} \textbf{165} (2007), 15--53, \href{https://arxiv.org/abs/math.AG/0110142}{arXiv:math.AG/0110142}.

\bibitem{Coates-Iritani:convergence}
Coates T., Iritani H., On the convergence of {G}romov--{W}itten potentials and
 {G}ivental's formula, \href{https://doi.org/10.1307/mmj/1441116660}{\textit{Michigan Math.~J.}} \textbf{64} (2015),
 587--631, \href{https://arxiv.org/abs/1203.4193}{arXiv:1203.4193}.

\bibitem{Coates-Iritani:Fock}
Coates T., Iritani H., A {F}ock sheaf for {G}ivental quantization,
 \href{https://doi.org/10.1215/21562261-2017-0036}{\textit{Kyoto~J. Math.}} \textbf{58} (2018), 695--864, \href{https://arxiv.org/abs/1411.7039}{arXiv:1411.7039}.

\bibitem{CIJ}
Coates T., Iritani H., Jiang Y., The crepant transformation conjecture for
 toric complete intersections, \href{https://doi.org/10.1016/j.aim.2017.11.017}{\textit{Adv. Math.}} \textbf{329} (2018),
 1002--1087, \href{https://arxiv.org/abs/1410.0024}{arXiv:1410.0024}.

\bibitem{CIT:wallcrossing}
Coates T., Iritani H., Tseng H.-H., Wall-crossings in toric {G}romov--{W}itten
 theory. {I}.~{C}repant examples, \href{https://doi.org/10.2140/gt.2009.13.2675}{\textit{Geom. Topol.}} \textbf{13} (2009),
 2675--2744, \href{https://arxiv.org/abs/math.AG/0611550}{arXiv:math.AG/0611550}.

\bibitem{Cotti-Dubrovin-Guzzetti}
Cotti G., Dubrovin B., Guzzetti D., Isomonodromy deformations at an irregular
 singularity with coalescing eigenvalues, \href{https://doi.org/10.1215/00127094-2018-0059}{\textit{Duke Math.~J.}} \textbf{168}
 (2019), 967--1108, \href{https://arxiv.org/abs/1706.04808}{arXiv:1706.04808}.

\bibitem{CLS}
Cox D.A., Little J.B., Schenck H.K., Toric varieties, \textit{Graduate Studies
 in Mathematics}, Vol.~124, \href{https://doi.org/10.1090/gsm/124}{Amer. Math. Soc.}, Providence, RI, 2011.

\bibitem{DKK:compactification_LG}
Diemer C., Katzarkov L., Kerr G., Compactifications of spaces of
 {L}andau--{G}inzburg models, \href{https://doi.org/10.4213/im8019}{\textit{Izv. Math.}} \textbf{77} (2013),
 487--508, \href{https://arxiv.org/abs/1207.0042}{arXiv:1207.0042}.

\bibitem{DKK:symplectomorphism}
Diemer C., Katzarkov L., Kerr G., Symplectomorphism group relations and
 degenerations of {L}andau--{G}inzburg models, \href{https://doi.org/10.4171/JEMS/640}{\textit{J.~Eur. Math. Soc.}}
 \textbf{18} (2016), 2167--2271, \href{https://arxiv.org/abs/1204.2233}{arXiv:1204.2233}.

\bibitem{Douai-Sabbah:I}
Douai A., Sabbah C., Gauss--{M}anin systems, {B}rieskorn lattices and
 {F}robenius structures.~{I}, \href{https://doi.org/10.5802/aif.1974}{\textit{Ann. Inst. Fourier (Grenoble)}}
 \textbf{53} (2003), 1055--1116, \href{https://arxiv.org/abs/math.AG/0211352}{arXiv:math.AG/0211352}.

\bibitem{Douai-Sabbah:II}
Douai A., Sabbah C., Gauss--{M}anin systems, {B}rieskorn lattices and
 {F}robenius structures.~{II}, in Frobenius Manifolds, \textit{Aspects Math.},
 Vol. E36, Friedr. Vieweg, Wiesbaden, 2004, 1--18, \href{https://arxiv.org/abs/math.AG/0211353}{arXiv:math.AG/0211353}.

\bibitem{Dubrovin:ICM}
Dubrovin B., Geometry and analytic theory of {F}robenius manifolds,
 \textit{Doc. Math.} (1998), extra Vol.~II, 315--326,
 \href{https://arxiv.org/abs/math.AG/9807034}{arXiv:math.AG/9807034}.

\bibitem{Dubrovin:Painleve}
Dubrovin B., Painlev\'{e} transcendents in two-dimensional topological field
 theory, in The {P}ainlev\'{e} Property, \textit{CRM Ser. Math. Phys.}, \href{https://doi.org/10.1007/978-1-4612-1532-5_6}{Springer}, New
 York, 1999, 287--412, \href{https://arxiv.org/abs/math.AG/9803107}{arXiv:math.AG/9803107}.

\bibitem{Fang-Zhou}
Fang B., Zhou P., Gamma {II} for toric varieties from integrals on {T}-dual
 branes and homological mirror symmetry, \href{https://arxiv.org/abs/1903.05300}{arXiv:1903.05300}.

\bibitem{FMN}
Fantechi B., Mann E., Nironi F., Smooth toric {D}eligne--{M}umford stacks,
 \href{https://doi.org/10.1515/CRELLE.2010.084}{\textit{J.~Reine Angew. Math.}} \textbf{648} (2010), 201--244,
 \href{https://arxiv.org/abs/0708.1254}{arXiv:0708.1254}.

\bibitem{Galkin:conifold}
Galkin S., The conifold point, \href{https://arxiv.org/abs/1404.7388}{arXiv:1404.7388}.

\bibitem{GGI:gammagrass}
Galkin S., Golyshev V., Iritani H., Gamma classes and quantum cohomology of
 {F}ano manifolds: gamma conjectures, \href{https://doi.org/10.1215/00127094-3476593}{\textit{Duke Math.~J.}} \textbf{165}
 (2016), 2005--2077, \href{https://arxiv.org/abs/1404.6407}{arXiv:1404.6407}.

\bibitem{GKZ:discriminants}
Gel'fand I.M., Kapranov M.M., Zelevinsky A.V., Discriminants, resultants, and
 multidimensional determinants, Mathematics: Theory \& Applications,
 \href{https://doi.org/10.1007/978-0-8176-4771-1}{Birkh\"{a}user Boston, Inc.}, Boston, MA, 1994.

\bibitem{GKZ:hypergeom}
Gel'fand I.M., Zelevinskii A.V., Kapranov M.M., Hypergeometric functions and
 toric varieties, \href{https://doi.org/10.1007/BF01078777}{\textit{Funct. Anal. Appl.}} \textbf{23} (1989), 94--106.

\bibitem{Givental:ICM}
Givental A.B., Homological geometry and mirror symmetry, in Proceedings of the
 {I}nternational {C}ongress of {M}athematicians, {V}ols.~1,~2 ({Z}\"{u}rich,
 1994), \href{https://doi.org/10.1007/978-3-0348-9078-6_40}{Birkh\"{a}user}, Basel, 1995, 472--480.

\bibitem{Givental:elliptic}
Givental A.B., Elliptic {G}romov--{W}itten invariants and the generalized mirror
 conjecture, in Integrable Systems and Algebraic Geometry ({K}obe/{K}yoto,
 1997), World Sci. Publ., River Edge, NJ, 1998, 107--155,
 \href{https://arxiv.org/abs/math.AG/9803053}{arXiv:math.AG/9803053}.

\bibitem{Givental:quadratic}
Givental A.B., Gromov--{W}itten invariants and quantization of quadratic
 {H}amiltonians, \href{https://doi.org/10.17323/1609-4514-2001-1-4-551-568}{\textit{Mosc. Math.~J.}}
 \textbf{1} (2001), 551--568, \href{https://arxiv.org/abs/math.AG/0108100}{arXiv:math.AG/0108100}.

\bibitem{Givental:higher_genus}
Givental A.B., Semisimple {F}robenius structures at higher genus,
 \href{https://doi.org/10.1155/S1073792801000605}{\textit{Int. Math. Res. Not.}} \textbf{2001} (2001), 1265--1286,
 \href{https://arxiv.org/abs/math.AG/0008067}{arXiv:math.AG/0008067}.

\bibitem{Gonzalez-Woodward:tmmp}
Gonz\'{a}lez E., Woodward C.T., Quantum cohomology and toric minimal model
 programs, \href{https://doi.org/10.1016/j.aim.2019.07.004}{\textit{Adv. Math.}} \textbf{353} (2019), 591--646,
 \href{https://arxiv.org/abs/1207.3253}{arXiv:1207.3253}.

\bibitem{Gorchinskiy-Orlov:geometric_phantom}
Gorchinskiy S., Orlov D., Geometric phantom categories, \href{https://doi.org/10.1007/s10240-013-0050-5}{\textit{Publ. Math.
 Inst. Hautes \'Etudes Sci.}} \textbf{117} (2013), 329--349,
 \href{https://arxiv.org/abs/1209.6183}{arXiv:1209.6183}.

\bibitem{Graber-Pandharipande}
Graber T., Pandharipande R., Localization of virtual classes, \href{https://doi.org/10.1007/s002220050293}{\textit{Invent.
 Math.}} \textbf{135} (1999), 487--518, \href{https://arxiv.org/abs/alg-geom/9708001}{arXiv:alg-geom/9708001}.

\bibitem{Gross:tropical_book}
Gross M., Tropical geometry and mirror symmetry, \textit{CBMS Regional
 Conference Series in Mathematics}, Vol.~114, \href{https://doi.org/10.1090/cbms/114}{Amer. Math. Soc.}, Providence, RI, 2011.

\bibitem{Halpern-Leistner:GIT}
Halpern-Leistner D., The derived category of a {GIT} quotient, \href{https://doi.org/10.1090/S0894-0347-2014-00815-8}{\textit{J.~Amer.
 Math. Soc.}} \textbf{28} (2015), 871--912, \href{https://arxiv.org/abs/1203.0276}{arXiv:1203.0276}.

\bibitem{HHP}
Herbst M., Hori K., Page D., Phases of {$\mathcal{N}=2$} theories in $1+1$
 dimensions with boundary, \href{https://arxiv.org/abs/0803.2045}{arXiv:0803.2045}.

\bibitem{Hertling-Sevenheck:nilpotent}
Hertling C., Sevenheck C., Nilpotent orbits of a generalization of {H}odge
 structures, \href{https://doi.org/10.1515/CRELLE.2007.060}{\textit{J.~Reine Angew. Math.}} \textbf{609} (2007), 23--80,
 \href{https://arxiv.org/abs/math.AG/0603564}{arXiv:math.AG/0603564}.

\bibitem{Hori-Vafa}
Hori K., Vafa C., Mirror symmetry, \href{https://arxiv.org/abs/hep-th/0002222}{arXiv:hep-th/0002222}.

\bibitem{Iritani:coLef}
Iritani H., Convergence of quantum cohomology by quantum {L}efschetz,
 \href{https://doi.org/10.1515/CRELLE.2007.067}{\textit{J.~Reine Angew. Math.}} \textbf{610} (2007), 29--69, \href{https://arxiv.org/abs/math.DG/0506236}{arXiv:math.DG/0506236}.

\bibitem{Iritani:KIAS}
Iritani H., Fourier--{M}ukai transformation and toric quantum cohomology,
 {T}alk at Korean Institute for Advanced Study on June~10, 2008, as part of
 the Workshop on Gromov--Witten Theory and Related Topics, June 09--13, 2008,
 \url{http://workshop.kias.re.kr/grw/}.

\bibitem{Iritani:genmir}
Iritani H., Quantum {$D$}-modules and generalized mirror transformations,
 \href{https://doi.org/10.1016/j.top.2007.07.001}{\textit{Topology}} \textbf{47} (2008), 225--276, \href{https://arxiv.org/abs/math.DG/0411111}{arXiv:math.DG/0411111}.

\bibitem{Iritani:Integral}
Iritani H., An integral structure in quantum cohomology and mirror symmetry for
 toric orbifolds, \href{https://doi.org/10.1016/j.aim.2009.05.016}{\textit{Adv. Math.}} \textbf{222} (2009), 1016--1079, \href{https://arxiv.org/abs/0903.1463}{arXiv:0903.1463}.

\bibitem{Iritani:MIT}
Iritani H., Wall-crossings in toric {G}romov--{W}itten theory {III}, {T}alk at
 Massachusetts Institute of Technology on June~25, 2009, as part of the
 Workshop on Mirror Symmetry and Related Topics, June 22--26, 2009, Notes and
 slides available at \url{https://math.berkeley.edu/~auroux/frg/mit09-notes/}.

\bibitem{Iritani:periods}
Iritani H., Quantum cohomology and periods, \href{https://doi.org/10.5802/aif.2798}{\textit{Ann. Inst. Fourier (Grenoble)}} \textbf{61} (2011), 2909--2958, \href{https://arxiv.org/abs/1101.4512}{arXiv:1101.4512}.

\bibitem{Iritani:shift_mirror}
Iritani H., A mirror construction for the big equivariant quantum cohomology of
 toric manifolds, \href{https://doi.org/10.1007/s00208-016-1437-7}{\textit{Math. Ann.}} \textbf{368} (2017), 279--316,
 \href{https://arxiv.org/abs/1503.02919}{arXiv:1503.02919}.

\bibitem{Iwanari1}
Iwanari I., The category of toric stacks, \href{https://doi.org/10.1112/S0010437X09003911}{\textit{Compos. Math.}} \textbf{145}
 (2009), 718--746, \href{https://arxiv.org/abs/math.AG/0610548}{arXiv:math.AG/0610548}.

\bibitem{Iwanari2}
Iwanari I., Logarithmic geometry, minimal free resolutions and toric algebraic
 stacks, \href{https://doi.org/10.2977/prims/1260476654}{\textit{Publ. Res. Inst. Math. Sci.}} \textbf{45} (2009), 1095--1140,
 \href{https://arxiv.org/abs/0707.2568}{arXiv:0707.2568}.

\bibitem{Jiang}
Jiang Y., The orbifold cohomology ring of simplicial toric stack bundles,
 \href{http://projecteuclid.org/euclid.ijm/1248355346}{\textit{Illinois~J. Math.}} \textbf{52} (2008), 493--514, \href{https://arxiv.org/abs/math.AG/0504563}{arXiv:math.AG/0504563}.

\bibitem{Jinzenji:genmir}
Jinzenji M., Coordinate change of {G}auss--{M}anin system and generalized
 mirror transformation, \href{https://doi.org/10.1142/S0217751X05020641}{\textit{Internat.~J. Modern Phys.~A}} \textbf{20}
 (2005), 2131--2156, \href{https://arxiv.org/abs/math.AG/0310212}{arXiv:math.AG/0310212}.

\bibitem{KKP:Hodge}
Katzarkov L., Kontsevich M., Pantev T., Hodge theoretic aspects of mirror
 symmetry, in From {H}odge theory to integrability and {TQFT} $tt^*$-geometry,
 \textit{Proc. Sympos. Pure Math.}, Vol.~78, \href{https://doi.org/10.1090/pspum/078/2483750}{Amer. Math. Soc.}, Providence, RI, 2008, 87--174, \href{https://arxiv.org/abs/0806.0107}{arXiv:0806.0107}.

\bibitem{Kawamata:log_crepant}
Kawamata Y., Log crepant birational maps and derived categories,
 \textit{J.~Math. Sci. Univ. Tokyo} \textbf{12} (2005), 211--231,
 \href{https://arxiv.org/abs/math.AG/0311139}{arXiv:math.AG/0311139}.

\bibitem{Kawamata:Dcat_birat}
Kawamata Y., Derived categories and birational geometry, in Algebraic
 Geometry~-- {S}eattle 2005, {P}art~2, \textit{Proc. Sympos. Pure Math.},
 Vol.~80, \href{https://doi.org/10.1090/pspum/080.2/2483950}{Amer. Math. Soc.}, Providence, RI, 2009, 655--665, \href{https://arxiv.org/abs/0804.3150}{arXiv:0804.3150}.

\bibitem{Kerr:weighted}
Kerr G., Weighted blowups and mirror symmetry for toric surfaces, \href{https://doi.org/10.1016/j.aim.2008.04.005}{\textit{Adv.
 Math.}} \textbf{219} (2008), 199--250, \href{https://arxiv.org/abs/math.AG/0609162}{arXiv:math.AG/0609162}.

\bibitem{Kontsevich:HSE2019}
Kontsevich M., Birational invariants from quantum cohomology, {T}alk at Higher
 School of Economics on May~27, 2019, as part of Homological Mirror Symmetry
 at HSE, May~27 -- June~1, 2019,
 \url{https://euro-math-soc.eu/event/mon-27-may-19-0000/homological-mirror-symmetry-hse}.

\bibitem{Kouchnirenko:Newton}
Kouchnirenko A.G., Poly\`edres de {N}ewton et nombres de {M}ilnor,
 \href{https://doi.org/10.1007/BF01389769}{\textit{Invent. Math.}} \textbf{32} (1976), 1--31.

\bibitem{Kuznetsov:HH_SOD}
Kuznetsov A., Hochschild homology and semiorthogonal decompositions,
 \href{https://arxiv.org/abs/0904.4330}{arXiv:0904.4330}.

\bibitem{Lee-Lin-Wang:flops}
Lee Y.-P., Lin H.-W., Wang C.-L., Flops, motives, and invariance of quantum rings,
 \href{https://doi.org/10.4007/annals.2010.172.243}{\textit{Ann. of Math.}} \textbf{172} (2010), 243--290,
 \href{https://arxiv.org/abs/math.AG/0608370}{arXiv:math.AG/0608370}.

\bibitem{Lee-Lin-Wang:StringMath15}
Lee Y.-P., Lin H.-W., Wang C.-L., Quantum cohomology under birational maps and
 transitions, in String-{M}ath 2015, \textit{Proc. Sympos. Pure Math.},
 Vol.~96, Amer. Math. Soc., Providence, RI, 2017, 149--168,
 \href{https://arxiv.org/abs/1705.04799}{arXiv:1705.04799}.

\bibitem{Lee-Lin-Wang:flips_I}
Lee Y.-P., Lin H.-W., Wang C.-L., Quantum flips~{I}: local model,
 \href{https://arxiv.org/abs/1912.03012}{arXiv:1912.03012}.

\bibitem{AMLi-Ruan}
Li A.-M., Ruan Y., Symplectic surgery and {G}romov--{W}itten invariants of
 {C}alabi--{Y}au 3-folds, \href{https://doi.org/10.1007/s002220100146}{\textit{Invent. Math.}} \textbf{145} (2001),
 151--218, \href{https://arxiv.org/abs/math.AG/9803036}{arXiv:math.AG/9803036}.

\bibitem{Mann-Reichelt}
Mann E., Reichelt T., Logarithmic degenerations of {L}andau--{G}inzburg models
 for toric orbifolds and global $tt^*$ geometry, \href{https://arxiv.org/abs/1605.08937}{arXiv:1605.08937}.

\bibitem{Matsumura:comm_ring_theory}
Matsumura H., Commutative ring theory, 2nd~ed., \textit{Cambridge Studies in Advanced
 Mathematics}, Vol.~8, \href{https://doi.org/10.1017/CBO9781139171762}{Cambridge University Press}, Cambridge, 1989.

\bibitem{MMW}
Matusevich L.F., Miller E., Walther U., Homological methods for hypergeometric
 families, \href{https://doi.org/10.1090/S0894-0347-05-00488-1}{\textit{J.~Amer. Math. Soc.}} \textbf{18} (2005), 919--941, \href{https://arxiv.org/abs/math.AG/0406383}{arXiv:math.AG/0406383}.

\bibitem{Milnor:singular}
Milnor J., Singular points of complex hypersurfaces, \textit{Annals of
 Mathematics Studies}, Vol.~61, Princeton University Press, Princeton, N.J.,
 University of Tokyo Press, Tokyo, 1968.

\bibitem{Nemethi-Zaharia:Milnor}
N\'{e}methi A., Zaharia A., Milnor fibration at infinity, \href{https://doi.org/10.1016/0019-3577(92)90039-N}{\textit{Indag. Math.
 (N.S.)}} \textbf{3} (1992), 323--335.

\bibitem{Oda-Park}
Oda T., Park H.S., Linear {G}ale transforms and
 {G}el'fand--{K}apranov--{Z}elevinskij decompositions, \href{https://doi.org/10.2748/tmj/1178227461}{\textit{Tohoku
 Math.~J.}} \textbf{43} (1991), 375--399.

\bibitem{Orlov:proj}
Orlov D.O., Projective bundles, monoidal transformations, and derived
 categories of coherent sheaves, \href{https://doi.org/10.1070/IM1993v041n01ABEH002182}{\textit{Izv. Math.}} \textbf{41} (1993), 133--141.

\bibitem{PPZ:relations}
Pandharipande R., Pixton A., Zvonkine D., Relations on {$\overline{\mathcal
 M}_{g,n}$} via {$3$}-spin structures, \href{https://doi.org/10.1090/S0894-0347-2014-00808-0}{\textit{J.~Amer. Math. Soc.}}
 \textbf{28} (2015), 279--309, \href{https://arxiv.org/abs/1303.1043}{arXiv:1303.1043}.

\bibitem{Parusinski}
Parusi\'{n}ski A., On the bifurcation set of complex polynomial with isolated
 singularities at infinity, \textit{Compositio Math.} \textbf{97} (1995),
 369--384.

\bibitem{Pham:Lefschetz}
Pham F., La descente des cols par les onglets de {L}efschetz, avec vues sur
 {G}auss--{M}anin, \textit{Ast\'{e}risque} \textbf{130} (1985), 11--47.

\bibitem{Reichelt-Sevenheck:logFrob}
Reichelt T., Sevenheck C., Logarithmic {F}robenius manifolds, hypergeometric
 systems and quantum {$\mathcal D$}-modules, \href{https://doi.org/10.1090/S1056-3911-2014-00625-1}{\textit{J.~Algebraic Geom.}}
 \textbf{24} (2015), 201--281, \href{https://arxiv.org/abs/1010.2118}{arXiv:1010.2118}.

\bibitem{Ruan:crepant}
Ruan Y., The cohomology ring of crepant resolutions of orbifolds, in
 Gromov--{W}itten theory of spin curves and orbifolds, \href{https://doi.org/10.1090/conm/403/07597}{\textit{Contemp.
 Math.}}, Vol.~403, Amer. Math. Soc., Providence, RI, 2006, 117--126,
 \href{https://arxiv.org/abs/math.AG/0108195}{arXiv:math.AG/0108195}.

\bibitem{Sabbah:tame}
Sabbah C., Hypergeometric period for a tame polynomial, \href{https://doi.org/10.1016/S0764-4442(99)80254-7}{\textit{C.~R.~Acad.
 Sci. Paris S\'er.~I Math.}} \textbf{328} (1999), 603--608, {A}~longer version
 published in \textit{Port. Math.~(N.S.)} \textbf{63} (2006), 173--226,
 \href{https://arxiv.org/abs/math.AG/9805077}{arXiv:math.AG/9805077}.

\bibitem{Sabbah:twistor_D_modules}
Sabbah C., Polarizable twistor {$\mathcal D$}-modules, \textit{Ast\'{e}risque}
 \textbf{300} (2005), vi+208~pages, \href{https://arxiv.org/abs/math.AG/0503038}{arXiv:math.AG/0503038}.

\bibitem{Sabbah:isomonodromic}
Sabbah C., Isomonodromic deformations and {F}robenius manifolds. An
 introduction, \textit{Universitext}, \href{https://doi.org/10.1007/978-1-84800-054-4}{Springer-Verlag London, Ltd.}, London, 2007.

\bibitem{SaitoK:higherresidue}
Saito K., The higher residue pairings {$K_{F}^{(k)}$} for a family of
 hypersurface singular points, in Singularities, {P}art~2 ({A}rcata, {C}alif.,
 1981), \textit{Proc. Sympos. Pure Math.}, Vol.~40, Amer. Math. Soc.,
 Providence, RI, 1983, 441--463.

\bibitem{SaitoK:primitiveform}
Saito K., Period mapping associated to a primitive form, \href{https://doi.org/10.2977/prims/1195182028}{\textit{Publ. Res.
 Inst. Math. Sci.}} \textbf{19} (1983), 1231--1264.

\bibitem{SaitoM:Brieskorn}
Saito M., On the structure of {B}rieskorn lattice, \href{https://doi.org/10.5802/aif.1157}{\textit{Ann. Inst. Fourier (Grenoble)}} \textbf{39} (1989), 27--72.

\bibitem{Sanda:thesis}
Sanda F., Fukaya categories and blow-ups, Ph.D.~Thesis, {N}agoya University,
 2015, available at \url{https://doi.org/10.15083/00008467}.

\bibitem{Sanda:computation_QC}
Sanda F., Computation of quantum cohomology from {F}ukaya categories,
 \href{https://arxiv.org/abs/1712.03924}{arXiv:1712.03924}.

\bibitem{Sanda-Shamoto}
Sanda F., Shamoto Y., An analogue of {D}ubrovin's conjecture, \textit{Ann.
 Inst. Fourier (Grenoble)}, {t}o appear, \href{https://arxiv.org/abs/1705.05989}{arXiv:1705.05989}.

\bibitem{Shadrin:BCOV}
Shadrin S., B{COV} theory via {G}ivental group action on cohomological fields
 theories, \href{https://doi.org/10.17323/1609-4514-2009-9-2-411-429}{\textit{Mosc. Math.~J.}} \textbf{9} (2009), 411--429, \href{https://arxiv.org/abs/0810.0725}{arXiv:0810.0725}.

 \bibitem{stacks-project}
The Stacks Project Authors, {S}tacks {P}roject, 2020,
 \url{https://stacks.math.columbia.edu/}.

\bibitem{Teleman:semisimple}
Teleman C., The structure of 2{D} semi-simple field theories, \href{https://doi.org/10.1007/s00222-011-0352-5}{\textit{Invent.
 Math.}} \textbf{188} (2012), 525--588, \href{https://arxiv.org/abs/0712.0160}{arXiv:0712.0160}.

\bibitem{Tseng:QRR}
Tseng H.-H., Orbifold quantum {R}iemann--{R}och, {L}efschetz and {S}erre,
 \href{https://doi.org/10.2140/gt.2010.14.1}{\textit{Geom. Topol.}} \textbf{14} (2010), 1--81, \href{https://arxiv.org/abs/math.AG/0506111}{arXiv:math.AG/0506111}.

\bibitem{Tyomkin:tropical}
Tyomkin I., Tropical geometry and correspondence theorems via toric stacks,
 \href{https://doi.org/10.1007/s00208-011-0702-z}{\textit{Math. Ann.}} \textbf{353} (2012), 945--995, \href{https://arxiv.org/abs/1001.1554}{arXiv:1001.1554}.

\bibitem{Wasow:book}
Wasow W., Asymptotic expansions for ordinary differential equations,
 \textit{Pure and Applied Mathematics}, Vol.~14, Interscience Publishers John
 Wiley \& Sons, Inc., New York~-- London~-- Sydney, 1965.

\bibitem{Zariski-Samuel}
Zariski O., Samuel P., Commutative algebra, {V}ol.~{II}, \textit{Graduate Texts
 in Mathematics}, Vol.~29, Springer-Verlag, New York~-- Heidelberg, 1975.

\bibitem{Zong:Givental_GKM}
Zong Z., Equivariant {G}romov--{W}itten theory of {GKM} orbifolds,
 \href{https://arxiv.org/abs/1604.07270}{arXiv:1604.07270}.

\end{thebibliography}
\end{document}